\documentclass[reqno]{amsart}

\usepackage{mathabx}
\usepackage[utf8]{inputenc} 
\usepackage[english]{babel}
\usepackage{amstext}
\usepackage{euscript}
\usepackage{bbm}
\usepackage{mathrsfs}
\usepackage{enumerate}
\usepackage{graphicx}
\usepackage{caption}
\usepackage{amscd}
\usepackage{verbatim}
\usepackage{amssymb}
\usepackage{amsthm}
\usepackage{amsmath}
\usepackage{amstext}
\usepackage{amsfonts}
\usepackage{latexsym}
\usepackage{mathtools} 
\usepackage{enumitem} 
\usepackage{accents} 
\usepackage[foot]{amsaddr} 
\usepackage{textcomp} 
\usepackage{cite} 

\usepackage[usenames,dvipsnames]{xcolor} 
\definecolor{dkblue}{RGB}{1,31,91} 
\usepackage[colorlinks=true, pdfstartview=FitV, linkcolor=dkblue, citecolor=dkblue, urlcolor=dkblue]{hyperref}

\newcommand{\newK}{\tilde{K}}

\newcommand{\eqdef }{\overset{\mbox{\tiny{def}}}{=}}
\newcommand{\pv}{p}
\newcommand{\pZ}{\pv^0}
\newcommand{\qv}{q}
\newcommand{\qZ}{\qv^0}

\newcommand{\utilde}[1]{\underaccent{\tilde}{#1}}

\newcommand{\rth}{{\mathbb{R}^3}}
\newcommand{\rfo}{{\mathbb{R}^4}}

\newcommand{\singA}{a}
\newcommand{\singB}{b}
\newcommand{\singS}{\rho}

\newcommand{\zetaL}{\zeta_L}

\newcommand{\zetaTL}{\tilde{\zeta}_L}

\newcommand{\zetaLTone}{\tilde{\zeta}_{L}^1}

\newcommand{\zetaLTtwo}{\tilde{\zeta}_{L}^2}

\newcommand{\zetaTone}{\tilde{\zeta}_1}

\newcommand{\zetaTZ}{\tilde{\zeta}_{0}}
\newcommand{\zetaZ}{\zeta_{0}}

\newcommand{\zetaTZm}{\tilde{\zeta}_{0,m}}
\newcommand{\zetaZm}{{\zeta}_{0,m}}
\newcommand{\zetaTLm}{\tilde{\zeta}_{L,m}}
\newcommand{\zetaLm}{{\zeta}_{L,m}}

\newcommand{\secref}[1]{\S\ref{#1}}

\newcommand{\dbar}[1]{\bar{\bar{#1}}}

\newcommand{\qlep}{{|q|\le \frac{1}{2}|p|^{1/m}}}
\newcommand{\qgep}{{|q|\ge \frac{1}{2}|p|^{1/m}}}

\def\R{\mathbb R}
\def\seq#1{\left<#1\right>}
\def\sep#1{\left(#1\right)}

\theoremstyle{definition}
\newtheorem{theorem}{Theorem}

\newtheorem{lemma}[theorem]{Lemma}
\newtheorem{proposition}[theorem]{Proposition}
\newtheorem{remark}[theorem]{Remark}
\newtheorem{definition}[theorem]{Definition}

\numberwithin{equation}{section}
\numberwithin{theorem}{section}

\begin{document}

\keywords{Boltzmann Equation, Special Relativity, Non Angular Cut-off, Collisional Kinetic Theory.}
\subjclass[2010]{Primary 35Q20, 35R11, 76P05, 83A05, 82C40, 35B65, 26A33. }

\title[Relativistic Boltzmann Equation without Cut-off]{Asymptotic Stability of the Relativistic Boltzmann Equation without Angular Cut-off}

\author[J. W. Jang]{Jin Woo Jang$^\dagger$}
\address{$^\dagger$Department of Mathematics, Pohang University of Science and Technology (POSTECH), Pohang, Republic of Korea 37673. \href{mailto:jangjw@postech.ac.kr}{jangjw@postech.ac.kr} (\href{https://orcid.org/0000-0002-3846-1983}{https://orcid.org/0000-0002-3846-1983})}
\thanks{$^\dagger$Supported by the German DFG grant CRC 1060 and the Korean Basic Science Research Institute Fund NRF-2021R1A6A1A10042944, previously supported by the Korean IBS grant IBS-R003-D1 and partially by the NSF grant DMS-1500916 of the USA}

\author[R. M. Strain]{Robert M. Strain$^{\ddagger}$}
\address{$^\ddagger$Department of Mathematics, University of Pennsylvania, Philadelphia, PA 19104, USA. \href{mailto:strain@math.upenn.edu}{strain@math.upenn.edu} (\href{https://orcid.org/0000-0002-1107-8570}{https://orcid.org/0000-0002-1107-8570})}
\thanks{$^\ddagger$Partially supported by the NSF grants DMS-1764177 and DMS-2055271 of the USA}

\begin{abstract}
This paper is concerned with the relativistic Boltzmann equation without angular cutoff.  We establish the global-in-time existence, uniqueness and asymptotic stability for solutions nearby the relativistic Maxwellian.
We work in the case of a spatially periodic box. We assume the generic hard-interaction and  soft-interaction conditions on the collision kernel that were derived by Dudy\'nski and Ekiel-Je$\dot{\text{z}}$ewska (Comm. Math. Phys. \textbf{115}(4):607--629, 1985) in \cite{MR933458}, and our assumptions include the case of Israel particles (J. Math. Phys. \textbf{4}:1163--1181, 1963) in \cite{MR165921}.   In this physical situation, the angular function in the collision kernel is not locally integrable, and the collision operator behaves like a fractional diffusion operator. The coercivity estimates that are needed rely crucially on the sharp asymptotics for the frequency multiplier that has not been previously established.   We further derive the relativistic analogue of the Carleman dual representation for the Boltzmann collision operator.  This resolves the open question of perturbative global existence and uniqueness without the Grad's angular cut-off assumption.
\end{abstract}

\thispagestyle{empty}

\maketitle
\tableofcontents

\section{Introduction}\label{sec:introduction}

In 1872, Ludwig Boltzmann \cite{MR0158708} derived a fundamental equation which mathematically models the dynamics of a gas represented as a collection of molecules. 
This is a model for the collisional dynamics between non-relativistic particles.  For the collisional dynamics between special relativistic particles whose speed is comparable to the speed of light, 
Lichnerowicz and Marrot \cite{MR0004796} have derived the relativistic Boltzmann equation in 1940, which is a fundamental model for fast moving particles.
The relativistic Boltzmann equation is written as
\begin{equation}
\label{RBE}
p^\mu\partial_\mu F={p^0}\partial_t F+cp\cdot\nabla_x F= C(F,F),
\end{equation}
where $c>0$ is the speed of light.  For this equation the unknown is $F=F(t,x,p)$ where the time variable is $t\ge 0$, the spatial variable is $x\in \mathbb{T}^3$ and the momentum satisfies $p\in \R^3$.  Then the collision operator $C(F,F)$ is given by
\begin{equation}
\label{Colop}
C(F,G)=\int_{\mathbb{R}^3}\frac{dq}{{q^0}}\int_{\mathbb{R}^3}\frac{dq'\ }{{q'^0}}\int_{\mathbb{R}^3}\frac{dp'\ }{{p'^0}} W(p,q|p',q')[F(q')G(p')-F(q)G(p)].
\end{equation}
Here, the transition rate $W(p,q|p',q')$ is 
\begin{equation}\label{FREQ:transition.rate.RBE}
W(p,q|p',q')=\frac{c}{2}s\sigma(g,\theta)\delta^{(4)}(p^\mu +q^\mu -p'^\mu -q'^\mu),
\end{equation}
where $\sigma(g,\theta)$ is the scattering kernel measuring the interactions between particles and the Dirac $\delta$-function expresses the conservation of energy and momentum.  The notation for $p^\mu$, $q^\mu$, $p'^\mu$ and $q'^\mu$ will be defined in \secref{Lor} and \eqref{conservation}.  Here $s$, $g$ and $\theta$ will be defined in \eqref{s.c}, \eqref{g.c} and \eqref{cost} respectively.

This equation is a relativistic generalization of the Newtonian Boltzmann equation:
\begin{equation}\label{Boltz.Newtonian}
    \partial_t F +v\cdot\nabla_x F=\int_{\rth}dv_*\int_{\mathbb{S}^2}d\omega ~B(v-v_*,\omega)[F(v'_*)F(v')-F(v_*)F(v)],
\end{equation}
where 
\begin{equation}\notag
    v'=\frac{v+v_*}{2}+\frac{|v-v_*|}{2}\omega, \quad v'_*=\frac{v+v_*}{2}-\frac{|v-v_*|}{2}\omega,
\end{equation}
and the collision kernel $B$ depends only on the relative velocity $|v-v_*|$ and the scattering angle $\omega$. The mathematical analysis of the Boltzmann equation such as the well-posedness of the equation or the regularity of the solution crucially depends on the assumptions on the scattering kernel $B(v-v_*,\omega)$. 
The kernel $B$ is in general assumed to be in the form of a product in its arguments as 
$$B(v-v_*,\omega)=\Psi(|v-v_*|)b_0(\omega),$$ where both $\Psi$ and $b_0$ are assumed to be non-negative. This assumption is general and it includes the varied kinds of  collision kernels such as the hard-sphere collision kernel $\Psi(|v-v_*|)\approx |v-v_*|$, the collision kernel for Maxwellian molecules $\Psi(|v-v_*|)\approx 1$, the collision kernel for the  inverse-power law potential $\psi(r)=\frac{1}{r^{p-1}}$ with $B\approx |v-v_*|^{\gamma} \theta^{-\gamma'} b'_0(\theta)$ where $\gamma=\frac{p-5}{p-1}$, $\gamma'=\frac{2p}{p-1}$, $b'_0$ is bounded, and $\cos\theta =\frac{v-v_*}{|v-v_*|}\cdot \omega$,  and the assumption also includes many other kernels.

Both equations model the evolution of a large number of particles interacting via collisions.  The classical non-relativistic Boltzmann equation has been widely studied in many aspects. However, the relativistic Boltzmann equation has received relatively less attention perhaps because of its complicated structure and the computational difficulty on dealing with relativistic post-collisional momenta.  The relativistic Boltzmann equation is a correction to the Newtonian equation which will do a better job of describing fast-moving particles whose speeds may be closer to the speed of light. Understanding the behavior of fast-moving special relativistic particles is crucial in describing many astrophysical and cosmological processes \cite{MR3189734}.   Especially, the description of the dynamics of quark plasma formed in heavy ion collisions will have to involve a satisfactory understanding in a certain relativistic hydrodynamical equation \cite{Dudynski2}, and the relativistic Boltzmann equation is a good candidate for describing those relativistic collisional hydrodynamical models.  Further fast moving particles are precisely the situation where non-cutoff effects can become important.  References on the relativistic Boltzmann equation include \cite{MR1898707,MR635279,MR1379589,Stewart,MR0088362}.

For short range interactions we have collision kernels such as $\sigma(g,\theta)=\text{constant}$ or $s\sigma(g,\theta)=\text{constant}$, and these are the relativistic analogue of the classical hard-sphere model (although there is no relativistic hard-sphere). However, once we consider the long-range interactions when particles are fast moving, then $\sigma(g,\theta)$ can be very singular  and non-integrable near $\theta=0$. This occurs especially for long-range interactions and grazing collisions.
 
The difficulty with the angular singularity can be removed with Grad's ``cut-off'' assumption \cite{MR0156656} that $\sigma\in L^1_{loc}(\rth\times\mathbb{S}^2)$.
In this case, we say that the Boltzmann equation is in the ``cutoff" regime. Otherwise, we call it the Boltzmann equation without angular cutoff; sometimes this is called the ``non-cutoff'' regime.  This cut-off assumption is indeed very powerful in the mathematical analysis, as it removes the singularity from the angular kernel and allows one to split the gain and the loss terms of the Boltzmann operator. 

However, it has been well-known that the regularity of a solution to the Boltzmann equation depends crucially on the assumption. For the angular kernel with the Grad cutoff, it has been known to propagate singularities \cite{MR1798557, MR2435186}. On the other hand, it has been known that the Boltzmann equation without angular cutoff has smoothing effects \cite{Li94,MR2679369,1909.12729,MR4195746}. In the case without angular cutoff, one has to make use of the cancellation between the gain and the loss terms to estimate the angular singularity.  Without angular cutoff the Boltzmann operator behaves as the fractional Laplacian on a lifted paraboloid of the energy-momentum four-vector \cite{MR2807092}.

 Unfortunately, to the best of our knowledge, the relativistic Boltzmann equation has not been studied  without the ``cut-off'' hypothesis though the case when the collisions tend to be grazing is very important. In this paper we study the relativistic Boltzmann equation without assuming the angular cut-off hypothesis which would give a better understanding on the long-range interactions of relativistic particles.

\subsection{A brief history of previous results on the relativistic Boltzmann equation}
The special relativistic Boltzmann equation was first derived in the paper by Lichnerowicz and Marrot \cite{MR0004796} in 1940. In 1967, 
Bichteler \cite{MR0213137} showed the local existence of the solutions to the relativistic Boltzmann equation. In 1989, Dudynski and Ekiel-Jezewska \cite{MR933458} showed that 
there exist unique $L^2$ solutions to the linearized equation. Afterwards, Dudynski \cite{MR1031410} studied the long time and small-mean-free-path limits of these solutions. 
Regarding large data global in time weak solutions, Dudynski and Ekiel-Jezewska \cite{MR1151987} in 1992 extended DiPerma-Lions renormalized solutions \cite{MR1014927} 
to the relativistic Boltzmann equation using their causality results from 1985 \cite{MR818441}.
Recently, Wang \cite{MR3868725} proved the global well-posedness of the relativistic Boltzmann equation with perturbative large amplitude initial data. 

In 1996, Andreasson \cite{MR1402446} studied the regularity of the gain term and the strong $L^1$ convergence of 
the solutions to the J{\"u}ttner equilibrium which were generalizations of Lions' results \cite{MR1284432,MR1295942} in the non-relativistic case. He showed that the gain term is regularizing. 
In 1997, Wennberg \cite{MR1480243} showed the regularity of the gain term in both non-relativistic and relativistic cases. 

Regarding the Newtonian limit for the Boltzmann equation, there is a local result by Calogero \cite{MR2098116} and a global result by Strain \cite{MR2679588}.
Also, Andreasson, Calogero and Illner \cite{MR2102321} proved that there is a blow-up if only with
gain-term in 2004. Then, in 2009, Ha, Lee, Yang, and Yun \cite{MR2543323} provided uniform $L^2$-stability estimates for the relativistic Boltzmann equation. 
In 2011, Speck and Strain \cite{MR2793935} connected the relativistic Boltzmann equation to the relativistic Euler equation via the Hilbert expansions. The gain of regularity for the gain operator was proved in \cite{MR3880739} for hard- and soft-interactions. The propagations of $L^1$, $L^\infty$, and $L^p$ estimates were proved in \cite{StrainYun2014s}, \cite{1907.05784}, and \cite{MR4156121}, respectively. 

Regarding problems with the initial data nearby the global Maxwellian equilibrium \eqref{jutter.equilibrium} that we consider in this paper, Glassey and Strauss \cite{MR1211782} first proved there exist unique global 
smooth solutions to the equation on the torus $\mathbb{T}^3$ for the hard-interactions in 1993. Also, in the same paper they have shown that the convergence rate to the relativistic Maxwellian
is exponential. Their assumptions on the differential
cross-section covered the case of cut-off hard-interactions. In 1995 \cite{MR1321370}, they extended their results to the whole space and have shown that the convergence rate to the equilibrium solution
is polynomial. 
Under reduced restrictions on the cross-sections, Hsiao and Yu \cite{MR2249574} gave results on the asymptotic stability of Boltzmann equation using energy methods in 2006.  In 2010, Yang and Yu \cite{MR2593052} proved time decay rates in the whole space for the relativistic
Boltzmann equation with hard-interactions and for the relativistic Landau equation.  In 2010, Strain \cite{MR2728733} showed that unique global-in-time solutions to the relativistic Boltzmann equation exist for 
the soft-interactions with cut-off. Recently, Duan and Yu \cite{MR3635814} have shown the global wellposedness for the relativistic Boltzmann equation for soft-interactions in the weighted $L^\infty$ perturbation framework. We also mention a recent result \cite{MR4264953} on the global wellposedness for the relativistic quantum Boltzmann equation for both Bosons and Fermions near equilibrium.  In addition, we would like to mention that Glassey and Strauss \cite{MR1105532} in 1991 computed the Jacobian determinant of the relativistic collision map.

\subsubsection{On the Newtonian Boltzmann equation without angular cut-off}
Regarding non-relativistic results for the spatially homogeneous Boltzmann equation without angular cutoff, we refer to results on moment propagation \cite{MR4014786, MR3810839, MR3759871} and the results on instantaneous smoothing effect \cite{MR3665667}. For the existence of measure-valued solutions, we have \cite{MR3572500}.

Regarding non-relativistic results with non-cutoff assumptions, we would like to mention \cite{ADVW} for the entropy dissipation and regularizing effect and \cite{MR1750040} for the instantaneous smoothing effect. For the existence theory, we have the work by Alexandre and Villani \cite{MR1857879} on renormalized weak solutions with non-negative defect measure. Also, we would like to record the work of Gressman and Strain \cite{GressmanStrain2010g,MR2784329} on the global existence of unique solutions close to the Maxwellian equilibrium. The large time decay in the whole space for these solutions was shown in \cite{MR3213301,MR2972454}. We also mention that Alexandre, Morimoto, Ukai, Xu, and Yang \cite{MR2795331,MR2679369,MR2863853,MR2793203,MR2847536} obtained the proof of the global existence of unique solutions with non-cutoff assumptions, using different methods. We would like to mention the work by the same group \cite{MR3177640} from 2013 on the local existence with mild regularity for the non-cutoff Boltzmann equation where they work with an improved initial condition and do not assume that the initial data is close to a global equilibrium.  Local existence with large polynomially decaying initial data  has been recently proven by Henderson, Snelson and Tarfulea in \cite{MR4112183}.  Morimoto and Sakamoto \cite{MR3532066} have recently proven the existence of unique global solutions close to equilibrium in a critical Chemin-Lerner space.   We mention also the works in \cite{MR3950012,2010.10065,MR4107942,2017arXiv171000315H}.   Stability of the vacuum state has been recently established in \cite{MR4270852}, building upon \cite{MR3948345}. Also a new regularization mechanism was developed in \cite{MR3551261}, a weak Harnack-type inequality for the Boltzmann equation has been proved in \cite{MR4049224}, and the $C^\infty$ regularization estimates are proven in \cite{1909.12729}.   Recently, the construction of unique global solutions with low regularity using the Wiener algebra $L^1_k$ in the $x$ variables was introduced in \cite{1904.12086}.

\section{Statement of the main results and strategies}

In this section, we will introduce a reformulation of the equation \eqref{RBE} by the linearization around the relativistic Maxwellian equilibrium. Before we introduce the reformulated problem including stating our main hypothesis on the scattering kernel and our main theorems, we first introduce several notations that we will use throughout the paper.

\subsection{Notations}\label{Lor}Throughout the paper, we denote $A\lesssim B$ if there exists a uniform constant $C>0$ such that $A\le CB.$ If $A\lesssim B$ and $B\lesssim A$, then we denote $A\approx B.$   We define $B_{r}= B(0,r)$ to be the standard ball of center zero and radius $r>0$.

The relativistic momentum of a particle is denoted by a four-vector representation $p^\mu$ where $\mu=0,1,2,3$. Without loss of generality we normalize the mass of each particle $m=1$. 
We raise and lower the indices with the Minkowski metric $p_\mu=\eta_{\mu\nu}p^\nu$, where the metric is defined as $\eta_{\mu\nu}=\text{diag}(-1, 1, 1, 1)$ is a $4 \times 4$ matrix.
The signature of the metric throughout this paper is $(-+++)$.   The inverse of the Minkowski metric is denoted $\eta^{\mu\nu}=\text{diag}(-1, 1, 1, 1)$.  In general, Latin indices $i,j,$ etc., take values in $\{1,2,3\},$ while Greek indices $\kappa, \lambda, \mu, \nu,$ etc., take on the values $\{0,1,2,3\}$. With $p=(p^1, p^2, p^3)\in  \rth$, we write $p^\mu=({p^0},p)$ where ${p^0}$, which is the energy of a relativistic particle with momentum $p$, is defined as ${p^0}=\sqrt{c^2+|p|^2}$ where $|p|^2 = p \cdot p$. 
We use the standard Euclidean dot product: $p\cdot q \eqdef \sum_{i=1}^{3} p^i q^i$. We use the notation $p^\mu$ to both denote the component $\mu$ and also to denote the vector $({p^0},p)$ without ambiguity.  We furthermore use the Einstein convention of implicit summation over repeated indices with one up and one down.  The product between the four-vectors with raised and lowered indices is the Lorentz inner product which is then given by
$$
p^\mu q_\mu= p^\mu \eta_{\mu\nu} q^\mu= \sum_{\mu =0}^3\sum_{\nu =0}^3 p^\mu \eta_{\mu\nu} q^\mu= -{p^0}{q^0}+p\cdot q.
$$
Note that the momentum for each particle satisfies the mass shell condition $p^\mu p_\mu=-c^2$ with ${p^0}>0$. Also, the product $p^\mu q_\mu$ is Lorentz invariant as described in Definition \ref{rcop:LTdef}.

By expanding the relativistic Boltzmann equation \eqref{RBE} and dividing both sides by ${p^0}$ we write the relativistic Boltzmann equation as
\begin{equation}
    \partial_t F+\hat{p}\cdot \nabla_x F= Q(F,F), \quad 
F(t=0, x, p) = F_0(x,p),
\label{rBoltz.eqn}
\end{equation}
where $Q(F,F)=C(F,F)/{p^0}$ and the normalized velocity of a particle $\hat{p}$ is given by
$$
\hat{p}=c\frac{p}{{p^0}}=\frac{p}{\sqrt{1+|p|^2/c^2}}.
$$
We define the quantities $s$ and $g$ which respectively are the square of the energy and the relative momentum in the \textit{center-of-momentum} system, $p+q=0$, as
\begin{equation}
\label{s.c}
s=s(p^\mu,q^\mu)=-(p^\mu+q^\mu)(p_\mu+q_\mu)=2(-p^\mu q_\mu+c^2)\geq 0,
\end{equation}
and 
\begin{equation}
\label{g.c}
g=g(p^\mu,q^\mu)=\sqrt{(p^\mu-q^\mu)(p_\mu-q_\mu)}=\sqrt{2(-p^\mu q_\mu-c^2)}.
\end{equation}
Note that we have $s=g^2+4c^2$. We now rewrite the quantities $s$ and $g$ from \eqref{s.c} and \eqref{g.c} with $c=1$ as
\begin{equation}
\label{s}
s=s(p^\mu,q^\mu)=-(p^\mu+q^\mu)(p_\mu+q_\mu)=2(-p^\mu q_\mu+1)\geq 0,
\end{equation}
and 
\begin{equation}
\label{g}
g=g(p^\mu,q^\mu)=\sqrt{(p^\mu-q^\mu)(p_\mu-q_\mu)}=\sqrt{2(-p^\mu q_\mu-1)}.
\end{equation}
Now we have that $s=g^2+4$. 
Similarly we can define $\bar{g}$ as the relative momentum between $p'^\mu$ and $p^\mu$ in the \textit{center-of-momentum} system. It is defined as
\begin{equation}
\label{gbar}
\begin{split}
\bar{g}&\eqdef g(p'^\mu,p^\mu)=\sqrt{(p'^\mu-p^\mu)(p'_\mu-p_\mu)}=\sqrt{2(-p'^\mu p_\mu-1)}\\&=\sqrt{2(p'^0p^0-p'\cdot p -1)}=\sqrt{2\frac{|p-p'|^2+|p\times p'|^2}{ p^0  p'^0 +p\cdot p'+1}}.\\
\end{split}
\end{equation}
In the same manner, we define the relative momentum between $p'^\mu$ and $q^\mu$ as
\begin{equation}\begin{split}\label{gtilde}
\tilde{g} &\eqdef g(p'^\mu,q^\mu)=\sqrt{(p'^\mu-q^\mu)(p'_\mu-q_\mu)} =\sqrt{2(-p'^\mu q_\mu-1)}\\&=\sqrt{2(p'^0q^0-p'\cdot q -1)}=\sqrt{2\frac{|p'-q|^2+|p'\times q|^2}{ p'^0  q^0 +p'\cdot q+1}}.
\end{split}
\end{equation}
Again we have $\bar{s}=\bar{g}^2+4$ and $\tilde{s}=\tilde{g}^2+4$.  These important quantities will be used extensively in the proofs below.

The conservation of energy and momentum for elastic collisions is described as 
\begin{equation}
\label{conservation}
p^\mu+q^\mu=p'^\mu+q'^\mu.
\end{equation}
Then the scattering angle $\theta$ is defined by
\begin{equation}
\label{cost}
\cos\theta=\frac{(p^\mu-q^\mu)(p'_\mu-q'_\mu)}{g^2}.
\end{equation}
Together with the conservation of energy and momentum in \eqref{conservation}, it can be shown that the angle and $\cos\theta$ are well-defined \cite{MR1379589}. Note that the numerator of $\cos\theta$ can be further written as 
\begin{equation}
 \label{FREQ:cos}
 \begin{split}
(p^\mu-q^\mu)(p'_\mu-q'_\mu)=&(p^\mu-q^\mu)(p_\mu+q_\mu-2q'_\mu)\\
=& (p^\mu-q^\mu)(p_\mu-q_\mu)+2 (p^\mu-q^\mu)(q_\mu-q'_\mu)\\
=&g^2+2(p^\mu-q^\mu)(q_\mu-q'_\mu)\\
=& g^2 +2(p^\mu-p'^\mu+p'^\mu-q^\mu)(p'_\mu-p_\mu)\\
=&g^2 -2(p'^\mu-p^\mu)(p'_\mu-p_\mu)+2(p'^\mu-q^\mu)(p'_\mu-p_\mu)\\
=&g^2 -2\bar{g}^2+2(p'^\mu-q^\mu)(p'_\mu-p_\mu)
=g^2-2\bar{g}^2.
 \end{split}
 \end{equation} 
 As above, we note that it follows from the collision geometry \eqref{conservation} that 
 $$
 (p'^\mu-q^\mu)(p'_\mu-p_\mu)=0.
 $$
 Therefore, using \eqref{FREQ:cos} with \eqref{g} and \eqref{gbar} we can write
\begin{equation}\label{FREQ:cosine.angle.formula}
    1-2\sin^2\frac{\theta}{2} =  \cos\theta=1-2\frac{\bar{g}^2}{g^2},
\end{equation}
 and hence we obtain that $\theta\approx \frac{\bar{g}}{g}$. This estimate will be used frequently in \secref{sec:frequency}.

\begin{remark}\label{FREQ:angle.remark}
Since we are dealing with the non-cutoff relativistic Boltzmann equation then there will be an angular singularity when $\cos\theta=1$ as in \eqref{angassumption}.  The purpose of this remark is to explain the collisional geometry when $\cos\theta=1$.  
By \eqref{FREQ:cosine.angle.formula} when $\cos\theta=1$ we have $\frac{\bar{g}^2}{g^2} =0$ which means 
$$
0=\bar{g}^2=(p^\mu-p'^\mu)(p_\mu-p'_\mu).
$$
Equivalently, this means that 
$$
({p'^0}-{p^0})^2=|p'-p|^2.
$$
And this implies that ${p^0}={p'^0}$ and $p=p'$ because otherwise
$$
|{p'^0}-{p^0}| = \left|\frac{|p'|^2-|p|^2}{{p'^0}+{p^0}}\right|< |p'-p|.
$$
Therefore, if $\cos\theta=1$, we have $p'^\mu=p^\mu$ and also $q'^\mu=q^\mu$ by \eqref{conservation}. 
\end{remark}

Here we would like to introduce the relativistic Maxwellian which models the equilibrium solutions, also known as J{\"u}ttner solutions. 
These are characterized as a particle distribution which maximizes the entropy subject to constant mass, momentum, and energy. They are given by
$$
J(p)=\frac{e^{-\frac{c{p^0}}{k_BT}}}{4\pi ck_BTK_2(\frac{c^2}{k_BT})},
$$
where $k_B$ is Boltzmann constant, $T$ is the temperature, and $K_2$ stands for the Bessel function $K_2(z)=\frac{z^2}{2}\int_1^\infty dt\ e^{-zt}(t^2-1)^\frac{3}{2}.$
Throughout this paper, we normalize all physical constants to 1, including the speed of light $c=1$.  Then we observe that the relativistic Maxwellian is given by
\begin{equation}\label{jutter.equilibrium}
    J(p)=\frac{e^{-{p^0}}}{4\pi}.
\end{equation}

\subsection{Relativistic collision operator}
We now consider the \textit{center-of-momentum} expression for the relativistic collision operator. Note that this expression has appeared in the physics literature; see \cite{MR635279}. 
For other representations of the operator such as Glassey-Strauss coordinate expression, see \cite{MR1402446,MR1105532,MR1211782}. Also, see \cite{MR2679588,MR2765751} for the relationship between those two representations of the collision operator. 
As in \cite{MR635279}, one can reduce the collision operator (\ref{Colop}) using Lorentz transformations and get
\begin{equation}
\label{omegaint}
Q(f,h)=\int_\rth dq\int_{\mathbb{S}^2}d\omega\   v_{\text\o}\sigma(g,\theta)[f(q')h(p')-f(q)h(p)],
\end{equation}
where $v_{\text{\o}}=v_{\text{\o}}(p,q)$ is the M{\o}ller velocity given by
\begin{equation}\label{Moller.def}
v_{\text{\o}}(p,q)=\sqrt{\Big|\frac{p}{{p^0}}-\frac{q}{{q^0}}\Big|^2-\Big|\frac{p}{{p^0}}\times\frac{q}{{q^0}}\Big|^2}=\frac{g\sqrt{s}}{{p^0}{q^0}}.
\end{equation}
The post-collisional momenta in the \textit{center-of-momentum} expression are written as 
\begin{equation}
\label{p'}
p'=\frac{p+q}{2}+\frac{g}{2}\Big(\omega+(\xi-1)(p+q)\frac{(p+q)\cdot\omega}{|p+q|^2}\Big),
\end{equation}
and
\begin{equation}
\label{q'}
q'=\frac{p+q}{2}-\frac{g}{2}\Big(\omega+(\xi-1)(p+q)\frac{(p+q)\cdot\omega}{|p+q|^2}\Big),
\end{equation}
where $\xi\eqdef \frac{p^0+q^0}{\sqrt{s}}$.  

For $F,G$ smooth and vanishing sufficiently rapidly at infinity, it turns out \cite{MR1379589} that the collision operator satisfies
\begin{equation}\label{eq.colop.property}
\int Q(F,G)\ dp=\int pQ(F,G)\ dp=\int {p^0}Q(F,G)\ dp=0,
\end{equation}
and 
\begin{equation}
\label{entropy}
\int Q(F,F)(1+\log F)\ dp\ \leq 0.
\end{equation}
Note that \eqref{eq.colop.property} leads to the conservation laws of total mass, momentum, and energy as
\begin{equation}
    \notag
    \frac{d}{dt}\int_{\mathbb{T}^3}dx \int_\rth dp \begin{pmatrix}1\\p\\p^0\end{pmatrix} F(t,x,p)=0.
\end{equation}
Also, \eqref{entropy} leads to the Boltzmann H-theorem which states that the entropy of the system is a non-decreasing function of $t$; i.e., we have
\begin{equation}
    \notag
\frac{d}{dt}\int_{\mathbb{T}^3}dx \int_\rth dp \ (-F\log F)(t,x,p)\ge 0,
\end{equation}where the expression $-F\log F$ is 
called the \textit{entropy density}.

\subsection{Main hypothesis on the collision kernel $\sigma$}\label{hypo}
The relativistic Boltzmann collision kernel $\sigma(g,\theta)$ is a non-negative function which only depends on the relative velocity $g$ and the scattering angle 
$\theta$. We assume that $\sigma$ takes the form of the product in its arguments; i.e., 
\begin{equation}
    \label{define.kernel}
    \sigma(g,\theta)\eqdef \Phi(g)\sigma_0(\theta).
\end{equation}
In general, we suppose that both $\Phi$ and $\sigma_0$ are non-negative functions.

Without loss of generality, we may assume that the collision kernel $\sigma$ is supported only when $\cos\theta\geq 0$ throught this paper;
i.e., $0\leq \theta \leq \frac{\pi}{2}$. 
Otherwise, the following \textit{symmetrization} \cite{MR1379589} will reduce to this case:
 $$
\bar{\sigma}(g,\theta)=[\sigma(g,\theta)+\sigma(g,-\theta)]1_{\cos\theta\geq 0},
$$
where $1_A$ is the indicator function of the set $A$.

We suppose that the angular function $\theta \mapsto \sigma_0(\theta)$ is not locally integrable; for some $C>0$, it satisfies 
 \begin{equation}
 \label{angassumption}
\frac{1}{C\theta^{1+\gamma}} \leq \sin\theta\cdot\sigma_0(\theta) \leq \frac{C}{\theta^{1+\gamma}}, \hspace*{5mm}\gamma \in (0,1), \hspace*{5mm}\forall \theta \in \Big(0,\frac{\pi}{2}\Big].
\end{equation}
Notice that we do not assume any ``cut-off'' hypothesis on the angular function \cite{MR0156656} that $\sigma_0\in L^1_{loc}(\mathbb{S}^2)$.  We further assume the collision kernel satisfies the following hard-interaction assumption:
\begin{equation}
\label{hard}
 \Phi(g) =  C_{\Phi} g^{\singA},
 \quad - \gamma\leq {\singA} <2, \quad C_{\Phi}>0.
\end{equation}
In the soft-interaction case we assume that
\begin{equation}
\label{soft}
 \Phi(g) =  C_{\Phi} g^{-\singB},
  \quad  \gamma<b<\min\left\{\frac{3}{2}+\gamma,2\right\}, \quad C_{\Phi}>0.
\end{equation}
For these expressions we introduce the following unified notation
\begin{equation} \label{singS.defin}
\singS=
\left\{
\begin{array}{ccc}
 \singA  & \text{for the hard-interactions}   &  \eqref{hard},  \\
 - \singB  & \text{for the soft-interactions}   &  \eqref{soft}. 
\end{array}
\right.
\end{equation}
Then we generally have for both hard \eqref{hard} and soft \eqref{soft} interactions that
\begin{equation}
\notag
 \Phi(g) =  C_{\Phi} g^{\singS},
  \quad   \quad C_{\Phi}>0.
\end{equation}
We further remark that the conditions above imply that $0 \le \singA +\gamma < 2+ \gamma$ and $\max\{-\frac{3}{2},-2+\gamma\} <  -\singB +\gamma < 0$ so that in general 
\begin{equation*}
    \max\left\{-\frac{3}{2},-2+\gamma\right\} <  \singS +\gamma < 2+ \gamma.
\end{equation*}
These are the assumptions on the kernel that we will use throughout this paper.

These assumptions on our collision kernel have been motivated from important physical interactions. Conditions on our collision kernel are generic  in the sense of the collision kernel assumptions derived by Dudy\'nski and Ekiel-Je$\dot{\text{z}}$ewska in \cite{MR933458}.  In this work we do not study the high order singularities when $\gamma \in [1,2)$ for \eqref{angassumption}.  Our results can further cover the case of Israel particles from \cite{MR165921}.   Unfortunately, to the best of our knowledge, the relativistic Boltzmann equation has not been studied without the ``cut-off'' hypothesis.   This problem was also discussed in the appendix to \cite{MR3186493}.   In this paper we will study the relativistic Boltzmann equation without assuming the Grad's angular cut-off hypothesis in order to try to obtain a better understanding of relativistic gases.    We include several additional physical references that discuss the special relativistic Boltzmann collision kernels  \cite{MR1958975,Polak_1973,MR933458,Dudynski2,MR3186493,MR1898707,MR471665,MR635279,MR165921,MR1402248}  including those with an angular singularity such as in \eqref{angassumption}.   Some of these are also discussed in \cite[Appendix B]{MR2679588}.

\subsection{Linearization and reformulation of the Boltzmann equation}

We will consider the linearization of the collision operator and the perturbation around the relativistic J{\"u}ttner equilibrium state
\begin{equation}
\label{pert}
F(t,x,p)=J(p)+\sqrt{J(p)}f(t,x,p).
\end{equation}
Without loss of generality, we suppose that the mass, momentum, and energy conservation laws for the perturbation $f(t,x,p)$ hold for all $t\geq 0$ as
\begin{equation}
\label{zero}
\int_\rth dp\ \int_{\mathbb{T}^3} dx\  \left(\begin{array}{c}
1\\
p\\
{p^0}\end{array}\right) \sqrt{J(p)}f(t,x,p)=0.
\end{equation}
We will now linearize the relativistic Boltzmann equation \eqref{rBoltz.eqn} with \eqref{omegaint} around the relativistic Maxwellian equilibrium state \eqref{pert}. We obtain that
\begin{equation}
\label{Linearized B}
\partial_t f+\hat{p}\cdot\nabla_x f+L(f)=\Gamma(f,f), \hspace{10mm} f(0,x,v)=f_0(x,v),
\end{equation}
where the linearized relativistic Boltzmann operator $L$ is given by 
\begin{multline*}
L(f)\eqdef -J^{-1/2}Q(J,\sqrt{J}f)-J^{-1/2}Q(\sqrt{J}f,J)\\
= \int_{\rth}dq \int_{\mathbb{S}^2} d\omega\ v_{\text{\o}} \sigma(g,\omega)\Big(f(q)\sqrt{J(p)}\\
 +f(p)\sqrt{J(q)}-f(q')\sqrt{J(p')}-f(p')\sqrt{J(q')}\Big)\sqrt{J(q)}, 
\end{multline*}
and the bilinear operator $\Gamma$ is given by
\begin{equation}
\label{Gamma1}
\begin{split}
\Gamma(f,h)&\eqdef J^{-1/2}Q(\sqrt{J}f,\sqrt{J}h)\\
&=\int_{\rth}dq \int_{\mathbb{S}^2} d\omega\ v_{\text{\o}} \sigma(g,\theta)\sqrt{J(q)}(f(q')h(p')-f(q)h(p)).
\end{split}
\end{equation}
Then notice that we have 
 \begin{equation}\label{L.def}
L(f)=-\Gamma(f,\sqrt{J})-\Gamma(\sqrt{J},f).
\end{equation}
We will further decompose $L=\mathcal{N}+\mathcal{K}$. 

We call $\mathcal{N}$ as the norm part and $\mathcal{K}$ as the compact part. First, we define the weight function $\tilde{\zeta}$ such that
 \begin{multline}
 \label{25}
 \Gamma(\sqrt{J},f)=\left(\int_{\rth}dq\int_{\mathbb{S}^2}d\omega\ v_{\text{\o}} \sigma(g,\theta)(f(p')-f(p))\sqrt{J(q')}\sqrt{J(q)}\right)\\-\tilde{\zeta}(p)f(p),
 \end{multline}
where \begin{equation}\label{FREQ:tildezeta}
	\tilde{\zeta}(p) \eqdef\int_{\rth}dq\int_{\mathbb{S}^2}d\omega\ v_{\text{\o}} \sigma(g,\theta)(\sqrt{J(q)}-\sqrt{J(q')})\sqrt{J(q)}.
\end{equation}
We now call $\tilde{\zeta}(p)$ the \textit{frequency multiplier} of the linearized Boltzmann collision operator. 
It is crucial to obtain the sharp asymptotic behavior of $\tilde{\zeta}(p)$ for the proof of the coercivity estimates of the linearized relativistic Boltzmann operator without angular cutoff, which will be used crucially for the proof of the global well-posedness of the relativistic Boltzmann equation without angular cutoff nearby the Maxwellian equilibrium \eqref{jutter.equilibrium}.

The weight function $\tilde{\zeta}(p)$ can be split into the sum of two weight functions as 
$$\tilde{\zeta}=\zeta+\zeta_{\mathcal{K}}$$ 
where the weights satisfy the following asymptotics; for any $\varepsilon\in (0,\gamma/2)$, there exists a finite constant $C_\varepsilon>0$ such that under \eqref{singS.defin} we have
\begin{equation}
\label{Paos}
|\zeta_\mathcal{K}(p)|\lesssim C_\varepsilon {(p^0)}^{\frac{\singS}{2}+\varepsilon}
\hspace{5mm}\text{and}\hspace{5mm} 
\zeta(p)\approx {(p^0)}^{\frac{{\singS+\gamma}}{2}}.
\end{equation}
These asymptotics are proven in Theorem \ref{FREQ:main.thm}.  Further $\zeta$ and $\zeta_{\mathcal{K}}$ are defined precisely in \eqref{FREQ:def.zeta} and \eqref{FREQ:zetaK.def}.

This splitting motivates the following splitting of the linearized operator $L$: the compact part $\mathcal{K}$ of the linearized Boltzmann operator $L$ is defined by
\begin{multline}\label{defK}\mathcal{K}f =\zeta_\mathcal{K}(p)f-\Gamma(f,\sqrt{J})\\
 =\zeta_\mathcal{K}(p)f-\int_{\rth}dq\int_{\mathbb{S}^2}d\omega\ v_{\text{\o}} \sigma(g,\theta)\sqrt{J(q)}(f(q')\sqrt{J(p')}-f(q)\sqrt{J(p)}),
\end{multline}
and the sharp norm part is called $\mathcal{N}$ and it is defined by
\begin{multline}\label{defN}
\mathcal{N}f =-\Gamma(\sqrt{J},f)-\zeta_\mathcal{K}(p)f\\
 =\zeta(p)f-\int_{\rth}dq \int_{\mathbb{S}^2} d\omega\ v_{\text{\o}} \sigma(g,\omega)(f(p')-f(p))\sqrt{J(q')}\sqrt{J(q)}.
\end{multline}
Then, the norm part satisfies that
\begin{multline}\label{mainPartNorm.Nf}
\langle \mathcal{N}f,f\rangle =\frac{1}{2}\int_{\rth}dp \int_{\rth}dq \int_{\mathbb{S}^2} d\omega\ v_{\text{\o}} \sigma(g,\theta)(f(p')-f(p))^2\sqrt{J(q')}\sqrt{J(q)}\\
+\int_{\rth} dp \ \zeta(p)|f(p)|^2.
\end{multline}
This holds because a pre-post collisional change of variables $(p,q) \rightarrow (p',q')$ as in \eqref{prepost.change} provides
\begin{equation}\notag
\begin{split}
-&\int_\rth dp\int_{\rth}dq\int_{\mathbb{S}^2}d\omega\ v_{\text{\o}} \sigma(g,\theta)(f(p')-f(p))h(p)\sqrt{J(q')}\sqrt{J(q)}\\
=&-\frac{1}{2}\int_\rth dp\int_{\rth}dq\int_{\mathbb{S}^2}d\omega\ v_{\text{\o}} \sigma(g,\theta)(f(p')-f(p))h(p)\sqrt{J(q')}\sqrt{J(q)}\\
&-\frac{1}{2}\int_\rth dp\int_{\rth}dq\int_{\mathbb{S}^2}d\omega\ v_{\text{\o}} \sigma(g,\theta)(f(p)-f(p'))h(p')\sqrt{J(q)}\sqrt{J(q')}\\
=&\frac{1}{2}\int_\rth dp\int_{\rth}dq\int_{\mathbb{S}^2}d\omega\ v_{\text{\o}} \sigma(g,\theta)(f(p')-f(p))(h(p')-h(p))\sqrt{J(q')}\sqrt{J(q)}.
\end{split}
\end{equation}
With this computation in mind we define a fractional semi-norm as 
 $$
|f|^2_{\mathtt{B}}\eqdef \frac{1}{2}\int_\rth dp\int_\rth dq \int_{\mathbb{S}^2}d\omega ~ v_{\text{\o}}  \sigma(g,\theta)
(f(p')-f(p))^2\sqrt{J(q)J(q')}.
$$
For $l\in \mathbb{R}$, we also define the unified weight function
\begin{equation}\label{weight.function}
w^{l}(p)=(p^0)^{l}.
\end{equation}
Then we further define a non-local fractional weighted semi-norm
$|f|_{\mathtt{B}_l}$ as:
\begin{equation}
\label{weighted.non.local.fractional.norm}
|f|^2_{\mathtt{B}_l} \eqdef \frac{1}{2}\int_\rth dp\int_\rth dq\ \int_{\mathbb{S}^2}d\omega ~ v_{\text{\o}}  \sigma(g,\theta) w^{l}(p)(f(p')-f(p))^2\sqrt{J(q)J(q')}.
\end{equation}
These norms appear in the process of linearization of the collision operator.

\subsection{Spaces}\label{space} Before introducing our methods and strategies, we would like to define several function spaces that we use throughout this paper.

 We will use $\langle\cdot, \cdot\rangle$ to denote the standard $L^2(\mathbb{R}^3_p)$ inner product. Also, we will use $(\cdot, \cdot)$ to denote the $L^2(\mathbb{T}^3_x\times \mathbb{R}^3_p)$ inner product. As will be seen, the construction of our solutions depends on the following weighted fractional Sobolev space:
 $$
I^{\singS,\gamma}\eqdef \{f\in L^2(\mathbb{R}^3_p):|f|_{I^{\singS,\gamma}}<\infty\},
$$
where the norm is described as
 \begin{equation} \label{fractional}
 |f|^2_{I^{\singS,\gamma}}
 \eqdef 
 |f|^2_{L^2_{\frac{\singS+\gamma}{2}}}
 +
 \int_\rth dp\int_\rth dp'\ \frac{(f(p')-f(p))^2}{|p-p'|^{3+\gamma}}({p'^0}{p^0})^{\frac{{\singS+\gamma}}{4}}1_{|p-p'|\leq 1},
 \end{equation}
 and we use the notation \eqref{singS.defin} to define $\singS$ with \eqref{hard} and \eqref{soft}.  Here we further define the weighted $L^2$ norm $|\cdot|_{L^2_{l}}$ for $l\in \mathbb{R}$ as
$$
|f|_{L^2_{l}}\eqdef \int_{\rth}dp ~w^{l}(p)|f(p)|^2.
$$
This is a standard weighted isotropic $L^2$ based fractional derivative norm that is known to be finite for a large class of functions.   We remark that the norm is flat and not geometric, and this is one of the main differences from the non-relativistic case. We discuss this further in \eqref{Nsrho.norm}--\eqref{triple.norm}--\eqref{PDO.form}--\eqref{simple.calc2}.

The notation on the norm $|\cdot|$ refers to function space norms acting on $\mathbb{R}^3_p$ only. The analogous norm acting on $\mathbb{T}^3_x\times \mathbb{R}^3_p$ is denoted by $\|\cdot\|$. So that we have
$$
\|f\|^2_{I^{\singS,\gamma}}\eqdef \|\ |f|_{I^{\singS,\gamma}}\ \|^2_{L^2(\mathbb{T}^3)}.
$$
Given the weight \eqref{weight.function}, using $\singS$ from \eqref{singS.defin} we also define a general weighted fractional Sobolev norm as
\begin{multline}
\label{fractionalw}
|f|^2_{I^{\singS,\gamma}_l}
\\
\eqdef |w^lf|^2_{L^2_{\frac{\singS+\gamma}{2}}}+\int_\rth dp\int_\rth dp'\ w^{2l}(p)\frac{(f(p')-f(p))^2}{|p-p'|^{3+\gamma}}({p'^0}{p^0})^{\frac{{\singS+\gamma}}{4}}1_{|p-p'|\leq 1},
\end{multline}
where as usual $\singS = \singA$ under \eqref{hard}, and $\singS = -\singB$ under \eqref{soft}.  Then similarly
$$
\|f\|^2_{I^{\singS,\gamma}_l}\eqdef \|\ |f|_{I^{\singS,\gamma}_l}\ \|^2_{L^2(\mathbb{T}^3)}.
$$
The multi-indices $\alpha=(\alpha^1,\alpha^2,\alpha^3)$ 
will be used to record spatial derivatives.  We write 
$$
\partial^\alpha=\partial^{\alpha^1}_{x_1}\partial^{\alpha^2}_{x_2}\partial^{\alpha^3}_{x_3}.
$$ 
If each component of $\alpha$ is not greater than that of $\alpha'$, we write $\alpha\leq\alpha'$. Also, $\alpha<\alpha'$ means $\alpha\leq\alpha'$ and $|\alpha|<|\alpha'|$ where $|\alpha|=\alpha^1+\alpha^2+\alpha^3$.  
We further define the derivative space $I^{\singS,\gamma}_{N}(\mathbb{T}^3\times\rth)$ with integer $N\geq 0$ spatial derivatives by
$$
\|f\|^2_{I^{\singS,\gamma}_{N}}=\|f\|^2_{I^{\singS,\gamma}_{N}(\mathbb{T}^3\times\rth)}=\sum_{|\alpha|\leq N}\|\partial^\alpha f\|^2_{I^{\singS,\gamma}(\mathbb{T}^3\times\rth)}.
$$We also define the  weighted derivative space $I^{\singS, \gamma}_{l,N}(\mathbb{T}^3\times\rth)$ whose norm is given by
$$
\|f\|^2_{I^{\singS,\gamma}_{l,N}}=\|f\|^2_{I^{\singS,\gamma}_{l,N}(\mathbb{T}^3\times\rth)}=\sum_{|\alpha|\le N}\|\partial^\alpha f\|^2_{I^{\singS,\gamma}_{l}(\mathbb{T}^3\times\rth)}.
$$
We define the space $H^N=H^N(\mathbb{T}^3\times\rth)$ with integer $N\geq 0$ spatial derivatives as
 $$
\|f\|^2_{H^N}=\|f\|^2_{H^N(\mathbb{T}^3\times\rth)}=\sum_{|\alpha|\leq N}\|\partial^\alpha f\|^2_{L^2(\mathbb{T}^3\times\rth)}.
$$
We then define the space $H^N_l=H^N_l(\mathbb{T}^3\times\rth)$ by
$$
\|f\|^2_{H^N_l}=\|f\|^2_{H^N_l(\mathbb{T}^3\times\rth)}=\sum_{|\alpha|\le N}\|w^{l}\partial^\alpha f\|^2_{L^2(\mathbb{T}^3\times\rth)}.
$$
We sometimes denote the norm $\|f\|^2_{H^{N}_l}$ as $\|f\|^2_H$ for simplicity when there is no risk of ambiguity.  We remark that we always use $N$ to denote the number of derivatives, and we always use $l$ to denote the order of the weights so there is no ambiguity.

We will also consider the spatial derivative of $\Gamma$.  Recall that the linearization of the collision operator is given by (\ref{Gamma1}) and that the post-collisional variables $p'$ and $q'$ satisfy (\ref{p'}) and (\ref{q'}).
Then, we can define the spatial derivatives of the bilinear collision operator $\Gamma$ as
\begin{equation}\label{gamma.deriv}
    \partial^\alpha\Gamma(f,h)=\sum_{\alpha'\leq\alpha} C_{\alpha,\alpha'}\Gamma(\partial^{\alpha-\alpha'}f,\partial^{\alpha'}h),
\end{equation}
where $C_{\alpha,\alpha'}$ are non-negative constants.

Now, we state our main result as follows:
\begin{theorem}(Main Theorem)
	\label{MAIN}
	Fix $N\geq 2$, which represents the total number of spatial derivatives. Choose $f_0=f_0(x,p)\in H^N_{l+m}(\mathbb{T}^3\times\rth)$ in (\ref{pert})  which satisfies (\ref{zero}).  For the hard-interactions \eqref{hard} and the soft-interactions \eqref{soft} we can take any $m\ge 0$ and $l \ge 0$ for the existence and uniqueness.  
	
There is an $\eta_0>0$ such that if 
	$\|f_0\|_{H^N_{l+m}(\mathbb{T}^3\times\rth)} \leq \eta_0$, then there exists a unique global solution to the relativistic Boltzmann equation (\ref{RBE}), in the form (\ref{pert}), which satisfies
$$
f(t,x,p)\in L^\infty_t([0,\infty);H^N_l(\mathbb{T}^3\times\rth))\cap L^2_t((0,\infty);I^{\singS,\gamma}_{l,N}(\mathbb{T}^3\times\rth)),
$$
where we use the notation from \eqref{singS.defin}.  

For the hard-interactions \eqref{hard} we have exponential decay to equilibrium. For some fixed $\lambda >0$ and for any $l\ge 0$, we have the uniform estimate
	$$
	\|w^lf(t)\|_{H^N(\mathbb{T}^3\times\rth)}\lesssim e^{-\lambda t}\|w^lf_0\|_{H^N(\mathbb{T}^3\times\rth)}.
	$$ 
Furthermore, for the soft-interactions \eqref{soft}, fix any $m> 0$ and $l\geq |\singS+\gamma|/4>0$, then for $\|f_0\|_{H^N_{l+m}(\mathbb{T}^3\times\rth)}$ sufficiently small, we further have the polynomial decay  
$$
\|w^lf(t)\|_{H^N(\mathbb{T}^3\times\rth)}
\leq  \|w^{l+m}f_0\|_{H^N(\mathbb{T}^3\times\rth)}
\left(1+C_{l,m}
\frac{|\singS+\gamma|}{2m} t\right)^{-\frac{2m}{|\singS+\gamma|}},
$$
for some constant $C_{l,m}>0$.
\end{theorem}

\subsection{Main estimates}\label{subsec.mainest}
The proof of Theorem \ref{MAIN} heavily depends on the establishment of a global in time energy inequality. For this, we needed to obtain sharp upper- and lower-bound estimates for the linearized operator $L$ and the nonlinear operator $\Gamma$. 
In this section, we would like to record our main upper and lower bound estimates of the inner products that involve the operators $\tilde{\zeta}$, $\Gamma$, $L$, $\mathcal{K}$, and $\mathcal{N}$ from \eqref{Gamma1}-\eqref{defN}. The proofs of these estimates are given in \secref{sec:frequency}, \secref{main upper} and \secref{main coercive estimates}. In general we will prove these estimates in the class of Schwartz functions.  However all of these estimates can be justified in general by standard approximation procedures.

In the theorem and the lemmas below we use $\singS = \singA$  in the hard-interaction case \eqref{hard} and $\singS = -\singB$ in the soft-interaction case \eqref{soft}, as in \eqref{singS.defin}.

\begin{theorem}\label{FREQ:main.thm}The \textit{frequency multiplier} $\tilde{\zeta}(p)$ from \eqref{FREQ:tildezeta} can be split into the sum of two frequency multiplier functions as 
$$\tilde{\zeta}=\zeta+\zeta_{\mathcal{K}},$$ 
which are defined in \eqref{FREQ:def.zeta} and \eqref{FREQ:zetaK.def} respectively. These multiplier functions satisfy the following asymptotics: 
\begin{equation}
\label{FREQ:Paos}
\begin{split}
&|\zeta_\mathcal{K}(p)|\leq C_\varepsilon {(p^0)}^{\frac{\singS}{2}+\varepsilon}
\hspace{5mm}\text{and}\hspace{5mm}
\zeta(p)\approx {(p^0)}^{\frac{{\singS+\gamma}}{2}}.
\end{split}
\end{equation}
Here, for any small $\varepsilon>0$ there exists a finite constant $C_\epsilon>0$ as above.
\end{theorem}

\begin{theorem}
	\label{thm1}
	We have the following uniform estimate
	 $$
	|\langle \Gamma(f,h),\eta\rangle| \lesssim |f|_{L^2}|h|_{I^{\singS,\gamma}}|\eta|_{I^{\singS,\gamma}}.
	$$
\end{theorem}

\begin{lemma}
	\label{Lemma1}
	Suppose that $|\alpha|\leq N$ with $N\geq 2$   and $l\geq 0$. Then we have the estimate
	$$
	|\left(w^{2l} \partial^\alpha\Gamma(f,h),\partial^\alpha\eta\right)| \lesssim \|f\|_{H^{N}_l}\|h\|_{I^{\singS,\gamma}_{l,N}}\|\partial^\alpha\eta\|_{I^{\singS,\gamma}_l}.
$$

\end{lemma}

\begin{lemma}
	\label{Lemma2}
	We have the uniform inequality for $\mathcal{K}$ that
	 $$
	|\langle w^{2l}\mathcal{K}f,f\rangle | \leq \epsilon|  f|^2_{I^{\singS,\gamma}_l}+C_\epsilon|  f|^2_{L^2(B_{C_\epsilon})}
	$$
	where $\epsilon>0$ is any small number and $C_\epsilon>0$ is a finite constant.  
\end{lemma}

\begin{lemma}
	\label{Nfupperboundlemma}
	We have the uniform inequality for $\mathcal{N}$ that
	 $$
	|\langle w^{2l} \mathcal{N}f,f\rangle | \lesssim |f|^2_{I^{\singS,\gamma}_l}.
	$$
	
\end{lemma}
\begin{lemma}
	\label{Nfcoercivitylemma}
	We have the uniform coercive lower bound estimate:
	 $$
	\langle w^{2l} \mathcal{N}f,f\rangle \gtrsim |f|^2_{I^{\singS,\gamma}_l}-C|f|^2_{L^2(B_C)}
	$$
	for some $C\geq0$ and for any $l\in\mathbb{R}$. If $l=0$, then we can take $C=0$.
\end{lemma}

Lemma \ref{Nfupperboundlemma} and Lemma \ref{Nfcoercivitylemma}  together imply that the norm piece is comparable to the fractional Sobolev norm $I^{\singS,\gamma}$ as
\begin{equation}\label{NL2equiv}
\langle \mathcal{N}f,f\rangle \approx |f|^2_{I^{\singS,\gamma}}.
\end{equation}
We lastly have the coercive inequality for the linearized Boltzmann operator:
\begin{lemma}
	\label{2.10}
	For some $C>0$, we have the uniform lower bound 
	 $$
	\langle w^{2l} Lf,f\rangle \gtrsim |f|^2_{I^{\singS,\gamma}_l}-C|f|^2_{L^2(B_C)}. 
	$$
\end{lemma}
Note that Lemma \ref{2.10} is a direct consequence of Lemmas \ref{Lemma2} and  \ref{Nfcoercivitylemma} simply because $L=\mathcal{K}+\mathcal{N}$ from \eqref{defK} and \eqref{defN}.

\subsection{Main difficulties and our strategy} In this section, we will explain  the main difficulties that we have experienced and how we resolved those issues. And we will further explain several new ideas and techniques that we developed in the course of the proof.  We will begin with the following discussion regarding the sharp linearized norm for the Newtonian Boltzmann equation \eqref{Boltz.Newtonian} in comparison to the relativistic Boltzmann equation \eqref{rBoltz.eqn}.

We mention that the unique global solutions to the Newtonian non-cutoff Boltzmann equation constructed in \cite{MR2784329} depend on the non-isotropic geometric fractional Sobolev space $N^{s,\singS}$ with the following norm: 
\begin{equation}\label{Nsrho.norm}
 |f|^2_{N^{s,\singS}}
 \eqdef 
 |f|^2_{L^2_{\singS+2s}}
 +
 \int_\rth dv\int_\rth dv'\ \frac{(f(v')-f(v))^2}{d(v,v')^{3+2s}}(\seq{v'}\seq{v})^{\frac{{\singS+2s}}{2}}1_{d(v,v')\leq 1}.
 \end{equation} 
 Above the parameters satisfy $s\in (0,1)$ and $\rho > -3$, and $\seq{v} = \sqrt{1+|v|^2}$.  Further $d(v,v')$ is an anisotropic metric on the ``lifted" paraboloid:
 $$d(v,v')\eqdef \sqrt{|v-v'|^2+\left(\frac{|v|^2}{2}-\frac{|v'|^2}{2}\right)^2}.$$ 
 Note that the inclusion of the quadratic difference in the metric is essential and it is not a lower-order term.  Further \cite{MR2807092} shows that the sharp diffusive behavior of the full nonlinear Newtonian Boltzmann collision operator \eqref{Boltz.Newtonian} is the same as the linearized norm in \eqref{Nsrho.norm}.

Alexandre, Morimoto, Ukai, Xu, and Yang also proved the global existence of unique solutions to the non-cutoff Boltzmann equation \cite{MR2795331,MR2863853,MR2793203,MR2847536}, and they used the triple norm which is given by
\begin{multline}\label{triple.norm}
    ||| f ||| 
=
\int_ {\mathbb{R}^3}dv \int_ {\mathbb{R}^3} dv_*\int_{\mathbb{S}^2}d\sigma B(v-v_*,\sigma)
\\
\times
\left[\mu(v_*)(f(v')-f(v))^2+f^2(\sqrt{\mu(v')}-\sqrt{\mu(v)})^2\right].
\end{multline}
This norm $    ||| f ||| $ can be shown to be equivalent to $ |f|^2_{N^{s,\singS}}$ in  \eqref{Nsrho.norm} as in \cite{MR2795331,MR2863853,MR2793203,MR2847536,MR2784329}.

Further Alexandre, H{\'e}rau and Li \cite{MR3950012} have derived the sharp linearized diffusive behavior using pseudo-differential operators as follows:
\begin{multline}\label{PDO.form}
-\langle \mathcal{L} f, f\rangle + 
|f|_{L^2_{l}}
\\
\approx
\int_ {\mathbb{R}^3}dv 
\sep{  \seq{v}^{\rho } |\seq{D_v}^{s}f(v)|^2
 + \seq{v}^{\rho} |\seq{v \wedge D_v}^{s}f(v)|^2
 + \seq{v}^{\rho + 2s}|f(v)|^2 
 }.
\end{multline}
Here $\mathcal{L}$ is the linearized Newtonian Boltzmann collision operator \cite{MR3950012}.
This expression holds for Schwartz functions $f$ and for any  $l \in \R$, where the implicit constant will depend upon $l$.  Since these three  expressions \eqref{Nsrho.norm}, \eqref{triple.norm} and \eqref{PDO.form} all in some sense sharply characterize the diffusion of the linearized Dirichlet form $\langle \mathcal{L} f, f\rangle$ for the Newtonian Boltzmann equation \eqref{Boltz.Newtonian}, then they equivalently show the non-isotropic behavior of the fractional diffusion in the Newtonian case.

 On the other hand, when it comes to the relativistic situation, due to the collisional geometry as in \eqref{conservation} the analogous metric $d(p^\mu, p'^\mu)$ to \eqref{Nsrho.norm} is the metric on the ``lifted" hyperboloid between two energy-momentum 4-vectors:
 $$d(p^\mu, p'^\mu)\eqdef \sqrt{|p-p'|^2+|p^0-p'^0|^2}.$$ 
  Note that, different from the non-relativistic analogue \mbox{$|\cdot|_{N^{s, \rho}}$} in \eqref{Nsrho.norm}, our semi-norm for the relativistic case \eqref{fractional} behaves like a weighted fractional Sobolev norm, as we will observe that the Euclidean distance $d(p^\mu,p'^\mu)$ between energy-momentum 4-vectors $p^\mu$ and $p'^\mu$ on a ``lifted'' hyperboloid in $\rfo$  is indeed equivalent to the standard 3-dimensional Euclidean distance $|p-p'|$. This is because we have 
$$
|p-p'|^2\leq (p^0-p'^0)^2+|p-p'|^2\eqdef d(p^\mu,p'^\mu)^2,
$$
and that when $p\neq p'$ we have
\begin{equation}\notag
|p^0-p'^0|=\frac{||p|^2-|p'|^2|}{p^0+p'^0}= \frac{||p|-|p'||(|p|+|p'|)}{p^0+p'^0}< ||p|-|p'||\leq |p-p'|.\end{equation} 
These two expressions together result in
\begin{equation}
\label{simple.calc2}|p-p'|\leq d(p^\mu,p'^\mu)\leq \sqrt{2}|p-p'|.
\end{equation}
This discussion shows that in the special relativistic situation the bounded momentum, $p/\pZ$, and the linearly growing collisional energy conservation, $p'^0+q'^0= p^0 + q^0$, cause the non-cutoff diffusion of the relativistic Boltzmann collision operator to be isotropic as in \eqref{fractional}.  This contrasts with the Newtonian case \eqref{Boltz.Newtonian} where the momentum, $v$, grows linearly and the collisional energy conservation, $|v'|^2+|v_*'|^2=|v|^2+|v_*|^2$, is quadratic and the fractional diffusion is non-isotropic as in \eqref{Nsrho.norm}--\eqref{triple.norm}--\eqref{PDO.form}.  The isotropic diffusion in the special relativistic case allows us to use standard Littlewood-Paley operators when we prove our main estimates and avoids the complexity of non-isotropic diffusion.

Even with the isotropic diffusion, there are major difficulties in the special relativistic case in merely establishing the required cancellation estimates as we will now explain.    Indeed in 1991 Glassey and Strauss \cite[Proof of Theorem 2]{MR1105532} calculated a sharp estimate which showed that the $p$-derivatives of $p'$ and $q'$ exhibit some momentum growth in the Glassey-Strauss coordinates \cite[Equation (1.18)]{MR2891870}.  This momentum growth of the first derivatives of $p'$ and $q'$, which does not occur in the Newtonian case, can be highly problematic as seen in \cite{MR2891870}.

Further, to establish the cancellation estimates of the non-local diffusion the standard approach is to do a change of variables of the form $p' \to p$ or $q' \to q$.  This was the foundation of the ``cancellation lemma'' \cite{ADVW} for the Newtonian Boltzmann equation \eqref{Boltz.Newtonian}.  This Newtonian cancellation lemma can be stated as 
\begin{equation}\label{Newtonian.cancellation.lemma}
\int_\rth dv
\int_{\mathbb{S}^2}d\sigma ~
B(v-v_*,\sigma)(F(v')-F(v))
=
(F* S)(v_*),
\end{equation}
where 
\begin{equation}\notag
    S(z)
=
C_3|z|^\rho, \quad 0<C_3<\infty.
\end{equation}
This holds for a general class of functions $F$ for a Newtonian collision kernel such as 
$B(v-v_*,\sigma) \approx |v-v_*|^\rho \theta^{-2-2s}$ with $\rho > -3$ and $s \in (0,1)$.  The main tool in proving this cancellation lemma is the change of variables $v' \to v$ with Jacobian determinant
\begin{equation}\notag
\left|\frac{dv'}{dv} \right| = \frac{1}{4}(\cos(\theta/2))^2
\ge \frac{1}{8}>0, \quad 0 \le \theta \le \frac{\pi}{2}.
\end{equation}
This Newtonian cancellation lemma \eqref{Newtonian.cancellation.lemma} and the associated change of variables $v' \to v$, and generalizations, have been a foundation for proving cancellation estimates for the Newtonian Boltzmann equation without angular cutoff.

However the analogous relativistic Jacobian determinant has been shown to be highly problematic.   Indeed, it has been numerically calculated using high precision arithmetic recently in \cite{ChapmanJangStrain2020} that the Jacobian determinant $\left|\frac{\partial p'}{\partial p}\right|$ has a huge number of distinct points at which it is essentially zero.  This motivated us to look for  a counter-example.

Now in \secref{sec:frequency} we will introduce a counter-example to a relativistic  cancellation lemma, such as \eqref{Newtonian.cancellation.lemma},  in the following sense.  We formally write down the following relativistic quantity:
    \begin{equation}\notag
    	\tilde{\zeta}^B(p) \eqdef
	\frac{1}{2}
\int_{\rth}dq\int_{\mathbb{S}^2}d\omega\  v_{\text{\o}} \sigma(g,\theta)(J(q)-J(q')) = \tilde{\zeta}^B_1(p) - \tilde{\zeta}^B_2(p).
    \end{equation} 
Recall the relativistic Maxwellian, $J(q)$, (a Schwartz function) is given by \eqref{jutter.equilibrium}.  Then in \secref{sec:frequency} we have shown for a fixed constant $c'>0$ that
\begin{multline}\notag
\tilde{\zeta}^B_1(p)=\frac{c'}{p^0}e^{\frac{p^0}{2}}\int_{\rth}\frac{dq}{q^0}\frac{e^{-\frac{1}{2}q^0}}{g} \int_{0}^{\infty} \frac{rdr}{\sqrt{r^2+s}}s_\lambda\sigma(g_\lambda,\theta_\lambda)\\\times \exp\left(-\frac{p^0+q^0}{2\sqrt{s}}\sqrt{r^2+s}\right)I_0\left(\frac{|p\times q|}{g\sqrt{s}}r\right).
\end{multline}
We take this opportunity to record the following modified Bessel functions:
\begin{equation}\label{bessel0}
    I_0(y)\eqdef \frac{1}{2\pi}\int_0^{2\pi} d\phi \exp(y\cos\phi),
\quad 
I_1(y)=\frac{1}{2\pi} \int_0^{2\pi} d\phi \ \cos\phi
\exp(y\cos\phi).
\end{equation}
We also define the notations
\begin{equation}\notag  
g^2_\lambda = g^2+\frac{1}{2}\sqrt{s}(\sqrt{r^2+s}-\sqrt{s}), \quad s_\lambda = g^2_\lambda + 4,
\end{equation}
and 
\begin{equation}\notag 
\cos\theta_\lambda =\frac{2g^2}{g^2_\lambda}-1
=\frac{g^2-\frac{1}{2}\sqrt{s}(\sqrt{r^2+s}-\sqrt{s})}{g^2+\frac{1}{2}\sqrt{s}(\sqrt{r^2+s}-\sqrt{s})}.
\end{equation}
Then with all the terms defined we observe that $\tilde{\zeta}^B_1(p)$ represents  a finite integral because it contains exponential decay in both the $q$ and the $r$ variables.   However, if we consider the term in $\tilde{\zeta}^B_2(p)$ with $J(q')$ only, then in \secref{sec:frequency} we derive that
\begin{equation}\notag
\tilde{\zeta}^B_2(p)=\frac{c'}{p^0}\int_{\rth}\frac{dq}{q^0}\frac{e^{-q^0}}{g} \int_{0}^{\infty} \frac{rdr}{\sqrt{r^2+s}}s_\lambda\sigma(g_\lambda,\theta_\lambda).
\end{equation}
Here we can assume that the angular kernel $\sigma_0$ from \eqref{angassumption} is for example pointwise bounded.   Then we see that the $dr$ integration becomes infinite in $\tilde{\zeta}^B_2(p)$, since this term no longer contains sufficient decay in the $r$ variable.\footnote{Although notice that one could make the $dr$ integration finite by artificially assuming rapid decay in \eqref{soft}.  But this is outside the range of the physical assumptions.}  Therefore $\tilde{\zeta}^B_2(p) = \infty$, and we can rigorously justify this argument using standard approximation procedures.  This shows that the relativistic analog of the Newtonian cancellation lemma in \eqref{Newtonian.cancellation.lemma} is false.

\begin{remark}\label{FREQ:cancel.problem}
This illustrates that the crucial change of variable in cancellation lemma, $q' \to q$ (or $(q', \omega) \to (q, k)$ for some $|k|=1$) as stated in \cite{ADVW} for the Newtonian Boltzmann equation, is not well defined in the special relativistic case.  This statement is further independent of our choice of coordinate representations of $(p', q')$, such as for example \eqref{p'}.\footnote{We refer to \cite{MR2765751} for a discussion of a variety of coordinate representations of the relativistic Boltzmann collision operator.} Indeed, if the change of variables $q'\to q$ held with an integrable Jacobian then $\tilde{\zeta}^B_2(p)$ would be finite. 
\end{remark}

Since we do not have a cancellation lemma, or the crucial change of variables from  $q' \to q$ or  $p' \to p$, instead we introduce the following novel series of changes of variables in order to estimate the cancellation of the fractional derivatives in the relativistic Boltzmann equation.  If we consider the norm term from \eqref{defN}, and we take the $L^2(\mathbb{R}^3_p)$ inner product with $\eta(p)$ then we will have to estimate a term like 
\begin{equation}\label{typical.term.1}
    \frac{1}{2}\int_\rth dp ~\eta(p)\int_\rth dq\int_{\mathbb{S}^2} d\omega\ v_{\text{\o}}\sigma(g,\theta)(f(p')-f(p))\sqrt{J(q')J(q)}.
\end{equation}
For a term such as this one, we may use the reduction of the delta function of the conservation laws \eqref{conservation} in the center-of-momentum frame as in \cite{MR635279}, \cite[(5.39)]{MR2707256} and \cite[Theorem 1.2]{MR2765751} to obtain the following representation of an operator:

\begin{lemma}[Center-of-momentum reduction, Theorem 1.2 of \cite{MR2765751}]\label{lemma.reduction} 
   For an integrable function $G:\rfo\times\rfo\times\rfo\times \rfo \to \mathbb{R}$, we have
\begin{multline*}\int_\rth \frac{dp'}{p'^0}\int_\rth \frac{dq'}{q'^0}s\sigma(g,\theta)\delta^{(4)}(p^\mu+q^\mu-p'^\mu-q'^\mu)G(p^\mu,q^\mu,p'^\mu,q'^\mu)\\
= \frac{1}{2}\int_{\mathbb{S}^2}d\omega~ g\sqrt{s}~\sigma(g,\theta) G(p^\mu,q^\mu,p'^\mu,q'^\mu).\end{multline*} 
Above we assume that the function $G(p^\mu,q^\mu,p'^\mu,q'^\mu)$ has sufficient vanishing conditions so that the integrals above are well-defined.  We also use \eqref{p'} and \eqref{q'} to define $p'^\mu$ and $q'^\mu$ in the second integral.
\end{lemma}

Then we can use Lemma \ref{lemma.reduction}  to write \eqref{typical.term.1} as follows
\begin{multline}\label{typical.term.2}
\int_\rth \frac{dp}{p^0}\int_\rth \frac{dq}{q^0}
\int_\rth \frac{dp'}{p'^0}\int_\rth \frac{dq'}{q'^0}
\delta^{(4)}(p^\mu+q^\mu-p'^\mu-q'^\mu) s \sigma(g,\theta)
\\
\times \eta(p) (f(p')-f(p))\sqrt{J(q')J(q)}.
\end{multline}
To better understand this term, we now introduce the Carleman representation of the collision operator as follows:

\begin{lemma}[Carleman representation]\label{lemma.carleman}
We have the following equality:
\begin{multline}\notag
\int_{\mathbb{R}^3}\frac{dp}{{p^0}}
\int_{\mathbb{R}^3}\frac{dq}{{q^0}}
\int_{\mathbb{R}^3}\frac{dp'\ }{{p'^0}}
\int_{\mathbb{R}^3}\frac{dq'\ }{{q'^0}}
s\sigma(g,\theta)\delta^{(4)}(p'^\mu+q'^\mu-p^\mu-q^\mu)G(p,q,p')\\
=\int_{\mathbb{R}^3}\frac{dp}{{p^0}}\int_{\mathbb{R}^3}
\frac{dp'\ }{{p'^0}}\int_{E^{q}_{p'-p}} \frac{d\pi_{q}}{8\bar{g}{q^0}}s\sigma(g,\theta)G(p,q,p'),
\end{multline}
where we assume that $G$ has a sufficient vanishing condition so the integral is well-defined.  Here $E^q_{p'-p}$ is the two-dimensional hypersurface for relativistic collisions which is defined as $$E^q_{p'-p}=\{q\in\mathbb{R}^3:(p'^\mu-p^\mu)(p_\mu+q_\mu)=0\}.$$ 
Further $q'^\mu$ is defined by \eqref{conservation}, and the measure can be represented as
\begin{equation}\label{def.pi.meas}
    d\pi_{q}=dq\ u({p^0}+{q^0}-{p'^0})\delta\left(\frac{\bar{g}}{2}+\frac{q^\mu(p_\mu-p'_\mu)}{\bar{g}}\right),
\end{equation}
where above we use the function $u$ which is defined by 
\begin{equation}\label{def.u}u(x)=0 \text{ if }x<0\text{, and }u(x)=1\text{ if }x \ge 0.\end{equation} 
We also recall \eqref{gbar}.
\end{lemma}

The proof of Lemma \ref{lemma.carleman} is given in \secref{CarlemanAppendix.Carleman}.   Then using this Carleman representation from  Lemma \ref{lemma.carleman} we can represent \eqref{typical.term.2} as
\begin{equation}\label{typical.term.3}
\frac{1}{8}\int_{\mathbb{R}^3}\frac{dp}{{p^0}}\eta(p)  \int_{\mathbb{R}^3}\frac{dp'\ }{{p'^0}} (f(p')-f(p)) 
\frac{s}{\bar{g}}\frac{\sqrt{J(p)}}{\sqrt{J(p')}}\int_{E^{q}_{p'-p}} \frac{d\pi_{q}}{{q^0}}\sigma(g,\theta)
J(q).
\end{equation}
Typically we can estimate integrals such as the one above in $d\pi_{q}$.  And a term such as $\frac{\sqrt{J(p)}}{\sqrt{J(p')}}$ is bounded when we are close to the singularity and $|p'-p|$ is small.   Then we are left with linear dependence on $p$ and $p'$ in the remaining integrals of $dp~ dp'$.  These are the main ideas in one method that we developed to prove the required cancellation estimates.  Recalling the  bilinear operator \eqref{Gamma1}, this change of variables procedure allows us to estimate terms such as $\langle \Gamma(f,h),\eta \rangle$ when we estimate the cancellation using the function $\eta$.

However proving the upper bound estimates by performing summation of the Littlewood-Paley decomposed pieces of the trilinear estimates in \secref{subsec.upperbd} also requires an alternative cancellation estimate (Proposition \ref{prop:cancellation2}) of the difference of trilinear forms $(T^{k,l}_+-T^{k,l}_-)(f,h,\eta)$ as in \eqref{T+++} while estimating the cancellation using the function $h$.  Since the standard integration by parts does not work for the integral which contains fractional derivatives, we had to derive a second representations of the trilinear inner product of \eqref{Gamma1} as $\langle \Gamma(f,h), \eta\rangle$.  To achieve this, we derived the dual representation $T_f^*h$ of the abstract operator $T_f\eta$ for each fixed $f$ such that $$\langle \Gamma(f,h) ,\eta\rangle =\langle T_f\eta,h\rangle=\langle \eta, T_f^*h\rangle. $$ 
In the relativistic case the expression involving $T_f^*h$ is rather complicated, whose integrands involve special functions and Lorentz transformations.  This expression is written precisely in \eqref{dual4} and in decomposed form in \eqref{T+++.dual}, and more generally in Lemma \ref{transformation.Lemma.appendix}.   To the best of our knowledge, this ``dual representation'' has not been previously derived for the relativistic Boltzmann operator.    Now we note that the methods that we used in \secref{sec:frequency} for the dual-type representation including the splitting of the region of the $q$ integration into $\qlep$, and  $\qgep$ for some large $m\ge 1$ does not work for the sharp upper-bound estimates for similar types of operators to \eqref{T+++.dual}, since the factor $\sqrt{J(q)}$ in the operators from \eqref{T+++.dual} does not provide sufficient decay in the $|q|$ variable to control the integral if we implement the methods used in \secref{sec:frequency}.  Instead the breakthrough idea in our estimates involved the spliting of the upper-bound estimate for each operator as in \eqref{estimate.key.term} into two terms $D_1$ and $D_2$.  Then
  $D_1$ can be reverted back to the original representation and estimated using a strategy analogous to \eqref{typical.term.1}-\eqref{typical.term.2}-\eqref{typical.term.3}.  And  $D_2$ stays in the dual representation. In this way one can push all the polynomial growth in the $q$ variable into the original representation in $D_1$ which has some leftover exponential decay such that one can intentionally create the otherwise missing polynomial decay in $D_2$.    This idea was very effective for the control of Part $II$ in \eqref{define.Tkl2} and Part $III$ in \eqref{define.Tkl3} in \secref{sec.newcancel} even though they still need much more delicate sophistication on the choices in the splitting such that all the estimates work for the full range  of the soft interaction $0>\singS + \gamma > -3/2$ in \eqref{soft} and the hard interaction $0\le \singS+\gamma<2+\gamma$ in \eqref{hard}.  For example, we include the artificial terms $1=\left(\frac{\tilde{s}}{p'^0q^0}\right)^{7/4}\left(\frac{\tilde{s}}{p'^0q^0}\right)^{-7/4}$ for the Part $II$ estimates and $\left(\frac{\tilde{g}^2}{p'^0q^0}\right)^{3/4}\left(\frac{\tilde{g}^2}{p'^0q^0}\right)^{-3/4}$  for the Part $III$ estimates, and these choices are sharp in the sense that these powers $\frac{7}{4}$ and $\frac{3}{4}$ and the choices $\frac{\tilde{s}}{p'^0q^0}$ and $\frac{\tilde{g}^2}{p'^0q^0}$ are the only possible choices that can control each decomposed piece in our estimates.  For the  full details we refer to \secref{sec.newcancel}.

In addition to the previously discussed techniques for establishing the cancellation estimates, the proofs of our main upper-bound estimates in Theorem \ref{thm1} and Lemma \ref{Lemma1} use a dyadic decomposition of the linearized operator $\Gamma$ and its kernel $\sigma(g,\theta)$ nearby the angular singularity $\theta=0$ since the angular kernel $\sigma_0(\theta)$ is not integrable by itself. In \secref{main estimates}, we start by showing that $\theta\approx \frac{\bar{g}}{g}$ and consider the dyadic decomposition around $\bar{g}$.  Then for $\bar{g}\approx 2^{-k}$, we estimate the upper-bounds of the trilinear forms in \eqref{T+++} for the gain term $T^{k,l}_+$ and the loss term $T^{k,l}_-$ of the linearized operator $\Gamma$ separately.  If we consider the region nearby the singularity where $\bar{g}\approx 2^{-k}$ for $k>0$, then we cannot simply separate the gain and loss terms.  Thus in \secref{Cancellation} we rewrite the difference of trilinear forms $(T^{k,l}_+-T^{k,l}_-)(f,h,\eta)$ in terms of the difference $\eta(p')-\eta(p)$ or $J(q')-J(q)$ of the pre- and the post-collisional momenta and work to obtain extra decaying factors of $|p'-p|$ or $|q'-q|$ to reduce the order of angular singularity.  Then in \secref{sec.newcancel}, we prove the cancellation estimates using the dual representation in \eqref{T+++.dual}.  Then in \secref{subsec.additional}, we prove further estimates, including cancellation estimates, for the compact operator $\mathcal{K}$ from \eqref{defK}.    In order to manage the dyadic sum of the momentum derivatives, we further explain the isotropic Littlewood-Paley decomposition inequalities that we will use in \secref{LP decomp}. This decomposition allows us to bound the sum of those decomposed pieces containing momentum-derivatives above by the terms in our weighted fractional derivative norm $|\eta|_{I^{\singS,\gamma}}$ and $|h|_{I^{\singS,\gamma}}$.  Then in \secref{subsec.upperbd}, we use triple sum estimates together with all the previous estimates in this section in order to prove the main upper-bound estimates of Theorem \ref{thm1} and Lemma \ref{Lemma1}.

Additional difficulties regarding proving upper-bound estimates occur because the lower bound of the relativistic version of the relative momentum $g=g(p^\mu,q^\mu)$ from \eqref{g} depends on the weights of $p$ and $q$. More precisely, one has the following sharp inequality \cite{MR1211782}:
\begin{equation}\label{g.ineq.sharp}
\frac{\sqrt{|p-q|^2+|p\times q|^2}}{\sqrt{p^0q^0}}\leq \sqrt{2\frac{|p-q|^2+|p\times q|^2}{ \pZ  \qZ +p\cdot q+1}} =  g   \leq |p-q|.
\end{equation}
The proof of \eqref{g.ineq.sharp} directly follows from \eqref{g} and the Cauchy-Schwartz inequality: $|1+p\cdot q|\le \pZ \qZ$.
This affects the lower-bound estimates of the hard-interaction case \eqref{hard} for $g^{\singA+\gamma}$ and the upper-bound estimates in the soft-interaction case \eqref{soft} for $g^{-\singB+\gamma}$ because $ \singA+\gamma\ge 0$ and $-\singB+\gamma<0$. The emergence of the extra weight $(p^0q^0)^{-\frac{1}{2}}$ in the lower bound of each relative momentum $g$ causes the crucial difference in the order of $p$, and we resolved this issue using another inequality of $g\le \sqrt{s}\lesssim \sqrt{p^0q^0}$ instead of using the standard $g\le |p-q|$ in many estimates. This is one of the major differences from the non-relativistic case where each relative momentum $|p-q|$ creates one power of $p^0$, whereas each relative momentum $g(p^\mu,q^\mu)$ corresponds to only a half-power of $p^0$ growth.

Also, the appearance of extra momentum weights, $w^l(p)$, in our norms introduce additional difficulties.   These weights are necessary because the $L^2$ energy functional $\mathcal{E}(t)$ and the dissipation functional $\mathcal{D}(t)$ do not satisfy $\mathcal{E}(t)\lesssim \mathcal{D}(t)$ in the soft-interaction \eqref{soft} case when ${-\singB+\gamma}<0$ and this results in $| f|_{L^2_{{-\singB+\gamma}}}<|f|_{L^2}$. In order to overcome the difficulty of not having $|f|_{L^2_{{-\singB+\gamma}}}\geq |f|_{L^2}$, we put the extra weights $w^l$ into our norms and interpolate with stronger norms as in (\ref{key soft}) to obtain the $L^2$ decay-in-time estimate with the polynomial rates.  This interpolation technique had been developed in \cite{SG-CPDE}.  However this is extremely complicated and delicate in the relativistic situation because of the difficult algebraic structure that is present in the equation.

We hope that these results will be useful to study many further mathematical problems in the relativistic Kinetic theory such as relativistic fluid limit problems, the Newtonian limit, the relativistic Boltzmann equation coupled with relativistic matter models \cite{MR356790} to name a few.

\subsection{Outline of the article}  In this section we will outline the rest of this article.  

In \secref{sec:frequency} we prove the sharp asymptotics for the \textit{frequency multiplier} to obtain coercivity estimates for the linearized collision operator. Namely, we prove Theorem \ref{FREQ:main.thm}. 
To prove Theorem \ref{FREQ:main.thm}, we will use the two different representations that we have given in \secref{FREQ:sec:main.decomp} and these will be derived in \secref{FREQ:sec:derivation}.  We will further follow the proof strategy that will be outlined in \secref{FREQ:sec:proof.outline}.  We  prove that $\zeta$ from \eqref{FREQ:def.zeta} has a leading order positive lower bound in Proposition \ref{FREQ:prop.coercive} in \secref{FREQ:sec:leadingorder lower bound zeta}.  Then we will prove that $\zeta_0$  from \eqref{FREQ:zeta0By} has the leading order upper bound in Proposition \ref{FREQ:prop.zeta0.asymptotic} in \secref{FREQ:sec:fullsharpupper zeta0}.    In \secref{FREQ:sec:low order zeta1}, we prove in Proposition \ref{FREQ:prop.zetaL.asymptotic}  that $\zetaL(p)$ from \eqref{FREQ:zetaLBy} has a lower order upper bound and we further prove in Proposition \ref{FREQ:prop.zetaL.asymptoticnew} that $\zetaTL(p)$ from \eqref{FREQ:eq:tildezetanew1L} has a lower order upper bound.  

In \secref{main upper} we prove the main upper bound estimates on the linearized \eqref{L.def} and non-linear collision operator \eqref{Gamma1} that are stated in \secref{subsec.mainest}.  In particular in \secref{main estimates}, we will start by introducing the dyadic decomposition of the angular singularity in the non-linear collision operator. In the rest of \secref{main estimates}, we will make upper-bound estimates on the decomposed pieces of gain and loss terms away from the angular singularity.  Then in \secref{Cancellation}, we prove the cancellation estimates. Namely, we use the cancellation of $f(p)-f(p')$ in the region $p\simeq p'$ to cancel the angular singularity and to obtain the upper bound estimates for the decomposed pieces nearby the angular singularity. The proof heavily depends on the use of certain Lorentz transformations and the relativistic Carleman-type dual representation from \secref{CarlemanAppendix}.  In \secref{sec.newcancel}, we perform upper bound estimates that incorporate cancellation on the dual expression from \eqref{T+++.dual}.   In \secref{subsec.additional}, we prove additional estimates for the compact operator $\mathcal{K}$.   Those final upper-bounds contain momentum derivatives of the functions and the sum of those upper-bounds will further be bounded above in terms of our weighted fractional derivative norms via Littlewood-Paley type arguments.   In \secref{LP decomp}, we explain the main Littlewood-Paley inequalities that we will use to prove our main estimates.   Then in \secref{subsec.upperbd}, we use triple sum estimates to establish the final main upper-bound estimates on the linearized collision operators $\mathcal{K}$, $\mathcal{N}$, and the nonlinear operator $\Gamma$ using the upper-bound estimates on the decomposed pieces from the previous sections.

In \secref{main coercive estimates}, we prove the coercive lower bound of the norm part $\langle \mathcal{N}f,f\rangle$. We also show that the norm part $\langle \mathcal{N}f,f\rangle$ is comparable to the weighted fractional Sobolev norm $|\cdot|_{I^{\singS,\gamma}}$.  In this section, we use the Fourier redistribution argument from \cite{MR2784329}.

In \secref{globalExist}, we finally use the standard iteration method and the uniform energy estimate for the iterated sequence of approximate solutions to prove the local existence.  Our proof of the global well-posedness in \secref{globalExist} uses the nonlinear energy method introduced in \cite{MR2000470}.  In particular we derive the relativistic system of macroscopic equations and local conservation laws.   And we use these to prove that the local solutions are global by the standard continuity argument and the energy estimates. We also show that the $L^2$ functional of solutions decays exponentially in time for the hard-interactions \eqref{hard} and decays polynomially in time to zero for the soft-interactions \eqref{soft}.

In \secref{CarlemanAppendix}, we derive the relativistic Carleman-type dual representation for the gain and loss terms  and obtain the dual formulation of the trilinear form,
which are used in many places in the previous sections.

In \secref{FREQ:sec:derivation} we provide full derivations of the two different representations that we have given in \secref{FREQ:sec:main.decomp} for the proofs of the sharp asymptotics for the \textit{frequency multiplier} and the coercivity estimates. 

In \secref{FREQ:sec:pointwiseProof} we provide the proofs of the pointwise estimates: Lemma \ref{FREQ:lem:useful.ests}, Lemma \ref{FREQ:lem:integral.ests}, and Lemma \ref{pointwise.lemma2}.

\subsection{A brief description of Lorentz transformations}
In this section we define several notations and conventions involving Lorentz transformations which will be used in several key places throughout the article.

Let $\Lambda$ be a $4\times 4$ matrix (of real numbers) denoted by
$$
\Lambda=(\Lambda_{\ \nu}^{\mu})_{0\le \mu,\nu\le 3}.
$$
For the basics of Lorentz transformations, we refer to \cite{MR1898707} and \cite{Weinberg1972}.  We use the convention that the top index $\mu$ denotes the row of the matrix, and the bottom index $\nu$ denotes the column of the matrix.  We will also use the vector notation 
$$\Lambda^\mu\eqdef (\Lambda^\mu{}_0,\Lambda^\mu{}_1,\Lambda^\mu{}_2,\Lambda^\mu{}_3), \text{ for }\mu=0,1,2,3,$$
to express the $\mu$-th row of $\Lambda$.  
We will further use the notation $\Lambda$ to exclusively denote a Lorentz transformation.

\begin{definition}\label{rcop:LTdef}
$\Lambda$ is a (proper) Lorentz transformation if $\mbox{det}(\Lambda) = 1$ and 
$$
\Lambda_{\ \mu}^{\kappa}  \eta_{\kappa \lambda}  \Lambda_{\ \nu}^{\lambda} = \eta_{\mu \nu}, \qquad (\mu, \nu = 0,1,2,3).
$$
In matrix notation this can also be written $\Lambda^\top D \Lambda = D$, for $D=\text{diag}(-1,1,1,1).$ 
This condition implies the invariance:
$
p^\mu q_\mu = p^\mu \eta_{\mu\nu } q^{\nu} 
= 
(\Lambda^{\kappa}_{~\mu}  p^{\mu})\eta_{\kappa\lambda } (\Lambda^{\lambda}_{~\nu}  q^{\nu} ).
$
\end{definition}

In this paper we will use several times a specific Lorentz transformation $\Lambda=(\Lambda^\mu{}_\nu)_{0\le \mu,\nu\le 3}$ which maps into the ``center of momentum'' system as 
\begin{equation}\label{center.momentum.condition}
    A^\mu\eqdef \Lambda^\mu{}_\nu(p^\nu+q^\nu) =(\sqrt{s},0,0,0),\quad B^\mu\eqdef -\Lambda^\mu{}_\nu (p^\nu-q^\nu)=(0,0,0,g).
\end{equation}
The first condition in \eqref{center.momentum.condition} is the one that means that you are mapping the particle momentum to the center of momentum system where $p+q=0$. The second condition in \eqref{center.momentum.condition} is extremely useful for the changes of variables that we will use in the rest of this paper.

The explicit form of the matrix $\Lambda$ satisfying \eqref{center.momentum.condition} was derived in \cite[Section 5.3.1.3]{MR2707256}, it was also written in \cite[p. 593]{MR2728733}.
More precisely, we have
\begin{equation}\label{eq.LT}
\Lambda=(\Lambda^{\mu}{}_\nu)=\left(\begin{array}{cccc}\frac{p^0+q^0}{\sqrt{s}}& -\frac{p_1+q_1}{\sqrt{s}} &-\frac{p_2+q_2}{\sqrt{s}} &-\frac{p_3+q_3}{\sqrt{s}}\\ \Lambda^1{}_0 &\Lambda^1{}_1&\Lambda^1{}_2&\Lambda^1{}_3\\ 0 & \frac{(p\times q)_1}{|p\times q|} &\frac{(p\times q)_2}{|p\times q|}&\frac{(p\times q)_3}{|p\times q|}\\ \frac{p^0-q^0}{g} &-\frac{p_1-q_1}{g}&-\frac{p_2-q_2}{g}&-\frac{p_3-q_3}{g}\end{array}\right),
\end{equation}
with the second row given by
$$\Lambda^1{}_0=\Lambda^1{}_0(p,q)=\frac{2|p\times q|}{g\sqrt{s}},$$ 
and for $i=1,2,3$ we have 
$$\Lambda^1{}_i=\Lambda^1{}_i(p,q)=\frac{2\left(p_i\{p^0+q^0p^\mu q_\mu\}+q_i\{q^0+p^0p^\mu q_\mu\}\right)}{g\sqrt{s}|p\times q|}.$$
This Lorentz transformation satisfies \eqref{center.momentum.condition}.

Now any Lorentz transformation, $\Lambda$, is invertible and
the inverse matrix is denoted $\Lambda^{-1} = (\Lambda_{\mu}^{~\nu})_{0\le \mu,\nu\le 3}$ where we denote $(\Lambda^{-1})^{\nu}_{~\mu}  = \Lambda^{~\nu}_{\mu}$
so that 
$
\Lambda_{~\kappa}^{\nu}\Lambda_{\mu}^{~\kappa}= \delta_{~\mu}^{\nu},
$
where 
$
\delta_{~\mu}^{\nu}
$
is the standard Kronecker delta which is unity when the indices are equal and zero otherwise.
It follows from Definition \ref{rcop:LTdef} that
\begin{equation}\notag
(\Lambda^{-1})^{\nu}_{~\mu}
=
\Lambda_{\mu}^{~\nu}= \eta^{\nu \lambda }~  \Lambda_{~ \lambda }^{\kappa} ~  \eta_{\kappa \mu}.
\end{equation}
Definition \ref{rcop:LTdef} further implies that  $(\Lambda^{-1})^{\nu}_{~\mu}  = \Lambda^{~\nu}_{\mu}$ 
is a Lorentz transformation.  We can then directly calculate the inverse of \eqref{eq.LT} as 
\begin{equation}\notag
\Lambda^{-1}=(\Lambda_\nu{}^{\mu})
=
\left(\begin{array}{cccc}
\frac{p^0+q^0}{\sqrt{s}}& -\Lambda^1{}_0 & 0 & -\frac{p^0-q^0}{g}
\\
\frac{p_1+q_1}{\sqrt{s}} & \Lambda^1{}_1& \frac{(p\times q)_1}{|p\times q|} & -\frac{p_1-q_1}{g}
\\
\frac{p_2+q_2}{\sqrt{s}} & \Lambda^1{}_2& \frac{(p\times q)_2}{|p\times q|}&  -\frac{p_2-q_2}{g}
\\
\frac{p_3+q_3}{\sqrt{s}} & \Lambda^1{}_3& \frac{(p\times q)_3}{|p\times q|} &  -\frac{p_3-q_3}{g}
\end{array}\right).
\end{equation}
We will use the Lorentz transformation \eqref{eq.LT} and it's inverse in several of the proofs of the estimates below.

\subsection{Preliminary lemmas}
Here we introduce several useful pointwise estimates that will be used throughout the paper. The proofs of the following lemmas (Lemma \ref{FREQ:lem:useful.ests}, Lemma \ref{FREQ:lem:integral.ests}, and Lemma \ref{pointwise.lemma2}) will be given in \secref{FREQ:sec:pointwiseProof}.

\begin{lemma}\label{FREQ:lem:useful.ests}
With the notations \eqref{s} and \eqref{g} we have
\begin{equation}\label{FREQ:s.ge.g2}
    s=g^2+4 \ge \max\{g^2,4\},
\end{equation}
and
\begin{equation}\label{FREQ:s.le.pq}
    s\le 4p^0 q^0.
\end{equation}
We trivially conclude from \eqref{FREQ:s.le.pq}-\eqref{FREQ:s.ge.g2} that 
\begin{equation}\label{FREQ:g.le.sqrtpq}
    g\lesssim \sqrt{p^0 q^0}
\end{equation}
Recalling \eqref{g.ineq.sharp} we further have
\begin{equation}\label{FREQ:g.ge.lower}
\frac{|p-q|}{\sqrt{p^0q^0}}\leq g,
\end{equation}
and
\begin{equation}\label{FREQ:g.ge.2lower}
\frac{|p\times q|}{\sqrt{p^0q^0}}\leq g,
\end{equation}
and
\begin{equation}\label{FREQ:g.le.upper}
  g   \leq |p-q|.
\end{equation}
We also have
\begin{equation}\label{FREQ:p0q0.le.pq}
  |p^0 - q^0|   \leq |p-q|,
\end{equation}
and
\begin{equation}\label{FREQ:p0.plus.q0.le.p0q0}
  p^0 + q^0   \leq 2p^0 q^0.
\end{equation}

We now state a few pointwise estimates for \eqref{FREQ:lj}.  
We have the inequality:
\begin{equation}\label{FREQ:j.le.l}
    j\le l.
\end{equation}
Further
\begin{equation}\label{FREQ:l.upper.ineq}
j^2   \leq \frac{1}{2}p^0 q^0, \quad l   \leq \frac{1}{2}p^0 q^0.
\end{equation}
Also
\begin{equation}\label{FREQ:l2j2}
	l^2-j^2=\frac{(p^0+q^0)^2g^2-4|p\times q|^2}{16g^2}=\frac{s}{16g^2}|p-q|^2,
\end{equation} 
and
\begin{equation}\label{FREQ:l2j2size}
	\sqrt{l^2-j^2}=|p-q|\frac{\sqrt{g^2+4}}{4g}\geq \frac{1}{4}|p-q|.
\end{equation} 
Next for $g_\Lambda^2$ defined in \eqref{FREQ:g2y.variable} we have
\begin{equation}\label{FREQ:ineq.gL.here}
g^2 \max\{\sqrt{y^2+1}, \sqrt{2}\} g^2
\lesssim
 g_\Lambda^2 \lesssim   s_\Lambda \lesssim s \sqrt{y^2+1}, \quad \forall 0 \le y \le \infty.
\end{equation}

If the notations $l$, $j$, and $g_\Lambda$ are defined instead as 
\begin{equation}\label{lj}
l=l(p',q)\eqdef \frac{p'^0+q^0}{4},\ j=j(p',q)\eqdef \frac{|p'\times q|}{2\tilde{g}},\text{ and}
\end{equation}
\begin{equation}\label{g2y.variable}
    g^2_\Lambda = \tilde{g}^2+\frac{\tilde{s}}{2}(\sqrt{|z|^2+1}-1), \quad s_\Lambda = g^2_\Lambda + 4,\quad g_L^2 = \frac{1}{2}\tilde{s}(\sqrt{|z|^2+1}-1),
\end{equation}
 (cf. \eqref{g2}, \eqref{g2.eq.lambda} and \eqref{theta.eq.lambda} with $r=\sqrt{s}|z|$), then we have 
\begin{equation}\label{newcos.intro}
\cos\theta_\Lambda=\frac{2\tilde{g}^2}{g^2_\Lambda}-1
=
\frac{\tilde{g}^2-\frac{\tilde{s}}{2}(\sqrt{|z|^2+1}-1)}{\tilde{g}^2+\frac{\tilde{s}}{2}(\sqrt{|z|^2+1}-1)},
\end{equation}
\begin{equation}\label{j.le.l}
    j\le l,
\end{equation}
\begin{equation}\label{l.upper.ineq}
   j^2   \lesssim p'^0 q^0, \quad l   \lesssim p'^0 q^0,
\end{equation}
\begin{equation}\label{l2j2}
	l^2-j^2=\frac{(p'^0+q^0)^2\tilde{g}^2-4|p'\times q|^2}{16\tilde{g}^2}=\frac{\tilde{s}}{16\tilde{g}^2}|p'-q|^2,
\end{equation} 
\begin{equation}\label{l2j2size}
	\sqrt{l^2-j^2}=|p'-q|\frac{\sqrt{\tilde{g}^2+4}}{4\tilde{g}}\geq \frac{1}{4}|p'-q|,
\end{equation} and
\begin{equation}\label{ineq.gL.here}
\tilde{g}^2 \max\{\sqrt{|z|^2+1}, \sqrt{2}\} 
\lesssim
 g_\Lambda^2 \lesssim   s_\Lambda \lesssim \tilde{s} \sqrt{|z|^2+1}, \quad \forall 0 \le |z| \le \infty.
\end{equation}
\end{lemma}

Further, we will also need sharp estimates of the following integrals
\begin{equation}\label{FREQ:int.Kgamma}
    \bar{K}_\gamma(l,j)\eqdef \int_0^1dy\  y^{1-\gamma}  \exp(- l\sqrt{y^2+1})I_0(jy),
\end{equation}
where $\gamma \in (0,2)$
and
\begin{equation}\label{FREQ:J2.special}
J_2(l,j)\eqdef \int_{0}^\infty\ dy\ \frac{y\exp\left(-l\sqrt{y^2+1}\right) I_0(j y)}{\sqrt{1+y^2}}
\end{equation} 
Also define
\begin{equation}\label{FREQ:tildeK2def}
\tilde{K}_2(l,j)\eqdef \int_{0}^\infty\ dy\ y(y^2+1)^{1/2}\exp\left(-l\sqrt{y^2+1}\right) I_0(j y).
\end{equation} 

These integrals are known from \cite{Gradshteyn:1702455} and \cite{MR1211782}.  In particular a proof of Lemma \ref{FREQ:lem:integral.ests} below is given by combining the results from \cite[Lemma 3.5, Lemma 3.6 and Corollary 2]{MR1211782}.   We give the following lemma and proof for completeness.

\begin{lemma}\label{FREQ:lem:integral.ests}
For both \eqref{lj} and \eqref{FREQ:lj}, we have
\begin{equation}\label{FREQ:max.bound}\max_{0\leq x\leq 1}\exp(-l\sqrt{x^2+1}+jx) \lesssim \exp(-\sqrt{l^2-j^2}).\end{equation}
Then for \eqref{FREQ:int.Kgamma}, we have the uniform estimate
	\begin{equation}\label{FREQ:smally.lemma}
|\bar{K}_\gamma(l,j)|
	\lesssim
	 \exp\left(-\sqrt{l^2-j^2}\right).
	\end{equation}
For \eqref{FREQ:J2.special} we have the exact formula
\begin{equation}\label{FREQ:J2.lemma}
J_2(l,j)= (\sqrt{l^2-j^2})^{-1}\exp(-\sqrt{l^2-j^2}),
\end{equation} 
and then for \eqref{FREQ:tildeK2def} we have the formula
\begin{multline}\label{FREQ:k2lj.lemma}
\tilde{K}_2(l,j)= (\sqrt{l^2-j^2})^{-5}\exp(-\sqrt{l^2-j^2})\\\times \left((l^2-j^2+3\sqrt{l^2-j^2}+3)l^2-(l^2-j^2)-(\sqrt{l^2-j^2})^3\right).
\end{multline}	
\end{lemma}
    
    In addition, we have the following pointwise estimates:
    \begin{lemma}\label{pointwise.lemma2}
    If $k>-3$ then we have
        \begin{equation}\label{jutter.integral.est}
        \int_{\rth}dq\sqrt{J(q)}|p-q|^k\approx {(p^0)}^k.
    \end{equation}
    Also, for \eqref{Moller.def} we have
        \begin{equation}\label{moller.upper.est}
        v_{\text{\o}} \lesssim 1.
    \end{equation}
    In addition, if the collision kernel $\sigma(g,\theta)$ is supported only when $\cos\theta \ge 0,$ then we have 
    \begin{equation}
        \label{gtildeg.equiv}
        \tilde{g}\approx g.
    \end{equation}
    For \eqref{g2y.variable} we have \begin{equation}\label{frac sg.est}0\le \left(\frac{\tilde{s} \Phi(\tilde{g})\tilde{g}^4}{s_\Lambda \Phi(g_\Lambda)g^4_\Lambda}\right)\le 1,\end{equation} and
    \begin{equation}\label{colfre4.2}
    	\left|\frac{\tilde{s} \Phi(\tilde{g})\tilde{g}^4}{s_\Lambda \Phi(g_\Lambda)g^4_\Lambda} - 1\right|\lesssim \frac{\tilde{s}(\sqrt{|z|^2+1}-1)}{\tilde{g}^2_\Lambda},
\end{equation} for both hard and soft interactions \eqref{hard} and \eqref{soft}.  Also, using \eqref{conservation} we have \begin{equation}\label{pp'q'} p^0\leq p'^0+q'^0\leq 2p'^0q'^0\end{equation}
and 
\begin{equation}\label{eq.pqp'q'}
   p^\mu q_\mu=p'^\mu q'_\mu, \quad p'^\mu q_\mu=p^\mu q'_\mu, \quad p'^\mu p_\mu=q'^\mu q_\mu.
\end{equation}
\end{lemma}

We remark that \eqref{eq.pqp'q'} directly implies that $$
(p'^\mu-q^\mu)(p'_\mu-p_\mu)=0.$$

These pointwise estimates above will be used crucially for main upper-bound and lower-bound estimates in the rest of this paper.  We now also introduce a lifting of the 6-fold integral below into a 8-fold integral:

\begin{lemma}[Claim (7.5) of \cite{MR2728733}]\label{7.5ofCMP} 
    Let $\underline{g}=g(p'^\mu,q'^\mu)$ and $\underline{s}=s(p'^\mu,q'^\mu).$   Recall \eqref{def.u}.
Then we have for a function $G=G(p^\mu,q^\mu,p'^\mu,q'^\mu)$ that 
$$\int_\rth \frac{dp'}{p'^0}\int_\rth \frac{dq'}{q'^0}G(p^\mu,q^\mu,p'^\mu,q'^\mu)=\frac{1}{16} \int_{\mathbb{R}^4\times\mathbb{R}^4}d\Theta(p'^\mu,q'^\mu)\ G(p^\mu,q^\mu,p'^\mu,q'^\mu),$$ where $$d\Theta(p'^\mu,q'^\mu)\eqdef dp'^\mu dq'^\mu u(p'^0+q'^0)u(\underline{s}-4)\delta(\underline{s}-\underline{g}^2-4)\delta\left((p'^\mu+q'^\mu)(p'_\mu-q'_\mu)\right).$$ Note that $\underline{g}=g$ under \eqref{conservation}. 
\end{lemma}

  We will also need to use the following alternative integral formula's in our estimates.  We define the following integral
\begin{equation}\label{original.eq.Ig}
        I_G \eqdef 
\int_\rth dp \int_\rth dq \int_{\mathbb{S}^2}d\omega \ v_{\text{\o}} \sigma(g,\theta) G(q,p',p,q'^0). 
\end{equation}
Here we use the variables \eqref{p'} and \eqref{q'}, and we assume that $G(q,p',p,q'^0)$ is a Schwartz function for which the integral above is well defined.  

\begin{lemma}\label{transformation.Lemma.appendix}
For $I_G$ from \eqref{original.eq.Ig} we can alternatively represent the integral as
    \begin{multline}\label{second.eq.Ig}
    I_G \\ = 
\frac{c'}{2}\int_\rth\frac{dp'}{p'^0}\int_{\rth}\frac{dq}{q^0} \frac{\sqrt{\tilde{s}}}{\tilde{g}}
\int_{\mathbb{R}^2} \frac{dz}{\sqrt{|z|^2+1}}s_\Lambda\sigma(g_\Lambda,\theta_\Lambda)
  G(q,p',p'+A,q^0+A^0), 
\end{multline}
where we assume that $G=G(q,p',p'+A,q^0+A^0)$ is a Schwartz function for which the integrals \eqref{original.eq.Ig} and \eqref{second.eq.Ig} are well defined.  The constant satisfies $c'>0$.  This is the case if the function $G$ has suitable cancellation so that the integral in \eqref{second.eq.Ig} is finite.

On the other hand, more generally (for $c'>0$) we have that $I_G$ from \eqref{original.eq.Ig} can be alternatively expressed as 
    \begin{multline}\label{third.eq.Ig}
    I_G  = 
\frac{c'}{2}\int_\rth\frac{dp'}{p'^0}\int_{\rth}\frac{dq}{q^0} \frac{\sqrt{\tilde{s}}}{\tilde{g}}
\int_{\mathbb{R}^2} \frac{dz}{\sqrt{|z|^2+1}}s_\Lambda\sigma(g_\Lambda,\theta_\Lambda)
  \bigg[
  G(q,p',p'+A,q^0+A^0)
  \\
  - \frac{\tilde{s} \Phi(\tilde{g})\tilde{g}^4}{s_\Lambda \Phi(g_\Lambda)g^4_\Lambda}
  G(q,p',p',q^0)\bigg].
\end{multline}
The formula above may be used in the case that $G$ may not have enough cancellation by itself to make the integral above well defined.

 The transformation from \eqref{original.eq.Ig} to \eqref{second.eq.Ig} or \eqref{third.eq.Ig} has the following mapping properties for the variables $p$, $q$, $p'$, and $q'$  in \eqref{original.eq.Ig} from \eqref{p'} and \eqref{q'} etc:  It sends $q^\mu\to q^\mu$, $p'^\mu \to p'^\mu$, $p\to p'+A$, $q'^0\to q^0 + A^0$, $g \to g_\Lambda$, $s \to s_\Lambda$, $\theta \to \theta_\Lambda$, $\bar{g} \to g_L$,  where we use the definitions  \eqref{g2y.variable} or \eqref{g2.eq.lambda}, \eqref{newcos.intro} or \eqref{theta.eq.lambda},  \eqref{A.def.lambda} and \eqref{Azero.def.lambda}.
\end{lemma}

Lemma \ref{transformation.Lemma.appendix} is proven in  \secref{app.dual.rep}.      In the next subsection, we introduce additional conventions of notations that we use throughout the paper.

\subsection{Further notations}
We call the change of pre- and post-collisional momentum variables $(p,q)\mapsto (p',q')$ as the pre-post change of variables. It is known from \cite{MR1105532} that the Jacobian for this change of variables is equal to 
\begin{equation}\label{prepost.change}
    \left|\frac{\partial (p',q')}{\partial (p,q)}\right|=\frac{p'^0q'^0}{p^0q^0}.
\end{equation}
We note that this Jacobian is calculated in \cite{MR1105532} in the Glassey-Strauss coordinates as written in \cite[Page 98, Equation (3.359)]{MR1379589}.  Then it is explained in \cite[Equation (23)]{MR2765751} how to also use this change of variable in the center-of-momentum coordinates \eqref{p'} and \eqref{q'} that are used in this paper.

We further introduce the following notation. 
Given $h_1=h_1(p,q)$, we define the function $h=h(p)$ as an integral on $\rth$ as 
$$h(p)=\int_\rth h_1(p,q)dq.$$ In this case, for $A>0$ we split the integral into 
	$$\int_\rth = \int_{|q|\ge A} +\int_{|q|< A},$$ and abuse the notations to denote each term as
	\begin{equation}\label{convention}
	\begin{split}[h]_{|q|\ge A} &= |h|_{|q|\ge A} \eqdef \int_{|q|\ge A} h_1(p,q)dq,
	\text{ and }\\
	[h]_{|q|< A} &= |h|_{|q|< A} \eqdef  \int_{|q|< A} h_1(p,q)dq.\end{split}
	\end{equation}
This convention of the notations will be used in a few convenient places in the rest of this paper.

\section{Frequency multiplier estimates}
\label{sec:frequency}

The existence theory for the Boltzmann equation without angular cutoff was developed in the class of weak solutions via the method of renormalization \cite{MR1857879}.  Further the existence and uniqueness theory was developed using the energy method via linearization nearby Maxwellian equilibrium in \cite{MR2784329, MR2795331, MR2793203, MR2847536}.  This current paper is mainly concerned with the energy method nearby equilibrium.  One of the most crucial parts in the proof via the energy method is to create a positive dissipation term in the energy inequality. It turns out that the coercivity estimates for the dissipation term crucially depends on the asymptotics of the \textit{frequency multiplier} \eqref{FREQ:tildezeta} whose explicit form will be introduced later on \eqref{FREQ:def.zeta} and \eqref{FREQ:zetaK.def}.

Regarding the Newtonian Boltzmann equation \eqref{Boltz.Newtonian}, the estimates on the asymptotics of the \textit{frequency multiplier} have been proved by Pao \cite{Pao} using the symmetry of the linearized operator and using the sharp pointwise estimates of certain special functions.  This can also be proven using the procedure outlined in \secref{FREQ:sec:main.difficult}. These asymptotics have been crucially used in the coercivity estimates and the spectral theory for the linearized Boltzmann operator in \cite{MR2322149}.  These coercivity estimates, and in addition the Newtonian cancellation lemma from \cite{ADVW} has been crucially used for the proof of the global in time wellposedness in \cite{MR2784329,MR2795331,MR2793203} nearby equilibrium.

In this section, we are interested in proving analogous results for the relativistic Boltzmann equation \eqref{RBE}. Namely, we would like to establish the estimates on the asymptotics of the relativistic frequency multipliers for the linearized Boltzmann operator.  However, in the relativistic case, it turns out  that the collisional structure is substantially different \cite{ChapmanJangStrain2020}, and the crucial change of variable $p'\to p$ in the non-relativistic cancellation lemma does not hold in the relativistic case (which is explained in Remark \ref{FREQ:cancel.problem}).  This also shows the major difficulty in the relativistic case versus  the non-relativistic case (in regards to the lack of the change of variables from $p' \to p$) and in regards to the inability to use the standard proof of the behavior of the frequency multiplier term $\zeta$ from the non-relativstic case.

We believe that this estimate can be useful to study many problems in the relativistic Kinetic theory.  The sharp asymptotic leading order estimate in Theorem \ref{FREQ:main.thm} should be useful in mathematical studies of relativistic fluid limit problems, the Newtonian limit, the relativistic Boltzmann equation coupled with relativistic matter models to name a few.

\begin{remark}\label{FREQ:leading.order.remark}
Throughout this section, based upon \eqref{FREQ:Paos}, we call a term $A(p)$ a \textit{leading order} term, if $A(p)\approx (p^0)^{\frac{\singS+\gamma}{2}}.$  In addition, we call a term $B(p)$ a \textit{lower order} term if  for some positive constant $\epsilon>0$ there exists a finite constant $C_\epsilon>0$ such that $|B(p)|\leq C_\epsilon (p^0)^{\frac{\singS+\gamma}{2}-\epsilon}$. 
\end{remark}

\subsection{Comparison to the Newtonian case}\label{FREQ:sec:main.difficult}
In contrast to \cite{Pao}, one can prove the asymptotic behavior of the non-relativistic frequency multiplier in the following simple way.

In the non-relativisitic case \eqref{Boltz.Newtonian} the analog of the collision frequency multiplier \eqref{FREQ:tildezeta} is given  \cite[Page 11]{MR2784329} by
\begin{equation}\label{Boltz.Newtonian.nu}
\tilde{\nu}(v)=\int_{\rth}dv_*\int_{\mathbb{S}^2}d\omega\ B(v-v_*,\omega)
\left(\sqrt{\mu(v_*)}-\sqrt{\mu(v'_*)}\right)\sqrt{\mu(v_*)},
\end{equation}
where the Newtonian Maxwellian equilibrium is given by $$\mu(v)\eqdef (2\pi)^{-3/2} \exp\left(-|v|^2/2\right).$$  Note the similarity to  $\tilde{\zeta}$ in \eqref{FREQ:tildezeta}.  In the non-relativistic case, due to symmetry the following decomposition of the frequency multiplier is very useful:
$$
\left(\sqrt{\mu(v_*)}-\sqrt{\mu(v'_*)}\right)\sqrt{\mu(v_*)}
=    
\frac{1}{2}\left(\sqrt{\mu(v_*)}-\sqrt{\mu(v'_*)}\right)^2 
+ \frac{1}{2} \left(\mu(v_*)-\mu(v'_*)\right).
$$ 
This decomposition allows the splitting 
$$
\tilde{\nu}(v) = \nu(v) + \nu_{\mathcal{K}}(v),
$$
where 
\begin{equation}\notag
\nu(v)\eqdef
\frac{1}{2}
\int_{\rth}dv_*\int_{\mathbb{S}^2}d\omega\  B(v-v_*,\omega)
\left(\sqrt{\mu(v_*)}-\sqrt{\mu(v'_*)}\right)^2.
\end{equation}
Now $\nu(v)$ is clearly non-negative and it can be quickly shown that $\nu(v)$ has the expected leading order asymptotic behavior as $|v|\to\infty$.  

On the other hand from \eqref{Boltz.Newtonian.nu} we have
\begin{equation}\notag
\nu_{\mathcal{K}}(v) = 
\frac{1}{2}
\int_{\rth}dv_*\int_{\mathbb{S}^2}d\omega\  B(v-v_*,\omega)
\left(\mu(v_*)-\mu(v'_*)\right).
\end{equation}
Now one can  use the Newtonian cancellation lemma \cite{ADVW}, the change of variable from $v_*' \to v_*$, to show for some $C'>0$ that
\begin{equation}\notag
\nu_{\mathcal{K}}(v) = 
C'
\int_{\rth}dv_*\ \mu(v_*) \Psi(|v-v_*| ).
\end{equation}
This expression directly implies that $\nu_{\mathcal{K}}(v)$ has lower order asymptotic behavior as $|v|\to\infty$.  This decomposition is crucial to designing a norm that captures the sharp behavior of the linearized collision operator and to further prove the global in time existence of solutions nearby equilibrium.

Unfortunately in the relativistic case this approach fails as we now explain.  We recall that the main difference in the integrand of $\tilde{\zeta}$ in \eqref{FREQ:tildezeta} is $$ \sqrt{J(q)}(\sqrt{J(q)}-\sqrt{J(q')}).$$
The analogous splitting in the relativistic case is
\begin{equation}\label{FREQ:rel.difference}
    (\sqrt{J(q)}-\sqrt{J(q')})  \sqrt{J(q)}
=    
\frac{1}{2}\left(\sqrt{J(q)}-\sqrt{J(q')}\right)^2 
+ \frac{1}{2} \left(J(q)-J(q')\right).
\end{equation}
However, this decomposition does not help in the relativistic case and it is also closely related to the fact that the crucial change of variables $q'\to q$ in the cancellation lemma \cite{ADVW} is problematic in the relativistic case \cite{ChapmanJangStrain2020} even in the case with an angular cutoff.  We now provide the sketch of the argument.

We ignore the positive square term, $\left(\sqrt{J(q)}-\sqrt{J(q')}\right)^2$, and focus on the second term on the right side in \eqref{FREQ:rel.difference}. We can write this term from \eqref{FREQ:tildezeta} as
\begin{equation}\notag
	\tilde{\zeta}^B(p) \eqdef
	\frac{1}{2}\int_{\rth}dq\int_{\mathbb{S}^2}d\omega\  v_{\text{\o}} \sigma(g,\theta)(J(q)-J(q')) = \tilde{\zeta}^B_1(p) - \tilde{\zeta}^B_2(p).
\end{equation} 
We can further assume that the angular kernel $\sigma_0$ from \eqref{define.kernel} is just pointwise bounded (we do not need to assume that it is mean-zero) with an angular cutoff.  Then the term on the right side of \eqref{FREQ:rel.difference} containing $J(q)$ in \eqref{FREQ:tildezeta} corresponds to \eqref{FREQ:Ilossrepresentation} in  \secref{FREQ:sec:alternative.deriv}.  Then \eqref{FREQ:Ilossrepresentation} is transformed into \eqref{FREQ:final loss} so that
\begin{multline}\notag
\tilde{\zeta}^B_1(p)=\frac{c'}{p^0}e^{\frac{p^0}{2}}\int_{\rth}\frac{dq}{q^0}\frac{e^{-\frac{1}{2}q^0}}{g} \int_{0}^{\infty} \frac{rdr}{\sqrt{r^2+s}}s_\Lambda\sigma(g_\Lambda,\theta_\Lambda)\\\times \exp\left(-\frac{p^0+q^0}{2\sqrt{s}}\sqrt{r^2+s}\right)I_0\left(\frac{|p\times q|}{g\sqrt{s}}r\right).
\end{multline}
Above $c'>0$, and we further use the notations \eqref{FREQ:g2y.variable} below with $r=y\sqrt{s}$ in addition to the Bessel function \eqref{bessel0}.  We point out that the above is a finite integral since it contains exponential decay in both the $q$ and the $r$ variables in \eqref{FREQ:final loss}.  However, if we look at the new loss term with $J(q')$ only, then if we follow the same derivation, then we obtain \eqref{FREQ:final loss} without the exponential term and without the Bessel function $I_0$.  Indeed following the transformation procedure in  \secref{FREQ:sec:alternative.deriv} we obtain
\begin{equation}\notag
\tilde{\zeta}^B_2(p)=\frac{c'}{p^0}\int_{\rth}\frac{dq}{q^0}\frac{e^{-q^0}}{g} \int_{0}^{\infty} \frac{rdr}{\sqrt{r^2+s}}s_\Lambda\sigma(g_\Lambda,\theta_\Lambda).
\end{equation}
Thus we see following this procedure that the $dr$ integration becomes infinite in $\tilde{\zeta}^B_2(p)$, since this term no longer contains sufficient decay in the $r$ variable.  Therefore, the whole integral becomes infinity unless we artificially assume that the kernel $\sigma$ decays very rapidly for large $r$.  (We mention that this can be directly justified using standard approximation procedures.)

If the cancellation lemma were true then the integrand in $\tilde{\zeta}^B_2(p)$ would be integrable and $\tilde{\zeta}^B_2(p)$ would be finite.  Indeed, the expression $\tilde{\zeta}^B_2(p)$ integrates to infinity by this argument.  The factor $J(q')$ does not provide sufficient decay to control the integral. That is why such a decomposition, which was very effective in the Newtonian case, does not help in the relativistic case. It is adding and subtracting a term which is infinity.

\subsection{Main decomposition}\label{FREQ:sec:main.decomp}
Instead we perform in  \secref{FREQ:sec:derivation} the following transformation of \eqref{FREQ:tildezeta} as $\tilde{\zeta}=\zeta_0+\zetaL$  where for $c'>0$ we have
\begin{multline}\label{FREQ:zeta0By}
\zeta_0 \eqdef \frac{c'}{p^0}\int_{\rth}\frac{dq}{q^0}\frac{e^{-q^0}\sqrt{s}}{g}\int_{0}^\infty \frac{ydy}{\sqrt{y^2+1}}s_\Lambda\sigma(g_\Lambda,\theta_\Lambda) \frac{s\Phi(g)g^4}{s_\Lambda  \Phi(g_\Lambda)g^4_\Lambda}
\\
\times
\Big[1 - \exp\left(l(1-\sqrt{y^2+1})\right)I_0\left(jy\right) \Big],
\end{multline}and
\begin{multline}\label{FREQ:zetaLBy}
\zetaL \eqdef \frac{c'}{p^0}\int_{\rth}\frac{dq}{q^0}\frac{e^{-q^0}\sqrt{s}}{g}\int_{0}^\infty \frac{ydy}{\sqrt{y^2+1}}s_\Lambda\sigma(g_\Lambda,\theta_\Lambda)   \\
\times
\exp\left(l(1-\sqrt{y^2+1})\right) I_0\left(jy\right)
\left( \frac{s\Phi(g)g^4}{s_\Lambda  \Phi(g_\Lambda)g^4_\Lambda} -1\right).
\end{multline}
We recall the modified Bessel function from \eqref{bessel0}.
These expressions arise by applying the change of variables $r\mapsto y=\frac{r}{\sqrt{s}}$, to the expressions \eqref{FREQ:zeta0B.appendix} and \eqref{FREQ:zetaLB}.  Above we use the notations $l$ and $j$ that are defined as
\begin{equation}\label{FREQ:lj}
l=l(p,q)\eqdef \frac{p^0+q^0}{4},\ \text{and}\ j=j(p,q)\eqdef \frac{|p\times q|}{2g}.
\end{equation}
We also further define the notations
\begin{equation}\label{FREQ:g2y.variable}
    g^2_\Lambda = g^2+\frac{s}{2}(\sqrt{y^2+1}-1), \quad s_\Lambda = g^2_\Lambda + 4,
\end{equation}
and from \eqref{FREQ:newcos} we have 
\begin{equation}\notag 
\cos\theta_\Lambda=\frac{2g^2}{g^2_\Lambda}-1
=
\frac{g^2-\frac{s}{2}(\sqrt{y^2+1}-1)}{g^2+\frac{s}{2}(\sqrt{y^2+1}-1)}.
\end{equation}
It is shown in \secref{FREQ:sec:low order zeta1} that $\zetaL(p)$ in \eqref{FREQ:zetaLBy} has lower order asymptotic behavior, and in \secref{FREQ:sec:fullsharpupper zeta0} we see that the main part of $\zeta_0(p)$ in \eqref{FREQ:zeta0By} has a leading order upper bound.

We remark that the dynamics of each decomposed piece of $\tilde{\zeta}$ in \eqref{FREQ:zeta0By} and \eqref{FREQ:zetaLBy}  are essentially depending on the integral domain for the $q$ and $y$ variables, and the Bessel function $I_0(jy)$ from \eqref{bessel0}, which makes them extremely complex.  It turns out that the major difficulty involves the difference inside the inside the integrand in \eqref{FREQ:zeta0By}:
\begin{equation}\notag
    \Big[1 - \exp\left(l(1-\sqrt{y^2+1})+jy \cos\phi\right)\Big].
\end{equation}
This expression is zero at $y=0$ and converges to one as $y\to\infty$.  However this expression also 
has it's negative minimum at 
\begin{equation}\label{FREQ:y.min.difference}
y_{m} = \frac{j|\cos\phi|}{\sqrt{l^2 - j^2 \cos^2 \phi}},
\end{equation}
and the difference above remains negative, for say $\cos\phi>0$, until
\begin{equation}\label{FREQ:y.zero.change}
y_{s} = \frac{2l j \cos\phi}{l^2 - j^2 \cos^2 \phi},
\end{equation}
where $0 \le y_{m} \le y_{s} \to 0$ as $l \to \infty$.  Thus the whole integral in \eqref{FREQ:zeta0By} is negative in a large region nearby the minimum of the difference and nearby the singularity of the kernel $\sigma(g_\Lambda, \theta_\Lambda)$ at $y=0$.     This region where the integrand is negative and close to it's minimum makes it extremely problematic to prove the required leading order asymptotic positive lower bound.  Therefore, it is unclear from this point of view whether or not \eqref{FREQ:tildezeta} or  \eqref{FREQ:zeta0By} is positive or a leading order term.  It is essential to have a positive leading order term for the collision frequency multiplier in order to obtain the sharp behavior of the linearized collision operator.

For this reason we needed to derive another representation of $\tilde{\zeta}$ from \eqref{FREQ:tildezeta}.   As in \eqref{FREQ:eq:tildezetanew1.appendix}, we can alternatively write $\tilde{\zeta}$ as
\begin{multline}\notag
\tilde{\zeta}(p)=\frac{c'}{\pi p^0}\int_{\rth}\frac{dq}{q^0}\frac{e^{-q^0}\sqrt{s}}{g}\int_{0}^\infty \frac{ydy}{\sqrt{y^2+1}}s_\Lambda\sigma(g_\Lambda,\theta_\Lambda)\int_0^{\pi} d\phi\\\times\Big[\exp (2l-2l\sqrt{y^2+1}+2jy\cos\phi)-\exp(l-l\sqrt{y^2+1} +jy\cos\phi)\Big],
\end{multline}
where we use the notations \eqref{FREQ:lj} and \eqref{FREQ:g2y.variable} and $c'>0$.
Then $\tilde{\zeta}$ can be decomposed  into two terms $\tilde{\zeta}=\tilde{\zeta}_0+\tilde{\zeta}_L$ as 
\begin{multline}\label{FREQ:eq:tildezetanew1}
\tilde{\zeta}_0(p)=\frac{c'}{\pi p^0}\int_{\rth}\frac{dq}{q^0}\frac{e^{-q^0}\sqrt{s}}{g}\int_{0}^\infty \frac{ydy}{\sqrt{y^2+1}}s_\Lambda\sigma(g_\Lambda,\theta_\Lambda)\int_0^{\pi} d\phi\frac{s\Phi(g)g^4}{s_\Lambda  \Phi(g_\Lambda)g^4_\Lambda}  \\\times\Big[\exp (2l-2l\sqrt{y^2+1}+2jy\cos\phi)-\exp(l-l\sqrt{y^2+1} +jy\cos\phi)\Big],
\end{multline}
and
\begin{multline}\label{FREQ:eq:tildezetanew1L}
\tilde{\zeta}_L(p)
=\frac{c'}{ p^0}\int_{\rth}\frac{dq}{q^0}\frac{e^{-q^0}\sqrt{s}}{g}\int_{0}^\infty \frac{ydy}{\sqrt{y^2+1}}s_\Lambda\sigma(g_\Lambda,\theta_\Lambda) \left(1-\frac{s\Phi(g)g^4}{s_\Lambda  \Phi(g_\Lambda)g^4_\Lambda}  \right) \\\times\Big[\exp (2l-2l\sqrt{y^2+1})I_0(2jy)-\exp(l-l\sqrt{y^2+1})I_0(jy)\Big].
\end{multline}
The expression \eqref{FREQ:eq:tildezetanew1} looks like a better candidate for the leading order term because the difference 
$$
\Big[\exp (2l-2l\sqrt{y^2+1}+2jy\cos\phi)-\exp(l-l\sqrt{y^2+1} +jy\cos\phi)\Big]
$$
now has it's maximum, $y_m$, at \eqref{FREQ:y.min.difference} and it is positive on $0 \le y \le y_s$ from \eqref{FREQ:y.zero.change}.  However this difference is negative for $y \ge y_s$ and still remains large and negative in a significantly sized region after $y$ passes $y_s$ which causes further extreme difficulties in proving a leading order positive lower bound.

Instead the key point will be that when we add $\zeta_0$ and $\tilde{\zeta}_0$ then the resulting term is clearly positive and leading order.  And we will later show that $\tilde{\zeta}_L$ from \eqref{FREQ:eq:tildezetanew1L} and $\zeta_L$ from \eqref{FREQ:zetaLBy} are lower order terms.  In particular we can take the summation of $\frac{1}{2}\zetaZ(p)1_{|p|\ge 1}$ from \eqref{FREQ:zeta0By} and $\frac{1}{2}\zetaTZ(p)1_{|p|\ge 1}$ from \eqref{FREQ:eq:tildezetanew1}. By adding an additional term $\langle p \rangle^{(\rho+\gamma)/2} 1_{|p|\le 1}$,
we obtain that
\begin{multline}\label{FREQ:def.zeta.mod}
\zeta'(p) \eqdef \left[\frac{\zetaZ(p)+\zetaTZ(p)}{2}\right]1_{|p|\ge 1}
+\langle p \rangle^{(\rho+\gamma)/2} 1_{|p|\le 1}\\
=
\frac{c'1_{|p|\ge 1}}{2\pi p^0}\int_{\rth}\frac{dq}{q^0}\frac{e^{-q^0}\sqrt{s}}{g}\int_{0}^\infty \frac{ydy}{\sqrt{y^2+1}}s_\Lambda\sigma(g_\Lambda,\theta_\Lambda)\frac{s\Phi(g)g^4}{s_\Lambda  \Phi(g_\Lambda)g^4_\Lambda}\int_0^{\pi} d\phi\\\times\Big[\exp (2l-2l\sqrt{y^2+1}+2jy\cos\phi)-2\exp(l-l\sqrt{y^2+1} +jy\cos\phi)+1\Big]\\+\langle p \rangle^{(\rho+\gamma)/2} 1_{|p|\le 1}\\
=
\frac{c'1_{|p|\ge 1}}{2\pi p^0}\int_{\rth}\frac{dq}{q^0}\frac{e^{-q^0}\sqrt{s}}{g}\int_{0}^\infty \frac{ydy}{\sqrt{y^2+1}}s_\Lambda\sigma(g_\Lambda,\theta_\Lambda)\frac{s\Phi(g)g^4}{s_\Lambda  \Phi(g_\Lambda)g^4_\Lambda}
\\
\times \int_0^{\pi} d\phi \Big[\exp(l-l\sqrt{y^2+1} +jy\cos\phi)-1\Big]^2+\langle p \rangle^{(\rho+\gamma)/2} 1_{|p|\le 1}.
\end{multline}
Note that $\zeta'(p)$ is then automatically positive for all $p$.  The term $\langle p \rangle^{(\rho+\gamma)/2} 1_{|p|\le 1}$ guarantees that $\zeta(p)>0$ near $p=0$ (if necessary).

However, unfortunately it turns out that there is an additional severe difficulty to obtain the sharp pointwise asymptotic upper bound of the expression \eqref{FREQ:def.zeta.mod} for $\zeta(p)$.  The problem is that the following term 
\begin{equation}\notag
    \exp({-q^0})\exp (2l-2l\sqrt{y^2+1}+2jy\cos\phi)
\end{equation}
does not in general have uniform decay in the $q^0$ variable when we are close to the singularity in the $dy$ integral at $y=0$.  Here we recall the definitions \eqref{FREQ:lj}.  Therefore while the expression $\zeta'(p)$ in \eqref{FREQ:def.zeta.mod} appears good for obtaining a positive asymptotic lower bound, it is extremely difficult to obtain the required sharp upper bound.  There is the same difficulty for $\zetaTL(p)$ in \eqref{FREQ:eq:tildezetanew1L}.  To overcome this situation we split the $dq$ integral into the two regions $\qgep$ and $\qlep$.  

\begin{remark}
 The derivations of the alternative formula's in \eqref{FREQ:zeta0By}, \eqref{FREQ:zetaLBy}, \eqref{FREQ:eq:tildezetanew1} and \eqref{FREQ:eq:tildezetanew1L} of \eqref{FREQ:tildezeta} in  \secref{FREQ:sec:derivation} still hold under the restrictions to the region such as $\left[\tilde{\zeta}\right]_{|q|\ge \frac{1}{2}|p|^{1/m}}$ and $\left[\tilde{\zeta}\right]_{|q|\le \frac{1}{2}|p|^{1/m}}$ using the convention \eqref{convention}.  This is straightforward from the proofs in   \secref{FREQ:sec:derivation}.
 \end{remark}

With these splittings, we are then able to obtain the sharp asymptotic upper bound estimates of $\zeta(p)$ and $\zetaTL(p)$ on the region $\qlep$.   And on $\qgep$ from \eqref{FREQ:tildezeta} using \eqref{convention} we further define
 \begin{equation}\label{FREQ:def.zetatilde1}
     \zetaTone(p) \eqdef \left[\tilde{\zeta}(p)\right]_{|q|\ge \frac{1}{2}|p|^{1/m}}1_{|p|\ge 1}
 \end{equation}
We will prove that $\zetaTone(p)$ has low order asymptotic behavior in Proposition \ref{FREQ:prop.tildezeta.upper.qgep}.  Then these estimates together will be enough to prove our main theorem.

To this end, from \eqref{FREQ:zeta0By} and \eqref{FREQ:zetaLBy}, we also define the following two terms
\begin{equation}\label{FREQ:def.zetaMdef}
    \zeta_{0,m}(p) \eqdef [\zeta_0(p)]_{|q| \le \frac{1}{2} |p|^{1/m}} 1_{|p|\ge 1},
    \quad 
   \zeta_{L,m}(p) \eqdef [\zetaL(p)]_{|q| \le \frac{1}{2} |p|^{1/m}} 1_{|p|\ge 1},   
\end{equation}
In \eqref{FREQ:def.zetaMdef} above we use the notation convention from \eqref{convention}.  For example
\begin{multline}\label{FREQ:def.convention.zeta0By}
\zeta_{0,m}(p) = \frac{c'}{p^0}\int_{|q| \le \frac{1}{2} |p|^{1/m}}\frac{dq}{q^0}\frac{e^{-q^0}\sqrt{s}}{g}\int_{0}^\infty \frac{ydy}{\sqrt{y^2+1}}s_\Lambda\sigma(g_\Lambda,\theta_\Lambda) \frac{s\Phi(g)g^4}{s_\Lambda  \Phi(g_\Lambda)g^4_\Lambda}
\\
\times
\Big[1 - \exp\left(l(1-\sqrt{y^2+1})\right)I_0\left(jy\right) \Big],
\end{multline}
Similarly, from \eqref{FREQ:eq:tildezetanew1} and \eqref{FREQ:eq:tildezetanew1L}, we further define 
\begin{equation}\label{FREQ:def.TzetaMdef}
    \zetaTZm(p) \eqdef [\tilde{\zeta}_0(p)]_{|q| \le \frac{1}{2} |p|^{1/m}} 1_{|p|\ge 1},
    \quad 
   \zetaTLm(p) \eqdef [\tilde{\zeta}_L(p)]_{|q| \le \frac{1}{2} |p|^{1/m}} 1_{|p|\ge 1},   
\end{equation}
With these definitions instead of \eqref{FREQ:def.zeta.mod} we define the modified frequency multiplier
\begin{multline}\label{FREQ:def.zeta}
\zeta(p) \eqdef \frac{1}{2} \left[\zeta_{0,m}(p)+\zetaTZm(p) \right]
+\langle p \rangle^{(\rho+\gamma)/2} 1_{|p|\le 1}\\
=
\frac{c'1_{|p|\ge 1}}{2\pi p^0}\int_{\qlep}\frac{dq}{q^0}\frac{e^{-q^0}\sqrt{s}}{g}\int_{0}^\infty \frac{ydy}{\sqrt{y^2+1}}s_\Lambda\sigma(g_\Lambda,\theta_\Lambda)\frac{s\Phi(g)g^4}{s_\Lambda  \Phi(g_\Lambda)g^4_\Lambda}
\\
\times \int_0^{\pi} d\phi \Big[\exp(l-l\sqrt{y^2+1} +jy\cos\phi)-1\Big]^2+\langle p \rangle^{(\rho+\gamma)/2} 1_{|p|\le 1}.
\end{multline}
Again this is automatically positive.

As such, we introduce the construction of a positive leading order term and a lower order term in the frequency multiplier $\tilde{\zeta}(p)$ from \eqref{FREQ:tildezeta}, which is highly non-trivial.   We suggest the following novel decomposition of $\tilde{\zeta}=\zeta+\zeta_{\mathcal{K}}$ as
$$
\tilde{\zeta}(p)=\zeta(p)+\zeta_{\mathcal{K}}(p)
$$ 
where $\zeta$ is given by \eqref{FREQ:def.zeta}, using also \eqref{FREQ:tildezeta}, \eqref{FREQ:zetaLBy}, \eqref{FREQ:eq:tildezetanew1L} and \eqref{FREQ:def.zetatilde1} we have
\begin{equation}\label{FREQ:zetaK.def}
    \zeta_{\mathcal{K}}\eqdef \tilde{\zeta}(p)1_{|p|\le 1}
    +
     \zetaTone(p)
     +
     \frac{1}{2}\left(\zetaLm+\zetaTLm\right) 
     - \langle p \rangle^{(\rho+\gamma)/2} 1_{|p|\le 1},
\end{equation}
 Then we will show that $\zeta$ and $\zeta_{\mathcal{K}}$ satisfy the asymptotics from \eqref{FREQ:Paos}.   In particular the main positive term is \eqref{FREQ:def.zeta}.

Lastly, we also introduce two additional representations of $\tilde{\zeta}(p)$ from \eqref{FREQ:tildezeta}.  In particular it is shown in  \secref{FREQ:sec:derivation} that we have the following splitting of $\tilde{\zeta}=\zeta_0+\zetaL$ in \eqref{FREQ:zeta0B.appendix} with $c'>0$ as
\begin{multline}\label{FREQ:zeta0B}
\zeta_0 \eqdef \frac{c'}{p^0}e^{\frac{p^0}{4}}\int_{\rth}\frac{dq}{q^0}\frac{e^{-\frac{3}{4}q^0}}{g}\int_{0}^\infty \frac{rdr}{\sqrt{r^2+s}}s_\Lambda\sigma(g_\Lambda,\theta_\Lambda) \frac{s\Phi(g)g^4}{s_\Lambda  \Phi(g_\Lambda)g^4_\Lambda}
\\
\times
\Big[\exp\left(-\frac{p^0+q^0}{4}\right) - \exp\left(-\frac{p^0+q^0}{4\sqrt{s}}\sqrt{r^2+s}\right)I_0\left(\frac{|p\times q|}{2g\sqrt{s}}r\right) \Big],
\end{multline}
and $\zetaL$ is further given in \eqref{FREQ:zetaLB}.  Here $g_\Lambda$ is given by \eqref{FREQ:g2}, $s_\Lambda$ by \eqref{FREQ:slambda.def}, and $\theta_\Lambda$ by \eqref{FREQ:coslam}.  Note that \eqref{FREQ:zeta0B} and \eqref{FREQ:zetaLB} can alternatively be obtained by  applying the change of variables $y \mapsto r = \sqrt{s} y$ to the expressions \eqref{FREQ:zeta0By} and \eqref{FREQ:zetaLBy}.   We will use the formula \eqref{FREQ:zeta0B} in the proof of Proposition \ref{FREQ:prop.zeta0.asymptotic}.

We can also write \eqref{FREQ:zeta0B} in a further alternative form with other variables by using the following change of variables \begin{equation}\label{FREQ:changeofv}r\mapsto k\eqdef \frac{1}{2}\sqrt{s}(\sqrt{r^2+s}-\sqrt{s}).\end{equation} 
Then this gives $$dk=\frac{1}{2}\sqrt{s}\frac{rdr}{\sqrt{r^2+s}}.$$
Also, we have 
$$ 
\frac{\sqrt{r^2+s}}{\sqrt{s}}=1+\frac{2k}{s},
$$ 
and 
$$
r=\frac{2\sqrt{k^2+ks}}{\sqrt{s}}.
$$
Here $g_\Lambda$ from \eqref{FREQ:g2} and $\theta_\Lambda$ from \eqref{FREQ:coslam} now take the form
\begin{equation}\label{FREQ:gl}
g^2_\Lambda=g^2+k,
\end{equation}
and
$$\cos\theta_\Lambda =\frac{g^2-k}{g^2+k}=1-2\frac{k}{g^2+k}.$$
Therefore, we have 
\begin{equation}\label{FREQ:tl}\frac{\theta_\Lambda}{2}\approx \sin\frac{\theta_\Lambda}{2}=\sqrt{\frac{k}{g^2+k}}.
\end{equation}
With respect to the new variable $k$, then \eqref{FREQ:zeta0B} can be re-written as
\begin{multline}\notag
\zeta_0 \eqdef \frac{c'}{p^0}e^{\frac{p^0}{4}}\int_{\rth}\frac{dq}{q^0}\frac{e^{-\frac{3}{4}q^0}}{g}\int_{0}^\infty \frac{2dk}{\sqrt{s}}s_\Lambda\sigma(g_\Lambda,\theta_\Lambda)\frac{s\Phi(g)g^4}{s_\Lambda  \Phi(g_\Lambda)g^4_\Lambda}  
\\
\times
\left[\exp\left(-\frac{p^0+q^0}{4}\right) - \exp\left(-\frac{p^0+q^0}{4}\left(1+\frac{2k}{s}\right)\right)I_0\left(\frac{|p\times q|}{gs}\sqrt{k^2+ks}\right)  \right]\\
=\frac{c'}{p^0}\int_{\rth}\frac{dq}{q^0}\frac{e^{-q^0}}{g}\int_{0}^\infty \frac{2dk}{\sqrt{s}}s_\Lambda\sigma(g_\Lambda,\theta_\Lambda)\frac{s\Phi(g)g^4}{s_\Lambda  \Phi(g_\Lambda)g^4_\Lambda} \\
\times
\left[1- \exp\left(-\frac{(p^0+q^0)k}{2s}\right) I_0\left(\frac{|p\times q|}{gs}\sqrt{k^2+ks}\right) \right],
\end{multline}
where $s_\Lambda=g^2_\Lambda+4$ with \eqref{FREQ:gl}.  This representation of $\zeta_0$ in the $k$ variables above will be used in the proof of Lemma \ref{FREQ:lemma.zeta1}, which is one part of the leading order upper bound estimate of $\zeta_0$.

\subsection{Outline of the proof of Theorem \ref{FREQ:main.thm}}\label{FREQ:sec:proof.outline}
 Specifically, in the rest of \secref{sec:frequency}, in order to prove Theorem \ref{FREQ:main.thm} we will make upper- and lower-bound estimates for  $\zeta$ in \eqref{FREQ:def.zeta} and will conclude that it is a leading order term. In addition, we will show that $\zeta_{\mathcal{K}}$ in \eqref{FREQ:zetaK.def} is a lower-order term.

 We will first prove that $\zeta(p)$ from \eqref{FREQ:def.zeta} has a leading order positive lower bound in Proposition \ref{FREQ:prop.coercive}.  Then we will prove that $\zeta_0$  from \eqref{FREQ:zeta0By} has the leading order upper bound in Proposition \ref{FREQ:prop.zeta0.asymptotic}.  Then in Proposition \ref{FREQ:prop.zetaL.asymptotic} we prove that $\zetaL(p)$ from \eqref{FREQ:zetaLBy} has a lower order upper bound.  We further prove in Proposition \ref{FREQ:prop.zetaL.asymptoticnew} that $\zetaTLm(p)$ from \eqref{FREQ:def.TzetaMdef} with \eqref{FREQ:eq:tildezetanew1L} has a lower order upper bound. We then prove in Proposition \ref{FREQ:prop.tildezeta.upper.qgep} that $\zetaTone(p)$ from \eqref{FREQ:def.zetatilde1} has a lower order upper bound. Note that both $\tilde{\zeta}(p)1_{|p|\le 1}$ and $\langle p \rangle^{(\rho+\gamma)/2} 1_{|p|\le 1}$ have lower order upper bounds since  $1_{|p|\le 1} $ trivially makes $p^0\lesssim 1.$   All of these estimates combine to prove that $\zeta_{\mathcal{K}}$ has a lower order upper bound, and that $\tilde{\zeta}(p)$ from \eqref{FREQ:tildezeta} has a leading order asymptotic upper bound.

We remark that we have not estimated the asymptotic upper bound of $\zetaTZ(p)$ from \eqref{FREQ:eq:tildezetanew1} or more accurately we have not estimated $\zetaTZm(p)$ from \eqref{FREQ:def.TzetaMdef} and this is not necessary because from the splittings above we have
 \begin{equation}\notag
     \zetaTZm(p) = [\tilde{\zeta}(p)]_{\qlep}1_{|p|\ge 1} - \zetaTLm(p)
    = \zetaZm(p) + \zetaLm(p)- \zetaTLm(p).
 \end{equation}
Therefore using the estimates discussed in the previous paragraph we obtain that $\zetaTZm(p)$ and $\zeta(p)$
 both have the leading order asymptotic upper bound.  All of the estimates discussed in this sub-section together give the proof of Theorem \ref{FREQ:main.thm}.

\subsection{Leading order lower bound estimate}\label{FREQ:sec:leadingorder lower bound zeta}

The main result in this section is the following leading order lower bound.

\begin{proposition}\label{FREQ:prop.coercive}	Suppose $\gamma \in (0,2)$ in \eqref{angassumption}. Then for both hard \eqref{hard} and soft \eqref{soft} interactions, using the notation \eqref{singS.defin}, for \eqref{FREQ:def.zeta}, we have
	$$\zeta(p) \gtrsim (p^0)^{\frac{\singS+\gamma}{2}}.$$
This uniform lower bound also holds for \eqref{FREQ:def.zeta.mod}.
\end{proposition}

\begin{proof}[Proof of Proposition \ref{FREQ:prop.coercive}]In order to obtain the lower-bound estimate for $\zeta(p)$, we first study the lower bound of the perfect square term $$\Big[\exp(l-l\sqrt{y^2+1} +jy\cos\phi)-1\Big]^2$$ in \eqref{FREQ:def.zeta}. We first observe that, if $y\in [0,y^*]$ with
\begin{equation}\label{FREQ:def.y.star}
y^*\eqdef \frac{2lj\cos\phi}{l^2-j^2\cos^2\phi},
\end{equation}
then we have$$ l-l\sqrt{y^2+1} +jy\cos\phi\ge 0.$$
Notice that we also have
\begin{equation}\notag
    l-l\sqrt{y^2+1}+jy\cos\phi \ge \frac{1}{2} jy\cos\phi, \quad
\end{equation}
if  $0 \le y \le y_1,$ where 
\begin{equation}\notag
y_1 \eqdef \frac{2l j \cos\phi }{4l^2-j^2 \cos^2\phi}.
\end{equation}
Also $\frac{1}{2} \le \cos\phi \le \frac{\sqrt{2}}{2}$ for $\phi\in [\pi/4,\pi/3]$.
Recalling \eqref{FREQ:def.y.star}, we remark that $y^*\le 3$ because
$$
y^* = \frac{2lj\cos\phi}{l^2-j^2\cos^2\phi}\le \frac{
\sqrt{2}lj}{l^2-\frac{j^2}{2}}\le \frac{
\sqrt{2}l^2}{\frac{l^2}{2}}\le 2 \sqrt{2},
$$ as $j\le l$ and $\frac{
\sqrt{2}lj}{l^2-\frac{j^2}{2}}$ is an increasing function in $j$.
Recalling again \eqref{FREQ:def.y.star}, then $0\le y_1\le \frac{y^*}{4}\le \frac{\sqrt{2}}{2}$.
Since in particular with $\phi\in [\pi/4,\pi/3]$ we have
\begin{equation}\notag
    l-l\sqrt{y^2+1}+jy\cos\phi \ge 0, \quad 0 \le y \le y_1.
\end{equation}
Then by the Taylor expansion
\begin{equation}\notag
    \Big[\exp(l-l\sqrt{y^2+1} +jy\cos\phi)-1\Big]^2
    \ge (l-l\sqrt{y^2+1}+jy\cos\phi)^2.
\end{equation}
We will use this lower bound in the following developments.

Now we start by proving the stated lower bound for \eqref{FREQ:def.zeta.mod}.
Now we split each integral representation of the decomposed pieces based on a restriction of the $y$ and $\phi$ domains.  We will now define the term 
$$
\zeta_* \eqdef [\zeta']_{0\le y\le y_1\text{ and }\phi\in [\pi/4,\pi/3]},
$$
where $\zeta_*$ is $\zeta'$ when the integrals inside \eqref{FREQ:def.zeta.mod} are only on the restricted domains $0\le y\le y_1\text{ and }\phi\in [\pi/4,\pi/3]$. 
This notation is similar to \eqref{convention}.   Note that of course $\zeta'(p) \ge \zeta_*(p)$.  We will show that $\zeta_*(p)$ has a high-order lower bound. Note that inside this integration region, $0\le y\le y_1\text{ and }\phi\in [\pi/4,\pi/3]$, the integral is still non-negative.

First of all, we note from \eqref{FREQ:def.zeta.mod} that
\begin{multline*}
\zeta_*(p)
\ge \frac{c'}{2\pi p^0}\int_{\rth}\frac{dq}{q^0}\frac{e^{-q^0}\sqrt{s}}{g}\int_{\pi/4}^{\pi/3} d\phi \int^{y_1}_{0}\frac{ydy}{\sqrt{y^2+1}}s_\Lambda\sigma(g_\Lambda,\theta_\Lambda)
\\
\times (l-l\sqrt{y^2+1}+jy\cos\phi)^2\frac{s\Phi(g)g^4}{s_\Lambda  \Phi(g_\Lambda)g^4_\Lambda}+\langle p \rangle^{(\rho+\gamma)/2} 1_{|p|\le 1}
\\
\ge
\frac{c'}{2\pi p^0}\int_{\rth}\frac{dq}{q^0}\frac{e^{-q^0}\sqrt{s}}{g}\int_{\pi/4}^{\pi/3} d\phi \int^{y_1}_{0}\frac{ydy}{\sqrt{y^2+1}}s_\Lambda\sigma(g_\Lambda,\theta_\Lambda)\\\times\frac{1}{4}\left( jy\cos\phi \right)^2\frac{s\Phi(g)g^4}{s_\Lambda  \Phi(g_\Lambda)g^4_\Lambda}+\langle p \rangle^{(\rho+\gamma)/2} 1_{|p|\le 1}
\\
\ge \frac{c'}{2\pi p^0}\int_{\rth}\frac{dq}{q^0}\frac{e^{-q^0}\sqrt{s}}{g}\int_{\pi/4}^{\pi/3} d\phi \int^{y_1}_{0}\frac{ydy}{\sqrt{y^2+1}}s_\Lambda\sigma(g_\Lambda,\theta_\Lambda)\\\times
\frac{1}{4}\left(\frac{jy}{2}\right)^2\frac{s\Phi(g)g^4}{s_\Lambda  \Phi(g_\Lambda)g^4_\Lambda}+\langle p \rangle^{(\rho+\gamma)/2} 1_{|p|\le 1},
\end{multline*} where we used that $\cos\phi\ge \frac{1}{2}$ when $\phi\in[\pi/4,\pi/3].$

 Now we will estimate the kernel $\sigma(g_\Lambda,\theta_\Lambda)$ from \eqref{define.kernel}.  Here, by \eqref{angassumption} with \eqref{FREQ:g2y.variable}, \eqref{FREQ:changeofv}, \eqref{FREQ:gl} and \eqref{FREQ:tl} we have
\begin{equation}\label{FREQ:ang.equiv}
\sigma_0(\theta_\Lambda)\approx  \left(\frac{sy^2}{sy^2+2g^2(\sqrt{y^2+1}+1)}\right)^{-1-\gamma/2}.
\end{equation}
Next using \eqref{FREQ:g2y.variable} we have that 
\begin{multline}\label{FREQ:glambda.calc}
g^2_\Lambda=g^2+\frac{s}{2}(\sqrt{y^2+1}-1)=
g^2+\frac{sy^2}{2(\sqrt{y^2+1}+1)}\\
=\frac{sy^2+2g^2(\sqrt{y^2+1}+1)}{2(\sqrt{y^2+1}+1)}.
\end{multline}
 Thus, also recalling \eqref{singS.defin} and \eqref{FREQ:glambda.calc}, we have
\begin{multline}\notag
s_\Lambda\sigma(g_\Lambda,\theta_\Lambda)\frac{s\Phi(g)g^4}{s_\Lambda  \Phi(g_\Lambda)g^4_\Lambda}=s\Phi(g)	\sigma_0(\theta_\Lambda)\frac{g^4}{g^4_\Lambda}\\
\approx s\Phi(g)g^4 \left(\frac{sy^2}{sy^2+2g^2(\sqrt{y^2+1}+1)}\right)^{-1-\gamma/2}\frac{1}{g^4_\Lambda}.
\end{multline}
Thus
\begin{multline}\notag
s_\Lambda\sigma(g_\Lambda,\theta_\Lambda)\frac{s\Phi(g)g^4}{s_\Lambda  \Phi(g_\Lambda)g^4_\Lambda}
\approx sg^{\singS+4} \left(\frac{sy^2}{sy^2+2g^2(\sqrt{y^2+1}+1)}\right)^{-1-\gamma/2}
\\
\times\left(\frac{sy^2+2g^2(\sqrt{y^2+1}+1)}{2(\sqrt{y^2+1}+1)}\right)^{-2}.
\end{multline}
We conclude that
\begin{multline}\label{FREQ:kernel.equiv}
s_\Lambda\sigma(g_\Lambda,\theta_\Lambda)\frac{s\Phi(g)g^4}{s_\Lambda  \Phi(g_\Lambda)g^4_\Lambda}=s\Phi(g)	\sigma_0(\theta_\Lambda)\frac{g^4}{g^4_\Lambda}\\
\approx s^{-\gamma/2}g^{\singS+4}y^{-2-\gamma}(2(\sqrt{y^2+1}+1))^{2}\left(sy^2+2g^2(\sqrt{y^2+1}+1)\right)^{-1+\gamma/2}.
\end{multline}
Thus, since $\gamma \in (0,2)$, we have 
\begin{equation*}
 s_\Lambda\sigma(g_\Lambda,\theta_\Lambda)\frac{s\Phi(g)g^4}{s_\Lambda  \Phi(g_\Lambda)g^4_\Lambda}   \gtrsim s^{-1}y^{-2-\gamma} g^{\rho+4},
\end{equation*}
where we used $\gamma/2-1<0$ and $y\le y_1\le  \frac{\sqrt{2}}{2}$. 
Therefore,
\begin{multline*}
\zeta_*(p)
\gtrsim  \frac{c'}{ p^0}\int_{\rth}\frac{dq}{q^0}\frac{e^{-q^0}\sqrt{s}}{g} \int_{\pi/4}^{\pi/3} d\phi\\\times  \int^{y_1}_{0}\frac{ydy}{\sqrt{y^2+1}}j^2s^{-1}y^{-\gamma} g^{4+\rho}+\langle p \rangle^{(\rho+\gamma)/2} 1_{|p|\le 1}.
\end{multline*} 
Then we have
\begin{multline*}
\zeta_*(p)
\gtrsim  \frac{c'}{ p^0}\int_{\rth}\frac{dq}{q^0}\frac{e^{-q^0}\sqrt{s}}{g} 
j^2s^{-1} g^{4+\rho}
\int^{y_1}_{0} y^{1-\gamma}dy+\langle p \rangle^{(\rho+\gamma)/2} 1_{|p|\le 1}
\\
\gtrsim  \frac{c'}{ p^0}\int_{\rth}\frac{dq}{q^0}\frac{e^{-q^0}\sqrt{s}}{g} 
j^2s^{-1} g^{4+\rho}
 y_1^{2-\gamma}+\langle p \rangle^{(\rho+\gamma)/2} 1_{|p|\le 1}.
\end{multline*} 
We further have on $\phi \in [\pi/4, \pi/3]$, using also $j \le l$, that
\begin{equation}\notag
y_1^{2-\gamma}
= \left(\frac{2lj\cos\phi}{4l^2-j^2\cos^2\phi}\right)^{2-\gamma}
\geq \left(\frac{lj}{4l^2-j^2/4}\right)^{2-\gamma}
\gtrsim
\left(\frac{j}{l}\right)^{2-\gamma},
\end{equation} 
as $\cos\phi \ge \frac{1}{2}$.  Altogether, we have
\begin{multline*}
\zeta_*(p)
\gtrsim  \frac{c'}{ p^0}\int_{\rth}\frac{dq}{q^0}\frac{e^{-q^0}\sqrt{s}}{g}  j^2s^{-1} g^{4+\rho} \left(\frac{j}{l}\right)^{2-\gamma}+\langle p \rangle^{(\rho+\gamma)/2} 1_{|p|\le 1}
\\
\gtrsim  \frac{c'}{ p^0}\int_{\rth}\frac{dq}{q^0}\frac{e^{-q^0}\sqrt{s}}{g}  s^{-1} g^{4+\rho} j^{4-\gamma}l^{-2+\gamma}+\langle p \rangle^{(\rho+\gamma)/2} 1_{|p|\le 1}. 
\end{multline*}
Now we recall \eqref{FREQ:lj}, \eqref{FREQ:s.ge.g2}, \eqref{FREQ:g.ge.lower} and note that $\gamma\in(0,2)$. 
Then we obtain
\begin{multline}\label{FREQ:eq.sameuntilhere}
\zeta_*(p)
\gtrsim   \frac{1}{ p^0}\int_{\rth}\frac{dq}{q^0}e^{-q^0}  g^{-1+\gamma+\rho} |p\times q|^{4-\gamma}s^{-1/2}(p^0+q^0)^{-2+\gamma} +\langle p \rangle^{(\rho+\gamma)/2} 1_{|p|\le 1}\\
\gtrsim   \frac{1}{ p^0}\int_{\rth}\frac{dq}{q^0}e^{-q^0}  s^{-1/2}g^{\gamma+\rho-1} |p\times q|^{4-\gamma}(p^0q^0)^{-2+\gamma}+\langle p \rangle^{(\rho+\gamma)/2} 1_{|p|\le 1},
\end{multline}
above we also used \eqref{FREQ:p0.plus.q0.le.p0q0}.
Further, since \eqref{FREQ:s.le.pq}, we have 
$$s^{-1/2}\gtrsim (p^0q^0)^{-1/2}.$$
If $\gamma+\rho -1\ge 0$, 
then from \eqref{FREQ:g.ge.lower} and \eqref{FREQ:p0q0.le.pq} we have
$$g^{\gamma+\rho-1}\ge \left(\frac{|p-q|}{\sqrt{p^0q^0}}\right)^{\gamma+\rho-1}\ge \left(\frac{|p^0-q^0|}{\sqrt{p^0q^0}}\right)^{\gamma+\rho-1}.$$ Otherwise, when $\gamma+\rho -1< 0$, using \eqref{FREQ:s.le.pq} with \eqref{FREQ:s.ge.g2} we have
$$g^{\gamma+\rho-1}\ge (p^0q^0)^{\gamma/2+\rho/2-1/2}.$$
Finally, we use the spherical-coordinate representation of $q\mapsto (r,\theta_q,\phi_q)$. We let the $z$-axis be parallel to the direction of $p$ such that $\phi_q$ is the angle between $p$ and $q$. Then we have
\begin{multline}\label{FREQ:eq.q.last}
\zeta_*(p)
\gtrsim \frac{1}{p^0}\int_0^{\infty}dr \frac{r^2e^{-\sqrt{1+r^2}}}{\sqrt{r^2+1}} \int_0^\pi d\phi_q \ \sin\phi_q   \\\times (p^0q^0)^{-1/2}\min\left\{\left(\frac{|p^0-q^0|}{\sqrt{p^0q^0}}\right)^{\gamma+\rho-1}, (p^0q^0)^{\gamma/2+\rho/2-1/2}\right\}\\\times |p|^{4-\gamma}r^{4-\gamma} \sin^{4-\gamma}\phi_q(p^0q^0)^{-2+\gamma}+\langle p \rangle^{(\rho+\gamma)/2} 1_{|p|\le 1}\\
\approx |p|^{4-\gamma} (p^0)^{-1-1/2+\gamma/2+\rho/2-1/2-2+\gamma}+\langle p \rangle^{(\rho+\gamma)/2} 1_{|p|\le 1}\\
\approx |p|^{4-\gamma}(p^0)^{-4+3\gamma/2+\rho/2} +\langle p \rangle^{(\rho+\gamma)/2} 1_{|p|\le 1}.
\end{multline}
Now we remark that if $|p|\ge 1$ then we have $|p|\approx p^0$. We conclude
$$
\zeta'(p) 
\ge 
\zeta_*(p)\gtrsim (p^0)^{\frac{\rho}{2}+\frac{\gamma}{2}}.$$
This completes the proof for the high-order lower bound of $\zeta'(p)$.

Similarly, we can obtain the high-order lower bound of $\zeta(p)$ from \eqref{FREQ:def.zeta}. Note that the only difference between $\zeta(p)$ and $\zeta'(p)$ from \eqref{FREQ:def.zeta} and \eqref{FREQ:def.zeta.mod} is that the domain $\rth$ with respect to $q$ variable in \eqref{FREQ:def.zeta.mod} is now restricted to $|q|\le \frac{1}{2}|p|^{1/m}$ in \eqref{FREQ:def.zeta}. Then we note that the proof for the high-order lower bound of $\zeta(p)$ is exactly the same as $\zeta'(p)$ until 
\eqref{FREQ:eq.sameuntilhere} above except for the change from $\int_\rth dq$ into $\int_{|q|\le\frac{1}{2}|p|^{1/m}} dq $. 
Then in the spherical-coordinate representation of $q\mapsto (r,\theta_q,\phi_q)$ for \eqref{FREQ:eq.q.last}, we change the integral domain $\int_0^\infty dr$ in \eqref{FREQ:eq.q.last} to $\int_0^{\frac{1}{2}|p|^{1/m}} dr$. Then analogous to \eqref{FREQ:eq.q.last} we have
\begin{multline}\notag
\zeta(p)
\gtrsim 
\frac{1}{p^0}\int_0^{\frac{1}{2}|p|^{1/m}} dr \frac{r^2e^{-\sqrt{1+r^2}}}{\sqrt{r^2+1}} \int_0^\pi d\phi_q \ \sin\phi_q   \\\times (p^0q^0)^{-1/2}\min\left\{\left(\frac{|p^0-q^0|}{\sqrt{p^0q^0}}\right)^{\gamma+\rho-1}, (p^0q^0)^{\gamma/2+\rho/2-1/2}\right\}\\\times |p|^{4-\gamma}r^{4-\gamma} \sin^{4-\gamma}\phi_q(p^0q^0)^{-2+\gamma}+\langle p \rangle^{(\rho+\gamma)/2} 1_{|p|\le 1}.
\end{multline}
Now in the region ${|q|\le\frac{1}{2}|p|^{1/m}}$  with $|p| \ge 1$ and $m\ge 1$ sufficiently large inside \eqref{FREQ:eq.q.last}  we have 
\begin{multline*}
    \min\left\{\left(\frac{|p^0-q^0|}{\sqrt{p^0q^0}}\right)^{\gamma+\rho-1}, (p^0q^0)^{\gamma/2+\rho/2-1/2}\right\}
    \\
    \gtrsim
    (p^0)^{\gamma/2+\rho/2-1/2}
        \min\left\{(q^0)^{-\gamma/2-\rho/2+1/2}, (q^0)^{\gamma/2+\rho/2-1/2}\right\}
\end{multline*}
We further have on $|p| \ge 1$ with $q^0 = \sqrt{1+r^2}$ that 
\begin{multline}\notag
\int_0^{\frac{1}{2}|p|^{1/m}} dr  \frac{r^2e^{-\sqrt{1+r^2}}}{\sqrt{r^2+1}} 
r^{4-\gamma}
\min\left\{(q^0)^{-\gamma/2-\rho/2}, (q^0)^{\gamma/2+\rho/2-1}\right\} (q^0)^{-2+\gamma}
\\
\gtrsim
\int_0^{\frac{1}{2}} dr  \frac{r^2e^{-\sqrt{1+r^2}}}{\sqrt{r^2+1}} 
r^{4-\gamma}
\min\left\{(q^0)^{-\gamma/2-\rho/2}, (q^0)^{\gamma/2+\rho/2-1}\right\} (q^0)^{-2+\gamma}
\gtrsim
c_{\frac{1}{2}},
\end{multline}
for some constant $c_{\frac{1}{2}}>0$ if $|p|\ge 1.$ 
Therefore, the same proof with the modifications above works for the leading-order lower bound of $\zeta(p)$ from \eqref{FREQ:def.zeta}.  In particular the estimate \eqref{FREQ:eq.q.last} continues to hold,  and this completes the leading-order lower-bound estimates.   
\end{proof}

This completes the leading order lower bound estimates of $\zeta(p)$. In the next two sections, we will use the decomposition $\tilde{\zeta}(p)=\zeta_0(p)+\zeta_L(p)$ from \eqref{FREQ:zeta0By} and \eqref{FREQ:zetaLBy} to obtain the leading order upper bound of $\tilde{\zeta}(p)$, and the lower order upper bounds of $\zetaL(p)$ and $\zetaTL(p)$ from \eqref{FREQ:eq:tildezetanew1L}.

\subsection{Leading order upper bound estimates}\label{FREQ:sec:fullsharpupper zeta0}

We now prove the following leading order upper bound estimate for $\zetaZ$ from \eqref{FREQ:zeta0By} using the alternative representation \eqref{FREQ:zeta0B}:

\begin{proposition}\label{FREQ:prop.zeta0.asymptotic}
	Suppose $\gamma \in (0,2)$  in \eqref{angassumption}. Then for both hard \eqref{hard} and soft \eqref{soft} interactions, for \eqref{FREQ:zeta0B} when $|p|\ge 1$, we have
	$$|\zetaZ(p)|\lesssim (p^0)^{\frac{\singS+\gamma}{2}}.$$
This consequently implies the same uniform bound for $\zetaZm(p)$ from \eqref{FREQ:def.convention.zeta0By}.	
\end{proposition}

For the proof, 
we decompose $\zetaZ$ from \eqref{FREQ:zeta0B} as $\zetaZ=\zeta_1+\zeta_2$ where
\begin{multline}\label{FREQ:zeta1B}
\zeta_1 \eqdef \frac{c'}{p^0}e^{\frac{p^0}{4}}\int_{\rth}\frac{dq}{q^0}\frac{e^{-\frac{3}{4}q^0}}{g}\int_{0}^\infty \frac{rdr}{\sqrt{r^2+s}}s_\Lambda\sigma(g_\Lambda,\theta_\Lambda)  \frac{s\Phi(g)g^4}{s_\Lambda  \Phi(g_\Lambda)g^4_\Lambda}
\\
\times
\left[\exp\left(-\frac{p^0+q^0}{4}\right) - \exp\left(-\frac{p^0+q^0}{4\sqrt{s}}\sqrt{r^2+s}\right) \right],
\end{multline}
\begin{multline}\label{FREQ:zeta2B}
\zeta_2 \eqdef \frac{c'}{p^0}e^{\frac{p^0}{4}}\int_{\rth}\frac{dq}{q^0}\frac{e^{-\frac{3}{4}q^0}}{g}\int_{0}^\infty \frac{rdr}{\sqrt{r^2+s}}s_\Lambda\sigma(g_\Lambda,\theta_\Lambda)  \frac{s\Phi(g)g^4}{s_\Lambda  \Phi(g_\Lambda)g^4_\Lambda}
\\
\times
\exp\left(-\frac{p^0+q^0}{4\sqrt{s}}\sqrt{r^2+s}\right)\left[1-I_0\left(\frac{|p\times q|}{2g\sqrt{s}}r\right) \right].
\end{multline}
Clearly, $\zeta_1$ is positive.   We estimate $\zeta_1$ in Lemma \ref{FREQ:lemma.zeta1} and then we will estimate $\zeta_2$ in Lemma \ref{FREQ:lemma.zeta2};  Proposition \ref{FREQ:prop.zeta0.asymptotic} then follows directly.  First, we have

\begin{lemma}\label{FREQ:lemma.zeta1}
	Assuming either \eqref{hard} or \eqref{soft} with \eqref{angassumption}, then we have the following uniform asymptotic bound for $\zeta_1$ from \eqref{FREQ:zeta1B}:
	$$\zeta_1(p)\lesssim (p^0)^{\frac{\singS+\gamma}{2}}.$$
\end{lemma}

\begin{proof}The change of variables \eqref{FREQ:changeofv} on the representation \eqref{FREQ:zeta1B} yields that
\begin{multline*}\zeta_1 \eqdef \frac{c'}{p^0}e^{\frac{p^0}{4}}\int_{\rth}\frac{dq}{q^0}\frac{e^{-\frac{3}{4}q^0}}{g}\int_{0}^\infty \frac{2dk}{\sqrt{s}}s_\Lambda\sigma(g_\Lambda,\theta_\Lambda)  \frac{s\Phi(g)g^4}{s_\Lambda  \Phi(g_\Lambda)g^4_\Lambda}
\\
\times
\Big[\exp\left(-\frac{p^0+q^0}{4}\right) - \exp\left(-\frac{p^0+q^0}{4}\left(1+\frac{2k}{s}\right)\right) \Big]\\
=\frac{c'}{p^0}\int_{\rth}\frac{dq}{q^0}\frac{e^{-q^0}}{g}\int_{0}^\infty \frac{2dk}{\sqrt{s}}s_\Lambda\sigma(g_\Lambda,\theta_\Lambda)  \frac{s\Phi(g)g^4}{s_\Lambda  \Phi(g_\Lambda)g^4_\Lambda}
\Big[1- \exp\left(-\frac{(p^0+q^0)k}{2s}\right) \Big]\\
=\frac{c_4}{p^0}\int_{\rth}\frac{dq}{q^0}e^{-q^0}\sqrt{s}\Phi(g)g^3\int_{0}^\infty dk \frac{\sigma_0(\cos\theta_{\Lambda})}{ g^4_\Lambda}
\Big[1- \exp\left(-\frac{(p^0+q^0)k}{2s}\right) \Big].
\end{multline*}
Here $c_4 = 2c'$.
	We start by showing the upper-bound estimates of $\zeta_1$. By the fundamental theorem of calculus, we have
	\begin{multline}\label{FREQ:zetanew}
	\zeta_1=\frac{c_4}{p^0}\int_{\rth}\frac{dq}{q^0}e^{-q^0}\sqrt{s}\Phi(g)g^3\int_{0}^\infty dk ~\frac{\sigma_0(\cos\theta_{\Lambda})}{ g^4_\Lambda}
	\Big[1- \exp\left(-\frac{(p^0+q^0)k}{2s}\right) \Big]\\
	=\frac{c_4}{p^0}\int_{\rth}\frac{dq}{q^0}e^{-q^0}\sqrt{s}\Phi(g)g^3\int_{0}^\infty dk ~\frac{\sigma_0(\cos\theta_{\Lambda})}{ g^4_\Lambda}
	\\\times	\int_0^1d\vartheta\ \exp\left(-\frac{(p^0+q^0)k}{2s}\vartheta\right)\frac{(p^0+q^0)k}{2s}.
	\end{multline}
	Note that using \eqref{angassumption}, \eqref{hard}, \eqref{soft},  \eqref{singS.defin}, \eqref{FREQ:gl}, and \eqref{FREQ:tl}, we have
\begin{equation} \label{FREQ:new.ang.est.zeta.doit}
\Phi(g)\approx  g^{\singS}
\ \text{and} \
\sigma_0(\cos\theta_{\Lambda})\approx\left(\frac{k}{k+g^2}\right)^{-1-\gamma/2}\approx g_\Lambda^{2+\gamma} k^{-1-\gamma/2}.
\end{equation}
We will use this equivalence in the following developments.

We split into two cases: $k\leq 4$ and $k>4$.  First consider $k\leq 4$.
We use $\exp\left(-\frac{(p^0+q^0)k}{2s}\vartheta\right)\leq 1$ and $g\leq g_\Lambda = (g^2 + k)^{1/2}$ from \eqref{FREQ:gl}, then when $k\leq 4$ we have
	\begin{multline}\label{FREQ:zeta.use.hard.too}
	\zeta_1(p)\lesssim \frac{1}{p^0}\int_{\rth}\frac{dq}{q^0}e^{-q^0}\sqrt{s}g^{3}\Phi(g)\int_{0}^{4} dk \frac{g_\Lambda^{2+\gamma}}{ g^4_\Lambda} k^{-1-\frac{\gamma}{2}}
	\frac{(p^0+q^0)k}{2s}
	\\
	\lesssim \int_{\rth}dq\ e^{-q^0}\Phi(g)s^{\frac{\gamma}{2}}\int_{0}^{4} dk \ k^{-\frac{\gamma}{2}}.
	\end{multline}
	Here we used $g_\Lambda \le \sqrt{s}$ when $k\leq 4$.
	Since $\gamma\in(0,2)$, the integral converges.
	
	Now we use $g\lesssim \sqrt{p^0q^0}$ from \eqref{FREQ:s.ge.g2} and \eqref{FREQ:s.le.pq} in the hard interaction \eqref{hard} case.  Alternatively we will use 
	$g\geq \frac{|p-q|}{\sqrt{p^0q^0}}$ from \eqref{FREQ:g.ge.lower} in the soft interaction \eqref{soft} case, and $s\lesssim p^0q^0$  from \eqref{FREQ:s.le.pq}.  Then on $k\le 4$ we further have
	\begin{equation}\notag
	\zeta_1(p)\lesssim \int_{\rth}dq\ e^{-q^0}(p^0q^0)^{\frac{\singA+\gamma}{2}}
	\lesssim
	(p^0)^{\frac{\singA+\gamma}{2}},
	\end{equation}for the hard interactions, and
	\begin{equation}\notag
	\zeta_1(p)\lesssim \int_{\rth}dq\ e^{-q^0}\left(\frac{|p-q|}{\sqrt{p^0q^0}}\right)^{-\singB}(p^0q^0)^{\frac{\gamma}{2}}
	\lesssim
	(p^0)^{\frac{-\singB+\gamma}{2}},
	\end{equation}
	for the soft interactions.

	On the other hand, when $k>4$, we still have \eqref{FREQ:new.ang.est.zeta.doit} and \eqref{FREQ:zetanew}.
Hence
\begin{multline}\label{FREQ:k.ge.four.est.zeta}
	\zeta_1(p)
	\lesssim
	\frac{1}{p^0}\int_{\rth}\frac{dq}{q^0}e^{-q^0}\sqrt{s}g^{3}\Phi(g)\int_{4}^\infty dk \frac{\Big[1- \exp\left(-\frac{(p^0+q^0)k}{2s}\right) \Big]}{ k^{1+\gamma/2}(k+g^2)^{1-\gamma/2}}
\\
	\lesssim
	\frac{1}{p^0}\int_{\rth}\frac{dq}{q^0}e^{-q^0}\sqrt{s}g^{1+\gamma}\Phi(g)\int_{4}^\infty dk ~ k^{-1-\gamma/2}  \Big[1- \exp\left(-\frac{(p^0+q^0)k}{2s}\right) \Big]
\\
\lesssim
	\int_{\rth} dq ~ e^{-q^0}g^{\gamma}\Phi(g)\int_{4}^\infty dk ~ k^{-1-\gamma/2}
	\Big[1- \exp\left(-\frac{(p^0+q^0)k}{2s}\right) \Big]
\\
\lesssim
	\int_{\rth} dq ~ e^{-q^0}g^{\gamma}\Phi(g)\int_{4}^\infty dk ~ k^{-1-\gamma/2}.
\end{multline}
Above we used $g \lesssim \sqrt{s} \lesssim \sqrt{\pZ \qZ}$ from \eqref{FREQ:s.ge.g2}-\eqref{FREQ:s.le.pq} and $\Big[1- \exp\left(-\frac{(p^0+q^0)k}{2s}\right) \Big] \lesssim 1$.

 Then, also using $g\lesssim \sqrt{p^0q^0}$ for hard interactions \eqref{hard} and \eqref{FREQ:g.ge.lower} for the soft interactions \eqref{soft}, when $k\geq 4$, we have
	\begin{multline*}
	\zeta_1(p)
	\lesssim
	\int_{\rth} dq ~ e^{-q^0}g^{\singA+\gamma}\int_{4}^\infty dk ~ k^{-1-\gamma/2}
		\lesssim
	\int_{\rth} dq ~ e^{-q^0}g^{\singA+\gamma}
\\
\lesssim
	\int_{\rth} dq ~ e^{-q^0}(p^0q^0)^{\frac{\singA+\gamma}{2}}
	\lesssim
	(p^0)^{\frac{\singA+\gamma}{2}},
	\end{multline*}in the hard interaction case, and
	\begin{multline*}
	\zeta_1(p)
	\lesssim
	\int_{\rth} dq ~ e^{-q^0}g^{-\singB+\gamma}\int_{4}^\infty dk ~ k^{-1-\gamma/2}
		\lesssim
	\int_{\rth} dq ~ e^{-q^0}g^{-\singB+\gamma}
\\
\lesssim
	\int_{\rth} dq ~ e^{-q^0}\left(\frac{|p-q|}{\sqrt{p^0q^0}}\right)^{-\singB+\gamma}
	\lesssim
	(p^0)^{\frac{-\singB+\gamma}{2}},
	\end{multline*}in the soft interaction case.
	This completes the upper-bound estimate of $\zeta_1$.  
\end{proof}

On the other hand, we have the following upper-bound estimate for $\zeta_2$:
\begin{lemma}\label{FREQ:lemma.zeta2}
	Suppose $\gamma \in (0,2)$ in \eqref{angassumption}.  Then for both hard \eqref{hard} and soft \eqref{soft} interactions with \eqref{singS.defin} we have the following uniform upper bound for \eqref{FREQ:zeta2B} when $|p|\ge 1$:
	$$|\zeta_2(p)|\lesssim (p^0)^{\frac{\singS+\gamma}{2}}.$$
	\end{lemma}

	\begin{proof}
	We use the change of variables $r\mapsto y\eqdef \frac{r}{\sqrt{s}}$ on \eqref{FREQ:zeta2B}. This yields
	\begin{multline}\notag
\zeta_2 \eqdef \frac{c'}{p^0}e^{\frac{p^0}{4}}\int_{\rth}\frac{dq}{q^0}\frac{e^{-\frac{3}{4}q^0}}{g}\int_{0}^\infty \frac{\sqrt{s}ydy}{\sqrt{y^2+1}}s_\Lambda\sigma(g_\Lambda,\theta_\Lambda)  \frac{s\Phi(g)g^4}{s_\Lambda  \Phi(g_\Lambda)g^4_\Lambda}
\\
\times
\exp\left(-\frac{p^0+q^0}{4}\sqrt{y^2+1}\right)\left[1-I_0\left(\frac{|p\times q|}{2g}y\right) \right].
\end{multline}
Recall \eqref{FREQ:lj}.   Note that $\sigma(g_\Lambda,\theta_\Lambda)=\Phi(g_\Lambda)\sigma_0(\theta_\Lambda)\ge 0.$ Then we have
 \begin{multline}\label{FREQ:zeta2New}
 \zeta_2 \eqdef \frac{c'}{p^0}e^{\frac{p^0}{4}}\int_{\rth}\frac{dq}{q^0}\frac{e^{-\frac{3}{4}q^0}}{g}\int_{0}^\infty \frac{\sqrt{s}ydy}{\sqrt{y^2+1}}s\Phi(g)\sigma_0(\theta_\Lambda)  \frac{g^4}{g^4_\Lambda}
 \\
 \times
 \exp\left(-l\sqrt{y^2+1}\right)\left[1-I_0(jy) \right].
 \end{multline} 
 Note that $I_0 \ge 1$ so that $\zeta_2 \le 0$.
 By \eqref{FREQ:ang.equiv}, using $g^2\leq s$ from \eqref{FREQ:s.ge.g2} we have
$$
	\sigma_0(\theta_\Lambda)
	\lesssim (y^2+\sqrt{y^2+1})^{1+\gamma/2}y^{-2-\gamma},
$$
Plugging this into \eqref{FREQ:zeta2New}, we have
	 \begin{multline}\label{FREQ:zeta2abs}
	|\zeta_2|\lesssim  \frac{1}{p^0}e^{\frac{p^0}{4}}\int_{\rth}\frac{dq}{q^0}\frac{e^{-\frac{3}{4}q^0}}{g}s\sqrt{s}\Phi(g)\int_{0}^\infty \frac{y^{-1-\gamma}dy}{\sqrt{y^2+1}}(y^2+\sqrt{y^2+1})^{1+\gamma/2}
	\\
	\times
	\exp\left(-l\sqrt{y^2+1}\right)\left[I_0(jy)-1 \right],
	\end{multline} where we also used $\frac{g^4}{g^4_\Lambda}\leq 1$ from \eqref{FREQ:g2y.variable}. Also note that
	\begin{multline*}	\exp\left(-l\sqrt{y^2+1}\right)=\exp(-l)	\exp\left(-l(\sqrt{y^2+1}-1)\right)\\=e^{\frac{-p^0-q^0}{4}}\exp\left(-l(\sqrt{y^2+1}-1)\right).\end{multline*} Plugging this into \eqref{FREQ:zeta2abs}, we have
		 \begin{equation}\label{FREQ:zeta2abs2}
	|\zeta_2|\lesssim  \frac{1}{p^0}\int_{\rth}\frac{dq}{q^0}\frac{e^{-q^0}}{g}s\sqrt{s}\Phi(g)Y(p,q),
		 \end{equation}
	where we define
	\begin{multline}\label{FREQ:eq.Y}Y(p,q)\eqdef \int_{0}^\infty \frac{y^{-1-\gamma}dy}{\sqrt{y^2+1}}(y^2+\sqrt{y^2+1})^{1+\gamma/2}
	\\\times\exp\left(-l(\sqrt{y^2+1}-1)\right)\left[I_0(jy)-1 \right].\end{multline}
	For the upper-bound estimate of $Y(p,q)$ we split the region $[0,\infty)$ into two: 
	\begin{equation}\notag
	Y(p,q) = \tilde{Y}_1(p,q) + \tilde{Y}_2(p,q),
	\end{equation}   
	where $\tilde{Y}_1(p,q)$ is the integral in \eqref{FREQ:eq.Y} restricted to the integration region $y\ge 1$ and then $\tilde{Y}_2(p,q)$ is the expression in \eqref{FREQ:eq.Y} on the integration region $0<y<1$.

First we consider the case $\tilde{Y}_1(p,q)$  that $y\geq 1$. When $y\ge 1$, we have
$$
	\frac{y^{-1-\gamma}}{\sqrt{y^2+1}}(y^2+\sqrt{y^2+1})^{1+\gamma/2}
		\lesssim
	\frac{y^{-1-\gamma}}{\sqrt{y^2+1}}(y^2)^{1+\gamma/2}
	\lesssim\frac{y}{\sqrt{y^2+1}}.
$$ 
Therefore, on the region $y\geq 1$, using \eqref{FREQ:J2.special} we have
$$
\tilde{Y}_1(p,q)\lesssim \exp(l) \int_1^\infty dy\  \frac{y}{\sqrt{y^2+1}}\exp(-l\sqrt{y^2+1})I_0(jy)
\lesssim \exp(l) J_2(l,j).
$$ 
By \eqref{FREQ:J2.lemma} we then have
$$
\tilde{Y}_1(p,q)
\lesssim \exp(l) \frac{\exp(-\sqrt{l^2-j^2})}{\sqrt{l^2-j^2}}.
$$
Since $p^0-q^0\leq |p-q|$ from \eqref{FREQ:p0q0.le.pq}, we have 
	\begin{equation}\label{FREQ:exponential.bound.1}
	    \exp\left(\frac{p^0-q^0-|p-q|}{4}\right)\leq 1.
	\end{equation}
Thus, using \eqref{FREQ:g.le.upper}, \eqref{FREQ:l2j2}, \eqref{FREQ:l2j2size} and \eqref{FREQ:exponential.bound.1} we have 
\begin{multline*}
\tilde{Y}_1(p,q)
\lesssim \exp\left(\frac{p^0+q^0}{4}-\frac{\sqrt{s}}{4g}|p-q|\right)\frac{4g}{\sqrt{s}|p-q|}\\
	\lesssim \exp\left(\frac{q^0}{2}\right)\exp\left(\frac{p^0-q^0}{4}-\frac{|p-q|}{4}\right)\frac{4}{\sqrt{s}}\lesssim \frac{\exp\left(\frac{q^0}{2}\right)}{\sqrt{s}}.
\end{multline*} 
Now we will use $\Phi(g)\approx g^{\singS}$ from \eqref{hard}-\eqref{singS.defin}.  In the hard interaction case \eqref{hard} we use \eqref{FREQ:g.ge.lower} and \eqref{FREQ:g.le.sqrtpq} in \eqref{FREQ:zeta2abs2} to conclude that
	\begin{equation}\label{FREQ:zeta2firstcase.hard}
	\left[\zeta_2\right]_{y \ge 1}\lesssim  \int_{\rth}dq\frac{\exp\left(\frac{-q^0}{2}\right)}{|p-q|}(p^0q^0)^{\frac{\singA+1}{2}}
	\lesssim(p^0)^{\frac{\singA}{2}-\frac{1}{2}}.
	\end{equation}
Then in the soft interaction case \eqref{soft}, using $b<2$, we use \eqref{FREQ:s.le.pq} and \eqref{FREQ:g.ge.lower} to obtain
	\begin{equation}\label{FREQ:zeta2firstcase.soft}
	\left[\zeta_2\right]_{y \ge 1}\lesssim \int_{\rth}dq\frac{\exp\left(\frac{-q^0}{2}\right)}{|p-q|^{1+\singB}}(p^0q^0)^{\frac{\singB+1}{2}}
	\lesssim(p^0)^{-\frac{\singB}{2}-\frac{1}{2}}.
	\end{equation}
	Here $\left[\zeta_2\right]_{y \ge 1}$ is $\zeta_2$ restricted to the integration region ${y \ge 1}$ using the convention \eqref{convention}.
	This completes the proof for the upper bound of $\zeta_2$ when  ${y \ge 1}$.

Alternatively, using \eqref{FREQ:eq.Y} we will show that $|\zeta_2|$ on $0<y<1$ is bounded uniformly from above by $(p^0)^{\frac{\singS}{2}+\frac{\gamma}{2}}.$ 
We prove this using the known Taylor expansion of the modified Bessel function of the first kind $I_0$ \cite{Gradshteyn:1702455}  as follows:
	$$I_0(jy)=\sum_{M=0}^{\infty}\frac{1}{(M!)^2}\left(\frac{jy}{2}\right)^{2M}.$$
	Now, since $y<1$, recalling \eqref{FREQ:eq.Y} we have
	$$\frac{y^{-1-\gamma}}{\sqrt{y^2+1}}(y^2+\sqrt{y^2+1})^{1+\gamma/2}\lesssim y^{-1-\gamma},$$ and
	$$\exp\left(-l(\sqrt{y^2+1}-1)\right)=\exp\left(-l\frac{y^2}{\sqrt{y^2+1}+1}\right)\leq \exp\left(-l\frac{y^2}{\sqrt{2}+1}\right).$$
	Therefore, by \eqref{FREQ:eq.Y}, we have
	\begin{equation*}
	\begin{split}
	\tilde{Y}_2(p,q) &\lesssim  \int_0^1 dy \ y^{-1-\gamma}
	\exp\left(-l\frac{y^2}{\sqrt{2}+1}\right)\sum_{M=1}^{\infty}\frac{1}{(M!)^2}\left(\frac{jy}{2}\right)^{2M}\\
	&\lesssim \sum_{M=1}^{\infty} \frac{1}{(M!)^2}(j/2)^{2M} \int_0^1 dy \ y^{-1-\gamma+2M}
	\exp\left(-cly^2\right),
	\end{split}
	\end{equation*} 
	where we define 
\begin{equation}
    	\label{FREQ:c.def}
	c\eqdef \frac{1}{1+\sqrt{2}}.
\end{equation}
	For $M\ge 1$, we further define 
	$$
	Y_M \eqdef \int_0^1dy \ y^{-1-\gamma+2M}
	\exp\left(-cly^2\right).
	$$
	Here we take a change of variables $y\mapsto z=ly^2$ with $dz=2lydy$ and obtain
	\begin{multline}\notag
	Y_M\leq  \frac{l^{\frac{\gamma}{2}-M}}{2} \int_0^l dz \ z^{-1-\gamma/2+M}\exp\left(-cz\right) \\\leq\frac{l^{\frac{\gamma}{2}-M}}{2} \int_0^\infty dz \ z^{-1-\gamma/2+M}    	\exp\left(-cz\right)\\
	\leq\frac{l^{\frac{\gamma}{2}-M}}{2}3^{M-1} \sup_{z\in[0,\infty)} \left\{\left(\frac{z}{3}\right)^{M-1}    	\exp\left(-\frac{z}{3}\right)\right\}\int_0^\infty dz \ z^{-\gamma/2}    	\exp\left(-(c-1/3)z\right)\\
		\leq C_1 3^{M}\sup_{z\in[0,\infty)} \left\{\left(\frac{z}{3}\right)^{M-1}    	\exp\left(-\frac{z}{3}\right)\right\}l^{\frac{\gamma}{2}-M}, \end{multline}
where the constant $C_1$ is uniformly bounded since $\gamma \in(0,2)$ as
\begin{equation} \notag
    C_1 \eqdef \frac{1}{6} \int_0^\infty dz \ z^{-\gamma/2}    	\exp\left(-(c-1/3)z\right) < \infty.
\end{equation}
This holds because $c> \frac{1}{3}$ from \eqref{FREQ:c.def}.  We use the Stirling formula error bounds to obtain 
	$$\sup_{z\in[0,\infty)} \left\{\left(\frac{z}{3}\right)^{M-1}    	\exp\left(-\frac{z}{3}\right)\right\}
	\leq
	\frac{1}{\sqrt{2\pi}}
	\frac{(M-1)!}{\sqrt{M-1}}
		\leq
	\frac{1}{\sqrt{4\pi}}
	\frac{M!}{\sqrt{M}},\text{ if }M\ge 2.
	$$
	Alternatively if $M=1$ we have the bound 
		$$\sup_{z\in[0,\infty)} \left\{    	\exp\left(-\frac{z}{3}\right)\right\}
\leq 1, \text{ if }M=1.$$
	Therefore we have the general bound
	\begin{equation}\label{FREQ:yM.bound}
	    Y_M \leq C_1 3^M \frac{M!}{\sqrt{M}} l^{\frac{\gamma}{2}-M}, \quad M\ge 1.
	\end{equation}
We will use this bound to estimate $\left[\zeta_2\right]_{0<y<1}$ using the convention \eqref{convention}.

First we notice that using \eqref{FREQ:lj} we have
	\begin{equation}\label{FREQ:YM2spc}
	j^{2M} l^{\frac{\gamma}{2}-M} \leq (q^0)^{M}l^{\frac{\gamma}{2}},\end{equation} 
	where to prove \eqref{FREQ:YM2spc} we used $j^2/l\leq q^0$ which follows from \eqref{FREQ:g.ge.2lower} as 
	\begin{equation}\label{FREQ:j2l}\frac{j^2}{l}=\frac{|p\times q|^2}{g^2(p^0+q^0)}\leq \frac{p^0q^0}{p^0+q^0}\leq q^0.
	\end{equation} 
	Now we plug \eqref{FREQ:yM.bound} and \eqref{FREQ:YM2spc} into \eqref{FREQ:zeta2abs2} with $\tilde{Y}_2(p,q)$, to obtain
	\begin{multline*}
	\left[\zeta_2\right]_{0<y<1}
	\lesssim  \frac{1}{p^0}\int_{\rth}\frac{dq}{q^0}\frac{e^{-q^0}}{g}s\sqrt{s}\Phi(g)l^{\frac{\gamma}{2}}\sum_{M=1}^\infty \frac{1}{M!\sqrt{M}}\left(\frac{3}{4}\right)^{M}(q^0)^M
	\\
\lesssim 	 \frac{1}{p^0}\int_{\rth}\frac{dq}{q^0}\frac{e^{-q^0}}{g}s\sqrt{s}\Phi(g)l^{\frac{\gamma}{2}}\exp\left({\frac{3}{4}q^0}\right).
\end{multline*} 
We use $\Phi(g)\approx g^{\singS}$ with $-2 <\rho$ from \eqref{singS.defin}.  
In the hard interaction case \eqref{hard}, we will use   \eqref{FREQ:l.upper.ineq}, \eqref{FREQ:g.ge.lower} and \eqref{FREQ:g.le.sqrtpq} to conclude that
	\begin{equation}\label{FREQ:zeta2secondcase.hard}
	\left[\zeta_2\right]_{0<y<1}\lesssim \int_{\rth}\frac{dq\ e^{-\frac{q^0}{4}}}{|p-q|}(p^0q^0)^{\frac{1}{2}+\frac{\gamma}{2} +\frac{1+\singA}{2}}\lesssim (p^0)^{\frac{\singA}{2}+\frac{\gamma}{2}}.
	\end{equation}
And in the soft interaction case \eqref{soft} we will use 	\eqref{FREQ:g.ge.lower} and \eqref{FREQ:s.le.pq} to obtain
	\begin{equation}\label{FREQ:zeta2secondcase.soft}
	\left[\zeta_2\right]_{0<y<1}\lesssim \int_{\rth}dq\frac{e^{-\frac{q^0}{4}}}{|p-q|^{1+\singB}}(p^0q^0)^{\frac{1}{2}+\frac{\gamma}{2} +\frac{1+\singB}{2}}\lesssim (p^0)^{-\frac{\singB}{2}+\frac{\gamma}{2}},
	\end{equation}
where we recall that $1+\singB<3.$	This proves that $\left[\zeta_2\right]_{0<y<1}$ has the leading order upper bound.
\end{proof}

Thus we obtain Proposition \ref{FREQ:prop.zeta0.asymptotic} by combining Lemmas \ref{FREQ:lemma.zeta1} and \ref{FREQ:lemma.zeta2}. In the next section, we will prove that the remainder terms  $\zetaL$ from \eqref{FREQ:zetaLBy}, and $\zetaTLm$ from \eqref{FREQ:def.zetaMdef}  have lower order upper bounds.  We will also prove that $\zetaTone(p)$ from \eqref{FREQ:def.zetatilde1} has a lower order upper bound in Proposition \ref{FREQ:prop.tildezeta.upper.qgep}.

\subsection{Lower order upper bound estimates}\label{FREQ:sec:low order zeta1}

In this section, we study the upper bound estimates of $\zetaL$ from \eqref{FREQ:zetaLBy}, $\zetaTLm$ from \eqref{FREQ:def.zetaMdef} and $\zetaTone(p)$ from \eqref{FREQ:def.zetatilde1}, which together form part of $\zeta_{\mathcal{K}}$ in \eqref{FREQ:zetaK.def}.  Our goal will be to prove that $|\zetaL(p)|$, $|\zetaTLm(p)|$, and $|\zetaTone(p)|$ have lower order upper bounds.

\subsubsection{Lower order upper bound for $\zetaL(p)$} For the proof of the lower order upper bound of $|\zetaL|$ we will use the representation in \eqref{FREQ:zetaLBy}. We have the following uniform asymptotic bound:

\begin{proposition}\label{FREQ:prop.zetaL.asymptotic}
	Suppose $\gamma \in (0,2)$  in \eqref{angassumption}. Then for both hard \eqref{hard} and soft \eqref{soft} interactions, for \eqref{FREQ:zetaLBy} when $|p|\ge 1$, we have
	$$|\zetaL(p)| \lesssim (p^0)^{\frac{\singS}{2}}.$$
This bound then automatically also holds for $|\zetaLm(p)|$ from \eqref{FREQ:def.zetaMdef}.
\end{proposition}

\begin{proof}   By \eqref{FREQ:zetaLBy} and the definition of $l$ and $j$ of \eqref{FREQ:lj} we have
	\begin{multline}\label{FREQ:zetaL.def.here}
	\zetaL \eqdef \frac{c'}{p^0}e^{\frac{p^0}{4}}\int_{\rth}\frac{dq}{q^0}\frac{e^{-\frac{3}{4}q^0}}{g}\sqrt{s}\\
	\times\int_{0}^\infty \frac{ydy}{\sqrt{y^2+1}}s_\Lambda\sigma(g_\Lambda,\theta_\Lambda)
	\exp\left(-l\sqrt{y^2+1}\right)I_0\left(jy\right) \left(\frac{s\Phi(g)g^4}{s_\Lambda  \Phi(g_\Lambda)g^4_\Lambda}-1 \right).
	\end{multline} 
	In the hard interaction case \eqref{hard}, we have $\frac{\Phi(g)}{\Phi(g_\Lambda)}=\left(\frac{g}{g_\Lambda}\right)^{a}$ with $a<2.$ Since $g\leq g_\Lambda$ from \eqref{FREQ:g2y.variable},  we have $\frac{\Phi(g)}{\Phi(g_\Lambda)}\geq \frac{g^2}{g^2_\Lambda}$. Then this  implies
	$$
	\left|\frac{s\Phi(g)g^4}{s_\Lambda  \Phi(g_\Lambda)g^4_\Lambda} - 1\right|=1-\frac{s\Phi(g)g^4}{s_\Lambda  \Phi(g_\Lambda)g^4_\Lambda} \leq 1-\frac{sg^6}{s_\Lambda  g^6_\Lambda}=\frac{s_\Lambda g^6_\Lambda-sg^6}{s_\Lambda g^6_\Lambda}.
	$$
	Further note that we have 
	\begin{multline*}
	s_\Lambda g^6_\Lambda-sg^6=(g^8_\Lambda-g^8)+4(g^6_\Lambda-g^6)\\= (g^2_\Lambda-g^2)\left((g_\Lambda^4+g^4)(g^2_\Lambda+g^2)+4g^4_\Lambda+4g^2_\Lambda g^2+4g^4\right)\\\lesssim \frac{s}{2}(\sqrt{y^2+1}-1)
	g^4_\Lambda s_\Lambda,
	\end{multline*}
	since $g^2_\Lambda-g^2=\frac{s}{2}(\sqrt{y^2+1}-1)$ from \eqref{FREQ:g2y.variable}, $s_\Lambda\eqdef g^2_\Lambda+4$ and again $g\leq g_\Lambda$.
	Therefore, we have 
\begin{equation}\label{FREQ:difference.estimate.here}
    	\left|\frac{s\Phi(g)g^4}{s_\Lambda  \Phi(g_\Lambda)g^4_\Lambda} - 1\right|\leq\frac{s_\Lambda g^6_\Lambda-sg^6}{s_\Lambda g^6_\Lambda}\lesssim \frac{\frac{s}{2}(\sqrt{y^2+1}-1)}{g^2_\Lambda}.
\end{equation}
	This is the main estimate  for this difference in the hard interaction case.
	
We now consider the same estimate in the soft interaction case \eqref{soft}.   Since $g\leq g_\Lambda$ and $\frac{\Phi(g)}{\Phi(g_\Lambda)}=\left(\frac{g}{g_\Lambda}\right)^{-b}$ with $b\in[\gamma,2),$ then we have $\frac{\Phi(g)}{\Phi(g_\Lambda)}\geq 1$. Then $b\in[\gamma,2)$ 
	further implies
	$$1-\frac{s\Phi(g)g^4}{s_\Lambda  \Phi(g_\Lambda)g^4_\Lambda} \leq 1-\frac{sg^4}{s_\Lambda  g^4_\Lambda}=\frac{s_\Lambda g^4_\Lambda-sg^4}{s_\Lambda g^4_\Lambda}.
	$$
In this case we also have
\begin{multline*}
	s_\Lambda g^4_\Lambda-sg^4=(g^6_\Lambda-g^6)+4(g^4_\Lambda-g^4)= (g^2_\Lambda-g^2)\left(g^4_\Lambda+g^2_\Lambda g^2+g^4+4g^2_\Lambda+4g^2\right)\\=(g^2_\Lambda-g^2)\left(g^4_\Lambda+g^2_\Lambda g^2+g^4+4g^2_\Lambda+4g^2\right)
	\leq \frac{s}{2}(\sqrt{y^2+1}-1)(3g^4_\Lambda+8g^2_\Lambda)\\
	\lesssim
	\frac{s}{2}(\sqrt{y^2+1}-1)g^2_\Lambda s_\Lambda,
\end{multline*}
because again $g^2_\Lambda-g^2=\frac{s}{2}(\sqrt{y^2+1}-1)$ from \eqref{FREQ:g2y.variable}.  Therefore, we have
	\begin{equation}\label{FREQ:upperbound.allcases.zeta1}
	\left|\frac{s\Phi(g)g^4}{s_\Lambda  \Phi(g_\Lambda)g^4_\Lambda} - 1\right|\leq\frac{s_\Lambda g^4_\Lambda-sg^4}{s_\Lambda g^4_\Lambda}\lesssim \frac{\frac{s}{2}(\sqrt{y^2+1}-1)}{g^2_\Lambda}.
	\end{equation}
Note that in both the hard interaction case and the soft interaction case the final upper bounds are the same in  \eqref{FREQ:difference.estimate.here} and \eqref{FREQ:upperbound.allcases.zeta1}.

	In both cases, then plugging \eqref{FREQ:difference.estimate.here} and \eqref{FREQ:upperbound.allcases.zeta1} into \eqref{FREQ:zetaL.def.here}  we have 
	\begin{equation}\label{FREQ:zetaLm.upper.first}
	|\zetaL|  \lesssim  \frac{1}{p^0}e^{\frac{p^0}{4}}\int_{\rth}\frac{dq}{q^0}\frac{e^{-\frac{3}{4}q^0}}{g}s^{1/2}
K_2(p,q),
	\end{equation}
where we define $K_2 = K_2(p,q)$ by
	\begin{equation}\label{FREQ:K2definition.appendix}
	K_2\eqdef  \int_{0}^\infty \frac{ydy}{\sqrt{y^2+1}}s_\Lambda\sigma(g_\Lambda,\theta_\Lambda)
	\exp\left(-l\sqrt{y^2+1}\right)I_0\left(jy\right) \frac{\frac{s}{2}(\sqrt{y^2+1}-1)}{g^2_\Lambda}.
	\end{equation}
	We will split into two cases:
	 $y\leq 1$ and $y>1$. We write $K_2=K_{2,\le 1}+K_{2,\ge 1}$ below where  $K_{2,\le 1}$ and $ K_{2,\ge 1}$ denote $K_2$ on $y\le 1$ and $y\ge 1$, respectively.

First let us generally estimate the kernel.  We will use the product form \eqref{define.kernel} with the estimates \eqref{hard}-\eqref{soft}-\eqref{singS.defin} to obtain
\begin{equation}\notag
    s_\Lambda\sigma(g_\Lambda,\theta_\Lambda) \frac{\frac{s}{2}(\sqrt{y^2+1}-1)}{g^2_\Lambda}
    \approx s_\Lambda g_\Lambda^{\rho-2} s \sigma_0(\theta_\Lambda) (\sqrt{y^2+1}-1) 
    \approx   \frac{s_\Lambda g_\Lambda^{\rho} s \sigma_0(\theta_\Lambda) y^2 }{g_\Lambda^{2}(\sqrt{y^2+1}+1)}.
\end{equation}
Additionally using \eqref{FREQ:ang.equiv} with \eqref{FREQ:glambda.calc} we have 
\begin{equation}\notag
\sigma_0(\theta_\Lambda)
\approx
\left(\frac{s}{ g_\Lambda^2} \frac{y^2}{2(\sqrt{y^2+1}+1)}\right)^{-1-\gamma/2}
\lesssim  y^{-2-\gamma}(\sqrt{1+y^2})^{1+\gamma/2} \left(\frac{ g_\Lambda^2}{s}\right)^{1+\gamma/2}
\end{equation}
We plug this into the previous estimate to obtain
\begin{multline}\notag
s_\Lambda\sigma(g_\Lambda,\theta_\Lambda) \frac{\frac{s}{2}(\sqrt{y^2+1}-1)}{g^2_\Lambda}
\lesssim  
y^{-\gamma}(1+y^2)^{\gamma/4}
\frac{s_\Lambda g_\Lambda^{\rho} s }{g_\Lambda^{2}}\left(\frac{ g_\Lambda^2}{s}\right)^{1+\gamma/2}
\\
\lesssim  
y^{-\gamma}(1+y^2)^{\gamma/4}
s_\Lambda g_\Lambda^{\rho} \left(\frac{ g_\Lambda^2}{s}\right)^{\gamma/2}.
\end{multline}
We conclude from \eqref{FREQ:ineq.gL.here} and the above that in general we have
\begin{equation}\notag
s_\Lambda\sigma(g_\Lambda,\theta_\Lambda) \frac{\frac{s}{2}(\sqrt{y^2+1}-1)}{g^2_\Lambda}
\lesssim  
y^{-\gamma}(1+y^2)^{\gamma/2}
s_\Lambda g_\Lambda^{\rho} 
\lesssim  
y^{-\gamma}(1+y^2)^{(1+\gamma)/2}
s g_\Lambda^{\rho}.
\end{equation}
In particular, recalling \eqref{hard}-\eqref{soft}-\eqref{singS.defin} and using \eqref{FREQ:difference.estimate.here}, \eqref{FREQ:upperbound.allcases.zeta1}, and \eqref{FREQ:ineq.gL.here}, then in general we have
\begin{equation}\label{FREQ:general.kernel.est.here}
s_\Lambda\sigma(g_\Lambda,\theta_\Lambda)
	\left|\frac{s\Phi(g)g^4}{s_\Lambda  \Phi(g_\Lambda)g^4_\Lambda} - 1\right|
\lesssim  \mathbb{S} ~ 
y^{-\gamma}(1+y^2)^{1+\gamma/2}, \quad \forall 0 \le y \le \infty.
\end{equation}
This holds in particular since $a<2$ in \eqref{hard}.  Here we define
	\begin{equation}\label{FREQ:S.hard.soft.def}
	\begin{split}\mathbb{S}&\eqdef s^{1+\frac{a}{2}} \text{  for hard interactions \eqref{hard},}\\
	&\eqdef sg^{-b}\text{  for soft interactions \eqref{soft}.}\end{split}\end{equation}
These are the specific estimates that we will use on the kernel of  \eqref{FREQ:K2definition.appendix}.

	Now we return to estimating \eqref{FREQ:K2definition.appendix} on the region when $y\leq 1$.  
	Then using the above calculations we have the following bound for $\left|K_{2,\le 1}\right|$:
	\begin{multline}\notag
	|K_{2,\le 1}|
	\lesssim
	\mathbb{S}\int_0^{1}dy\  y^{1-\gamma}
	(1+y^2)^{(1+\gamma)/2}
	\exp\left(-l\sqrt{y^2+1}\right) I_0(j y)
	\\
	\lesssim
\mathbb{S}\int_0^{1}dy\  y^{1-\gamma}
\exp\left(-l\sqrt{y^2+1}\right) I_0(j y)
	\lesssim
\mathbb{S}\bar{K}_\gamma(l,j),
	\end{multline}  
where we used $(1+y^2)^{(1+\gamma)/2}\lesssim 1$ as $y\in (0,1)$, and $\gamma>0$.  	Since $\gamma\in (0,2)$, this integral converges.  In the last upper bound we used \eqref{FREQ:int.Kgamma}. From \eqref{FREQ:smally.lemma} we conclude
	\begin{equation}\label{FREQ:smally}
|K_{2,\le 1}|
	\lesssim
	\mathbb{S} \exp\left(-\sqrt{l^2-j^2}\right).
	\end{equation}
	This completes our estimate on the region when $y \le 1$.

	On the other region when $y>1$, using \eqref{FREQ:general.kernel.est.here} and \eqref{FREQ:S.hard.soft.def},   for the integral defined in \eqref{FREQ:K2definition.appendix}  we have for both hard and soft interactions, that
$$
|K_{2,\ge 1}|
\lesssim
\mathbb{S}\int_{1}^\infty\ dy\ y(y^2+1)^{1/2}\exp\left(-l\sqrt{y^2+1}\right) I_0(j y)
\lesssim \mathbb{S} \tilde{K}_2(l,j).
$$
Here $\tilde{K}_2(l,j)$ is defined in \eqref{FREQ:tildeK2def}.  The formula for $\tilde{K}_2$ is \eqref{FREQ:k2lj.lemma} 
and we have
\begin{equation}\label{FREQ:k2lj}
\tilde{K}_2(l,j)
	\lesssim (\sqrt{l^2-j^2})^{-5}\exp(-\sqrt{l^2-j^2}) (l^2-j^2+1)l^2.
\end{equation}
By \eqref{FREQ:l2j2}, we have $l^2-j^2=\frac{s}{16g^2}|p-q|^2.$
Thus, we obtain
\begin{multline}\notag\tilde{K}_2(l,j)\lesssim \frac{l^2}{(\sqrt{l^2-j^2})^3} \left(1+\frac{1}{l^2-j^2}\right)\exp(-\sqrt{l^2-j^2}) \\
 	\lesssim
 	\frac{(p^0+q^0)^2}{16}\left(\frac{4g}{\sqrt{s}|p-q|}\right)^3\left(1+\left(\frac{4g}{\sqrt{s}|p-q|}\right)^2\right)\exp(-\sqrt{l^2-j^2}).\end{multline}
 	We point out that due to \eqref{FREQ:g.le.upper} the above is not singular when $|p-q|=0$.

 	Note that using \eqref{FREQ:s.ge.g2} and \eqref{FREQ:g.le.upper} we have
 \begin{equation}\label{FREQ:tildeK2size0} 1+\left(\frac{4g}{\sqrt{s}|p-q|}\right)^2\le 1+\frac{16}{s}\lesssim 1.  
\end{equation} 
Also using $g\le \sqrt{s}$, which follows from \eqref{FREQ:s.ge.g2}, we have
 \begin{multline}\label{FREQ:tildeK2size}\tilde{K}_2(l,j)\lesssim  \frac{(p^0+q^0)^2}{16}\left(\frac{4g}{\sqrt{s}|p-q|}\right)^3\exp(-\sqrt{l^2-j^2})\\
 	\lesssim  	\frac{(p^0+q^0)^2}{16}\left(\frac{4}{|p-q|}\right)^3\exp(-\sqrt{l^2-j^2})\\
\lesssim  	\frac{(p^0+q^0)^2}{|p-q|^3}\exp(-\sqrt{l^2-j^2}).
 \end{multline}
 We will use estimate \eqref{FREQ:tildeK2size} to control the size of $|K_{2,\ge 1}|$ below.

We will now to use the region $|q|\le \frac{1}{2}|p|^{1/m}$ and $|p| \ge 1$  to complete our estimate of $|K_{2,\ge 1}|$.  Then later we will do separate estimates on the complementary region: $|q|\ge \frac{1}{2}|p|^{1/m}$.  Now since $|q|\le \frac{1}{2}|p|^{1/m}$, $m>1$ and $|p| \ge 1$, 
then we have 
\begin{equation}\label{FREQ:pminusq bound}\frac{p^0}{4}\le \frac{|p|}{2}\le |p-q|\le \frac{3}{2}|p|,
\end{equation}
and
\begin{equation}\label{FREQ:q0 bound}1\le q^0 \le2 (p^0)^{1/m}.\end{equation}
Thus we have
$$\frac{(p^0+q^0)^2}{|p-q|^3}\lesssim \frac{1}{p^0}.$$
Hence we obtain from \eqref{FREQ:tildeK2size} that 
 \begin{equation}
     \label{FREQ:tildeK2size.ref2}
 \tilde{K}_2(l,j)
\lesssim  	\frac{1}{p^0}\exp(-\sqrt{l^2-j^2}).
 \end{equation}
We therefore conclude from \eqref{FREQ:tildeK2size.ref2} that 
\begin{equation}
	\label{FREQ:largey}
		|K_{2,\ge 1}|\lesssim \mathbb{S}\frac{1}{p^0}\exp\left(-\sqrt{l^2-j^2}\right).
	\end{equation}
	By \eqref{FREQ:S.hard.soft.def}, \eqref{FREQ:smally} and \eqref{FREQ:largey}, using \eqref{convention}  we finally obtain
	\begin{multline*}
	[|\zetaL(p)|]_{|q|\le \frac{1}{2}|p|^{1/m}} \lesssim \frac{1}{p^0}e^{\frac{p^0}{4}}\int_{\qlep} \frac{dq}{q^0}\frac{\sqrt{s}}{g}e^{-\frac{3}{4}q^0}\mathbb{S}\exp\left(-\sqrt{l^2-j^2}\right)\\
\lesssim \frac{1}{p^0}\int_{\qlep} \frac{dq}{q^0}e^{-\frac{q^0}{2}}\frac{\sqrt{s}}{g}\mathbb{S}\exp\left(\frac{p^0-q^0}{4}-\sqrt{l^2-j^2}\right).
	\end{multline*} 
	We will use the inequality above to obtain the final upper bounds.

	In the hard interaction case \eqref{hard}, we use \eqref{FREQ:s.le.pq}, \eqref{FREQ:g.ge.lower} and \eqref{FREQ:l2j2size} to obtain
	\begin{equation}\label{FREQ:zeta2final.hard}
	[|\zetaL(p)|]_{|q|\le \frac{1}{2}|p|^{1/m}}\lesssim \int_{\rth}dq\ \frac{(p^0q^0)^{1+\frac{\singA}{2}}}{|p-q|}e^{-\frac{q^0}{2}}\exp\left(\frac{p^0-q^0-|p-q|}{4}\right).
	\end{equation}
	Therefore, from \eqref{FREQ:exponential.bound.1} we have
$$	[|\zetaL(p)|]_{|q|\le \frac{1}{2}|p|^{1/m}}\lesssim(p^0)^{\frac{\singA}{2}}.$$
	In the soft interaction case \eqref{soft}, we use \eqref{FREQ:g.ge.lower} and \eqref{FREQ:l2j2size} to obtain
	\begin{multline}\label{FREQ:zeta2final.soft}
	[|\zetaL(p)|]_{|q|\le \frac{1}{2}|p|^{1/m}} \\ 
	\lesssim \int_{	{|q|\le \frac{1}{2}|p|^{1/m}}}dq\ \frac{(p^0q^0)^{\frac{\singB}{2}+1}}{|p-q|^{1+\singB}} e^{-\frac{q^0}{2}}\exp\left(\frac{p^0-q^0-|p-q|}{4}\right).
	\end{multline}
	By \eqref{FREQ:exponential.bound.1}, since $\singB<2$, we then have
	\begin{equation*}	[|\zetaL(p)|]_{|q|\le \frac{1}{2}|p|^{1/m}} \lesssim(p^0)^{\frac{\singB}{2}+1}\int_{\rth}dq\ |p-q|^{-1-\singB} e^{-\frac{q^0}{2}}(q^0)^{\frac{\singB}{2}+1}
	\lesssim (p^0)^{-\frac{\singB}{2}}.
	\end{equation*} 
This completes the desired estimates on the region $\qlep$

Next we perform the estimates on the complementary region where $|q|\ge \frac{1}{2}|p|^{1/m}$, $|p| \ge 1$ and $m>1$.   Therefore, in this region we have
\begin{equation}\notag
q^0 \ge |q|  \ge \frac{1}{2} \left(|p|^2 \right)^{\frac{1}{2m}}
\ge 
\frac{1}{2}\left(\frac{1}{2}\right)^{\frac{1}{2m}} \left(p^0 \right)^{\frac{1}{m}}.
\end{equation}
Then, for some $c=c_m>0$, we have additional exponential decay from
\begin{equation}\label{FREQ:exponential.more}
e^{-q^0} = e^{-q^0/2} e^{-q^0/2}\le e^{-q^0/2} e^{-c(p^0)^{1/m}},
\end{equation}
Then with \eqref{FREQ:exponential.more} we have exponential decay in $p^0$.

Now we need to replace the estimates on $\tilde{K}_2(l,j)$ above, which is defined in \eqref{FREQ:tildeK2def}.  Recalling the estimates \eqref{FREQ:k2lj} and \eqref{FREQ:tildeK2size0}, instead of \eqref{FREQ:tildeK2size} we use \eqref{FREQ:s.ge.g2} and \eqref{FREQ:g.le.upper} to obtain
 \begin{multline}\label{FREQ:tilde.k2.estimate.region}
 \tilde{K}_2(l,j)\lesssim  \frac{(p^0+q^0)^2}{16}\left(\frac{4g}{\sqrt{s}|p-q|}\right)^3\exp(-\sqrt{l^2-j^2})\\
 	\lesssim (p^0+q^0)^2
 	\exp(-\sqrt{l^2-j^2}).
 \end{multline}
 We conclude from the above estimate, recalling also \eqref{FREQ:S.hard.soft.def}, that
 \begin{equation}\notag
		|K_{2,\ge 1}|\lesssim \mathbb{S}(p^0+q^0)^2
 	\exp(-\sqrt{l^2-j^2}).
	\end{equation}
Then by \eqref{FREQ:zetaLm.upper.first}, \eqref{FREQ:K2definition.appendix}, \eqref{FREQ:S.hard.soft.def}, \eqref{FREQ:smally} and the above,  we further obtain
	\begin{multline}\notag
	[|\zetaL(p)|]_{\qgep} \lesssim \frac{1}{p^0}e^{\frac{p^0}{4}}\int_{\qgep}\frac{dq}{q^0}\frac{\sqrt{s}}{g}e^{-\frac{3}{4}q^0}\mathbb{S}(p^0+q^0)^2\exp\left(-\sqrt{l^2-j^2}\right)\\
\lesssim \frac{1}{p^0}\int_{\qgep}\frac{dq}{q^0}e^{-\frac{q^0}{2}}\frac{\sqrt{s}}{g}\mathbb{S}(p^0+q^0)^2\exp\left(\frac{p^0-q^0}{4}-\sqrt{l^2-j^2}\right)
\\
\lesssim p^0e^{-c(p^0)^{1/m}}\int_{\qgep } dq ~ q^0 e^{-\frac{q^0}{4}}\frac{\sqrt{s}}{g}\mathbb{S},
	\end{multline} 	
where in the last inequality we used \eqref{FREQ:p0.plus.q0.le.p0q0}, \eqref{FREQ:l2j2size}, 	\eqref{FREQ:exponential.bound.1} and \eqref{FREQ:exponential.more}. Then from the same procedures we used to prove \eqref{FREQ:zeta2final.hard} and \eqref{FREQ:zeta2final.soft}, using the exponential decay in $p^0$ above, we obtain for some uniform $c'>0$ that
	\begin{equation}\label{FREQ:zetaLqgep}
	[|\zetaL(p)|]_{\qgep} 
\lesssim e^{-c'(p^0)^{1/m}}.
	\end{equation} 	
Combining the previous estimates, this completes the proof.
\end{proof}

This completes the proof of the lower order upper bound estimates for $\zetaL$ from \eqref{FREQ:zetaLBy}. Next, we prove that $\zetaTLm$ from \eqref{FREQ:def.TzetaMdef} has a lower order upper bound.

\subsubsection{Lower order upper bound for $\zetaTLm$}
We now prove the following proposition:

\begin{proposition}\label{FREQ:prop.zetaL.asymptoticnew}
	Suppose $\gamma \in (0,2)$  in \eqref{angassumption}, and recall \eqref{singS.defin}.  For any given small $\varepsilon>0$, assume that $m$ is sufficiently large such that  
	$\frac{|\singS|+8}{2m}\le \varepsilon.$  Then for both hard \eqref{hard} and soft \eqref{soft} interactions there exists a finite constant $C_\varepsilon>0$ such that for \eqref{FREQ:def.TzetaMdef} we have the following uniform asymptotic estimate
	$$\left|\zetaTLm(p)\right| \leq C_\varepsilon (p^0)^{\frac{\singS}{2}+\varepsilon}.$$
\end{proposition}

We recall \eqref{FREQ:eq:tildezetanew1L} and \eqref{FREQ:def.TzetaMdef} with \eqref{convention}.  Then we use the following representation in this section (implicitly assuming $|p| \ge 1$): 
\begin{multline}\notag
	\zetaTLm(p)
	=\frac{c'}{ p^0}\int_{\qlep}\frac{dq}{q^0}\frac{e^{-q^0}\sqrt{s}}{g}\int_{0}^\infty \frac{ydy}{\sqrt{y^2+1}}s_\Lambda\sigma(g_\Lambda,\theta_\Lambda)\left(1-\frac{s\Phi(g)g^4}{s_\Lambda  \Phi(g_\Lambda)g^4_\Lambda} \right)\\\times \left( \exp(2  l -2 l\sqrt{y^2+1})I_0(2jy)
-\exp(  l - l\sqrt{y^2+1})I_0(jy)\right)
	\eqdef \zetaLTone-\zetaLTtwo.\end{multline}
The splitting of $\zetaTL$ into $\zetaLTone$ and $\zetaLTtwo$ allows us to realize $\zetaLTtwo=-\zetaLm$ from \eqref{FREQ:def.zetaMdef} with \eqref{FREQ:zetaLBy} and the lower order upper bound estimate for $\zetaL$ was already given in Proposition \ref{FREQ:prop.zetaL.asymptotic}.

\begin{proof}
Based on the above discussion, in this proof we only need to give the asymptotic upper bound for $\zetaLTone$.  We start with
\begin{multline}\notag
	\zetaLTone =\frac{c'}{ p^0}\int_{\qlep}\frac{dq}{q^0}\frac{e^{-q^0}\sqrt{s}}{g}\int_{0}^\infty \frac{ydy}{\sqrt{y^2+1}}s_\Lambda\sigma(g_\Lambda,\theta_\Lambda)\left(1-\frac{s\Phi(g)g^4}{s_\Lambda  \Phi(g_\Lambda)g^4_\Lambda} \right)\\\times  \exp(2  l -2 l\sqrt{y^2+1})I_0(2jy).
\end{multline} 
We recall that we have the kernel estimate 
 \eqref{FREQ:general.kernel.est.here} with the notation \eqref{FREQ:S.hard.soft.def}.

Thus when $y\le 1$, since $0<\gamma<2$, we have
\begin{equation}\notag
	[ \zetaLTone]_{ y\le1} \lesssim\frac{1}{ p^0}\int_{\qlep}\frac{dq}{q^0}\frac{e^{-q^0}s^{1/2}}{g}\exp(2l)\bar{K}_\gamma(2l,2j) \mathbb{S},
\end{equation}
where we defined $\bar{K}_\gamma(l,j)$ in \eqref{FREQ:int.Kgamma}.
In particular from \eqref{FREQ:smally.lemma} we have
\begin{equation}\notag
\bar{K}_\gamma(2l,2j) \lesssim \exp (-\sqrt{(2l)^2-(2j)^2}).\end{equation} 
Thus, when $y\le 1$, we have
\begin{equation}\label{FREQ:smally.zetaL}
	[ \zetaLTone]_{ y\le1} \lesssim\frac{1}{ p^0}\int_{\qlep}\frac{dq}{q^0}\frac{e^{-q^0}s^{1/2}}{g}\exp (2l-\sqrt{4l^2-4j^2})\mathbb{S}.
\end{equation}
Above we are using the convention from \eqref{convention}.

 On the other hand, if $y\ge 1$, we again use the kernel estimate 
 \eqref{FREQ:general.kernel.est.here} with \eqref{FREQ:S.hard.soft.def}.  Then, for both hard and soft interactions, we have
 \begin{multline}\notag
 	[ \zetaLTone]_{ y\ge1} \lesssim\frac{1}{ p^0}\int_{\qlep} \frac{dq}{q^0}\frac{e^{-q^0}s^{1/2}}{g}\exp(2l)\mathbb{S}\\\times\int_1^\infty dy\  y(y^2+1)^{1/2}  \exp(-2 l\sqrt{y^2+1})I_0(2jy).
 \end{multline}
 Then note that 
 $$\int_1^\infty dy\  y(y^2+1)^{1/2}  \exp( -2 l\sqrt{y^2+1})I_0(2jy) \le \tilde{K}_2(2l,2j),$$ where $\tilde{K_2}$ is defined in \eqref{FREQ:tildeK2def}. 
 Then by \eqref{FREQ:tildeK2size.ref2}, on the region $|q|\le \frac{1}{2}|p|^{1/m}$, we have
\begin{equation}\label{FREQ:kernel.region.estimate}
    \tilde{K}_2(2l,2j)\lesssim \frac{1}{p^0}\exp(-\sqrt{4l^2-4j^2}).
\end{equation}
Hence if $y\ge 1$, we have
 \begin{equation}\label{FREQ:largey.zetaL}
 	[ \zetaLTone ]_{ y\ge1} \lesssim\frac{1}{ p^0}\int_\qlep \frac{dq}{q^0}\frac{e^{-q^0}s^{1/2}}{gp^0}\mathbb{S}\exp(2l-\sqrt{4l^2-4j^2}).
 \end{equation}
 Thus, combining \eqref{FREQ:smally.zetaL} and \eqref{FREQ:largey.zetaL}, we obtain \begin{equation}\notag
 	 \zetaLTone \lesssim \frac{1}{p^0}\int_\qlep \frac{dq}{q^0}\ \frac{e^{-q^0}s^{1/2}}{g}\mathbb{S}\exp(2l-\sqrt{4l^2-4j^2}).
 \end{equation}
 We will now split this estimate into the hard \eqref{hard} and soft \eqref{soft} interaction cases.

 In the hard interaction case \eqref{hard}, we use \eqref{FREQ:g.ge.lower}, \eqref{FREQ:s.le.pq} and \eqref{FREQ:l2j2size} to obtain
 \begin{equation}\notag
 	\zetaLTone 
 	\lesssim \int_\qlep dq\ |p-q|^{-1}(p^0q^0)^{1+\frac{\singA}{2}}\exp\left(\frac{p^0-q^0-|p-q|}{2}\right).
 \end{equation}
 Then by \eqref{FREQ:pminusq bound}, \eqref{FREQ:q0 bound}, and \eqref{FREQ:exponential.bound.1}, we have
 \begin{multline}\label{FREQ:zetaLfinal.hard}
 	 \zetaLTone 
 	\lesssim \int_\qlep dq\ |p|^{-1}(p^0)^{\left(1+\frac{\singA}{2}\right)\left(1+\frac{1}{m}\right)}\\\lesssim |p|^{-1+\frac{3}{m}}(p^0)^{\left(1+\frac{\singA}{2}\right)\left(1+\frac{1}{m}\right)}\lesssim (p^0)^{\frac{\singA}{2}+\frac{\singA+8}{2m}},
 \end{multline}
since $|p|\ge 1.$ Then for any given small $\varepsilon>0$, we choose $m$ sufficiently large such that  
 $\frac{\singA+8}{2m}\le \varepsilon.$ This yields Proposition \ref{FREQ:prop.zetaL.asymptoticnew} in the  hard interaction case on the region $|q|\le \frac{1}{2}|p|^{1/m}$.

 In the soft interaction case \eqref{soft}, we use \eqref{FREQ:g.ge.lower} and \eqref{FREQ:l2j2size} to obtain
 \begin{equation}\notag
 	 \zetaLTone 
 	\lesssim \int_\qlep dq\ |p-q|^{-1-\singB} (p^0q^0)^{\frac{\singB}{2}+1}\exp\left(\frac{p^0-q^0-|p-q|}{2}\right).
 \end{equation} 
 Now we use \eqref{FREQ:pminusq bound}, \eqref{FREQ:q0 bound}, and \eqref{FREQ:exponential.bound.1} to obtain for $|p|\ge 1$ that
 \begin{multline}\label{FREQ:zetaLfinal.soft}
 	 \zetaLTone
	\lesssim \int_\qlep dq\ |p|^{-1-\singB}(p^0)^{\left(1+\frac{\singB}{2}\right)\left(1+\frac{1}{m}\right)}\\\lesssim |p|^{-1-\singB+\frac{3}{m}}(p^0)^{\left(1+\frac{\singB}{2}\right)\left(1+\frac{1}{m}\right)}\lesssim (p^0)^{-\frac{\singB}{2}+\frac{\singB+8}{2m}}.
\end{multline}
For any given small $\varepsilon>0$, we choose $m$ sufficiently large such that  
$\frac{\singB+8}{2m}\le \varepsilon.$ 
This yields Proposition \ref{FREQ:prop.zetaL.asymptoticnew} in the soft interaction case.  This completes the proof.
 \end{proof}

\subsubsection{Low-order upper-bound for $\zetaTone$} Lastly, we introduce the following proposition on the lower-order upper bound estimate for $|\zetaTone|$ from \eqref{FREQ:def.zetatilde1}.
\begin{proposition}\label{FREQ:prop.tildezeta.upper.qgep}
	Suppose $\gamma \in (0,2)$ and $m>0.$ Then for both hard \eqref{hard} and soft \eqref{soft} interactions, for some $c>0$, we have the uniform upper bound for \eqref{FREQ:def.zetatilde1}:
	$$\left|\zetaTone(p)\right| \lesssim e^{-c(p^0)^{1/m}},$$ 
\end{proposition}

\begin{proof}
Note that on $\qgep$, then we will prove that each decomposed piece $\zeta_0$ from \eqref{FREQ:zeta0B} and $\zeta_L$ from \eqref{FREQ:zetaLBy} of $\tilde{\zeta}$ is lower order as above.  

In \secref{FREQ:sec:fullsharpupper zeta0}, in both \eqref{FREQ:zeta.use.hard.too} and \eqref{FREQ:k.ge.four.est.zeta}, if we restrict the domain to the case $|q|\ge \frac{1}{2}|p|^{1/m}$, then we have the bound \eqref{FREQ:exponential.more}
for some uniform $c>0$.  Therefore, $\zeta_1$ in this subregion is lower order in $p^0$, as it has additional exponential decay $e^{-c(p^0)^{1/m}}$. Similarly, for $\zeta_2$, in \eqref{FREQ:zeta2firstcase.hard}, \eqref{FREQ:zeta2firstcase.soft}, \eqref{FREQ:zeta2secondcase.hard}, and \eqref{FREQ:zeta2secondcase.soft}, can again use \eqref{FREQ:exponential.more}
on the region $|q|\ge \frac{1}{2}|p|^{1/m}$. Thus, $\zeta_2$ is also lower order in this region, and hence $\zeta_0$ is lower order when $\qgep$.

On the other hand for $\zeta_L$ from \eqref{FREQ:zetaLBy}, we have \eqref{FREQ:zetaLqgep} in \secref{FREQ:sec:low order zeta1} which is exactly the desired estimate.  Thus, $\zeta_L$ is also lower order in $p^0$, and hence $\tilde{\zeta}$ is also lower order in this sub-region.  This completes the proof.\end{proof}

This concludes our discussion on the sharp asymptotics of the \textit{frequency multipliers}. In this next section, we provide upper-bound estimates on the nonlinear linearized Boltzmann operator $\Gamma.$

\section{Main upper bound estimates}
\label{main upper}
In this section we prove the main upper bound estimates on the linearized \eqref{L.def} and non-linear collision operator \eqref{Gamma1}. In particular in \secref{main estimates} we dyadically decompose the singularity and perform size estimates on the decomposed pieces of \eqref{Gamma1}.  In \secref{Cancellation}, we perform the upper bound estimates that incorporate the cancellation when we are nearby the singularity.  Then in \secref{sec.newcancel}, we perform upper bound estimates that incorporate cancellation on the dual expression from \eqref{T+++.dual}.   In \secref{subsec.additional} we give some additional estimates on the decomposed pieces that will be useful in proving in particular Lemma \ref{Lemma2}.  In \secref{LP decomp}, we explain the main Littlewood-Paley inequalities that we will use to prove our main estimate.  Then in \secref{subsec.upperbd}, we use triple sum estimates together with all the previous estimates in the section in order to prove the main estimates from \secref{subsec.mainest}.

\subsection{Estimates on the single decomposed pieces}
\label{main estimates}
In this section, we mainly discuss about the estimates on the decomposed pieces of the trilinear product $\langle \Gamma(f,h), \eta\rangle$.
For the usual 8-fold representation, we recall  \eqref{Gamma1} and obtain that
 \begin{align}\notag
 \langle w^{2l}\Gamma(f,h), \eta\rangle &=\int_\rth dp \int_\rth dq \int_{\mathbb{S}^2}d\omega \ v_{\text{\o}} \sigma(g,\theta)w^{2l}(p)\eta(p)\sqrt{J(q)}\\
 &\hspace{10mm}\times\left(f(q')h(p')-f(q)h(p)\right)
 \label{eqn.trilinear.form}\\
 &= T^l_+-T^l_-,
 \notag
 \end{align}
 where the gain term $T^l_+$ and the loss term $T^l_-$ are defined as
 \begin{equation}\notag
 \begin{split}
 T^l_+(f,h,\eta)&\eqdef \int_\rth dp \int_\rth dq \int_{\mathbb{S}^2}d\omega ~v_{\text{\o}} \sigma(g,\theta)w^{2l}(p)\eta(p)\sqrt{J(q)}f(q')h(p'),
 \\
 T^l_-(f,h,\eta)&\eqdef \int_\rth dp \int_\rth dq \int_{\mathbb{S}^2}d\omega ~v_{\text{\o}} \sigma(g,\theta)w^{2l}(p)\eta(p)\sqrt{J(q)}f(q)h(p).
\end{split}
 \end{equation}
 And when $l=0$ we denote $T^0_\pm = T_\pm$.

 In the following, we will use the dyadically decomposed pieces $T^l_+$ and $T^l_-$ around the angular singularity. We let $\{\chi_k\}^\infty_{k=-\infty}$ be a partition of unity on $(0,\infty)$ such that $|\chi_k|\leq 1$ and supp$(\chi_k) \subset [2^{-k-1},2^{-k}]$. Then, using \eqref{define.kernel} we define 
\begin{equation}
     \label{kernel.k.diadic.define}
 \sigma_k(g,\theta)\eqdef \sigma(g,\theta)\chi_k(\bar{g}),
\end{equation}
 where we recall  $\bar{g}\eqdef g(p^\mu,p'^\mu)$ defined in \eqref{gbar}. The reason that we dyadically decompose around $\bar{g}$ is that we have $\theta \approx \frac{\bar{g}}{g}$ for small $\theta$ using \eqref{bargoverg}.   We refer to Remark \ref{FREQ:angle.remark} and \eqref{FREQ:cosine.angle.formula} for further explanations of this cancellation.
 
 Then we write the decomposed pieces $T^{k,l}_+$ and $T^{k,l}_-$ as
 \begin{equation}
 \label{T+++}
 \begin{split}
 T^{k,l}_+(f,h,\eta)&\eqdef \int_\rth dp \int_\rth dq \int_{\mathbb{S}^2}d\omega \ v_{\text{\o}} \sigma_k(g,\theta)w^{2l}(p)\eta(p)\sqrt{J(q)}f(q')h(p')\\
 T^{k,l}_-(f,h,\eta)&\eqdef \int_\rth dp \int_\rth dq \int_{\mathbb{S}^2}d\omega \ v_{\text{\o}} \sigma_k(g,\theta)w^{2l}(p)\eta(p)\sqrt{J(q)}f(q)h(p).\\
 \end{split}
 \end{equation}
 Thus, for $f,h,\eta\in S(\rth)$, where $S(\rth)$ denotes the standard Schwartz space on $\rth$:
 $$
 \langle  w^{2l}\Gamma(f,h),\eta\rangle=\sum^\infty_{k=-\infty}\{T^{k,l}_+(f,h,\eta)-T^{k,l}_-(f,h,\eta)\}.
 $$ 
 We will also use the definitions 
 \begin{equation} \label{kernel.diadic.define}
 \tilde{\sigma}_k\eqdef \frac{s\sigma(g,\theta)}{\tilde{g}}\chi_k(\bar{g}),\hspace*{5mm} \bar{g}\eqdef g(p^\mu,p'^\mu), \hspace*{5mm}\tilde{g}\eqdef g(p'^\mu,q^\mu),
\end{equation}
 where we further recall \eqref{g}, \eqref{gbar} and \eqref{gtilde}.

Now, we start making some size estimates for the decomposed pieces $T^{k,l}_-$ and $T^{k,l}_+$ for Schwartz functions. Then the estimates can be justified in general by approximation.

\begin{proposition}\label{T-proposition}
	For any integer $k\in \mathbb{Z}$ and for any $l\geq 0$ and $m\geq 0$, we have the uniform estimate for both hard- and soft-interactions \eqref{hard} and \eqref{soft}:
	\begin{equation}
    \label{T-}
	\begin{split}
	|T^{k,l}_-(f,h,\eta)|\lesssim 2^{k\gamma}|f|_{L^2_{-m}}| w^l h|_{L^2_{\frac{\singS+\gamma}{2}}}|w^l  \eta|_{L^2_{\frac{\singS+\gamma}{2}}}.
	\end{split}
	\end{equation}
\end{proposition}
\begin{proof}
	The term $T^{k,l}_-$ is given as in \eqref{T+++}.
Then the condition $\bar{g}\approx 2^{-k}$ is equivalent to saying that the angle 
	$\theta$ is comparable to $2^{-k}g^{-1}$ by \eqref{bargoverg}. 
	Given the size estimates for $\sigma(g,\omega)\eqdef \Phi(g)\sigma_0(\cos\theta)$ with \eqref{angassumption} and the support of $\chi_k$, we obtain
	\begin{equation}
	\label{Bk}
	\begin{split}
	\int_{\mathbb{S}^2}d\omega\ \sigma_k(g,\omega)&\lesssim  \Phi(g)\int_{\mathbb{S}^2}d\omega\  \sigma_0(\cos \theta) \chi_k(\bar{g}) \\
	&\lesssim  \Phi(g)\int_{2^{-k-1}g^{-1}}^{2^{-k}g^{-1}}d\theta\ \sigma_0(\cos \theta) \sin\theta \\
	&\lesssim  \Phi(g)\int_{2^{-k-1}g^{-1}}^{2^{-k}g^{-1}}d\theta\ \frac{1}{\theta^{1+\gamma}} \\
	&\lesssim  \Phi(g) 2^{k\gamma} g^\gamma.
	\end{split}
	\end{equation}
	Thus, under kernel assumptions \eqref{hard} and \eqref{soft}, we have
$$
	|T^{k,l}_-(f,h,\eta)| \lesssim 2^{k\gamma}\int_{\rth}dp\int_{\rth}dq\ g^{\singS+\gamma}v_{\text{\o}} |f(q)||h(p)|\sqrt{J(q)}|\eta(p)|w^{2l}(p)\eqdef I.$$ 
	
	In the hard-interaction case \eqref{hard} with $\singS+\gamma\geq 0$, we use
	\eqref{FREQ:g.le.sqrtpq} and \eqref{moller.upper.est}.  Then, we obtain
	 $$
	I\lesssim 2^{k\gamma}\int_{\rth}dp\int_{\rth}dq\ ({p^0}{q^0})^{\frac{\singS+\gamma}{2}} |f(q)||h(p)|\sqrt{J(q)}|\eta(p)|w^{2l}(p).
$$
	By the Cauchy-Schwarz inequality, 
	\begin{equation}
	\label{T-esimate1}
	\begin{split}
	I\lesssim &\ 2^{k\gamma}\left(\int_{\rth}dp\int_{\rth}dq\ |f(q)|^2|w^lh(p)|^2\sqrt{J(q)}{(p^0)}^{\frac{\singS+\gamma}{2}}\right)^\frac{1}{2}\\
	& \times\left(\int_{\rth}dp\  |w^l\eta(p)|^2{(p^0)}^{\frac{\singS+\gamma}{2}}\int_{\rth}dq\sqrt{J(q)}{(q^0)}^{\singS+\gamma}\right)^\frac{1}{2}.
	\end{split}
	\end{equation}
	Since $\int_{\rth}dq\sqrt{J(q)}{(q^0)}^{\singS+\gamma}\approx 1$, we have
	\begin{equation}
	\label{T-estimate}
	\begin{split}
	I\lesssim&\ 2^{k\gamma}\left(\int_{\rth}dp\int_{\rth}dq\ |f(q)|^2|w^lh(p)|^2\sqrt{J(q)}
	{(p^0)}^{\frac{\singS+\gamma}{2}}\right)^\frac{1}{2}\\
    &\times\left(\int_{\rth}dp\  |w^l\eta(p)|^2{(p^0)}^{\frac{\singS+\gamma}{2}}\right)^\frac{1}{2}\\
	\lesssim&\ 2^{k\gamma}|f|_{L^2_{-m}}|  w^lh|_{L^2_{\frac{\singS+\gamma}{2}}}|  w^l\eta|_{L^2_{\frac{\singS+\gamma}{2}}} \hspace{5mm} \text{for any} \ m\geq 0.
	\end{split}
	\end{equation}

	On the other hand, in the soft-interaction case with \eqref{soft}, we have $-\frac{3}{2}<\singS+\gamma<0$.  
	Then we use \eqref{FREQ:g.ge.lower} to obtain 
	 $$
	I\lesssim 2^{k\gamma}\int_{\rth}dp\int_{\rth}dq\ |p-q|^{\singS+\gamma}({p^0}{q^0})^\frac{-\singS-\gamma}{2} |f(q)||h(p)|\sqrt{J(q)}|\eta(p)|w^{2l}(p).
	$$
	With the Cauchy-Schwarz inequality, we have
	\begin{multline*}
	I\lesssim  2^{k\gamma}\left(\int_{\rth}dp\int_{\rth}dq\ |f(q)|^2|w^lh(p)|^2\sqrt{J(q)}{(p^0)}^{\frac{1}{2}(\singS+\gamma)}{(q^0)}^{-\singS-\gamma}\right)^\frac{1}{2}\\
	  \times\left(\int_{\rth}dp\  |w^l\eta(p)|^2{(p^0)}^{-\frac{1}{2}(\singS+\gamma)}{(p^0)}^{-\singS-\gamma}\int_{\rth}dq\sqrt{J(q)}|p-q|^{2(\singS+\gamma)}\right)^\frac{1}{2}.
	\end{multline*}
	Now we use \eqref{jutter.integral.est} to obtain
	\begin{multline*}
	I\lesssim  2^{k\gamma}\left(\int_{\rth}dp\int_{\rth}dq\ |f(q)|^2|w^lh(p)|^2\sqrt{J(q)}{(p^0)}^{\frac{1}{2}(\singS+\gamma)}{(q^0)}^{-\singS-\gamma}\right)^\frac{1}{2}\\
	 \times\left(\int_{\rth}dp\  |w^l\eta(p)|^2{(p^0)}^{\frac{1}{2}(\singS+\gamma)}\right)^\frac{1}{2}\\
	\lesssim  2^{k\gamma}|f|_{L^2_{-m}}|  w^lh|_{L^2_{\frac{1}{2}(\singS+\gamma)}}|  w^l\eta|_{L^2_{\frac{1}{2}(\singS+\gamma)}} \hspace{5mm} \text{for any} \ m\geq 0.
	\end{multline*}
	This completes the proof.
\end{proof}

Before we do the size estimates for the $T^{k,l}_+$ term, we first prove a useful inequality as in the following proposition. 

\begin{proposition}Suppose $k$ is any integer. Then, we have
	\label{usefulineq}
	\begin{equation}\notag
\tilde{g}\sqrt{\tilde{s}}\int_\rth \frac{dp}{p^0}\int_\rth \frac{dq'}{q'^0}\chi_k(\bar{g})\bar{g}^{-2-\gamma}\delta^{(4)}(p^\mu+q^\mu-p'^\mu-q'^\mu)\lesssim 2^{k\gamma},
	\end{equation}where $\bar{g}=g(p'^\mu,p^\mu)$ from \eqref{gbar} and $\tilde{g}\eqdef g(p'^\mu,q^\mu)$ and $\tilde{s}\eqdef \tilde{g}^2+4$ from \eqref{gtilde}.
\end{proposition}

\begin{proof}
Define $$k_2(p',q)\eqdef \int_\rth \frac{dp}{p^0}\int_\rth \frac{dq'}{q'^0}\chi_k(\bar{g})\bar{g}^{-2-\gamma}\delta^{(4)}(p^\mu+q^\mu-p'^\mu-q'^\mu).$$
By Lemma \ref{7.5ofCMP}, 
we have
$$k_2(p',q)\eqdef\frac{1}{16} \int_{\mathbb{R}^4\times\mathbb{R}^4}d\Theta(p^\mu,q'^\mu)\chi_k(\bar{g})\bar{g}^{-2-\gamma}\delta^{(4)}(p^\mu+q^\mu-p'^\mu-q'^\mu),$$ where $$d\Theta(p^\mu,q'^\mu)\eqdef dp^\mu dq'^\mu u(p^0+q'^0)u(\utilde{s}-4)\delta(\utilde{s}-\utilde{g}^2-4)\delta\left((p^\mu+q'^\mu)(p_\mu-q'_\mu)\right),$$ with $\utilde{g}\eqdef g(p^\mu,q'^\mu)$ and $\utilde{s}\eqdef \utilde{g}^2+4.$ Notice that $\utilde{g}\eqdef g(p^\mu,q'^\mu) = g(p'^\mu,q^\mu) = \tilde{g}$ from \eqref{gtilde} using \eqref{conservation}. 
Here, $u(x)$ is defined in \eqref{def.u}. 
We make the change of variables from $(p^\mu,q'^\mu)$ to
$$\bar{p}^\mu=p^\mu+q'^\mu,\ \bar{q}^\mu=p^\mu-q'^\mu.$$ 
Then we have
$$k_2(p',q)=\frac{1}{256} \int_{\mathbb{R}^4\times\mathbb{R}^4}d\Theta(\bar{p}^\mu,\bar{q}^\mu)\chi_k(\bar{g})\bar{g}^{-2-\gamma}\delta^{(4)}(q^\mu-p'^\mu+\bar{q}^\mu),$$ where now $$d\Theta(\bar{p}^\mu,\bar{q}^\mu)\eqdef d\bar{p}^\mu d\bar{q}^\mu u(\bar{p}^0)u(-\bar{p}^\mu\bar{p}_\mu-4)\delta(-\bar{p}^\mu\bar{p}-\bar{q}^\mu\bar{q}_\mu-4)\delta\left(\bar{p}^\mu\bar{q}_\mu\right).$$ 
Here, $\bar{g}$ originally from \eqref{gbar} is now redefined as
$$\bar{g}^2= (p'^\mu-p^\mu)(p'_\mu-p_\mu)=\left(p'^\mu-\frac{\bar{p}^\mu+\bar{q}^\mu}{2}\right)\left(p'_\mu-\frac{\bar{p}_\mu+\bar{q}_\mu}{2}\right).$$
Then we carry out the delta function $\delta^{(4)}(q^\mu-p'^\mu+\bar{q}^\mu)$ and obtain
$$k_2(p',q)=\frac{1}{256}\int_{\mathbb{R}^4}d\Theta(\bar{p}^\mu)\chi_k(\bar{g})\bar{g}^{-2-\gamma},$$ where now 
$$d\Theta(\bar{p}^\mu)\eqdef d\bar{p}^\mu u(\bar{p}^0)u(-\bar{p}^\mu\bar{p}_\mu-4)\delta(-\bar{p}^\mu\bar{p}_\mu-\tilde{g}^2-4)\delta\left(\bar{p}^\mu(p'_\mu-q_\mu)\right),$$ where $\tilde{g}=g(p'^\mu,q^\mu).$ Above, $\bar{g}$ is now redefined as
$$\bar{g}^2= \left(p'^\mu-\frac{\bar{p}^\mu+\bar{q}^\mu}{2}\right)\left(p'_\mu-\frac{\bar{p}_\mu+\bar{q}_\mu}{2}\right)=\frac{1}{4}(p'^\mu+q^\mu-\bar{p}^\mu)(p'_\mu+q_\mu-\bar{p}_\mu).$$
Since $\tilde{s}=\tilde{g}^2+4$, we have
 \begin{multline}\label{tilde.calc.LT}
 u(\bar{p}^0)\delta(-\bar{p}^\mu \bar{p}_\mu-\tilde{g}^2-4)=u(\bar{p}^0)\delta(-\bar{p}^\mu \bar{p}_\mu-\tilde{s})\\
 =u(\bar{p}^0)\delta((\bar{p}^0)^2-|\bar{p}|^2-\tilde{s})
 =\frac{\delta(\bar{p}^0-\sqrt{|\bar{p}|^2+\tilde{s}})}{2\sqrt{|\bar{p}|^2+\tilde{s}}}.
 \end{multline}
Now note that $u(-\bar{p}^\mu\bar{p}_\mu-4)=1$ because $-\bar{p}^\mu\bar{p}_\mu-4=\tilde{g}^2\ge 0.$ By carrying out the delta function $\delta(\bar{p}^0 -\sqrt{|\bar{p}|^2 +\tilde{s}})$, we have
$$k_2(p',q)=\frac{1}{512}\int_\rth \frac{d\bar{p}}{\sqrt{|\bar{p}|^2 +\tilde{s}}}\delta(\bar{p}^\mu(p'_\mu-q_\mu))\bar{g}^{-2-\gamma}\chi_k(\bar{g}).$$
We will further change variables inside this integral to evaluate the delta function.
For the reduction we move to a new Lorentz frame as below.

We recall Definition \ref{rcop:LTdef} and then consider the Lorentz transform $\Lambda=\Lambda(p',q)$ of \eqref{eq.LT} where we exchange the role of $p$ in \eqref{eq.LT} with $p'$.  As in \eqref{center.momentum.condition}, with $p$ replaced by $p'$, this transformation maps into the \textit{center-of-momentum} frame as 
$$
A^\nu=\Lambda^\mu{}_\nu(p'_\mu+q_\mu)=(\sqrt{\tilde{s}},0,0,0), \quad B^\nu=-\Lambda^\mu{}_\nu(p'_\mu-q_\mu)=(0,0,0,\tilde{g}).
$$ 
After applying this change of variables, we have
$$k_2(p',q)=\frac{1}{512}\int_\rth \frac{d\bar{p}}{\sqrt{|\bar{p}|^2 +\tilde{s}}}\delta(\bar{p}^\mu B_\mu))\bar{g}^{-2-\gamma}\chi_k(\bar{g}),$$ where $\frac{d\bar{p}}{\sqrt{|\bar{p}|^2 +\tilde{s}}}$ is a Lorentz invariant measure. 
Here, $\bar{g}$ is now redefined as
$$\bar{g}^2=-\frac{\tilde{s}}{2}+\frac{\bar{p}^0\sqrt{\tilde{s}}}{2}=\frac{\sqrt{\tilde{s}}}{2}(\sqrt{|\bar{p}|^2 +\tilde{s}}-\sqrt{\tilde{s}}).$$
We write $\bar{p}$ in the polar coordinate system $(y,\phi,\psi)\in [0,\infty)\times [0,2\pi)\times [0,\pi)$ such that we have
$$k_2(p',q)=\frac{1}{512}\int_0^{2\pi}d\phi \int_0^\pi \sin\psi d\psi \int_0^\infty \frac{y^2dy}{\sqrt{y^2 +\tilde{s}}}\delta(y\tilde{g}\cos\psi)\bar{g}^{-2-\gamma}\chi_k(\bar{g}),$$ where $\bar{g}$ is now
$$\bar{g}^2=\frac{\sqrt{\tilde{s}}}{2}(\sqrt{y^2 +\tilde{s}}-\sqrt{\tilde{s}}).$$
We carry out the delta function at $\psi=\frac{\pi}{2}$ and obtain
$$k_2(p',q)=\frac{\pi}{256\tilde{g}} \int_0^\infty \frac{ydy}{\sqrt{y^2 +\tilde{s}}}\bar{g}^{-2-\gamma}\chi_k(\bar{g}).$$
Note that the support condition $\chi_k(\bar{g})$ implies that $\bar{g}\in [2^{-k-1},2^{-k}].$ Then this is equivalent to 
$$Y_1\eqdef \sqrt{\tilde{s}}+\frac{2^{-2k-1}}{\sqrt{\tilde{s}}}\leq \sqrt{y^2+\tilde{s}}\leq \sqrt{\tilde{s}}+\frac{2^{-2k+1}}{\sqrt{\tilde{s}}}\eqdef Y_2.$$
Then we consider the change of variables $y\mapsto y'=\sqrt{y^2+\tilde{s}}$ with $ydy=y'dy'$ and obtain
\begin{multline*}k_2(p',q)=\frac{\pi}{256\tilde{g}} \int_{Y_1}^{Y_2}\frac{y'dy'}{y'}\bar{g}^{-2-\gamma}\chi_k(\bar{g})
\approx 2^{k(\gamma+2)}\frac{\pi}{\tilde{g}} \int_{Y_1}^{Y_2} dy'\chi_k(\bar{g})\\\lesssim 2^{k(\gamma+2)}\frac{\pi}{\tilde{g}}(Y_2-Y_1)
\approx 2^{k(\gamma+2)}\frac{\pi}{\tilde{g}}\left(\frac{2^{-2k+1}}{\sqrt{\tilde{s}}}-\frac{2^{-2k-1}}{\sqrt{\tilde{s}}}\right)
\approx 2^{k\gamma}\frac{1}{\tilde{g}\sqrt{\tilde{s}}}.
\end{multline*}Therefore, we obtain
$$\tilde{g}\sqrt{\tilde{s}}k_2(p',q)\lesssim 2^{k\gamma},$$and this completes the proof for the proposition.
\end{proof}

We are now ready to estimate the operator $T^{k,l}_+$. This is more difficult and requires more refined techniques because it contains the post-collisional momenta.   By taking a pre-post change of variables $(p,q) \to (p',q')$ as in \eqref{prepost.change}, 
we obtain from (\ref{T+++}) that the term $T^{k,l}_+$ is equal to
	\begin{equation}
	\label{T++}
	\begin{split}
	T^{k,l}_+(f,h,\eta)=\int_{\rth}dp\int_{\rth}dq\int_{\mathbb{S}^2}d\omega\ \sigma_k(g,\omega)v_{\text{\o}} f(q)h(p)\sqrt{J(q')}\eta(p')w^{2l}(p'),
	\end{split}
	\end{equation}
	where $\sigma_k(g,\omega)$ is defined in \eqref{kernel.k.diadic.define} with \eqref{define.kernel}.  We note that no momentum weight is gained for the plus term $T^{k,l}_+$ in $| w^l f|_{L^2}$ in this proposition below.

\begin{proposition}\label{T+proposition}
	For any integer $k\in \mathbb{Z}$ and $l\geq 0$, we have the following uniform estimate for both hard and soft-interactions \eqref{hard} and \eqref{soft}:
	\begin{equation}
	\label{T+}
	\begin{split}
	|T^{k,l}_+(f,h,\eta)|&\lesssim 2^{k\gamma}| w^l f|_{L^2}|  w^lh|_{L^2_{\frac{\singS+\gamma}{2}}}|  w^l\eta|_{L^2_{\frac{\singS+\gamma}{2}}}.
	\end{split}
	\end{equation}
\end{proposition}
\begin{proof}
	Using \eqref{hard} and \eqref{soft}, we have
	\begin{multline}\label{T++I}
	 |T^{k,l}_+(f,h,\eta)|\\
	 \lesssim \int_{\rth}dp\int_{\rth}dq\int_{\mathbb{S}^2}d\omega\ g^{\singS}v_{\text{\o}}\sigma_0\chi_k(\bar{g}) |f(q)||h(p)|\sqrt{J(q')}|\eta(p')|w^{2l}(p')
	 \eqdef I.
	\end{multline}
	We will separately estimate the hard and soft interactions cases.
	
We start with the hard-interaction case \eqref{hard} with $\singS+\gamma\geq 0$.	We first note that 
\begin{equation}\notag w^{2l}(p')\lesssim w^l(p')w^l(p)w^l(q),\end{equation} as $l\geq 0$ and $p'^0\leq p^0+q^0.$ By the Cauchy-Schwarz inequality,
	\begin{multline}\label{CauchyT+hard}
 I \lesssim\left(\int_{\rth}dp\int_{\rth}dq\int_{\mathbb{S}^2}d\omega\ v_{\text{\o}}\frac{g^{\singS}\sigma_0\chi_k(\bar{g})}{\tilde{g}^{\singS+\gamma}}
	|w^lf(q)|^2|w^l\eta(p')|^2\sqrt{J(q')}({p^0})^{\frac{\singS+\gamma}{2}}\right)^{\frac{1}{2}}\\
	 \times \left(\int_{\rth}dp\int_{\rth}dq\int_{\mathbb{S}^2}d\omega\ v_{\text{\o}} g^{\singS}\sigma_0\chi_k(\bar{g})\tilde{g}^{\singS+\gamma}|w^lh(p)|^2\sqrt{J(q')}({p^0})^{\frac{-\singS-\gamma}{2}} \right)^{\frac{1}{2}}\\
	 \eqdef I_{1}\cdot I_{2},
	\end{multline}where $\tilde{g}=g(p'^\mu,q^\mu).$
    We estimate $I_2$ first. We can rewrite $I_2$ as follows: 
    \begin{multline*}I_2 = \left(\int_\rth dp\int_\rth dq \int_{\mathbb{S}^2}d\omega\  v_{\text{\o}} g^{\singS} \sigma_0\chi_k(\bar{g})\tilde{g}^{\singS+\gamma}|w^lh(p)|^2\sqrt{J(q')}({p^0})^{\frac{-\singS-\gamma}{2}}\right)^{\frac{1}{2}}\\
    \approx \bigg(\int_\rth \frac{dp}{p^0}\int_\rth \frac{dq}{q^0}\int_\rth \frac{dp'}{p'^0}\int_\rth \frac{dq'}{q'^0} s g^{\singS} \sigma_0\chi_k(\bar{g})\tilde{g}^{\singS+\gamma}|w^lh(p)|^2\sqrt{J(q')}({p^0})^{\frac{-\singS-\gamma}{2}}\\\times \delta^{(4)}(p'^\mu+q'^\mu-p^\mu-q^\mu)\bigg)^{\frac{1}{2}}.
    \end{multline*}
    Then, we take a pre-post change of variables $(p,q)\mapsto (p',q')$ and use $\tilde{g}=g(p'^\mu,q^\mu)=g(p^\mu,q'^\mu)$ to obtain
     \begin{multline}\label{T+I2hard}I_2 
    \approx \bigg(\int_\rth \frac{dp'}{p'^0}\int_\rth \frac{dq'}{q'^0}\int_\rth \frac{dp}{p^0}\int_\rth \frac{dq}{q^0} s g^{\singS} \sigma_0\chi_k(\bar{g})\tilde{g}^{\singS+\gamma}|w^lh(p')|^2\sqrt{J(q)}\\\times({p'^0})^{\frac{-\singS-\gamma}{2}} \delta^{(4)}(p^\mu+q^\mu-p'^\mu-q'^\mu)\bigg)^{\frac{1}{2}}\\
    \approx\bigg(\int_\rth \frac{dp'}{p'^0}\int_\rth \frac{dq}{q^0}      \tilde{s}\tilde{g}^{2\singS+2\gamma+2}|w^lh(p')|^2\sqrt{J(q)}({p'^0})^{\frac{-\singS-\gamma}{2}}\\\times  \int_\rth \frac{dp}{p^0}\int_\rth \frac{dq'}{q'^0}\chi_k(\bar{g})\bar{g}^{-2-\gamma}\delta^{(4)}(p^\mu+q^\mu-p'^\mu-q'^\mu)\bigg)^{\frac{1}{2}},
    \end{multline}
    where we used $g\approx \tilde{g}$, $s\approx \tilde{s}$ with $\tilde{s}\eqdef \tilde{g}^2+4$, and $\sigma_0(\theta)\approx \theta^{-2-\gamma}\approx (\frac{\bar{g}}{g})^{-2-\gamma}$ by \eqref{angassumption}, \eqref{gtildeg.equiv},  and \eqref{bargoverg}. Then we use Proposition \ref{usefulineq} and  obtain that
$$
I_2\lesssim 2^{\frac{k\gamma}{2}}\left(\int_{\rth}\frac{dp'}{p'^0}\int_{\rth}\frac{dq}{q^0}\ \sqrt{\tilde{s}} \tilde{g}^{2\singS+2\gamma+1}|w^lh(p')|^2\sqrt{J(q)}({p'^0})^{\frac{-\singS-\gamma}{2}}\right)^{\frac{1}{2}}.
$$
    We further use $\tilde{g}\lesssim \sqrt{p'^0q^0}$ and $\tilde{s}\lesssim p'^0q^0$ from \eqref{FREQ:g.le.sqrtpq} and \eqref{FREQ:s.le.pq} to conclude that
    \begin{multline*}
I_2\lesssim 2^{\frac{k\gamma}{2}} \left(\int_{\rth}dp'\ ({p'^0})^{\frac{\singS+\gamma}{2}}|w^lh(p')|^2\int_{\rth}dq\ (q^0)^{\singS+\gamma}\sqrt{J(q)} \right)^{\frac{1}{2}}
\approx 2^{\frac{k\gamma}{2}}|w^lh|_{L^2_{\frac{\singS+\gamma}{2}}}.
\end{multline*}This completes the estimate for $I_2$.

   Now we estimate $I_1$. We first observe that 
   $$(p^0)^{\frac{\singS+\gamma}{2}} \sqrt{J(q')}\lesssim    (p'^0q'^0)^{\frac{\singS+\gamma}{2}} \sqrt{J(q')}\lesssim (p'^0)^{\frac{\singS+\gamma}{2}} \sqrt{J^\alpha(q')}\lesssim (p'^0)^{\frac{\singS+\gamma}{2}} ,$$ for some $\alpha \in (0,1)$ by \eqref{pp'q'} and $\singS+\gamma\ge 0$.  
   Thus, we obtain that
   $$
	I_1\lesssim \left(\int_{\rth}dp\int_{\rth}dq\int_{\mathbb{S}^2}d\omega\ v_{\text{\o}}\frac{g^{\singS}\sigma_0\chi_k(\bar{g})}{\tilde{g}^{\singS+\gamma}} 
	|w^lf(q)|^2|w^l\eta(p')|^2({p'^0})^{\frac{\singS+\gamma}{2}}\right)^{\frac{1}{2}}.
	$$We raise the 8-fold integration into 12-fold integration like  \eqref{T+I2hard} without the pre-post change of variable this time. Then we obtain
	\begin{multline*}
	I_1\lesssim\bigg(\int_\rth \frac{dp'}{p'^0}\int_\rth \frac{dq}{q^0}  \tilde{s}\tilde{g}^2
	|w^lf(q)|^2|w^l\eta(p')|^2({p'^0})^{\frac{\singS+\gamma}{2}}  \\\times  \int_\rth \frac{dp}{p^0}\int_\rth \frac{dq'}{q'^0}\chi_k(\bar{g})\bar{g}^{-2-\gamma}\delta^{(4)}(p^\mu+q^\mu-p'^\mu-q'^\mu)\bigg)^{\frac{1}{2}},\end{multline*}   where we used $g\approx \tilde{g}$, $s\approx \tilde{s}$ with $\tilde{s}\eqdef \tilde{g}^2+4$, and $\sigma_0(\theta)\approx \theta^{-2-\gamma}\approx (\frac{\bar{g}}{g})^{-2-\gamma}$ by \eqref{angassumption} and \eqref{bargoverg}, as in $I_2$ case.
	By Proposition \ref{usefulineq}, we have$$
	I_1\lesssim2^{\frac{k\gamma}{2}}\bigg(\int_\rth \frac{dp'}{p'^0}\int_\rth \frac{dq}{q^0}  \sqrt{\tilde{s}}\tilde{g}
	|w^lf(q)|^2|w^l\eta(p')|^2({p'^0})^{\frac{\singS+\gamma}{2}}  \bigg)^{\frac{1}{2}}.$$
	Then, using \eqref{FREQ:s.le.pq} and \eqref{FREQ:g.le.sqrtpq} we obtain
	\begin{equation}
	\label{T12}
	\begin{split}
	I_1&\lesssim 2^\frac{k\gamma}{2}\left(\int_{\rth}\frac{dp'\ }{{p'^0}}\int_{\rth}\frac{dq}{{q^0}}|w^lf(q)|^2 |w^l\eta(p')|^2({p'^0})^{\frac{\singS+\gamma}{2}}p'^0q^0\right)^\frac{1}{2}\\
	&\lesssim 2^\frac{k\gamma}{2}\left(\int_{\rth}dp'\ ({p'^0})^{\frac{\singS+\gamma}{2}}|w^l\eta(p')|^2\int_{\rth}dq|w^lf(q)|^2  \right)^\frac{1}{2}\\
	&\lesssim 2^\frac{k\gamma}{2}|w^lf|_{L^2}|  w^l\eta|_{L^2_{\frac{\singS+\gamma}{2}}}
	\end{split}
	\end{equation}
    by the Cauchy-Schwarz inequality.
	Thus, from \eqref{CauchyT+hard} we have
	 $$
	I\lesssim 2^{k\gamma}|  w^lf|_{L^2}|  w^lh|_{L^2_{\frac{\singS+\gamma}{2}}}|  w^l\eta|_{L^2_{\frac{\singS+\gamma}{2}}}.
	$$ This completes the proof for the hard-interaction case \eqref{hard}.
	
	On the other hand, in the soft-interaction case \eqref{soft} when $0>\singS+\gamma>-\frac{3}{2}$, we have the following by the Cauchy-Schwarz inequality for $I$ defined in \eqref{T++I}:
	\begin{multline}\label{CauchyT+soft}
I \lesssim \left(\int_{\rth}dp\int_{\rth}dq\int_{\mathbb{S}^2}d\omega\ v_{\text{\o}}\frac{g^{\singS}\sigma_0\chi_k(\bar{g})}{g^{\singS+\gamma}}
	|f(q)|^2|h(p)|^2w^{2l}(p')\sqrt{J(q')}({p'^0})^{\frac{\singS+\gamma}{2}}\right)^{\frac{1}{2}}\\
	 \times \left(\int_{\rth}dp\int_{\rth}dq\int_{\mathbb{S}^2}d\omega\ v_{\text{\o}} g^{\singS}\sigma_0\chi_k(\bar{g})g^{\singS+\gamma}|w^l\eta(p')|^2\sqrt{J(q')}({p'^0})^{\frac{-\singS-\gamma}{2}}\right)^{\frac{1}{2}}\\
	 =I_{1}\cdot I_{2}.
	\end{multline}
	For $I_1$, we split the region of $p'$ into two: ${p'^0} \leq \frac{1}{2}({p^0}+{q^0})$ and ${p'^0}\geq \frac{1}{2}({p^0}+{q^0})$.
	If ${p'^0} \leq \frac{1}{2}({p^0}+{q^0})$, ${p^0}+{q^0}-{q'^0}\leq \frac{1}{2}({p^0}+{q^0})$ by the conservation laws \eqref{conservation}. Thus, $-{q'^0}\leq -\frac{1}{2}({p^0}+{q^0})$ and $J(q')\leq \sqrt{J(p)}\sqrt{J(q)}=\left(J(p)J(q)J(p')J(q')\right)^{1/4}$ by using the conservation laws \eqref{conservation}.  
	Since the exponential decay is faster than any polynomial decay, we have
	 $$
	w^{2l}(p')({p'^0})^{\frac{1}{2}(\singS+\gamma)}\sqrt{J(q')}\lesssim ({p^0})^{-m}({q^0})^{-m}
$$ for any fixed $m>0$.
	On the other region, we have ${p'^0} \geq \frac{1}{2}({p^0}+{q^0})$ and hence ${p'^0} \approx ({p^0}+{q^0})$ because ${p'^0} \leq ({p^0}+{q^0})$.	Then $w^{2l}(p')\lesssim w^{2l}(p)w^{2l}(q)$ as $l\geq 0$.
	Also, we have $({p'^0})^{\frac{1}{2}(\singS+\gamma)}\lesssim({p^0})^{\frac{1}{2}(\singS+\gamma)}$ because $\singS+\gamma<0$. Thus, we obtain
	 $$
	w^{2l}(p')({p'^0})^{\frac{1}{2}(\singS+\gamma)}\sqrt{J(q')}\lesssim  w^{2l}(p)w^{2l}(q) ({p^0})^{\frac{1}{2}(\singS+\gamma)}.
$$
After computing $d\omega$ integral as in (\ref{Bk}) in both cases above, 
we obtain 
	\begin{multline*}
	I_1 \lesssim\left(\int_{\rth}dp\int_{\rth}dq\frac{g^{\singS}2^{k\gamma}g^\gamma}{g^{\singS+\gamma}} |f(q)|^2|h(p)|^2 w^{2l}(p)w^{2l}(q) ({p^0})^{\frac{1}{2}(\singS+\gamma)}\right)^{\frac{1}{2}}\\
	 \approx \left(\int_{\rth}dp\int_{\rth}dq\ 2^{k\gamma} |w^lf(q)|^2|w^lh(p)|^2  ({p^0})^{\frac{1}{2}(\singS+\gamma)}\right)^{\frac{1}{2}}
	 \lesssim2^{\frac{k\gamma}{2}}| w^l f|_{L^2}|  w^lh|_{L^2_{\frac{\singS+\gamma}{2}}},
	\end{multline*}
	by the Cauchy-Schwarz inequality.

Next we estimate $I_2$ from \eqref{CauchyT+soft}. Note that $v_{\text{\o}}=\frac{g\sqrt{s}}{{p^0}{q^0}}$ by \eqref{Moller.def}. By a pre-post change of variables $(p',q')\mapsto (p,q)$ as in \eqref{prepost.change}, we have 
$$
I_2 =\left(\int_{\rth}dp\int_{\rth}dq\int_{\mathbb{S}^2}d\omega\ v_{\text{\o}} g^{\singS}\sigma_0\chi_k(\bar{g})g^{\singS+\gamma}|w^l\eta(p)|^2\sqrt{J(q)}({p^0})^{\frac{-\singS-\gamma}{2}}\right)^{\frac{1}{2}}.
$$ 
Then by \eqref{Bk}, we have
\begin{equation}\label{I2forcompact}
	I_2\lesssim \left(\int_{\rth}dp\ \int_{\rth}dq\ v_{\text{\o}} 2^{k\gamma}g^{2(\singS+\gamma)}|w^l\eta(p)|^2\sqrt{J(q)}({p^0})^{\frac{-\singS-\gamma}{2}}\right)^{\frac{1}{2}}.
\end{equation}
Using \eqref{FREQ:g.ge.lower} for $0>\singS+\gamma>-\frac{3}{2}$ and \eqref{moller.upper.est}, we have
	 $$
	I_2\lesssim \left(\int_{\rth}dp\ \int_{\rth}dq\ 2^{k\gamma}\frac{|p-q|^{2(\singS+\gamma)}}{({p^0}{q^0})^{\singS+\gamma}}|w^l\eta(p)|^2\sqrt{J(q)}({p^0})^{\frac{-\singS-\gamma}{2}}\right)^{\frac{1}{2}}.
	$$
Also $({q^0})^{-\singS-\gamma}\sqrt{J(q)}\lesssim \sqrt{J^\alpha(q)}$ for some $\alpha>0$. Then using \eqref{jutter.integral.est} we have
	\begin{multline}\label{I22}
	I_2 \lesssim \left(\int_{\rth}dp\ 2^{k\gamma}|w^l\eta(p)|^2({p^0})^{\frac{3}{2}(-\singS-\gamma)}\int_{\rth}dq\ \frac{\sqrt{J^\alpha(q)}}{|p-q|^{2(-\singS-\gamma)}}\right)^{\frac{1}{2}}\\
	 \lesssim \left(\int_{\rth}dp\ 2^{k\gamma}|w^l\eta(p)|^2({p^0})^{\frac{3}{2}(-\singS-\gamma)}({p^0})^{2(\singS+\gamma)}\right)^{\frac{1}{2}}
	 =2^\frac{k\gamma}{2}|  w^l\eta|_{L^2_{\frac{\singS+\gamma}{2}}}.
	\end{multline}
	Together, we obtain that 
	 $$
	I\lesssim 2^{k\gamma}|w^l  f|_{L^2}| w^l h|_{L^2_{\frac{\singS+\gamma}{2}}}|  w^l\eta|_{L^2_{\frac{\singS+\gamma}{2}}}.
	$$
	This completes the proof.
\end{proof}

\begin{remark}\label{rem:hardI2est}
For future use in the proof of Proposition \ref{prop:GradRap} we point out that estimate \eqref{I22} also holds for the hard potentials \eqref{hard} when $\rho+\gamma \ge 0$.  Indeed in this case we use \eqref{moller.upper.est} and  \eqref{FREQ:g.le.sqrtpq} in \eqref{I2forcompact} to obtain
	 $$
	I_2\lesssim \left(\int_{\rth}dp\ \int_{\rth}dq\ 2^{k\gamma}({p^0}{q^0})^{\singS+\gamma}|w^l\eta(p)|^2\sqrt{J(q)}({p^0})^{\frac{-\singS-\gamma}{2}}\right)^{\frac{1}{2}}.
	$$
Again $({q^0})^{\singS+\gamma}\sqrt{J(q)}\lesssim \sqrt{J^\alpha(q)}$ for some $\alpha>0$, and using \eqref{jutter.integral.est} we have
	\begin{equation}\label{I22.hard}
	I_2 \lesssim \left(\int_{\rth}dp\ 2^{k\gamma}|w^l\eta(p)|^2({p^0})^{\frac{1}{2}(\singS+\gamma)}\right)^{\frac{1}{2}}
	 \lesssim 2^\frac{k\gamma}{2}|  w^l\eta|_{L^2_{\frac{\singS+\gamma}{2}}}.
	\end{equation}
And this is the desired estimate.	
\end{remark}

This concludes our discussion of the upper-bound estimates for the collision operators away from the angular singularity. In
the next section we will make upper-bound estimates for the same operators nearby the angular singularity. The key point for that is to utilize the cancellation properties between the gain and the loss terms.

\subsection{Cancellation estimates}
\label{Cancellation}
In this section we will establish uniform upper bound estimates for the difference $T^{k,l}_+ - T^{k,l}_-$ in the case when $k>0$.   We will need our upper bound estimates to have a dependency on a negative power of $2^k$ so we have a good estimate after summation in $k>0$. 
Before we move onto the actual estimates, we will now introduce the following useful inequality:

\begin{lemma}\label{lemmaqq}
Suppose $\bar{g}<1$. Then,  we have $q'^0\approx q^0$ and $p'^0\approx p^0$. In particular if $\bar{g}\approx 2^{-k}$ for some $k>0$, we then have $p^0\leq \sqrt{5}p'^0$ and $p'^0\leq \sqrt{5}p^0$.
\end{lemma}

This equivalence is one of the advantages that we take on this region $\bar{g}< 1$ nearby the angular singularity. We prove the case $p^0\approx p'^0$ and the case for $q^0\approx q'^0$ is the same because $\bar{g}=g(p'^\mu,p^\mu)=g(q'^\mu,q^\mu)<1$ in the same region. 
\begin{proof}
	Recall the inequality $\frac{|p-p'|}{\sqrt{p^0p'^0}}\leq \bar{g}$ from \eqref{FREQ:g.ge.lower}. If $\bar{g}\leq 2^{-k}$, we have 
\begin{multline}\label{qq'}
	(p'^0)^2 =1+|p'|^2 \leq 1+2(|p'-p|^2+|p|^2) 
	\\
	\leq 1+2(2^{-2k}p'^0p^0+|p|^2)\le 2((p^0)^2+2^{-2k}p'^0p^0).
\end{multline}
	By Young's inequality, we have $$2^{-2k}p^0p'^0 \leq 2^{-4k}(p^0)^2+\frac{1}{4}(p'^0)^2.$$ We put this into \eqref{qq'} and obtain that \begin{equation}\notag
	(p'^0)^2 \leq 2((p^0)^2+2^{-4k}(p^0)^2+\frac{1}{4}(p'^0)^2).
	\end{equation} Thus, 
\begin{equation}\label{qq'2}
	(p'^0)^2 \leq 4(1+2^{-4k})(p^0)^2\le \frac{17}{4}(p^0)^2,
	\end{equation} 
		as $k$ is a positive integer. 
	Therefore, $(p'^0)^2\leq 5(p^0)^2$ holds and hence $p'^0\leq \sqrt{5}p^0$. The same proof works if we interchange the roles of $p$ and $p'$. Therefore, $p^0\approx p'^0$.
\end{proof}

For the upper-bound estimates of the difference $T^{k,l}_+-T^{k,l}_-$, we also define paths from $p$ to $p'$ and from $q$ to $q'$. Fix any two $p,p'\in\rth$ and consider $\kappa:[0,1] \rightarrow \rth$  given by
\begin{equation}
    \kappa(\vartheta)\eqdef \vartheta p+(1-\vartheta)p'\text{ for }\vartheta\in [0,1].
\label{path.kappa}
\end{equation}
Similarly, we define the following for the path from $q'$ to $q$;
$$
\kappa_q(\vartheta)\eqdef \vartheta q+(1-\vartheta)q'\text{ for }\vartheta\in [0,1].
$$
Then we can easily notice that $\kappa(\vartheta)+\kappa_q(\vartheta)=p'+q'=p+q$.

We also define the length of the gradient as:
\begin{equation}
\label{nabla}
|\nabla|^iH(p)\eqdef \max_{0\leq j\leq i}\sup_{|\chi|\leq 1}\Big|\big(\chi\cdot\nabla\big)^jH(p)\Big|, \hspace{5mm} i=0,1,2,
\end{equation}
where $\chi\in\rth$ and $|\chi|$ is the usual Euclidean length. Note that we have $|\nabla|^0H=|H|$ and we will write $|\nabla|^1 H=|\nabla| H$ without ambiguity.

Now we start estimating the term $|T^{k,l}_+ - T^{k,l}_-|$ under the condition $\bar{g}< 1.$
We recall from  (\ref{T+++}) and  \eqref{T++} that $|(T^{k,l}_+-T^{k,l}_-)(f,h,\eta)|$ is defined as
\begin{multline}\label{eq.T+-T-}
 |(T^{k,l}_+-T^{k,l}_-)(f,h,\eta)|
= \bigg|\int_{\rth}dp\int_{\rth}dq\int_{\mathbb{S}^2}d\omega\ \sigma_k(g,\omega)v_{\text{\o}} f(q)h(p)\\\times (w^{2l}(p')\sqrt{J(q')}\eta(p')-w^{2l}(p)\sqrt{J(q)}\eta(p))\bigg|,
\end{multline}
The key part is to estimate $|w^{2l}(p')\sqrt{J(q')}\eta(p')-w^{2l}(p)\sqrt{J(q)}\eta(p)|$.

We have the following proposition for the cancellation estimate:
\begin{proposition}\label{Cancellationproposition}For any $k>0$ and for $0<\gamma<1$ and $m\geq 0$, we have the uniform estimate: 
\begin{equation}\label{can+-}
|(T^{k,l}_+-T^{k,l}_-)(f,h,\eta)|
\lesssim   \ 2^{(\gamma-1)k}|f|_{L^2_{-m}}|  w^lh|_{L^2_{\frac{\singS+\gamma}{2}}}  \left|w^l|\nabla| \eta \right|_{L^2_{\frac{\singS+\gamma}{2}}}.
\end{equation}
\end{proposition}
\begin{proof} We want our kernel to have a good dependency on $2^{-k}$ so
	we end up with the negative power on 2 as $2^{(\gamma-1)k}.$  
	We first split the term $w^{2l}(p')\sqrt{J(q')}\eta(p')-w^{2l}(p)\sqrt{J(q)}\eta(p)$ into two parts as
	\begin{multline}
    \label{canceldecompp}
	w^{2l}(p')\sqrt{J(q')}\eta(p')-w^{2l}(p)\sqrt{J(q)}\eta(p)\\
	=w^{2l}(p')\sqrt{J(q')}\Big(\eta(p')-\eta(p)\Big)
	+w^{2l}(p')\Big(\sqrt{J(q')}-\sqrt{J(q)}\Big)\eta(p)\\
	+(w^{2l}(p')-w^{2l}(p))\sqrt{J(q)}\eta(p)
	\eqdef \text{I}+\text{II}+\text{III}.
	\end{multline}

		For part I, we define the whole integral with part I $(T^{k,l}_{+,\text{I}}-T^{k,l}_{-,\text{I}})$ as
		\begin{multline*}
		T^{k,l}_{+,\text{I}}-T^{k,l}_{-,\text{I}}
		\eqdef \int_\rth dp \int_\rth dq \int_{\mathbb{S}^2}d\omega \  v_{\text{\o}} \sigma_k(g,\omega) f(q)h(p)w^{2l}(p')\\\times\sqrt{J(q')}\Big(\eta(p')-\eta(p)\Big).
		\end{multline*} 
		For the cancellation terms, we obtain
		$$
		\eta(p')-\eta(p)=(p'-p)\cdot\int_{0}^{1}d\vartheta\ (\nabla\eta)(\kappa(\vartheta)).
		$$We first note that under $\bar{g}<1$ we have $w^{2l}(p')\approx w^{2l}(p)$ by Lemma \ref{lemmaqq}. 
		Also note that 
	$|p'-p|=|q'-q|\leq g(q^\mu,q'^\mu)\sqrt{q^0q'^0}=\bar{g}\sqrt{q^0q'^0}\approx 2^{-k}\sqrt{q^0q'^0}$. 
		In addition, we have that 
		$
		(q^0q'^0)^{\frac{1}{2}}\sqrt{J(q')}\lesssim (J(q)J(q'))^\epsilon
		$
		for a sufficiently small $\epsilon>0$. 	This is because Lemma \ref{lemmaqq} implies that $q^0+q'^0\approx q'^0$ if $\bar{g}\leq 2^{-k}$ for a positive integer $k$.
Thus, we obtain that 
		\begin{multline}\notag
		\left|T^{k,l}_{+,\text{I}}-T^{k,l}_{-,\text{I}}\right|
		\lesssim 2^{-k} \int_0^1d\vartheta \int_\rth dp \int_\rth dq \int_{\mathbb{S}^2}d\omega \  v_{\text{\o}} \sigma_k(g,\omega) |f(q)||h(p)|\\\times w^{2l}(p)(J(q)J(q'))^\epsilon|\nabla|\eta(\kappa(\vartheta)).
		\end{multline}
We thus conclude 
		\begin{multline}\label{page28I2}
		\left|T^{k,l}_{+,\text{I}}-T^{k,l}_{-,\text{I}}\right|
		\lesssim  2^{-k}\Bigg(\int_0^1d\vartheta \int_\rth dp \int_\rth dq \int_{\mathbb{S}^2}d\omega \  v_{\text{\o}} \sigma_k(g,\omega)|f(q)|^2|h(p)|^2\\\times w^{2l}(p)(J(q)J(q'))^\epsilon\Bigg)^{\frac{1}{2}}\\\times \Bigg(\int_0^1d\vartheta \int_\rth dp \int_\rth dq \int_{\mathbb{S}^2}d\omega \  v_{\text{\o}} \sigma_k(g,\omega)w^{2l}(p)\\\times (J(q)J(q'))^\epsilon\left||\nabla|\eta(\kappa(\vartheta))\right|^2\Bigg)^{\frac{1}{2}}\eqdef 2^{-k}I_1^{1/2}\times I_2^{1/2},
		\end{multline}
		where the second inequality is by the Cauchy-Schwarz inequality. For the first part $I_1$, we use $J(q')^\epsilon\lesssim 1$ and follow exactly the same argument as in the estimate for $|T^{k,l}_-|$ as in \eqref{Bk} and \eqref{T-} to obtain that $$I_1\lesssim 2^{k\gamma}|f|^2_{L^2_{-m}}|w^lh|^2_{L^2_{\frac{\singS+\gamma}{2}}},~\text{for any}~m\ge 0.$$ For the second part $I_2$, we recall that $\kappa(\vartheta)=\vartheta p+(1-\vartheta)p'$. Following Lemma \ref{lemma.reduction} we first recover the integrals with respect to the post-collisional momenta $p'$ and $q'$ and write $I_2$ as a 13-fold integral as follows:
		\begin{multline}\label{I2.eta.int}
			I_2=I_2(\eta)=\int_0^1d\vartheta \int_\rth \frac{dp}{p^0} \int_{\rth}\frac{dp'}{p'^0}  \int_\rth \frac{dq}{q^0}\int_\rth \frac{dq'}{q'^0}\   s\sigma_k(g,\omega)\delta^{(4)}(p'^\mu+q'^\mu-p^\mu-q^\mu)\\\times w^{2l}(p)(J(q)J(q'))^\epsilon\left||\nabla|\eta(\vartheta p+(1-\vartheta)p')\right|^2.
		\end{multline}
		We recall $\sigma_k(g,\omega)\lesssim g^{\singS}\sigma_0(\omega)\chi_k(\bar{g})\approx g^{\singS}\left(\frac{\bar{g}}{g}\right)^{-2-\gamma}\chi_k(\bar{g})$ as in \eqref{bargoverg} for both hard and soft-interactions \eqref{hard} and \eqref{soft}.
		Then, we have\begin{multline*}
		I_2\lesssim \int_0^1d\vartheta \int_\rth \frac{dp}{p^0} \int_{\rth}\frac{dp'}{p'^0}  \int_\rth \frac{dq}{q^0}\int_\rth \frac{dq'}{q'^0}\   sg^{\singS}\left(\frac{\bar{g}}{g}\right)^{-2-\gamma}\chi_k(\bar{g})\\\times \delta^{(4)}(p'^\mu+q'^\mu-p^\mu-q^\mu) w^{2l}(p)(J(q)J(q'))^\epsilon\left||\nabla|\eta(\vartheta p+(1-\vartheta)p')\right|^2.
		\end{multline*}	Here, we split $\bar{g}^{-2-\gamma}=\bar{g}^{-3-\gamma'}\times \bar{g}^{-\gamma+1+\gamma'}$ for some constant  $\gamma'>0$ to be chosen. Then note that 
		$-\gamma+1+\gamma' >0$ for any $\gamma' \ge 0.$  For the first component $\bar{g}^{-3-\gamma'}$, we use $$\bar{g}(p^\mu,p'^\mu)=\bar{g}(q^\mu,q'^\mu)\geq \frac{|q-q'|}{\sqrt{q^0q'^0}}=\frac{|p'-p|}{\sqrt{q^0q'^0}}$$ and obtain that $$\bar{g}^{-3-\gamma'}(J(q)J(q'))^\epsilon\lesssim |p-p'|^{-3-\gamma'}(J(q)J(q'))^{\epsilon'},$$ for some $\epsilon>\epsilon'>0$.   
		For the second component $\bar{g}^{1+\gamma'-\gamma}$, we use $\bar{g}\approx 2^{-k}$ from the support condition $\chi_k(\bar{g})$ to obtain that $\bar{g}^{1+\gamma'-\gamma} \approx 2^{(\gamma-1-\gamma')(k+1)}$. Thus, we obtain 
		$$\bar{g}^{-2-\gamma}(J(q)J(q'))^{\epsilon}\lesssim 2^{(\gamma-1-\gamma')(k+1)} |p-p'|^{-3-\gamma'}(J(q))^{\epsilon'},
		$$
		where we also used $J(q')^{\epsilon'}\leq 1$.
	Hence, \begin{multline*}
	I_2\lesssim 2^{(\gamma-1-\gamma')(k+1)} \int_0^1d\vartheta \int_\rth \frac{dp}{p^0} \int_{\rth}\frac{dp'}{p'^0}  \int_\rth \frac{dq}{q^0}\int_\rth \frac{dq'}{q'^0}\   \frac{sg^{\singS+\gamma+2}}{|p-p'|^{3+\gamma'}}\chi_k(\bar{g})\\\times \delta^{(4)}(p'^\mu+q'^\mu-p^\mu-q^\mu) w^{2l}(p)(J(q))^{\epsilon'}\left||\nabla|\eta(\vartheta p+(1-\vartheta)p')\right|^2.
	\end{multline*}
	By Lemma \ref{lemma.carleman}, we can reduce this integral to the integral on the set $E^q_{p'-p}$ and obtain
		\begin{multline*}  
		I_2\lesssim 2^{(\gamma-1-\gamma')(k+1)}  \int_0^1d\vartheta \int_ \rth \frac{dp}{{p^0}} \int_ \rth \frac{dp'}{{p'^0}} \int_{E^{q}_{p'-p}} \frac{d\pi_{q}}{8\bar{g}{q^0}}\frac{sg^{\singS+\gamma+2}}{|p-p'|^{3+\gamma'}}\chi_k(\bar{g})\\\times w^{2l}(p)(J(q))^{\epsilon'}\left||\nabla|\eta(\vartheta p+(1-\vartheta)p')\right|^2
		\end{multline*}  where  
		$d\pi_{q}=dq\ u({p^0}+{q^0}-{p'^0})\delta\left(\frac{\bar{g}}{2}+\frac{q^\mu(p_\mu-p'_\mu)}{\bar{g}}\right)$.
		We use another $\bar{g}^{-1}\leq 2^{k+1}$ and reorder terms to obtain that 
				\begin{multline*}
		I_2\lesssim 2^{(\gamma-\gamma')(k+1)} \int_0^1d\vartheta \int_ \rth \frac{dp}{{p^0}} \int_ \rth \frac{dp'}{{p'^0}} \int_{E^{q}_{p'-p}} \frac{d\pi_{q}}{{q^0}}\frac{sg^{\singS+\gamma+2}}{|p-p'|^{3+\gamma'}}\chi_k(\bar{g})\\\times w^{2l}(p)(J(q))^{\epsilon'}\left||\nabla|\eta(\vartheta p+(1-\vartheta)p')\right|^2\\
		\approx 2^{(\gamma-\gamma')(k+1)} \int_0^1d\vartheta \int_ \rth \frac{dp}{{p^0}} \int_ \rth \frac{dp'}{{p'^0}}\frac{\left||\nabla|\eta(\vartheta p+(1-\vartheta)p')\right|^2}{|p-p'|^{3+\gamma'}}\chi_k(\bar{g})\\\times w^{2l}(p) \int_{E^{q}_{p'-p}} \frac{d\pi_{q}}{{q^0}}sg^{\singS+\gamma+2}(J(q))^{\epsilon'}.
		\end{multline*} 
			We use that $\singS+\gamma+2 \geq 0$ for both hard and soft-interactions \eqref{hard} and \eqref{soft} and \eqref{FREQ:s.le.pq}, \eqref{FREQ:g.le.sqrtpq}, and Lemma \ref{lemmaqq} to obtain  
		\begin{multline*}
		\int_{E^{q}_{p'-p}} \frac{d\pi_{q}}{{q^0}}sg^{\singS+\gamma+2} (J(q))^{\epsilon'}\lesssim (p^0)^{2+\frac{\singS+\gamma}{2}}\int_{E^{q}_{p'-p}}\frac{d\pi_q}{q^0} (q^0)^{2+\frac{\singS+\gamma}{2}}(J(q))^{\epsilon'}
		\\
		\lesssim (p^0p'^0)^{1+\frac{\singS+\gamma}{4}}\int_{E^{q}_{p'-p}}\frac{d\pi_q}{q^0} (J(q))^{\epsilon''},
\end{multline*}	
				where $\epsilon''$ is some uniform constant such that $0<\epsilon''<\epsilon'$.

Then we {\it claim} that
\begin{equation}
\label{dpiq1}
\int_{E^{q}_{p'-p}}\frac{d\pi_q}{q^0} (J(q))^{\epsilon''}
\lesssim 1.
\end{equation}
This can be seen by a direct computation.  
We can justify the {\it claim} as follows	
	$$	
		\int_{E^{q}_{p'-p}}\frac{d\pi_q}{q^0} (J(q))^{\epsilon''}
		=\int_{\rth}\frac{dq}{q^0}\ u({p^0}+{q^0}-{p'^0})\delta\left(\frac{\bar{g}}{2}+\frac{q^\mu(p_\mu-p'_\mu)}{\bar{g}}\right)(J(q))^{\epsilon''}.$$
	We take the change of variables on $q$ into angular coordinates as $q\in \rth \mapsto (r,\theta,\phi)$ and choose the z-axis parallel 
	to $p-p'$ such that the angle between $q$ and $p-p'$ is equal to $\phi.$   The terms in the delta function can be rewritten as
	\begin{multline*}
		\frac{\bar{g}}{2}+\frac{q^\mu(p_\mu-p'_\mu)}{\bar{g}}=\frac{1}{2\bar{g}}(\bar{g}^2+2q^\mu(p_\mu-p'_\mu))\\
		=\frac{1}{2\bar{g}}(\bar{g}^2-2q^0(p^0-p'^0)+2q\cdot(p-p'))\\
		=\frac{1}{2\bar{g}}(\bar{g}^2-2\sqrt{1+r^2}(p^0-p'^0)+2r|p-p'|\cos\phi).		
	\end{multline*}
		Thus, we obtain that 
\begin{multline*}	
\int_{E^{q}_{p'-p}}\frac{d\pi_q}{q^0} (J(q))^{\epsilon''}
=\int_{0}^\infty\frac{dr}{\sqrt{1+r^2}}\int_{0}^\pi d\phi \int_{0}^{2\pi}d\theta ~ r^2\sin\phi\ u({p^0}+{\sqrt{1+r^2}}-{p'^0})\\\times \delta\left(\frac{1}{2\bar{g}}(\bar{g}^2-2\sqrt{1+r^2}(p^0-p'^0)+2r|p-p'|\cos\phi)\right)e^{-\epsilon''\sqrt{1+r^2}}
\\
=\int_{0}^\infty\frac{dr}{\sqrt{1+r^2}}\int_{-1}^1 d(-\cos\phi) \int_{0}^{2\pi}d\theta \ r^2 \  u({p^0}+{\sqrt{1+r^2}}-{p'^0})
\\
\times \frac{\bar{g}}{r|p-p'|}\delta\left(\cos\phi+\frac{\bar{g}^2-2\sqrt{1+r^2}(p^0-p'^0)}{2r|p-p'|}\right)
e^{-\epsilon''\sqrt{1+r^2}}
\\
\leq \int_{0}^\infty\frac{rdr}{\sqrt{1+r^2}}\int_{0}^{2\pi}d\theta \ u({p^0}+{\sqrt{1+r^2}}-{p'^0}) \frac{\bar{g}}{|p-p'|}e^{-\epsilon''\sqrt{1+r^2}}\\
\lesssim \int_{0}^\infty\frac{rdr}{\sqrt{1+r^2}}e^{-\epsilon''\sqrt{1+r^2}}\lesssim 1,
\end{multline*}
where we have used $\bar{g}\leq |p-p'|$ and $u(x)\leq 1$. Thus we obtain
		\begin{multline*}
			I_2\lesssim 2^{(\gamma-\gamma')(k+1)}  \int_0^1d\vartheta \int_ \rth dp' \int_ \rth dp\ (p^0p'^0)^{\frac{\singS+\gamma}{4}}w^{2l}(p)\\\times \frac{\left||\nabla|\eta(\vartheta p+(1-\vartheta)p')\right|^2}{|p-p'|^{3+\gamma'}} \chi_k(\bar{g}).
		\end{multline*}
		Now we define $u=p-p'$ and consider the change of variables $p\mapsto u=p-p'$. Then we have $\vartheta p+(1-\vartheta)p'=p'+\vartheta u$. Note that the support condition $\chi_k(\bar{g})$ gives $p^0\approx p'^0$ and $w^l(p)\approx w^l(p')$.  Also, the support condition $\chi_k(\bar{g})$ indicates that $|u|=|p-p'|\geq \bar{g}\geq 2^{-k-1}$ and that $|u|\leq 2^{-k}5^{1/4}p'^0$. This is because we can use $p^0\leq \sqrt{5}p'^0$ from Lemma \ref{lemmaqq} to obtain that $$|u|=|p-p'|\leq 2^{-k}\sqrt{p^0p'^0}\leq 2^{-k}5^{1/4}p'^0.$$   	
				Then we have 
		\begin{multline*}
		I_2\lesssim 2^{(\gamma-\gamma')(k+1)}  \int_0^1d\vartheta \int_ \rth dp' (p'^0)^{\frac{\singS+\gamma}{2}}w^{2l}(p')\\\times\int_ {2^{-k-1}\leq |u|\leq 2^{-k}5^{1/4}p'^0 } du\  \frac{\left||\nabla|\eta(p'+\vartheta u)\right|^2}{|u|^{3+\gamma'}}.
		\end{multline*}
		We will consider this expression briefly and estimate some of the terms in the integrand. 
In a moment we will use the change of variables $p'\mapsto v=p'+\vartheta u$.

We first {\it claim} that $p'^0\approx v^0\eqdef \sqrt{1+|v|^2}$. This can be proved as follows. We show that $v^0\lesssim p'^0$ first.  Note that we have 	
		\begin{multline*}
	(v^0)^2=1+|v|^2=1+|p'+\vartheta u|^2= 1+|p'|^2 + \vartheta^2 |u|^2 +2\vartheta p'\cdot u \\
	\leq 2(p'^0)^2+ 2\vartheta^2 |u|^2\leq 2(p'^0)^2+ 2\vartheta^2 2^{-2k}5^{1/2}(p'^0)^2\leq \frac{7}{2}(p'^0)^2 
\end{multline*} as $\vartheta $ is in $(0,1)$ and $k$ is a positive integer. So we have $v^0\lesssim p'^0.$ On the other hand, in order to prove $p'^0\lesssim v^0$, we recall that $v=p'+\vartheta u$ and obtain 
		\begin{equation}\label{pv}
				(p'^0)^2=1+|p'|^2 = 1+|v-\vartheta u|^2 =(1+|v|^2)-2\vartheta v \cdot u+ \vartheta^2|u|^2.
		\end{equation}
		By Young's inequality, we have $2\vartheta\left| v \cdot u\right|\leq 4|v|^2 +\frac{1}{4}\vartheta^2 |u|^2$. We plug this back into (\ref{pv}) to obtain that  $$(p'^0)^2\leq 5(v^0)^2 +\frac{5}{4}\vartheta^2|u|^2\leq 5(v^0)^2 +\frac{5}{4}\vartheta^22^{-2k}5^{1/2}(p'^0)^2 \leq 5(v^0)^2 +\frac{12}{16} (p'^0)^2, 
		 $$ because $k$ is a positive integer and $\vartheta\in (0,1)$. Therefore, we obtain $(p'^0)^2\leq 20(v^0)^2$. We conclude that $p'^0\approx v^0.$

		We plug in these estimates, and then we use Fubini's theorem to change the order of integration to obtain
		\begin{multline*}
		I_2\lesssim 2^{(\gamma-\gamma')(k+1)}  \int_0^1d\vartheta \int_ {2^{-k-1}\leq |u|} du\  \frac{1}{|u|^{3+\gamma'}}\int_ \rth dp' (v^0)^{\frac{\singS+\gamma}{2}}w^{2l}(v)\left||\nabla|\eta(v)\right|^2.
			\end{multline*}
		Now we consider the change of variables $p'\mapsto v=p'+\vartheta u$.  We obtain 	
\begin{multline}\label{I2.int.est}
		I_2\lesssim 
		2^{(\gamma-\gamma')(k+1)}  \int_0^1d\vartheta \int_ {2^{-k-1}\leq |u| } du\  \frac{1}{|u|^{3+\gamma'}}\\\times\int_ \rth dv (v^0)^{\frac{\singS+\gamma}{2}}w^{2l}(v)\left||\nabla|\eta(v)\right|^2\\
		\lesssim 2^{(\gamma-\gamma')(k+1)} \left|w^l|\nabla|\eta\right|^2_{L^2_\frac{\singS+\gamma}{2}} \int_0^1d\vartheta \int_ {2^{-k-1}\leq |u| } du\  \frac{1}{|u|^{3+\gamma'}}
		\\
\lesssim 2^{(\gamma-\gamma')(k+1)} \left|w^l|\nabla|\eta\right|^2_{L^2_\frac{\singS+\gamma}{2}} \int_0^1d\vartheta \int_ {2^{-k-1}}^{\infty} d|u|\   \frac{|u|^2}{|u|^{3+\gamma'}}
		\\
\lesssim 2^{(\gamma-\gamma')(k+1)} \left|w^l|\nabla|\eta\right|^2_{L^2_\frac{\singS+\gamma}{2}}2^{(k+1)\gamma'}
\lesssim 
2^{(k+1)\gamma} \left|w^l|\nabla|\eta\right|^2_{L^2_\frac{\singS+\gamma}{2}}\\
\lesssim 
2^{k\gamma} \left|w^l|\nabla|\eta\right|^2_{L^2_\frac{\singS+\gamma}{2}}.
\end{multline}
		 From \eqref{page28I2} and \eqref{I2.int.est},  we obtain
\begin{multline*}
		|(T^{k,l}_{+,\text{I}}-T^{k,l}_{-,\text{I}})(f,h,\eta)|	\lesssim  2^{-k}I_1^{1/2}\times I_2^{1/2}\\
		 	\lesssim  \ 2^{(\gamma-1)k}|f|_{L^2_{-m}}|  w^lh|_{L^2_{\frac{\singS+\gamma}{2}}}|   w^l|\nabla|\eta |_{L^2_{\frac{\singS+\gamma}{2}}}.
		 \end{multline*}	
This completes the estimate for part I.

    For part II of \eqref{canceldecompp} in \eqref{eq.T+-T-}, we use the fundamental theorem of calculus and obtain that 
  $$
   \sqrt{J(q')}-\sqrt{J(q)}=\int_0^1d\vartheta\ (q'-q)\cdot(\nabla\sqrt{J})(\kappa_q(\vartheta)).
   $$ Thus, we have
   $$
   \left|\left( \sqrt{J(q')}-\sqrt{J(q)}\right)\eta(p)\right|\leq \int_0^1d\vartheta\  |q'-q||\nabla|\sqrt{J}(\kappa_q(\vartheta))|\eta(p)|.
   $$ Now we observe that $|q'-q|\leq g(q^\mu,q'^\mu)\sqrt{q^0q'^0}=\bar{g}\sqrt{q^0q'^0}\approx 2^{-k}\sqrt{q^0q'^0}$. For $|\nabla|\sqrt{J}$, we use that $$|\nabla|\sqrt{J}(\kappa_q(\vartheta))\lesssim \sqrt{J}(\kappa_q(\vartheta)).$$ Also, we have that 
   $
   (q^0q'^0)^{\frac{1}{2}}\sqrt{J(\kappa_q(\vartheta))}\lesssim (J(q)J(q'))^\epsilon
   $
 for any $\vartheta\in[0,1]$   for sufficiently small $\epsilon>0$.
    Thus, using $w^{2l}(p')\approx w^{2l}(p)$ by Lemma \ref{lemmaqq},  the difference $\left|T^{k,l}_{+,\text{II}}-T^{k,l}_{-,\text{II}}\right|$ of the part $\text{II}$ of \eqref{canceldecompp} in \eqref{eq.T+-T-} is bounded above as
    \begin{multline}\label{partII}
    \left|T^{k,l}_{+,\text{II}}-T^{k,l}_{-,\text{II}}\right|\\\lesssim 2^{-k} \int_\rth dp \int_\rth dq \int_{\mathbb{S}^2}d\omega \  v_{\text{\o}} \sigma_k(g,\omega) |f(q)||h(p)|(J(q)J(q'))^\epsilon| \eta(p)|w^{2l}(p).
    \end{multline}
    Now the rest of the proof follows exactly the same as in the estimate for $|T^{k,l}_-|$ as in  \eqref{Bk} and \eqref{T-}, and we obtain the upper bound in the right-hand side of the proposition because $|\eta|$ is less than or equal to $\left||\nabla| \eta\right|$ by the definition of the length of the gradient \eqref{nabla}. 
   Thus,
   $$\left|T^{k,l}_{+,\text{II}}-T^{k,l}_{-,\text{II}}\right|\lesssim 2^{(\gamma-1)k}|f|_{L^2_{-m}}|w^l h|_{L^2_{\frac{\singS+\gamma}{2}}}|w^l |\nabla|\eta|_{L^2_{\frac{\singS+\gamma}{2}}}.$$
   The term $\text{III}$ is then handled the same way as the term $\text{II}.$
	Together with the previous estimates, we obtain the proposition.
	\end{proof}

\subsection{Dual cancellation estimates}\label{sec.newcancel}
We will also derive the following cancellation estimates that have the momentum derivative acting on $h$ instead of $\eta$.

\begin{proposition}\label{prop:cancellation2}For any $k> 0$ and for $0<\gamma<1$ and any $m> 0$, we have the uniform estimate: 
\begin{equation}\notag
\left|(T^{k,l}_{+,d}-T^{k,l}_{-,d})(f,h,\eta)\right|
\lesssim   2^{k(\gamma-1)} |f|_{L^2_{-m}}|w^l|\nabla|h|_{L^2_{\frac{\singS+\gamma}{2}}}|w^l\eta|_{L^2_{\frac{\singS+\gamma}{2}}}.
\end{equation}

\end{proposition}

For the proof of Proposition \ref{prop:cancellation2}, we will have to derive and use the Carleman dual representation of the trilinear form $\langle w^{2l}\Gamma(f,h),\eta\rangle$. The dual representation of the trilinear term is derived in \eqref{dual4} in  \secref{CarlemanAppendix}.  

If we consider the dyadic decomposition of the region $\bar{g}\le 1$ into $\bar{g}\approx  2^{-k}$ in the original representation as in \eqref{kernel.k.diadic.define}, then this decomposition corresponds to the dyadic decomposition of $g_{L}\approx 2^{-k}$ in the new representation \eqref{dual4} after we applied the Lorentz transformation.  As before \eqref{kernel.k.diadic.define}, we let $\{\chi_k\}^\infty_{k=-\infty}$ be the partition of unity on $(0,\infty)$ such that $|\chi_k|\leq 1$ and supp$(\chi_k) \subset [2^{-k-1},2^{-k}]$.  Then we define 
$$
\frac{c'}{2}\sigma(g_\Lambda, \theta_\Lambda)\chi_k(g_L)\eqdef  \sigma_k(g_\Lambda, \theta_\Lambda).
$$
Here $\frac{c'}{2}>0$ is the constant that comes from \eqref{dual4}, and $g_L$ is defined in \eqref{g2.eq.lambda}.  We can now write the decomposed pieces of the dual formulation as $T^{k,l}_{+,d}=T^{k,l}_{+,d}(f,h,\eta)$ and $T^{k,l}_{-,d}=T^{k,l}_{-,d}(f,h,\eta)$ from \eqref{dual4} as
 \begin{align}\notag
T^{k,l}_{+,d}\eqdef 
&\int_\rth\frac{dp'}{p'^0}\int_{\rth}\frac{dq}{q^0} \frac{\sqrt{\tilde{s}}}{\tilde{g}}
\int_{\mathbb{R}^2} \frac{dz}{\sqrt{|z|^2+1}}
s_\Lambda\sigma_k(g_\Lambda,\theta_\Lambda)
\sqrt{J(q)}w^{2l}(p') \eta(p')f(q)
\\ \notag&
\times h\left(A(p',q,z)+p'\right)\exp\left(-\frac{p'^0+q^0}{4}(\sqrt{|z|^2+1}-1)+\frac{|p'\times q|}{2\tilde{g}}z_1\right) 
 \\\notag
T^{k,l}_{-,d}
\eqdef 
&\int_\rth\frac{dp'}{p'^0}\int_{\rth}\frac{dq}{q^0} \frac{\sqrt{\tilde{s}}}{\tilde{g}}
\int_{\mathbb{R}^2} \frac{dz}{\sqrt{|z|^2+1}}
s_\Lambda\sigma_k(g_\Lambda,\theta_\Lambda)
\sqrt{J(q)}w^{2l}(p')\eta(p')f(q)
\\
 \label{T+++.dual} &
\times\frac{\tilde{s} \Phi(\tilde{g})\tilde{g}^4}{s_\Lambda \Phi(g_\Lambda)g^4_\Lambda}
h(p').
\end{align}
From the calculations in  \secref{CarlemanAppendix} we have that $$(T^{k,l}_{+,d}-T^{k,l}_{-,d})(f,h,\eta)
 =(T^{k,l}_{+}-T^{k,l}_{-})(f,h,\eta),$$ also using the definitions \eqref{T+++}.  The rest of the notation used above is defined in \eqref{g2y.variable} or \eqref{g2.eq.lambda}, \eqref{newcos.intro} or \eqref{theta.eq.lambda}, and  \eqref{A.def.lambda}.   This notation will also be defined the first time we use it in the proof below.    Thus, again, for $f,h,\eta\in S(\rth)$ Schwartz functions, we have:
 $$
 \langle  w^{2l}\Gamma(f,h),\eta\rangle=\sum^\infty_{k=-\infty}\{T^{k,l}_{+,d}(f,h,\eta)-T^{k,l}_{-,d}(f,h,\eta)\}.
 $$ 
Since this dual representation is written in $(p',q,z)$ variables as above, instead of the standard $(p,q,\omega)$ variables, we will first check what the condition $\bar{g}\approx 2^{-k}$ of the dyadic decomposition of the angular singularity would correspond to in the new variable $g_L$ which is defined in \eqref{g2y.variable}. The rest of this section is devoted to the proof of Proposition \ref{prop:cancellation2}. 

In order to estimate the trilinear terms in \eqref{T+++.dual} nearby the singularity, for the difference of \eqref{T+++.dual}, we split up the difference of the integrands as follows
\begin{multline*}
    h\left(A(p',q,z)+p'\right)
    \exp\left(-\frac{p'^0+q^0}{4}(\sqrt{|z|^2+1}-1)+\frac{|p'\times q|}{2\tilde{g}}z_1\right) 
    \\
    -\frac{\tilde{s} \Phi(\tilde{g})\tilde{g}^4}{s_\Lambda \Phi(g_\Lambda)g^4_\Lambda}h(p')
    \\
    =
\exp\left(-\frac{p'^0+q^0}{4}(\sqrt{|z|^2+1}-1)+\frac{|p'\times q|}{2\tilde{g}}z_1\right) 
\left(    h\left(A(p',q,z)+p'\right) - h(p') \right)
    \\
    +
    \frac{\tilde{s} \Phi(\tilde{g})\tilde{g}^4}{s_\Lambda \Phi(g_\Lambda)g^4_\Lambda}
    \left( \exp\left(-\frac{p'^0+q^0}{4}(\sqrt{|z|^2+1}-1)+\frac{|p'\times q|}{2\tilde{g}}z_1\right) -1 \right)
    h(p')
        \\
    +
    \exp\left(-\frac{p'^0+q^0}{4}(\sqrt{|z|^2+1}-1)+\frac{|p'\times q|}{2\tilde{g}}z_1\right) 
   \left(1 -  \frac{\tilde{s} \Phi(\tilde{g})\tilde{g}^4}{s_\Lambda \Phi(g_\Lambda)g^4_\Lambda} \right)
    h(p')
    \\
=    I+II+III.
\end{multline*}
In the rest of this section, we will make an upper-bound estimate for each part of the trilinear term $T^{k,l}_+-T^{k,l}_-$ which involves the parts $I,$ $II,$ and $III.$   Thus we define
 \begin{multline}\label{define.Tkl1}
   (T^{k,l}_{+,I}
   -T^{k,l}_{-,I})(f,h,\eta)
   \\
   \eqdef 
	\int_\rth\frac{dp'}{p'^0}\int_{\rth}\frac{dq}{q^0} \frac{\sqrt{\tilde{s}}}{\tilde{g}}\int_{\mathbb{R}^2} \frac{dz}{\sqrt{|z|^2+1}}s_\Lambda\sigma_k(g_\Lambda,\theta_\Lambda)\sqrt{J(q)}w^{2l}(p')\eta(p')f(q)\\
	\times\exp\left(-l(\sqrt{|z|^2+1}-1)+j z_1\right) \bigg[h\left(A(p',q,z)+p'\right)-h\left(
	p'\right)\bigg],
 \end{multline}
 where we recall the notation \eqref{lj}.  We also define
  \begin{multline}\label{define.Tkl2}
   (T^{k,l}_{+,II}
   -T^{k,l}_{-,II})(f,h,\eta)
   \\
   \eqdef 
\int_\rth\frac{dp'}{p'^0} h(p')w^{2l}(p')\eta(p')\int_{\rth}\frac{dq}{q^0} \sqrt{J(q)}f(q)\frac{\sqrt{\tilde{s}}}{\tilde{g}}
\int_{\mathbb{R}^2} \frac{dz}{\sqrt{|z|^2+1}}s_\Lambda\sigma_k(g_\Lambda,\theta_\Lambda)
\\\times
\frac{\tilde{s} \Phi(\tilde{g})\tilde{g}^4}{s_\Lambda \Phi(g_\Lambda)g^4_\Lambda}
            \left(\exp\left(-l(\sqrt{|z|^2+1}-1)+j z_1\right)-1\right),
 \end{multline}
 and lastly
  \begin{multline}\label{define.Tkl3}
   (T^{k,l}_{+,III}
   -T^{k,l}_{-,III})(f,h,\eta)
   \\
   \eqdef 
\int_\rth\frac{dp'}{p'^0} h(p')w^{2l}(p')\eta(p')\int_{\rth}\frac{dq}{q^0} \sqrt{J(q)}f(q)\frac{\sqrt{\tilde{s}}}{\tilde{g}}
\int_{\mathbb{R}^2} \frac{dz}{\sqrt{|z|^2+1}}s_\Lambda\sigma_k(g_\Lambda,\theta_\Lambda)
\\\times \exp\left(-l(\sqrt{|z|^2+1}-1)+jz_1\right) 
	\left(1 -  \frac{\tilde{s} \Phi(\tilde{g})\tilde{g}^4}{s_\Lambda \Phi(g_\Lambda)g^4_\Lambda} \right).
 \end{multline}
 Then note that 
 $
    T^{k,l}_{+,d}
   -T^{k,l}_{-,d}
   =
      T^{k,l}_{+,I}
   -T^{k,l}_{-,I}
   +
      T^{k,l}_{+,II}
   -T^{k,l}_{-,II}
   +
      T^{k,l}_{+,III}
   -T^{k,l}_{-,III}.
 $
For our estimates of each of these terms, we will make use of these dual representations that are written in the variables $(p',q,z)$. Therefore we will next study which conditions in the $z$ variable corresponds to the correct dyadic decomposition.

We first note that the condition $2^{-k-1}\le g_L\le2^{-k}$ is equivalent to
\begin{equation}\label{z.dyadic}2^{-2k-1}\tilde{s}^{-1}\le (\sqrt{|z|^2+1}-1)\le  2^{-2k+1}\tilde{s}^{-1}\end{equation} 
by \eqref{g2y.variable}. 
In the rest of this section, we will denote the condition \eqref{z.dyadic} simply as $$(\sqrt{|z|^2+1}-1)\approx 2^{-2k}\tilde{s}^{-1}.$$
We also note that \eqref{z.dyadic} implies that
\begin{equation}\label{z.dyadic2}0<2^{-k-1}\tilde{s}^{-1/2}\le |z|\le 2^{-k+1}\tilde{s}^{-1/2}\sqrt{1+2^{-2k}\tilde{s}^{-1}}\le 2^{-k+2}\tilde{s}^{-1/2}\le 1,\end{equation}
since $k>0$ and $\tilde{s}\ge 4$.

\begin{proof}[Proof of Proposition \ref{prop:cancellation2}]
We will split the proof into three parts $I$, $II$, and $III$.  In each part we estimate each of the terms in \eqref{define.Tkl1},  \eqref{define.Tkl2}, and  \eqref{define.Tkl3} respectively.

{\bf Estimates on part $I$}. For the estimate of the first part in \eqref{define.Tkl1},
we use the Cauchy-Schwarz inequality to obtain that 
\begin{multline}\label{estimate.key.term}
	|(T^{k,l}_{+,I}-T^{k,l}_{-,I})(f,h,\eta)|
	\\ 
	\lesssim 
\bigg(\int_\rth\frac{dp'}{p'^0}\int_{\rth}\frac{dq}{q^0} \left(\frac{\sqrt{\tilde{s}}}{\tilde{g}}\right)^2\int_{\mathbb{R}^2} \frac{dz}{\sqrt{|z|^2+1}}s_\Lambda\sigma_k(g_\Lambda,\theta_\Lambda)(J(q))^{1/2}w^{2l}(p')(q^0)^{m}\\
\times\exp\left(-l(\sqrt{|z|^2+1}-1)+jz_1\right) \bigg[h\left(A(p',q,z)+p'\right)-h\left(
	p'\right)\bigg]^2\bigg)^{1/2}\\
\times 	\bigg(\int_\rth\frac{dp'}{p'^0}\int_{\rth}\frac{dq}{q^0} \int_{\mathbb{R}^2} \frac{dz}{\sqrt{|z|^2+1}}s_\Lambda\sigma_k(g_\Lambda,\theta_\Lambda)(J(q))^{1/2}w^{2l}(p')|\eta(p')|^2|f(q)|^2\\
\times\exp\left(-l(\sqrt{|z|^2+1}-1)+j z_1\right)(q^0)^{-m} \bigg)^{1/2}
\eqdef D_1^{1/2}D_2^{1/2},
\end{multline}
for a sufficiently large $m>0$.   We note that the spltting with the term $(q^0)^{m}$ above is important because the term $D_2$ needs the extra $|q|$ decay and the term $D_1$ has plenty of exponential $|q|$ decay.

{\it The representation of $D_1$}.
We can recover the original representation of $D_1$ in the following way.  Specifically we notice that $D_1$ is of the form \eqref{second.eq.Ig} with $G$ in \eqref{second.eq.Ig} given by
\begin{multline*}
    G=\frac{c'}{2}(q^0)^{m}\frac{\sqrt{\tilde{s}}}{\tilde{g}}\chi_k(g_L)
    (J(q))^{1/2}w^{2l}(p')
    \exp\left(-l(\sqrt{|z|^2+1}-1)+jz_1\right)
    \\ \times
    \bigg[h\left(A(p',q,z)+p'\right)-h\left(
p'\right)\bigg]^2,
\end{multline*}
Thus since \eqref{second.eq.Ig} is also equal to \eqref{original.eq.Ig}, 
we obtain that
$D_1$ corresponds to
$$D_1\approx \int_\rth dp \int_\rth dq \int_{\mathbb{S}^2}d\omega\  \frac{s}{p^0q^0} \sigma(g,\theta)\chi_k(\bar{g})w^{2l}(p')J(q')^{1/2} \bigg[h(p')-h\left(
p\right)\bigg]^2(q^0)^{m},$$ 
where we used \eqref{gtildeg.equiv} such that $v_{\text{\o}}\frac{\sqrt{\tilde{s}}}{\tilde{g}}\approx \frac{s}{p^0q^0}$. We use the fundamental theorem of calculus to obtain
\begin{multline*}D_1\lesssim \int_\rth dp \int_\rth dq \int_{\mathbb{S}^2}d\omega \ \frac{s}{p^0q^0} \sigma(g,\theta)\chi_k(\bar{g})w^{2l}(p')J(q')^{1/2} \\\times |p'-p|^2\bigg(\int_0^1d\vartheta(|\nabla|h)\left(\kappa(\vartheta)\right)\bigg)^2(q^0)^{m}\\
	\lesssim  \int_0^1d\vartheta\int_\rth dp \int_\rth dq \int_{\mathbb{S}^2}d\omega \ \frac{s}{p^0q^0} \sigma(g,\theta)\chi_k(\bar{g})w^{2l}(p')J(q')^{1/2} \\\times |p'-p|^2\bigg((|\nabla|h)\left(\kappa(\vartheta)\right)\bigg)^2(q^0)^{m},
\end{multline*} where $\kappa(\vartheta)\eqdef \vartheta p+(1-\vartheta)p'$. 	Note that 
$|p'-p|^2=|q'-q|^2\leq g(q^\mu,q'^\mu)^2q^0q'^0=\bar{g}^2q^0q'^0\approx 2^{-2k}q^0q'^0$ by \eqref{FREQ:g.ge.lower}. 
And $w(p') \approx w(p)$.
Also, we have that 
$$
(q^0)^{m}(q^0q'^0)J(q')^{1/2} \lesssim (J(q)J(q'))^{\epsilon}
$$
for a sufficiently small $\epsilon>0$ by Lemma \ref{lemmaqq}. 
Then this term $D_1$ has an upper bound of $2^{-2k}I_2$ where $I_2$ is defined in \eqref{I2.eta.int}, but the notation $\eta$ is replaced by $h$; in other words in \eqref{I2.eta.int} we use $I_2 = I_2(h)$. Then by the same arguments leading to \eqref{I2.int.est} we obtain the upper bound of $D_1$ as
$$D_1\lesssim 2^{k(\gamma-2)} \left|w^l|\nabla|h\right|^2_{L^2_\frac{\singS+\gamma}{2}} .$$
This completes our estimate for $D_1$.

{\it The estimates for $D_2$}.
The term $D_2$ is given by
\begin{multline*}D_2=\int_\rth\frac{dp'}{p'^0}\int_{\rth}\frac{dq}{q^0} \int_{\mathbb{R}^2} \frac{dz}{\sqrt{|z|^2+1}}s_\Lambda\sigma_k(g_\Lambda,\theta_\Lambda)(J(q))^{1/2}w^{2l}(p')\\\times|\eta(p')|^2|f(q)|^2\exp\left(-l(\sqrt{|z|^2+1}-1)+jz_1\right)(q^0)^{-m},\end{multline*}
where $l=\frac{p'^0+q^0}{4}$ and $j=\frac{|p'\times q|}{2\tilde{g}}$ from \eqref{lj}. 
Note that $s_\Lambda=g_\Lambda^2+4,$ and from \eqref{g2.eq.lambda} we have
\begin{equation}\label{eq.gL}
g_\Lambda^2 = \tilde{g}^2+\frac{\tilde{s}}{2}(\sqrt{|z|^2+1}-1).
\end{equation} 
Above we recall \eqref{gtilde}.  
Also from \eqref{angassumption} and \eqref{theta.eq.lambda} we have 
\begin{equation}\label{eq.sigma0}\sigma_0(\theta_\Lambda)\approx \theta_\Lambda^{-2-\gamma}\approx \left(\frac{\tilde{s}(\sqrt{|z|^2+1}-1)}{g_\Lambda^2}\right)^{-1-\gamma/2}.\end{equation}
Therefore, we have
\begin{multline}\label{eq.ssigma}s_\Lambda\sigma(g_\Lambda,\theta_\Lambda)\approx s_\Lambda g_\Lambda^{\singS+\gamma+2} \tilde{s}^{-1-\gamma/2} \left(\sqrt{|z|^2+1}-1\right)^{-1-\gamma/2}\\\lesssim g_\Lambda^{\singS+\gamma+2}\tilde{s}^{-\gamma/2} \left(\sqrt{|z|^2+1}-1\right)^{-1-\gamma/2},
\end{multline} by \eqref{z.dyadic2}.
Then we further note using \eqref{z.dyadic} and \eqref{z.dyadic2} that
\begin{multline*}D_2\lesssim \int_\rth\frac{dp'}{p'^0}\int_{\rth}\frac{dq}{q^0} \int_{\left(\sqrt{|z|^2+1}-1\right)\approx 2^{-2k}\tilde{s}^{-1}} \frac{dz}{\sqrt{|z|^2+1}}\\
\times g_\Lambda^{\singS+\gamma+2}\tilde{s}^{-\gamma/2} \left(\sqrt{|z|^2+1}-1\right)^{-1-\gamma/2}(J(q))^{1/2}w^{2l}(p')\\\times|\eta(p')|^2|f(q)|^2\exp\left(-l(\sqrt{|z|^2+1}-1)+jz_1\right)(q^0)^{-m}\\
\approx 2^{k(\gamma+2)}\int_\rth\frac{dp'}{p'^0}\int_{\rth}\frac{dq}{q^0} \int_{\left(\sqrt{|z|^2+1}-1\right)\approx 2^{-2k}\tilde{s}^{-1}} \frac{d|z|}{\sqrt{|z|^2+1}}|z|(q^0)^{-m}\\
\times g_\Lambda^{\singS+\gamma+2}\tilde{s}(J(q))^{1/2}w^{2l}(p')|\eta(p')|^2|f(q)|^2\exp\left(-l(\sqrt{|z|^2+1}-1)\right)I_0(j|z|)\\
\lesssim 2^{k(\gamma+2)}\int_\rth\frac{dp'}{p'^0}\int_{\rth}\frac{dq}{q^0} e^{l} \max_{0\le |z|\le 1}\bigg\{\exp\left(-l\sqrt{|z|^2+1}\right)I_0(j|z|)\bigg\}(q^0)^{-m}
\\
\times\int_{\left(\sqrt{|z|^2+1}-1\right)\approx 2^{-2k}\tilde{s}^{-1}} \frac{d|z|}{\sqrt{|z|^2+1}}|z|g_\Lambda^{\singS+\gamma+2}\tilde{s}(J(q))^{1/2}w^{2l}(p')|\eta(p')|^2|f(q)|^2.
\end{multline*} 
In the following we further use the modified Bessel functions from \eqref{bessel0}.  In the inequality above we also used $I_0(j|z|)$ from \eqref{bessel0}.
Using \eqref{FREQ:max.bound} with \eqref{lj} we have
\begin{equation}\label{GS.estimate.exponential}
    \max_{0\le |z|\le 1}\exp\left(-l\sqrt{|z|^2+1}+j|z|\right)\lesssim \exp (-\sqrt{l^2-j^2}).
\end{equation}
The estimate \eqref{GS.estimate.exponential} is earlier proven in \cite{MR1211782}. 
 Then, by \eqref{l2j2size} we have
\begin{multline}\label{Jql2j2.bound}(J(q))^{1/2}e^{l}\exp (-\sqrt{l^2-j^2})
\le \exp\left(-\frac{q^0}{2}+\frac{p'^0+q^0}{4}-\frac{|p'-q|}{4}\right) \\
  \le  \exp\left( \frac{p'^0-q^0}{4}-\frac{|p'-q|}{4}\right)
  \le  1,
\end{multline}by \eqref{FREQ:p0q0.le.pq}.
Thus, 
\begin{multline*}D_2
\lesssim 2^{k(\gamma+2)}\int_\rth\frac{dp'}{p'^0}\int_{\rth}\frac{dq}{q^0}\tilde{s}w^{2l}(p')|\eta(p')|^2|f(q)|^2(q^0)^{-m} \\\times\int_{\left(\sqrt{|z|^2+1}-1\right)\approx 2^{-2k}\tilde{s}^{-1}} \frac{d|z|}{\sqrt{|z|^2+1}}|z|g_\Lambda^{\singS+\gamma+2}
,\end{multline*}
We make the change of variables $|z|\mapsto K\eqdef l(\sqrt{|z|^2+1}-1)$ with $$dK=\frac{l|z|d|z|}{\sqrt{|z|^2+1}}.$$ 
Since $\singS+\gamma+2\ge 0$, 
we use \eqref{FREQ:s.le.pq}, \eqref{z.dyadic2} and \eqref{eq.gL} to obtain
$$g_\Lambda^{\singS+\gamma+2}\lesssim \tilde{s}^{\singS/2+\gamma/2+1}\lesssim (p'^0q^0)^{\singS/2+\gamma/2+1}.$$ 
Then we have
\begin{multline*}D_2
\lesssim  2^{k(\gamma+2)}\int_\rth\frac{dp'}{p'^0}\int_{\rth}\frac{dq}{q^0} \int_{k\approx 2^{-2k}\tilde{s}^{-1}l} \frac{dK}{l}(p'^0q^0)^{\singS/2+\gamma/2+1}
\\
\times \tilde{s}w^{2l}(p')|\eta(p')|^2|f(q)|^2(q^0)^{-m}
\\
\lesssim 2^{k\gamma}\int_\rth dp'\int_{\rth}dq\  (p'^0q^0)^{\singS/2+\gamma/2}w^{2l}(p')|\eta(p')|^2|f(q)|^2(q^0)^{-m}
.\end{multline*}
Thus we obtain 
$$D_2 \lesssim 2^{k\gamma} |f|^2_{L^2_{\frac{\singS+\gamma}{2}-m}} |w^l\eta|^2_{L^2_{\frac{\singS+\gamma}{2}}} ,$$
for any $m>0.$

Therefore, altogether for any $m>0$ we have
\begin{equation}\label{estimate.TI}
    |(T^{k,l}_{+,I}-T^{k,l}_{-,I})(f,h,\eta)|\lesssim 2^{k(\gamma-1)} |f|_{L^2_{\frac{\singS+\gamma}{2}-m}}\left|w^l|\nabla|h\right|_{L^2_\frac{\singS+\gamma}{2}} |w^l\eta|_{L^2_{\frac{\singS+\gamma}{2}}}.
\end{equation}
This completes our estimate of $T^{k,l}_{+,I}-T^{k,l}_{-,I}$.

{\bf Estimates on Part $II$}.  
 We now estimate the upper-bound of the term \eqref{define.Tkl2}
 nearby singularity when $k> 0$.   Inside \eqref{define.Tkl2} we define
  \begin{multline}\label{define.Tkl2.k}
   K_{II}(q,p')
   \eqdef 
 \sqrt{J(q)}
\int_{\mathbb{R}^2} \frac{dz}{\sqrt{|z|^2+1}}s_\Lambda\sigma_k(g_\Lambda,\theta_\Lambda)
\\\times
\frac{\tilde{s} \Phi(\tilde{g})\tilde{g}^4}{s_\Lambda \Phi(g_\Lambda)g^4_\Lambda}
            \left(\exp\left(-l(\sqrt{|z|^2+1}-1)+j z_1\right)-1\right),
 \end{multline}
To estimate the differerence, by the fundamental theorem of calculus we have
\begin{multline}\label{eq.difference.exp}
\exp\left(-l(\sqrt{|z|^2+1}-1)+jz_1\right)-1\\
=(-l(\sqrt{|z|^2+1}-1)+jz_1)\int_0^1 d\vartheta \ \exp\left(-\vartheta l(\sqrt{|z|^2+1}-1)+\vartheta jz_1\right)\\
=\left(-\frac{l|z|^2}{\sqrt{|z|^2+1}+1}+j|z|\cos\phi\right)\int_0^1 d\vartheta \ \exp\left(-\vartheta l(\sqrt{|z|^2+1}-1)+\vartheta jz_1\right), 
\end{multline} 
where we use \eqref{lj}. 
Now we will use the angular variable $\phi\in [0,2\pi)$ for $z$, with $z_1 = |z| \cos\phi$ and we recall the modified Bessel functions \eqref{bessel0}.  We will use the known Bessel function inequality $I_1(x)\le xI_0(x)$ for $x \ge 0$, we have that 
$$I_1(\vartheta j|z|)\le  \vartheta j|z|I_0(\vartheta j|z|).$$
In order to estimate \eqref{define.Tkl2.k} using \eqref{eq.difference.exp} we will split the integral $\int_0^1 d\vartheta$ into two as $\vartheta\in (0,1/2]$ and $\vartheta \in (1/2,1)$.  In particular we define 
  \begin{multline}\label{define.Tkl2.k1}
   K_{II}^1(q,p')
   \eqdef 
 \sqrt{J(q)}
\int_{\mathbb{R}^2} \frac{dz}{\sqrt{|z|^2+1}}s_\Lambda\sigma_k(g_\Lambda,\theta_\Lambda)
\frac{\tilde{s} \Phi(\tilde{g})\tilde{g}^4}{s_\Lambda \Phi(g_\Lambda)g^4_\Lambda}
\\\times
\left(-\frac{l|z|^2}{\sqrt{|z|^2+1}+1}+j|z|\cos\phi\right)\int_0^{1/2} d\vartheta \ \exp\left(-\vartheta l(\sqrt{|z|^2+1}-1)+\vartheta jz_1\right), 
 \end{multline}
 and
   \begin{multline}\label{define.Tkl2.k2}
   K_{II}^2(q,p')
   \eqdef 
 \sqrt{J(q)}
\int_{\mathbb{R}^2} \frac{dz}{\sqrt{|z|^2+1}}s_\Lambda\sigma_k(g_\Lambda,\theta_\Lambda)
\frac{\tilde{s} \Phi(\tilde{g})\tilde{g}^4}{s_\Lambda \Phi(g_\Lambda)g^4_\Lambda}
\\\times
\bigg(-\frac{l|z|^2}{\sqrt{|z|^2+1}+1}+j|z|\cos\phi\bigg)\int_{1/2}^1 d\vartheta  \exp\bigg(-\vartheta l(\sqrt{|z|^2+1}-1)+\vartheta jz_1\bigg)
 \end{multline}
Then we have $K_{II}(q,p') = K_{II}^1(q,p') + K_{II}^2(q,p')$.

{\it The first case with $\vartheta\in (0,1/2]$ in \eqref{define.Tkl2.k1}}.
Using \eqref{g2y.variable}, \eqref{newcos.intro}, and \eqref{eq.sigma0} we have that the kernel satisfies 
\begin{multline*}s_\Lambda\sigma(g_\Lambda,\theta_\Lambda)\frac{\tilde{s} \Phi(\tilde{g})\tilde{g}^4}{s_\Lambda \Phi(g_\Lambda)g^4_\Lambda}
=\tilde{s} \Phi(\tilde{g})\sigma_0(\theta_\Lambda) \frac{\tilde{g}^4}{g^4_\Lambda}\approx\tilde{s} \Phi(\tilde{g})\left(\sin(\theta_\Lambda)\right)^{-2-\gamma} \frac{\tilde{g}^4}{g^4_\Lambda}\\
\approx\tilde{s} \Phi(\tilde{g})\left(\frac{\tilde{s}}{g^2_\Lambda}(\sqrt{|z|^2+1}-1)\right)^{-1-\gamma/2} \frac{\tilde{g}^4}{g^4_\Lambda}\\
\approx \tilde{s}^{-\gamma/2} \Phi(\tilde{g})\tilde{g}^4\left(\sqrt{|z|^2+1}-1\right)^{-1-\gamma/2} g_\Lambda^{\gamma-2}\\
\approx\tilde{s}^{-\gamma/2} \Phi(\tilde{g})\tilde{g}^4\left(\sqrt{|z|^2+1}-1\right)^{-1-\gamma/2} \left( \tilde{g}^2+\frac{1}{2}\tilde{s}(\sqrt{|z|^2+1}-1)\right)^{\gamma/2-1}.
\end{multline*}
Therefore, since $\gamma<2$ and $\tilde{s}\gtrsim \tilde{g}^2$, we have $$ \left( \tilde{g}^2+\frac{1}{2}\tilde{s}(\sqrt{|z|^2+1}-1)\right)^{\gamma/2-1}\lesssim  \tilde{g}^{\gamma-2}\left( 1+\frac{1}{2}(\sqrt{|z|^2+1}-1)\right)^{\gamma/2-1}\lesssim \tilde{g}^{\gamma-2},$$ and we have 
\begin{multline}\label{ssigmafrac.bound}\left|s_\Lambda\sigma(g_\Lambda,\theta_\Lambda)\frac{\tilde{s} \Phi(\tilde{g})\tilde{g}^4}{s_\Lambda \Phi(g_\Lambda)g^4_\Lambda}\right|
\lesssim \tilde{s}^{-\gamma/2} \Phi(\tilde{g})\tilde{g}^{2+\gamma}\left((\sqrt{|z|^2+1}-1)\right)^{-1-\gamma/2} \\
\lesssim \tilde{s}^{-\gamma/2} \Phi(\tilde{g})\tilde{g}^{2+\gamma}|z|^{-2-\gamma} (1+|z|^2)^{1/2+\gamma/4}.
\end{multline}
We will use these kernel estimates to further estimate $K_{II}^1(q,p')$ and $K_{II}^2(q,p')$.

 Therefore, for $K_{II}^1(q,p')$ using \eqref{ssigmafrac.bound} and $|z|\le 1$ from \eqref{z.dyadic2}, we have
\begin{multline*}
  J^{-1/2}(q) \bigg| K_{II}^1(q,p')\bigg|
\\
\lesssim
\int_{(\sqrt{|z|^2+1}-1)\approx 2^{-2k}\tilde{s}^{-1}} \frac{dz}{\sqrt{|z|^2+1}}\tilde{s}^{-\gamma/2} \Phi(\tilde{g})\tilde{g}^{2+\gamma}|z|^{-2-\gamma} (1+|z|^2)^{1/2+\gamma/4}
\\\times \int_0^{\frac{1}{2}} d\vartheta \ \left(\frac{l|z|^2}{\sqrt{|z|^2+1}+1}+ j^2|z|^2\right)\exp\left(-\vartheta l(\sqrt{|z|^2+1}-1)+\vartheta jz_1\right)
\\
\lesssim
(l+j^2)\tilde{s}^{-\gamma/2} \Phi(\tilde{g})\tilde{g}^{2+\gamma}\int_0^{\frac{1}{2}} d\vartheta\\\times \int_{(\sqrt{|z|^2+1}-1)\approx 2^{-2k}\tilde{s}^{-1}} dz\ |z|^{-\gamma}\ \exp\left(-\vartheta l(\sqrt{|z|^2+1}-1)+\vartheta jz_1\right).
\end{multline*}
So that
\begin{multline*}
  J^{-1/2}(q) \bigg| K_{II}^1(q,p')\bigg|
\lesssim
(l+j^2)\tilde{s}^{-\gamma/2} \Phi(\tilde{g})\tilde{g}^{2+\gamma}\int_0^{\frac{1}{2}} d\vartheta\ \exp\left(\vartheta l\right)
\\
\times \int_{(\sqrt{|z|^2+1}-1)\approx 2^{-2k}\tilde{s}^{-1}} d|z|\ |z|^{1-\gamma}\ \exp\left(-\vartheta l\sqrt{|z|^2+1}\right)I_0(\vartheta j|z|)
\end{multline*}
Then finally using \eqref{z.dyadic2} we have
\begin{multline*}
  J^{-1/2}(q) \bigg| K_{II}^1(q,p')\bigg|
\lesssim 2^{k(\gamma-1)} \tilde{s}^{\frac{\gamma-1}{2}}(l+j^2)\tilde{s}^{-\gamma/2} \Phi(\tilde{g})\tilde{g}^{2+\gamma}
\\
\times\int_0^{\frac{1}{2}} d\vartheta\ \exp\left(\vartheta l\right)\int_{|z|\approx 2^{-k}\tilde{s}^{-1/2}} d|z|\  \exp\left(-\vartheta l\sqrt{|z|^2+1}\right)I_0(\vartheta j|z|),
\end{multline*}
Now we recall \eqref{bessel0} and \eqref{GS.estimate.exponential}.  We use those with \eqref{l2j2size} to obtain
\begin{multline*}
  \sqrt{J(q)}\exp(\vartheta l)   \exp\left(-\vartheta l\sqrt{|z|^2+1}\right)I_0(\vartheta j|z|)\\\lesssim  \exp\left(-\frac{q^0}{2}+\vartheta \frac{p'^0+q^0}{4}-\vartheta\frac{|p'-q|}{4}\right) \\
  \lesssim  \exp\left(-(1-\vartheta)\frac{q^0}{2}+\vartheta\bigg( \frac{p'^0-q^0}{4}-\frac{|p'-q|}{4}\bigg)\right)
  \lesssim  \exp\left(-(1-\vartheta)\frac{q^0}{2}\right),
\end{multline*}
by \eqref{FREQ:p0q0.le.pq}. We further have 
$$\int_0^{\frac{1}{2}} d\vartheta \exp\left(-(1-\vartheta)\frac{q^0}{2}\right)
= \frac{1}{q^0} \exp\left(-\frac{q^0}{4}\right).$$  Also note that 
\begin{equation}\label{eq.radial est}\int_{|z|\approx 2^{-k}\tilde{s}^{-1/2}} d|z|\ |z|^{m}\approx (2^{-k}\tilde{s}^{-1/2})^{m+1}. \end{equation}
Thus, collecting all of these estimates for \eqref{define.Tkl2.k1} we have
\begin{equation*}
\bigg| K_{II}^1(q,p')\bigg|
\lesssim 2^{k(\gamma-2)} (l+j^2)\tilde{s}^{-1} \Phi(\tilde{g})\tilde{g}^{2+\gamma}
\frac{1}{q^0} \exp\left(-\frac{q^0}{4}\right).
\end{equation*}
Then we can plug the estimate above into \eqref{define.Tkl2} on the region where $0 \le \vartheta \le 1/2$ using the convention \eqref{convention}, also using 
 $l,j^2\lesssim p'^0q^0$ and $\tilde{g}\lesssim \sqrt{\tilde{s}}$ from \eqref{l.upper.ineq} and \eqref{FREQ:s.ge.g2},
to obtain
\begin{multline}\notag
\bigg|(T^{k,l}_{+,II}-T^{k,l}_{-,II})(f,h,\eta)\bigg|_{\vartheta\le1/2}
\lesssim
2^{k(\gamma-2)}
\int_\rth\frac{dp'}{p'^0}\int_{\rth}\frac{dq}{q^0} 
\left| h(p')w^{2l}(p')\eta(p') f(q)\right|
\\
\times  \frac{\sqrt{\tilde{s}}}{\tilde{g}}
\frac{1}{q^0}\exp\left(-\frac{q^0}{4}\right)
(l+j^2)\tilde{s}^{-1} \Phi(\tilde{g})\tilde{g}^{2+\gamma}
\\
\lesssim 
2^{k(\gamma-2)}\int_\rth dp'\int_{\rth}dq  \left| h(p')w^{2l}(p')\eta(p')f(q) \right|  \tilde{g}^{\singS+\gamma}\frac{1}{q^0}\exp\left(-\frac{q^0}{4}\right). 
\end{multline}
We next use the Cauchy-Schwarz inequality to obtain that
\begin{equation} \notag
    \bigg|(T^{k,l}_{+,II}-T^{k,l}_{-,II})(f,h,\eta)\bigg|_{\vartheta\le1/2}
\lesssim 2^{k(\gamma-2)}I_1^{1/2} I_2^{1/2},
\end{equation}
where
\begin{equation} \notag
I_1
\eqdef
\int_\rth dp' |h(p')|^2w^{2l}(p')\int_{\rth}dq\  |f(q)|^2\tilde{g}^{\singS+\gamma+3/2}\frac{1}{q^0}(p'^0q^0)^{-3/4}\exp\left(-\frac{q^0}{4}\right),
\end{equation}
and
\begin{equation} \notag
I_2
\eqdef
\int_\rth dp' |\eta(p')|^2w^{2l}(p')\int_{\rth} dq\  \tilde{g}^{\singS+\gamma-3/2}\frac{1}{q^0}(p'^0q^0)^{3/4}\exp\left(-\frac{q^0}{4}\right).
\end{equation}
We will see below that it was important to add the term such as 
$\tilde{g}^{-3/2}(p'^0q^0)^{3/4}$ above.  The choice of the power $3/4$ here is sharp in the sense that this is the only possible value that makes both $I_1$ and $I_2$ above can be controlled.
We will now estimate both $I_1$ and $I_2$. 

For the estimate of $I_1,$ note that we have $\singS+\gamma+3/2\ge 0$ from \eqref{angassumption}-\eqref{singS.defin} and $$\tilde{g}^{\singS+\gamma+3/2}\lesssim (p'^0q^0)^{\singS/2+\gamma/2+3/4},$$
where we used \eqref{FREQ:g.le.sqrtpq}.
Then we have
$$\notag
I_1
\lesssim  |f|_{L^2_{-m}}^2|w^l h|_{L^2_{\frac{\singS+\gamma}{2}}}^2, $$ for any $m>0$.

For the estimate of $I_2$, if $\singS+\gamma-3/2<0,$ from \eqref{jutter.integral.est} we have that
\begin{multline*}
\int_{\rth}dq\  (p'^0q^0)^{3/4}\tilde{g}^{\singS+\gamma-3/2}\exp\left(-\frac{q^0}{4}\right)\\ \le\int_{\rth}dq\  (p'^0q^0)^{3/4}|p'-q|^{\singS+\gamma-3/2}(p'^0q^0)^{-\singS/2-\gamma/2+3/4}\exp\left(-\frac{q^0}{4}\right)
\approx (p'^0)^{\singS/2+\gamma/2}.\end{multline*}On the other hand, if $\singS+\gamma-3/2\ge 0$, then we have\begin{multline*}
\int_{\rth}dq\  (p'^0q^0)^{3/4}\tilde{g}^{\singS+\gamma-3/2}\exp\left(-\frac{q^0}{4}\right)\\ \lesssim \int_{\rth}dq\  (p'^0q^0)^{3/4}(p'^0q^0)^{\singS/2+\gamma/2-3/4}\exp\left(-\frac{q^0}{4}\right)
\approx (p'^0)^{\singS/2+\gamma/2}.\end{multline*}
Therefore, in general we have
$$I_2\lesssim  |w^l \eta|_{L^2_{\frac{\singS+\gamma}{2}}}^2.$$
Altogether, we conclude that we have
\begin{equation}\label{previous.small.theta}
    \bigg|(T^{k,l}_{+,II}-T^{k,l}_{-,II})(f,h,\eta)\bigg|_{\vartheta\le1/2}\lesssim  2^{k(\gamma-2)}|f|_{L^2_{-m}}|w^l h|_{L^2_{\frac{\singS+\gamma}{2}}}|w^l \eta|_{L^2_{\frac{\singS+\gamma}{2}}}.
\end{equation}
This completes our estimate of $T^{k,l}_{+,II}-T^{k,l}_{-,II}$ when $\vartheta\le1/2$.

{\it The other case with $\vartheta>\frac{1}{2}$}.
 In this case, we recall \eqref{define.Tkl2.k2} and then we have
\begin{multline*}
   J^{-1/2}(q)\bigg| K_{II}^2(q,p')\bigg|
            \lesssim \int_{\mathbb{R}^2} \frac{dz}{\sqrt{|z|^2+1}}s_\Lambda\sigma_k(g_\Lambda,\theta_\Lambda)\frac{\tilde{s} \Phi(\tilde{g})\tilde{g}^4}{s_\Lambda \Phi(g_\Lambda)g^4_\Lambda}
          \\
          \times \int_{\frac{1}{2}}^1 d\vartheta \ \left(\frac{l|z|^2}{\sqrt{|z|^2+1}+1}+ j^2|z|^2\right)\exp\left(-\vartheta l(\sqrt{|z|^2+1}-1)+\vartheta jz_1\right)
          \\
          \lesssim 
          (l+j^2)\int_{\frac{1}{2}}^1 d\vartheta 
          \int_{\mathbb{R}^2}
          \frac{dz}{\sqrt{|z|^2+1}}|z|^2 s_\Lambda\sigma_k(g_\Lambda,\theta_\Lambda)\frac{\tilde{s} \Phi(\tilde{g})\tilde{g}^4}{s_\Lambda \Phi(g_\Lambda)g^4_\Lambda}
          \\\times \exp\left(-\vartheta l(\sqrt{|z|^2+1}-1)+\vartheta jz_1\right)\\
          \lesssim(l+j^2) 
          \int_{\mathbb{R}^2}
          \frac{dz}{\sqrt{|z|^2+1}}|z|^2 s_\Lambda\sigma_k(g_\Lambda,\theta_\Lambda)\frac{\tilde{s} \Phi(\tilde{g})\tilde{g}^4}{s_\Lambda \Phi(g_\Lambda)g^4_\Lambda}
          \\\times \exp\left(- l(\sqrt{|z|^2+1}-1)+ jz_1\right)
          \\\times \int_{\frac{1}{2}}^1 d\vartheta\exp\left(-(\vartheta-1) l(\sqrt{|z|^2+1}-1)+(\vartheta-1) jz_1\right).
\end{multline*}
Plugging this into \eqref{define.Tkl2} on $1/2 < \vartheta \le 1$ using the convention \eqref{convention} we have
\begin{multline}\label{II.theta.large}
\bigg|(T^{k,l}_{+,II}-T^{k,l}_{-,II})(f,h,\eta)\bigg|_{\vartheta>1/2}
\lesssim
\int_\rth\frac{dp'}{p'^0} h(p')w^{2l}(p')\eta(p')\int_{\rth}\frac{dq}{q^0} \sqrt{J(q)}f(q)\frac{\sqrt{\tilde{s}}}{\tilde{g}}
\\
\times(l+j^2) \int_{\mathbb{R}^2}
\frac{dz}{\sqrt{|z|^2+1}}|z|^2 s_\Lambda\sigma_k(g_\Lambda,\theta_\Lambda)\frac{\tilde{s} \Phi(\tilde{g})\tilde{g}^4}{s_\Lambda \Phi(g_\Lambda)g^4_\Lambda}
          \\\times \exp\left(- l(\sqrt{|z|^2+1}-1)+ jz_1\right)
          \\\times \int_{\frac{1}{2}}^1 d\vartheta\exp\left(-(\vartheta-1) l(\sqrt{|z|^2+1}-1)+(\vartheta-1) jz_1\right)
   .
\end{multline}
Then by the Cauchy-Schwarz inequality, we have
\begin{equation}\notag
    \bigg|(T^{k,l}_{+,II}-T^{k,l}_{-,II})(f,h,\eta)\bigg|_{\vartheta>1/2}
\lesssim D_3^{1/2}\times D_4^{1/2},
\end{equation}
where for any $m>0$ we define 
\begin{multline}\notag
D_3\eqdef
\int_{\frac{1}{2}}^1 d\vartheta \int_\rth\frac{dp'}{p'^0} |h(p')|^2 w^{2l}(p')\int_{\rth}\frac{dq}{q^0} \sqrt{J(q)}\left(\frac{\sqrt{\tilde{s}}}{\tilde{g}}\right)^2(q^0)^{m} \left(\frac{p'^0q^0}{\tilde{s}}\right)^{\frac{7}{4}}
\\
\times
\int_{\mathbb{R}^2}
\frac{dz}{\sqrt{|z|^2+1}} s_\Lambda\sigma_k(g_\Lambda,\theta_\Lambda)\exp\left(- l(\sqrt{|z|^2+1}-1)+ jz_1\right), 
\end{multline}
and 
\begin{multline}\notag
D_4\eqdef
\int_\rth\frac{dp'}{p'^0} |\eta(p')|^2 w^{2l}(p')\int_{\rth}\frac{dq}{q^0}|f(q)|^2 \sqrt{J(q)} 
\\
\times (l+j^2)^2 
\int_{\mathbb{R}^2}
\frac{dz}{\sqrt{|z|^2+1}}|z|^4 s_\Lambda\sigma_k(g_\Lambda,\theta_\Lambda)\left(\frac{\tilde{s} \Phi(\tilde{g})\tilde{g}^4}{s_\Lambda \Phi(g_\Lambda)g^4_\Lambda}\right)^2\left(\frac{\tilde{s}}{p'^0q^0}\right)^{\frac{7}{4}}
          \\\times \int_{\frac{1}{2}}^1 d\vartheta\exp\left(-(2\vartheta-1) l(\sqrt{|z|^2+1}-1)+(2\vartheta-1) jz_1\right)(q^0)^{-m}.
\end{multline}The choice of the power $7/4$ here is sharp in the sense that this is the only possible value that makes both $D_3$ and $D_4$ above can be controlled.
We will now estimate the terms $D_3$ and $D_4$.

For the estimates of $D_3$, we notice that $D_3$ is of the form \eqref{second.eq.Ig}   and so it can also be written in the form \eqref{original.eq.Ig} using Lemma \ref{transformation.Lemma.appendix}.  Thus we obtain that up to an unimportant constant 
$D_3$ corresponds to
$$
D_3\approx \int_\rth dp \int_\rth dq \int_{\mathbb{S}^2}d\omega\ v_{\text{\o}}\frac{\sqrt{\tilde{s}}}{\tilde{g}} \left(\frac{p'^0q^0}{\tilde{s}}\right)^{\frac{7}{4}} \sigma(g,\theta)\chi_k(\bar{g})w^{2l}(p')J(q')^{1/2} |h(p')|^2(q^0)^{m}.$$  
By taking the pre-post change of variables $(p,q)\mapsto (p',q')$ as in \eqref{prepost.change}, and using Lemma \ref{lemmaqq}, we have
$$D_3\approx\int_\rth dp \int_\rth dq \int_{\mathbb{S}^2}d\omega\  v_{\text{\o}}\frac{\sqrt{\tilde{s}}}{\tilde{g}} \left(\frac{p^0q'^0}{\tilde{s}}\right)^{\frac{7}{4}} \sigma(g,\theta)\chi_k(\bar{g})w^{2l}(p)J(q)^{1/2} |h(p)|^2(q^0)^{m}.$$
Then using \eqref{Moller.def}, Lemma \ref{lemmaqq} and \eqref{gtildeg.equiv}, we further have
$$D_3\approx\int_\rth dp \int_\rth dq \int_{\mathbb{S}^2}d\omega\  \left(\frac{p^0q^0}{s}\right)^{\frac{3}{4}} \sigma(g,\theta)\chi_k(\bar{g})w^{2l}(p)J(q)^{1/2} |h(p)|^2(q^0)^{m}.$$
Next using \eqref{FREQ:s.ge.g2} and \eqref{Bk}, we have
\begin{multline}\label{D3.est.mid}D_3\lesssim 2^{k\gamma}\int_\rth dp\  w^{2l}(p) |h(p)|^2\int_\rth dq  \left(\frac{p^0q^0}{s}\right)^{\frac{3}{4}} J(q)^{1/2} g^{\singS+\gamma} (q^0)^{m}\\
\lesssim 2^{k\gamma}\int_\rth dp\  w^{2l}(p) |h(p)|^2\int_\rth dq  \left(p^0q^0\right)^{\frac{3}{4}} J(q)^{1/2} g^{\singS+\gamma-\frac{3}{2}} (q^0)^{m}.
\end{multline}
Now if $\singS+\gamma-\frac{3}{2}\ge 0,$ then we use \eqref{FREQ:g.le.sqrtpq} and obtain 
$$2^{k\gamma}\int_\rth dp\  w^{2l}(p) |h(p)|^2\int_\rth dq  \left(p^0q^0\right)^{\frac{3}{4}} J(q)^{1/2} (p^0q^0)^{\frac{\singS+\gamma}{2}-\frac{3}{4}} (q^0)^{m}\lesssim 2^{k\gamma} \left|w^lh\right|^2_{L^2_\frac{\singS+\gamma}{2}}.$$
If $-3<\singS+\gamma-\frac{3}{2}<0,$ then we use \eqref{FREQ:g.ge.lower} and obtain  
\begin{multline*}2^{k\gamma}\int_\rth dp\  w^{2l}(p) |h(p)|^2\int_\rth dq  \left(p^0q^0\right)^{\frac{3}{4}} J(q)^{1/2} \left(\frac{|p-q|}{\sqrt{p^0q^0}}\right)^{\singS+\gamma-\frac{3}{2}} (q^0)^{m}\\
\lesssim 2^{k\gamma} \left|w^lh\right|^2_{L^2_\frac{\singS+\gamma}{2}}.\end{multline*}
Therefore in either case we have
\begin{equation}\label{D3.est}D_3\lesssim 2^{k\gamma} \left|w^lh\right|^2_{L^2_\frac{\singS+\gamma}{2}}.
\end{equation}
This completes our estimate for the term $D_3$.

On the other hand, for the estimates of $D_4$, we note that
\begin{multline*}
   D_4\lesssim  \int_\rth\frac{dp'}{p'^0} |\eta(p')|^2 w^{2l}(p')\int_{\rth}\frac{dq}{q^0}|f(q)|^2 \sqrt{J(q)} (p'^0q^0)^2\left(\frac{\tilde{s}}{p'^0q^0}\right)^{\frac{7}{4}}\\
\times  \int_{(\sqrt{|z|^2+1}-1)\approx2^{-2k}\tilde{s}^{-1}}\frac{dz}{\sqrt{|z|^2+1}}|z|^4 s_\Lambda\sigma(g_\Lambda,\theta_\Lambda)\left(\frac{\tilde{s} \Phi(\tilde{g})\tilde{g}^4}{s_\Lambda \Phi(g_\Lambda)g^4_\Lambda}\right)
          \\\times \int_{\frac{1}{2}}^1 d\vartheta\exp\left(-(2\vartheta-1) l(\sqrt{|z|^2+1}-1)+(2\vartheta-1) jz_1\right)(q^0)^{-m},
\end{multline*}where we used \eqref{l.upper.ineq} and \eqref{frac sg.est}. 
Also, we note that $2\vartheta -1 >0$ if $\vartheta >1/2$. Then using \eqref{z.dyadic2} and \eqref{Jql2j2.bound} with $l$ and $j$ in \eqref{Jql2j2.bound} replaced by $(2\vartheta -1)l$ and $(2\vartheta -1)j$ and \eqref{FREQ:p0q0.le.pq} we have
\begin{multline*}
\sqrt{J(q)}\max_{0\le |z|\le 1}\exp\left(-(2\vartheta-1) l(\sqrt{|z|^2+1}-1)+(2\vartheta-1) jz_1\right)\\
\lesssim    (J(q))^{1/2}e^{(2\vartheta -1)l}\exp (-(2\vartheta -1)\sqrt{l^2-j^2})\\
\lesssim (J(q))^{1-\vartheta}\exp\bigg((2\vartheta-1)\left(-\frac{q^0}{2}+\frac{p'^0+q^0}{4}-\frac{|p'-q|}{4}\right) \bigg)\\
  \lesssim (J(q))^{1-\vartheta}  \exp\left((2\vartheta-1)\left( \frac{p'^0-q^0}{4}-\frac{|p'-q|}{4}\right)\right)\lesssim (J(q))^{1-\vartheta}.
\end{multline*}
Therefore, we have
\begin{multline*}
   D_4\lesssim  \int_\rth\frac{dp'}{p'^0} |\eta(p')|^2 w^{2l}(p')\int_{\rth}\frac{dq}{q^0}|f(q)|^2  (p'^0q^0)^2(q^0)^{-m-1}\left(\frac{\tilde{s}}{p'^0q^0}\right)^{\frac{7}{4}}\\
\times  \int_{(\sqrt{|z|^2+1}-1)\approx2^{-2k}\tilde{s}^{-1}}\frac{dz}{\sqrt{|z|^2+1}}|z|^4 \tilde{s}^{-\gamma/2} \Phi(\tilde{g})\tilde{g}^{2+\gamma}|z|^{-2-\gamma} (1+|z|^2)^{1/2+\gamma/4},
\end{multline*}
where we used \eqref{ssigmafrac.bound} and 
$$\int_{\frac{1}{2}}^1d\vartheta\ (J(q))^{1-\vartheta}\lesssim \frac{1}{q^0}.$$
Then by \eqref{z.dyadic2}, using \eqref{FREQ:s.ge.g2} and \eqref{eq.radial est} we further have
\begin{multline*}
   D_4\lesssim  \int_\rth\frac{dp'}{p'^0} |\eta(p')|^2 w^{2l}(p')\int_{\rth}\frac{dq}{q^0}|f(q)|^2  (p'^0q^0)^2(q^0)^{-m-1}\left(\frac{\tilde{s}}{p'^0q^0}\right)^{\frac{7}{4}}\\
\times  \int_{|z|\approx2^{-k}\tilde{s}^{-1/2}}d|z|\  |z|^{3-\gamma} \tilde{s}^{-\gamma/2} \tilde{g}^{2+\singS+\gamma}\\
\lesssim  2^{k(\gamma-4)}\int_\rth dp' |\eta(p')|^2 w^{2l}(p')\int_{\rth}dq|f(q)|^2  (q^0)^{-m-1}
 \tilde{g}^{\frac{3}{2}+\singS+\gamma}\left(p'^0q^0\right)^{-\frac{3}{4}}.
\end{multline*}
Then since $\frac{3}{2}+\singS+\gamma>0$, we have
$$\tilde{g}^{\frac{3}{2}+\singS+\gamma}\lesssim (p'^0q^0)^{\frac{\singS+\gamma}{2}+\frac{3}{4}},$$ by \eqref{FREQ:g.le.sqrtpq}.
Therefore, we have
$$D_4\lesssim 2^{k(\gamma-4)}|f|_{L^2_{-m'}}^2|w^l \eta|_{L^2_{\frac{\singS+\gamma}{2}}}^2,$$ 
for any $m'\ge 0.$
Thus, together with \eqref{D3.est}, we have $$\bigg|(T^{k,l}_{+,II}-T^{k,l}_{-,II})(f,h,\eta)\bigg|_{\vartheta>1/2}\lesssim  2^{k(\gamma-2)}|f|_{L^2_{-m'}}|w^l h|_{L^2_{\frac{\singS+\gamma}{2}}}|w^l \eta|_{L^2_{\frac{\singS+\gamma}{2}}}.$$
Then altogether, combining the estimate for $\vartheta>1/2$ with the one for $\vartheta\le 1/2$ in \eqref{previous.small.theta} for any $m'\ge 0$ we have
\begin{equation}\label{estimate.T2I}
    \left|(T^{k,l}_{+,II}-T^{k,l}_{-,II})(f,h,\eta)\right|\lesssim  2^{k(\gamma-2)}|f|_{L^2_{-m'}}|w^l h|_{L^2_{\frac{\singS+\gamma}{2}}}|w^l \eta|_{L^2_{\frac{\singS+\gamma}{2}}}.
\end{equation}
This completes the estimate for Part $II$.

{\bf Estimates on Part $III$}.
\label{sec.newcancel.part3}
Finally, we estimate the last part from \eqref{define.Tkl3}.   Note that in \eqref{define.Tkl3} $T^{k,l}_{+,III}-T^{k,l}_{-,III}$ is non-negative since we have \eqref{frac sg.est}. 
Then,
using \eqref{colfre4.2} and \eqref{z.dyadic} we first note that 
$$\left|1 -  \frac{\tilde{s} \Phi(\tilde{g})\tilde{g}^4}{s_\Lambda \Phi(g_\Lambda)g^4_\Lambda} \right|\le \left|1 -  \frac{\tilde{s} \Phi(\tilde{g})\tilde{g}^4}{s_\Lambda \Phi(g_\Lambda)g^4_\Lambda} \right|^{\frac{1}{2}}\lesssim \left(\frac{\tilde{s}(\sqrt{|z|^2+1}-1)}{g^2_\Lambda}\right)^\frac{1}{2}\approx \frac{2^{-k}}{g_\Lambda}.$$ 
Therefore,  we obtain
\begin{multline}\notag
(T^{k,l}_{+,III}-T^{k,l}_{-,III})
\lesssim 2^{-k}	\int_\rth\frac{dp'}{p'^0} h(p')w^{2l}(p')\eta(p')\int_{\rth}\frac{dq}{q^0} \sqrt{J(q)}f(q)\frac{\sqrt{\tilde{s}}}{\tilde{g}}\\\times \int_{(\sqrt{|z|^2+1}-1)\approx 2^{-2k}\tilde{s}^{-1}} \frac{dz}{\sqrt{|z|^2+1}}s_\Lambda\sigma(g_\Lambda,\theta_\Lambda) \exp\left(-l(\sqrt{|z|^2+1}-1)+jz_1\right) 
	\frac{1}{g_\Lambda}.
\end{multline}
Then using the Cauchy-Schwarz inequality we have
\begin{multline}\notag
(T^{k,l}_{+,III}-T^{k,l}_{-,III})
\\
\lesssim 2^{-k}	\bigg( \int_\rth\frac{dp'}{p'^0} |h(p')|^2 w^{2l}(p')\int_{\rth}\frac{dq}{q^0} \sqrt{J(q)}\left(\frac{\sqrt{\tilde{s}}}{\tilde{g}}\right)^2(q^0)^{m} \left(\frac{p'^0q^0}{\tilde{g}^2}\right)^{\frac{3}{4}}
\\\times\int_{(\sqrt{|z|^2+1}-1)\approx2^{-2k}\tilde{s}^{-1}}\frac{dz}{\sqrt{|z|^2+1}} s_\Lambda\sigma(g_\Lambda,\theta_\Lambda)\exp\left(- l(\sqrt{|z|^2+1}-1)+ jz_1\right) \bigg)^{1/2}\\
\times\bigg(\int_\rth\frac{dp'}{p'^0} |\eta(p')|^2 w^{2l}(p')\int_{\rth}\frac{dq}{q^0}|f(q)|^2 \sqrt{J(q)} \\
\times  \int_{(\sqrt{|z|^2+1}-1)\approx2^{-2k}\tilde{s}^{-1}}\frac{dz}{\sqrt{|z|^2+1}} s_\Lambda\sigma(g_\Lambda,\theta_\Lambda)\left(\frac{\tilde{g}^2}{p'^0q^0}\right)^{\frac{3}{4}}\frac{1}{g^2_\Lambda}
          \\\times \exp\left(- l(\sqrt{|z|^2+1}-1)+ jz_1\right)(q^0)^{-m}\bigg)^{1/2}
          =: 2^{-k}D_5^{1/2}\times D_6^{1/2}.
\end{multline}
We estimate $D_5$ similarly to how we estimated for $D_3$ just below \eqref{II.theta.large}.  Note that $D_5$ and $2D_3$ are the same except for the term $\left(\frac{p'^0q^0}{\tilde{s}}\right)^{\frac{7}{4}}$ in $D_3$ replaced by $\left(\frac{p'^0q^0}{\tilde{g}^2}\right)^{\frac{3}{4}}$ in $D_5$, since $\int_{\frac{1}{2}}^1 d\vartheta=\frac{1}{2}.$ By the same argument, $D_5$ has exactly the same upper-bound in \eqref{D3.est.mid} by observing that the additional term $\frac{\tilde{s}}{p'^0q^0}$  in $D_5$ satisfies $\frac{\tilde{s}}{p'^0q^0} \lesssim 1$ (by \eqref{FREQ:s.le.pq}) and that the term $\left(\frac{1}{s}\right)^{\frac{3}{4}}$ in \eqref{D3.est.mid} was treated as $\left(\frac{1}{s}\right)^{\frac{3}{4}}\lesssim \left(\frac{1}{g^2}\right)^{\frac{3}{4}}$. Therefore, as in \eqref{D3.est}, we have\begin{equation}\label{D5.est}
D_5\lesssim 2^{k\gamma} \left|w^lh\right|^2_{L^2_\frac{\singS+\gamma}{2}} .\end{equation}
This completes our estimate for the term $D_5$.

For the estimates of $D_6$, we use \eqref{z.dyadic} and \eqref{eq.ssigma} to observe that 
$$ s_\Lambda\sigma(g_\Lambda,\theta_\Lambda)\frac{1}{g^2_\Lambda}
	\lesssim g_\Lambda^{\singS+\gamma} \tilde{s}^{-\gamma/2} (\sqrt{|z|^2+1}-1)^{-1-\gamma/2}
	\approx 2^{k(\gamma+2)} g_\Lambda^{\singS+\gamma} \tilde{s}.
$$
Thus using \eqref{g2y.variable} we have
\begin{multline*}
    D_6\lesssim 2^{k(\gamma+2)}\int_\rth\frac{dp'}{p'^0} |\eta(p')|^2 w^{2l}(p')\int_{\rth}\frac{dq}{q^0}|f(q)|^2 \sqrt{J(q)} \\
\times  \int_{(\sqrt{|z|^2+1}-1)\approx2^{-2k}\tilde{s}^{-1}}\frac{dz}{\sqrt{|z|^2+1}}  g_\Lambda^{\singS+\gamma+\frac{3}{2}} \tilde{s}\left(\frac{1}{p'^0q^0}\right)^{\frac{3}{4}}
          \\\times \exp\left(- l(\sqrt{|z|^2+1}-1)+ jz_1\right)(q^0)^{-m}.
\end{multline*}
Note that $\singS+\gamma+\frac{3}{2}>0$ and hence
$$g_\Lambda^{\singS+\gamma+\frac{3}{2}} \lesssim \tilde{s}^{\frac{\singS+\gamma}{2}+\frac{3}{4}} ,$$ by \eqref{g2y.variable} and \eqref{z.dyadic2}.  Next using
\eqref{GS.estimate.exponential} and
\eqref{Jql2j2.bound}, we have
$$ \sqrt{J(q)}\max_{0\le |z|\le 1}\exp\left(- l(\sqrt{|z|^2+1}-1)+ jz_1\right) \lesssim 1.$$ Thus, using \eqref{FREQ:s.le.pq}, we conclude that 
\begin{multline*}
    D_6\lesssim 2^{k(\gamma+2)}\int_\rth\frac{dp'}{p'^0} |\eta(p')|^2 w^{2l}(p')\int_{\rth}\frac{dq}{q^0}|f(q)|^2  \\
\times  \int_{(\sqrt{|z|^2+1}-1)\approx2^{-2k}\tilde{s}^{-1}}
\frac{|z|d|z|}{\sqrt{|z|^2+1}}   \tilde{s}^{\frac{\singS+\gamma}{2}+\frac{7}{4}}\left(\frac{1}{p'^0q^0}\right)^{\frac{3}{4}}
         (q^0)^{-m}.
\end{multline*} 
Thus using \eqref{eq.radial est} we have 
\begin{multline*}
    D_6\lesssim    2^{k(\gamma+2)}\int_\rth\frac{dp'}{p'^0} |\eta(p')|^2 w^{2l}(p')\int_{\rth}\frac{dq}{q^0}|f(q)|^2  \\
\times  \int_{(\sqrt{|z|^2+1}-1)\approx2^{-2k}\tilde{s}^{-1}}d(\sqrt{|z|^2+1}-1)\  \tilde{s}^{\frac{\singS+\gamma}{2}+\frac{7}{4}}\left(\frac{1}{p'^0q^0}\right)^{\frac{3}{4}}
         (q^0)^{-m}  \\
         \approx 2^{k\gamma}\int_\rth\frac{dp'}{p'^0} |\eta(p')|^2 w^{2l}(p')\int_{\rth}\frac{dq}{q^0}|f(q)|^2
\tilde{s}^{\frac{\singS+\gamma}{2}+\frac{3}{4}}\left(\frac{1}{p'^0q^0}\right)^{\frac{3}{4}}
         (q^0)^{-m}\\
        \lesssim 2^{k\gamma}\int_\rth\frac{dp'}{p'^0} |\eta(p')|^2 w^{2l}(p')\int_{\rth}\frac{dq}{q^0}|f(q)|^2
(p'^0q^0)^{\frac{\singS+\gamma}{2}}
         (q^0)^{-m}.
\end{multline*} 
Since $m>0$ can be any number arbitrarily large, we have
$$D_6\lesssim 2^{k\gamma}|f|_{L^2_{-m'}}^2
|w^l \eta|_{L^2_{\frac{\singS+\gamma}{2}-1}}^2,$$ for any $m'\ge 0.$
Thus, together with \eqref{D5.est},  for any $m'\ge 0$ we have 
$$
\left| T^{k,l}_{+,III}-T^{k,l}_{-,III}\right|\lesssim  2^{k(\gamma-1)}|f|_{L^2_{-m'}}|w^l h|_{L^2_{\frac{\singS+\gamma}{2}}}|w^l \eta|_{L^2_{\frac{\singS+\gamma}{2}-1}},
$$
This estimate combined with \eqref{estimate.TI} and \eqref{estimate.T2I} thus completes the proof of Proposition \ref{prop:cancellation2}. 
\end{proof}

This concludes our cancellation estimates for the differences involving three arbitrary smooth functions.

\subsection{Additional estimates}\label{subsec.additional}
  We will also need estimates when we have a more specific Schwartz function satisfying
the following uniform estimate
\begin{equation}
\label{derivESTa}
(|\nabla| \phi)(p) \le C_\phi e^{-c\pZ},
\quad
C_\phi \ge 0, \quad 
c> 0.
\end{equation}

With this in mind, we have the next estimates:
\begin{proposition}  
\label{compactCANCELe}
We assume \eqref{derivESTa}, then
for any $k \geq 0$ we have
\begin{gather}
\notag
 \left| (T^{k,l}_{+}-T^{k,l}_{-})(g,h,\phi) \right| 
   \lesssim 
   C_\phi ~ 2^{(\gamma-1)k} 
    | g|_{L^2_{-m}} | h|_{L^2_{-m}} .
\end{gather}
The above inequality holds uniformly for any $m\ge 0$ and $l \in \mathbb{R}$.
\end{proposition}

\begin{proof}
We decompose the cancellation term as in \eqref{canceldecompp}, with terms $I$, $II$, and $III$, except that $\eta$ in \eqref{canceldecompp} is replaced by $\phi$.   Then $\phi$ satisfying \eqref{derivESTa} has rapid decay as does $J$ from \eqref{jutter.equilibrium}.    In particular, in this case all the terms $I$, $II$, and $III$ can now be estimated exactly as in
 the estimate from \eqref{partII}, also using the exponential decay as in \eqref{derivESTa} and \eqref{jutter.equilibrium}.  We thus obtain directly Proposition \ref{compactCANCELe}. 
\end{proof}

 \begin{proposition}\label{prop:GradRap}
Fix ${l} \in \mathbb{R}$.
Assume \eqref{derivESTa}.  We have the uniform estimates
\begin{equation}
 \left| T^{k,l}_{+}(f,\phi,\eta) \right| 
 \lesssim 
  C_\phi ~ 
 2^{\gamma k} ~ 
 | w^{l} f|_{L^2_{\frac{\singS+\gamma-1}{2}}}  | w^{l} \eta|_{L^2_{\frac{\singS+\gamma}{2}}},
 \label{tminushRAPplus}
\end{equation}
for any $k \leq 0$.  
Additionally for any $m\ge 0$ and any $k$ we obtain
\begin{align}
    |T^{k,l}_+(f,h,\phi)|+|T^{k,l}_-(f,h,\phi)|\lesssim
      C_\phi ~ 2^{k\gamma}|f|_{L^2_{-m}}|h|_{L^2_{-m}}.
 \label{tminushRAPplus2}
\end{align}

 \end{proposition}

 \begin{proof}  We first explain the proof of \eqref{tminushRAPplus2}.  
     If $\phi$ is as in \eqref{derivESTa}, then both $T^{k,l}_+(f,h,\phi)$ and $T^{k,l}_-(f,h,\phi)$ have rapid decay in both $p$ and $q$ variables in \eqref{T++} and \eqref{T+++}. By applying the Cauchy-Schwarz inequality to \eqref{T+++} and \eqref{T++} and using \eqref{Bk}, we obtain \eqref{tminushRAPplus2}.

We will now prove the upper bound of $\left|T^{k,l}_+(f,\phi,\eta)\right|$ as in \eqref{tminushRAPplus}.    We consider $I$ as in \eqref{T++I} with $\phi=h$.  By the Cauchy-Schwartz inequality we have
\begin{multline}\label{CauchyT+COMPACT}
I 
\lesssim
\left(
\int_{\rth}dp\int_{\rth}dq\int_{\mathbb{S}^2}d\omega\ v_{\text{\o}}\frac{g^{\singS}\sigma_0\chi_k(\bar{g})}{g^{\singS+\gamma}}
	|f(q)|^2|\phi(p)|^2\sqrt{J(q')}({p'^0})^{\frac{\singS+\gamma}{2}}w^{2l}(p')
	\right)^{1/2}
	\\
	 \times \left(\int_{\rth}dp\int_{\rth}dq\int_{\mathbb{S}^2}d\omega\ v_{\text{\o}} g^{\singS}\sigma_0\chi_k(\bar{g})g^{\singS+\gamma}|w^l\eta(p')|^2\sqrt{J(q')}({p'^0})^{\frac{-\singS-\gamma}{2}}\right)^{\frac{1}{2}}\\
	 =I_{1}\cdot I_{2}.
\end{multline}
First we will estimate $I_2$ in \eqref{CauchyT+COMPACT}. We have the following uniform estimate for $I_2$:
$$I_2\lesssim 2^{\frac{k\gamma}{2}}|w^l \eta|_{L^2_{\frac{\singS+\gamma}{2}}}.$$
This estimate is given in \eqref{I22} for the soft potential case \eqref{soft} and in \eqref{I22.hard} for the hard potential case \eqref{hard}.

For the estimate of $I_1$ in \eqref{CauchyT+COMPACT}, it suffices to show the following {\it claim}:
\begin{multline}
\label{claimc2}
\int_{\rth}dp\int_{\rth}dq\int_{\mathbb{S}^2}d\omega\ v_{\text{\o}}\frac{g^{\singS}\sigma_0\chi_k(\bar{g})}{g^{\singS+\gamma}}
	|f(q)|^2|\phi(p)|^2\sqrt{J(q')}({p'^0})^{\frac{\singS+\gamma}{2}}w^{2l}(p')\\
	\lesssim 2^{k\gamma}|w^lf|^2_{L^2_{\frac{\singS+\gamma-1}{2}}}.  
\end{multline}
Then together with \eqref{I22}, we obtain \eqref{C2} after summation of $S_1$ and $S_3$, since  $0<\gamma<1$.

The {\it claim} \eqref{claimc2} can be proved as follows. Let the left-hand side of \eqref{claimc2} be $I_L$. By recovering $dp'$ and $dq'$ measures from $d\omega$, we obtain that 
\begin{multline*}
I_L=\int_{\rth}\frac{dp}{p^0}\int_{\rth}\frac{dq}{q^0}\int_{\rth}\frac{dp'}{p'^0}\int_{\rth}\frac{dq'}{q'^0}\ s\frac{g^{\singS}\sigma_0\chi_k(\bar{g})}{g^{\singS+\gamma}}
	|f(q)|^2|\phi(p)|^2\\\times \sqrt{J(q')}({p'^0})^{\frac{\singS+\gamma}{2}}w^{2l}(p')\delta^{(4)}\left(p'^\mu+q'^\mu-p^\mu-q^\mu\right).
\end{multline*}
We make the change of variables $(p,q,p',q')\mapsto (q,p,q',p').$ Then we obtain
\begin{multline*}
I_L=\int_{\rth}\frac{dp}{p^0}\int_{\rth}\frac{dq}{q^0}\int_{\rth}\frac{dp'}{p'^0}\int_{\rth}\frac{dq'}{q'^0}\ s\frac{g^{\singS}\sigma_0\chi_k(\bar{g})}{g^{\singS+\gamma}}
	|f(p)|^2|\phi(q)|^2\\\times \sqrt{J(p')}({q'^0})^{\frac{\singS+\gamma}{2}}w^{2l}(q')\delta^{(4)}\left(p'^\mu+q'^\mu-p^\mu-q^\mu\right).
\end{multline*} Using the energy conservation law from the delta function, we obtain that 
$$
({q'^0})^{\frac{\singS+\gamma}{2}}w^{2l}(q')=({q'^0})^{\frac{\singS+\gamma}{2}+2l}=(p^0+q^0-p'^0)^{\frac{\singS+\gamma}{2}+2l}.
$$Now we use the relativistic Carleman representation, in Lemma \ref{lemma.carleman}, to reduce the integral as
\begin{multline*}
I_L\approx \int_{\rth}\frac{dp}{p^0}\int_{\rth}\frac{dp'}{p'^0}\int_{E^q_{p'-p}}\frac{d\pi_q}{q^0}\ \frac{s}{\bar{g}}\frac{g^{\singS}\sigma_0\chi_k(\bar{g})}{g^{\singS+\gamma}}
	|f(p)|^2|\phi(q)|^2\\\times \sqrt{J(p')}(p^0+q^0-p'^0)^{\frac{\singS+\gamma}{2}+2l},
\end{multline*} 
where
 $d\pi_{q}$ 
 is defined in \eqref{def.pi.meas}.  Using \eqref{angassumption}, \eqref{bargoverg}, \eqref{FREQ:s.le.pq}, and \eqref{FREQ:g.le.sqrtpq}, we have$$
\frac{s}{\bar{g}}\frac{g^{\singS}\sigma_0\chi_k(\bar{g})}{g^{\singS+\gamma}}\lesssim 2^{k\gamma}\frac{(p^0q^0)^2}{\bar{g}^{3}}\chi_k(\bar{g}).$$
We further have that $$(p^0+q^0-p'^0)^{\frac{\singS+\gamma}{2}+2l}\lesssim (p^0q^0p'^0)^{\frac{\singS+\gamma}{2}+2l}$$ and that 
$$|\phi(q)|^2\sqrt{J(p')}(q^0)^{\frac{\singS+\gamma}{2}+2l+2}(p'^0)^{\frac{\singS+\gamma}{2}+2l}\lesssim (J(p')J(q))^{\epsilon},$$ for some sufficiently small $\epsilon>0$ because $\phi(q)$ is the product of a polynomial in $q$ and $\sqrt{J(q)}$. 
Altogether we obtain
$$
I_L\lesssim 2^{k\gamma} \int_{\rth}dp\ |f(p)|^2(p^0)^{\frac{\singS+\gamma}{2}+2l+1}\int_{\rth}\frac{dp'}{p'^0}\frac{1}{\bar{g}^3}(J(p'))^{\epsilon}\int_{E^q_{p'-p}}\frac{d\pi_q}{q^0}\ (J(q))^{\epsilon}.
$$
Now we use \eqref{dpiq1} to obtain
$$
I_L\lesssim 2^{k\gamma} \int_{\rth}dp\ |f(p)|^2(p^0)^{\frac{\singS+\gamma}{2}+2l+1}\int_{\rth}\frac{dp'}{p'^0}\frac{1}{\bar{g}^3}(J(p'))^{\epsilon}.
$$Here we note that, if $k \le 0,$ we have $\bar{g}\approx 2^{-k} \gtrsim 1.$  Thus, $\bar{g}^{-3}\approx (\bar{g}^2+1)^{-3/2}.$ 
We further use that $\frac{1}{\bar{g}^3}\leq \frac{(p^0p'^0)^{3/2}}{|p-p'|^3}$ to obtain that 
$$\frac{1}{\bar{g}^3}\lesssim \left(\frac{|p-p'|^2}{p^0p'^0}+1\right)^{-3/2}.$$ Then if $|p|\le 1$ we have
$$(p^0)^{\frac{\singS+\gamma}{2}+2l+1}\int_{\rth}\frac{dp'}{p'^0}\frac{1}{\bar{g}^3}(J(p'))^{\epsilon}\lesssim (p^0)^{\frac{\singS+\gamma}{2}+2l-m}\int_{\rth}\frac{dp'}{p'^0}(J(p'))^{\epsilon}\lesssim (p^0)^{\frac{\singS+\gamma}{2}+2l-m}, $$ for any $m> 0.$ On the other hand, if $|p|>1,$ then we further split the region $p'\in \rth $ into $|p'|\le \frac{|p|}{2}$ and $|p'|> \frac{|p|}{2}$. If $|p'|\le \frac{|p|}{2}$, then we have
$$J(p')^{\epsilon}\lesssim (J(p')J(p))^{\epsilon'},$$ for some $0<\epsilon' <\epsilon.$ Thus, we have
$$(p^0)^{\frac{\singS+\gamma}{2}+2l+1}\int_{\rth}\frac{dp'}{p'^0}\frac{1}{\bar{g}^3}(J(p'))^{\epsilon}\lesssim (p^0)^{\frac{\singS+\gamma}{2}+2l+1}J(p)^{\epsilon'}\int_{\rth}\frac{dp'}{p'^0}(J(p'))^{\epsilon'}\lesssim (p^0)^{-m},$$ for any $m>0.$ Lastly, if $|p'|> \frac{|p|}{2}$, we have
$$|p-p'|\ge |p|-|p'|\ge \frac{|p|}{2}\gtrsim p^0,$$ which leads us to
$$\int_{\rth}\frac{dp'}{p'^0}\left(\frac{|p-p'|^2}{p^0p'^0}+1\right)^{-3/2}(J(p'))^{\epsilon}\lesssim (p^0)^{-3/2}\int_{\rth} (p'^0)^{1/2}(J(p'))^{\epsilon}\lesssim (p^0)^{-3/2}.$$
Therefore, we can conclude that 
$$
I_L\lesssim 2^{k\gamma} \int_{\rth}dp\ |f(p)|^2(p^0)^{\frac{\singS+\gamma}{2}+2l-\frac{1}{2}}=2^{k\gamma}\left|w^lf\right|^2_{L^2_{\frac{\singS+\gamma-1}{2}}}.
$$
This finishes the proof for the {\it claim} \eqref{claimc2}.
 \end{proof}

This concludes our discussion of the cancellation estimates nearby the angular singularities.  In the next sub-section we briefly introduce the standard 3-dimensional Littlewood-Paley theory which allows us to make sharp estimates of the linearized Boltzmann collision operator.

\subsection{Littlewood-Paley decompositions}
\label{LP decomp}
We now introduce the 3-dimensional Littlewood-Paley theory.  We will see in \eqref{LP} that the sum of weighted $L^2$-norm of the Littlewood-Paley pieces is bounded above by our fractional derivative norm from \eqref{fractional}. We  further bound the sum of the weighted $L^2$-norms of the derivatives of the Littlewood-Paley pieces above by the fractional derivative norm in \eqref{LPd1}.

We choose a real valued function  $\phi(p)\in C_c^\infty(\mathbb{R}^3_p)$ such that it satisfies $\phi(p)=1$ if $|p|\leq 1/2$ and $\phi(p)=0$ if $|p|\geq 1$. Also define $\psi(p) = \phi(p) - 2^{-3} \phi(p/2)$. Using the standard scaling, we further define
\begin{equation}\notag
\begin{split}
&\phi_j(p)=2^{3j}\phi(2^jp),  ~ j \ge 0,\\
&\psi_j(p)=2^{3j}\psi(2^jp), ~ j \ge 1.
\end{split}
\end{equation}
Now define the partial sum operator
$$
S_j(f)=f*\phi_j=\int_\rth 2^{3j}\phi(2^j(p-q))f(q)dq, \quad j \ge 0,
$$
and the difference operator
$$
\Delta_j(f)=f*\psi_j=\int_\rth 2^{3j}\psi(2^j(p-q))f(q)dq, \quad j \ge 0.
$$
For $j=0$ we define $\Delta_0(f)=S_0(f)$.  We suppose that $\int_\rth \phi(p) dp =1$, so that
\begin{equation}
\label{lpcan}
\begin{split}
\Delta_j(1)(p)=(1*\psi_j)(p)=\int_{\rth}\psi_j(q)dq=0.
\end{split}
\end{equation} 
Throughout this sub-section, the variables $p$ and $p'$ are independent vectors in $\rth$ and we will not assume the variables $p$ and $p'$ are related by the collision geometry as in \eqref{p'} and \eqref{q'}.

We further have for all sufficiently smooth $f$ as $l\rightarrow\infty$ that
$$
\sum_{j=0}^l \Delta_j(f)(p)
=
S_l(f)(p)\rightarrow f(p),
$$
and that 
$$
\Big(\int_\rth dp\ |S_l(f)(p)|^{r}(\pZ)^\rho\Big)^\frac{1}{r}\lesssim \Big(\int_\rth dp\ |f(p)|^{r}(\pZ)^\rho\Big)^\frac{1}{r},
$$
uniformly in $j\geq 0$ for any fixed $\rho\in \mathbb{R}$ and any $r\in[1,\infty].$ This $L^{r}$-boundedness property also holds for the operators $\Delta_j$.

We are now interested in estimating the upper bound for 
$$
\sum_{j=0}^\infty 2^{\gamma j}\int_\rth dp\ |\Delta_jf|^2(\pZ)^\rho,
$$
when $0<\gamma<1$ and $\rho\in\mathbb{R}$.  To this end, for any $j\geq1$, we have 
	\begin{align}
	\notag
	&\frac{1}{2}\int_\rth dp\int_\rth dp' \int_\rth dz\  (f(p)-f(p'))^2\psi_j(z-p)\psi_j(z-p')(z^0)^\rho
	\\
	\notag
	&=-\int_\rth dp\ (\Delta_jf(p))^2(\pZ)^\rho + \int_\rth dp\int_\rth dz\ (f(p))^2\psi_j(z-p)\Delta_j(1)(z)(z^0)^\rho 
		\\
	&=-\int_\rth dp\ (\Delta_jf(p))^2(\pZ)^\rho,
	\label{LP.reduction}
	\end{align}
which follows from $\Delta_j(1)(z)=0$ in \eqref{lpcan}.   Also $z^0 = \sqrt{1+|z|^2}$ etc.

	On the other hand, from the support condition for $\psi_j(z-p)\psi_j(z-p')$ on $z$, we have $\pZ\approx p'^0 \approx z^0$. Also notice that $|\psi_j(z-p')|\lesssim 2^{3j}$.  
	Thus, we obtain that
	\begin{multline*}
	\int_\rth dz\ |\psi_j(z-p')||\psi_j(z-p)|(z^0)^\rho
	\lesssim 
		(\pZ p'^0)^\frac{\rho}{2}\int_\rth dz\ |\psi_j(z-p')||\psi_j(z-p)|
		\\
	\lesssim 
	2^{3j} (\pZ p'^0)^\frac{\rho}{2}\int_\rth dz\ |\psi_j(z-p)|
	\lesssim  2^{3j}(\pZ p'^0)^\frac{\rho}{2}.
	\end{multline*}
Now $|\psi_j(z-p')||\psi_j(z-p)|$ is supported only when $|p-p'|\leq 2^{-j+1}$.
Hence
	$$
	2^{\gamma j}\int_\rth dz\ |\psi_j(z-p')||\psi_j(z-p)|(z^0)^\rho\lesssim 2^{(3+\gamma)j}(\pZ p'^0)^\frac{\rho}{2}1_{|p-p'|\leq 2^{-j+1}}.
	$$
	Since there exists $j_0>0$ such that $2^{-j_0}<|p-p'|\leq 2^{-j_0+1},$ we have 
	\begin{align*}
	\sum_{j=1}^\infty  2^{(3+\gamma)j}1_{|p-p'|\leq 2^{-j+1}}
	&= \sum_{j=1}^{j_0}  2^{(3+\gamma)j}1_{|p-p'|\leq 2^{-j+1}}
	\lesssim 2^{(3+\gamma)j_0}1_{|p-p'|\leq 1}\\
	&\lesssim \frac{1_{|p-p'|\leq 1}}{|p-p'|^{\gamma+3}}.
	\end{align*}
	When $j=0$, the term $\int_{\rth}dp\ |\Delta_0 f|^2(\pZ)^\rho$ is bounded above by $|f|^2_{L^2_\rho}$.
Combining these estimates we obtain that
\begin{multline*}
\left|
\frac{1}{2}	\sum_{j=0}^\infty 2^{\gamma j}\int_\rth dp\int_\rth dp' \int_\rth dz\  (f(p)-f(p'))^2\psi_j(z-p)\psi_j(z-p')(z^0)^\rho 
\right|
\\
\lesssim \int_\rth dp\int_\rth dp'\  (\pZ p'^0)^\frac{\rho}{2}\frac{(f(p)-f(p'))^2}{|p-p'|^{\gamma+3}}1_{|p-p'|\leq 1}.
\end{multline*}	
	Therefore using \eqref{LP.reduction} we have shown, for any $\gamma\in(0,1)$ and any $\rho\in\mathbb{R}$, that the following inequality holds:
	\begin{equation}
	\label{LP}
	\begin{split}
	\sum_{j=0}^\infty 2^{\gamma j}&\int_\rth dp\ |\Delta_jf|^2(\pZ)^\rho\\
	&\lesssim |f|^2_{L_\rho^2}+\int_\rth dp\int_\rth dp'\  (\pZ p'^0)^{\frac{\rho}{2}}\frac{(f(p)-f(p'))^2}{|p-p'|^{3+\gamma}}1_{|p-p'|\leq 1}
		\lesssim |f|^2_{I^{\rho,\gamma}  }.
	\end{split}
	\end{equation}
	This holds uniformly for any smooth function $f$.

To use the cancellation estimates obtained in \secref{Cancellation}, it is also necessary to obtain an analogous inequality to \eqref{LP} for the derivatives of the Littlewood-Paley pieces.  We need to establish a similar inequality when $\Delta_j$'s are replaced by $2^{-kj}\nabla \Delta_j$ where $\nabla$ is the standard 3-dimensional gradient.
We denote a derivative by
$
\nabla^\alpha =(\partial^{\alpha^1}_{p_1},\partial^{\alpha^2}_{p_2},\partial^{\alpha^3}_{p_3}). 
$
For any partial derivative $\frac{\partial}{\partial p_i}\Delta_j f$, it holds that $\frac{\partial}{\partial p_i}\Delta_j f= 2^j\tilde{\Delta}_j
f$ where $\tilde{\Delta}_j$
is the $j^{\text{th}}$-Littlewood-Paley cut-off operator associated to a new cut-off
function $\tilde{\psi}$ which also satisfies the cancellation property  (\ref{lpcan}) that $\tilde{\Delta}_j(1)(p)=0$. 
Thus, for a multi-index $\alpha$, we can write 
$$
2^{-|\alpha|j}\nabla^\alpha \Delta_j f(p)= \Delta^\alpha_j(f)(p)
$$ 
where $\Delta^{\alpha}_j$ is the cut-off operator associated to some $\psi^\alpha$.

Then, we can repeat the same proof as for \eqref{LP} by considering the following integral instead and make an upper-bound estimate on the weighted $L^2$-norm of the derivatives of each Littlewood-Paley decomposed piece:
\begin{multline*}
\frac{1}{2}\int_\rth dp\int_\rth dp'\ \int_\rth dz (f(p)-f(p'))^2\psi^\alpha_j(z-p)\psi^\alpha_j(z-p')(z^0)^\rho
\\
=
- \int_{\rth} dp |\Delta^\alpha_j(f)(p)|^2(p^0)^{\rho}.
\end{multline*}
This follows from  the same condition as in  (\ref{lpcan}) that $\Delta^\alpha_j(1)(p)=0$.  Similarly the analogous estimates can be multiplied by $2^{\gamma j}$ and summed over $j$ to get 
\begin{align*}
\sum_{j=0}^\infty 2^{\gamma j}&\int_\rth dp\ |\Delta^\alpha_j(f)(p)|^2({p^0})^\rho\\
&\lesssim |f|^2_{L_\rho^2}+\int_\rth dp\int_\rth dp'\  ({p^0}{p'^0})^{\frac{\rho}{2}}\frac{(f(p)-f(p'))^2}{|p-p'|^{3+\gamma}}1_{|p-p'|\leq 1}.
\end{align*}
Therefore, for any multi-index $\alpha$, for any fixed $\rho, l \in \mathbb{R}$, it follows that we have
\begin{equation}
\label{LPd1}
\sum_{j=0}^\infty 2^{(\gamma-|\alpha|)j}\int_{\rth}dp\ |\nabla^\alpha \Delta_j f(p)|^2({p'^0})^{\frac{\rho+\gamma}{2}}w^{2l}(p)\lesssim |f|^2_{I^{\rho,\gamma}_l  }.
\end{equation}
This holds  uniformly since $p^0\approx p'^0$ if $|p-p'|\le 1$ by Lemma \ref{lemmaqq}.

This concludes our introduction to the standard Littlewood-Paley theory. In
the next sub-section we will make our final upper-bound estimates of the linearized Boltzmann operator by utilizing the previous propositions.

\subsection{Upper bound estimates}\label{subsec.upperbd}
We begin with the proof of Theorem {\ref{thm1}}.
We first decompose
$$
h=\Delta_0h+\sum^\infty_{i=1}\Delta_ih\eqdef\sum_{i=0}^\infty h_i\text{ and }\eta=\Delta_0\eta+\sum^\infty_{j=1}\Delta_j\eta\eqdef\sum_{j=0}^\infty \eta_j.
$$
Also we consider the dyadic decomposition of the gain and the loss terms and write the trilinear product as
\begin{equation}\notag
\begin{split}
\langle w^{2l}\Gamma(f,h),\eta\rangle
&= \sum_{i=0}^\infty\sum_{j=0}^\infty \langle w^{2l}\Gamma(f,h_i),\eta_j\rangle\\
&=\sum_{i=1}^\infty\sum_{j=0}^\infty \langle w^{2l}\Gamma(f,h_{i+j}),\eta_j\rangle+\sum_{i=0}^\infty\sum_{j=0}^\infty \langle w^{2l}\Gamma(f,h_{j}),\eta_{i+j}\rangle.
\end{split}
\end{equation}
We further split this up using \eqref{T+++} as
\begin{equation}
\label{trilinearsum}
\begin{split}
\langle w^{2l}\Gamma(f,h),\eta\rangle
&=\sum_{k=-\infty}^\infty\sum_{i=1}^\infty\sum_{j=0}^\infty\{T_+^{k,l}(f,h_{i+j},\eta_j)-T_-^{k,l}(f,h_{i+j},\eta_j)\}\\&\ \quad +\sum_{k=-\infty}^\infty\sum_{i=0}^\infty\sum_{j=0}^\infty\{T_+^{k,l}(f,h_j,\eta_{i+j})-T_-^{k,l}(f,h_j,\eta_{i+j})\}.
\end{split}
\end{equation}
We first consider the sum over $k$ for fixed $i$ as the following
\begin{multline}
\label{196}
\sum_{k=-\infty}^\infty\sum_{j=0}^\infty\{T_+^{k,l}(f,h_{i+j} ,\eta_{j})-T_-^{k,l}(f,h_{i+j}  ,\eta_{j})\}\\
=\sum_{j=0}^\infty\sum_{k=-\infty}^{j}\{T_+^{k,l}(f,h_{i+j}  ,\eta_{j})-T_-^{k,l}(f,h_{i+j}  ,\eta_{j})\}\\
+\sum_{j=0}^\infty\sum_{k=j+1}^\infty\{T_+^{k,l}(f,h_{i+j}  ,\eta_{j})-T_-^{k,l}(f,h_{i+j}  ,\eta_{j})\}
\eqdef S_1+S_2.
\end{multline}
When $f, h, \eta$ are Schwartz functions, the order of summation may be rearranged because the sum will be seen to be absolutely convergent.
Then by $(\ref{T+})$ and $(\ref{T-})$, we obtain
\begin{align}\notag
|S_1|\lesssim& \sum_{j=0}^\infty 2^{\gamma j}\left|w^{l}f\right|_{L^2}\left|w^{l}h_{i+j}  \right|_{L^2_{\frac{\singS+\gamma}{2}}}\left|w^{l}\eta_{j}\right|_{L^2_{\frac{\singS+\gamma}{2}}}\\
\notag
\lesssim& \ 2^{-\frac{\gamma i}{2}} \left|w^{l}f\right|_{L^2}\left(\sum_{j=0}^\infty 2^{\gamma (i+j)} |w^{l}h_{i+j}|^2_{L^2_{\frac{\singS+\gamma}{2}}}\right)^{\frac{1}{2}}\left(\sum_{j=0}^\infty 2^{\gamma j} |w^{l}\eta_{j}|^2_{L^2_{\frac{\singS+\gamma}{2}}}\right)^{\frac{1}{2}}\\
\lesssim& \ 2^{-\frac{\gamma i}{2}}\left|w^{l}f\right|_{L^2}\left|h \right|_{I^{\singS,\gamma}_{l}}
|\eta|_{I^{\singS,\gamma}_{l}},
\label{S1.ineq}
\end{align}
where the second inequality is by the Cauchy-Schwarz inequality and the last inequality is by  $(\ref{LP})$.
Regarding the sum $S_2$, we use the cancellation estimates from \secref{Cancellation}. By (\ref{can+-}), we can  sum in $k$ from $k=j+1$ to $\infty$, since $\gamma-1<0$. Then we obtain
\begin{multline}\label{S2cancellationsummation}
|S_2|\lesssim \sum_{j=0}^{\infty}2^{(\gamma-1)j}|f|_{L^2_{-m}}\left|w^lh_{i+j} \right|_{L^2_{\frac{\singS+\gamma}{2}}}\left|w^l |\nabla| \eta_j \right|_{L^2_{\frac{\singS+\gamma}{2}}}\\
\lesssim 2^{\frac{-\gamma i}{2}} \left|f\right|_{L^2_{-m}}\left(\sum_{j=0}^\infty 2^{\gamma (i+j)} |w^{l}h_{i+j}|^2_{L^2_{\frac{\singS+\gamma}{2}}}\right)^{\frac{1}{2}}\left(\sum_{j=0}^\infty 2^{(\gamma-2) j} \left|w^{l}|\nabla|\eta_{j}\right|^2_{L^2_{\frac{\singS+\gamma}{2}}}\right)^{\frac{1}{2}}
\\
 \lesssim  2^{\frac{-\gamma i}{2}}\left|f\right|_{L^2_{-m}}\left|h \right|_{I^{\singS,\gamma}_{l}}
|\eta|_{I^{\singS,\gamma}_{l}},
\end{multline}
where the third inequality is by (\ref{LPd1}).  Finally, we take the sum in $i$ from $i=1$ to $\infty$ on both $S_1$ and $S_2$ to obtain that
\begin{equation}
	\label{mainup}
	\sum_{i=1}^\infty\sum_{j=0}^\infty \langle w^{2l}\Gamma(f,h_{i+j}),\eta_j\rangle\lesssim \left|w^{l}f\right|_{L^2}\left|h \right|_{I^{\singS,\gamma}_{l}}
	|\eta|_{I^{\singS,\gamma}_{l}}.
\end{equation}
This completes the estimate for the first term in \eqref{trilinearsum}.

Now we move onto estimating the second part of \eqref{trilinearsum}.  We consider the sum over $k$ of the terms $\langle w^{2l}\Gamma(f,h_j),\eta_{i+j}\rangle$ for fixed $i$ as 
\begin{multline}
	\label{1962}
	\sum_{k=-\infty}^\infty\sum_{j=0}^\infty\{T_+^{k,l}(f,h_{j} ,\eta_{i+j})-T_-^{k,l}(f,h_{j} ,\eta_{i+j})\}\\
	=\sum_{j=0}^\infty\sum_{k=-\infty}^{j}\{T_+^{k,l}(f,h_{j} ,\eta_{i+j})-T_-^{k,l}(f,h_{j} ,\eta_{i+j})\}\\
	+\sum_{j=0}^\infty\sum_{k=j+1}^\infty\{T_+^{k,l}(f,h_{j} ,\eta_{i+j})-T_-^{k,l}(f,h_{j} ,\eta_{i+j})\}
	\eqdef S_3+S_4.
\end{multline}
Then by $(\ref{T+})$ and $(\ref{T-})$, we obtain
\begin{align}\notag
	|S_3|\lesssim& \sum_{j=0}^\infty 2^{\gamma j}\left|w^{l}f\right|_{L^2}\left|w^{l}h_{j}  \right|_{L^2_{\frac{\singS+\gamma}{2}}}\left|w^{l}\eta_{i+j}\right|_{L^2_{\frac{\singS+\gamma}{2}}}\\
	\notag
	\lesssim& \ 2^{-\frac{\gamma i}{2}} \left|w^{l}f\right|_{L^2}\left(\sum_{j=0}^\infty 2^{\gamma j} |w^{l}h_{j}|^2_{L^2_{\frac{\singS+\gamma}{2}}}\right)^{\frac{1}{2}}\left(\sum_{j=0}^\infty 2^{\gamma (i+j)} |w^{l}\eta_{i+j}|^2_{L^2_{\frac{\singS+\gamma}{2}}}\right)^{\frac{1}{2}},
\notag
\end{align}
where the second inequality is just the Cauchy-Schwarz inequality.  Then we use   \eqref{LP} to obtain
\begin{equation}
    	|S_3|
	\lesssim \ 2^{-\frac{\gamma i}{2}}\left|w^{l}f\right|_{L^2}\left|h \right|_{I^{\singS,\gamma}_{l}}
	|\eta|_{I^{\singS,\gamma}_{l}}.
	\label{S3.ineq}
\end{equation}
Regarding the sum $S_4$, we use an alternative cancellation estimate from \secref{Cancellation}. By Proposition \ref{prop:cancellation2}, we can conduct the sum in $k$ from $k=j+1$ to $\infty$, since $\gamma-1<0$. Then we obtain
\begin{multline}\label{S2cancellationsummation2}
	|S_4|\lesssim \sum_{j=0}^{\infty}2^{(\gamma-1)j}|f|_{L^2_{-m}}\left|w^l |\nabla| h_j \right|_{L^2_{\frac{\singS+\gamma}{2}}}\left|w^l\eta_{i+j} \right|_{L^2_{\frac{\singS+\gamma}{2}}}\\
	\lesssim 2^{\frac{-\gamma i}{2}} \left|f\right|_{L^2_{-m}}\left(\sum_{j=0}^\infty 2^{(\gamma-2) j} \left|w^{l}|\nabla|h_{j}\right|^2_{L^2_{\frac{\singS+\gamma}{2}}}\right)^{\frac{1}{2}}
	\left(\sum_{j=0}^\infty 2^{\gamma (i+j)} |w^{l}\eta_{i+j}|^2_{L^2_{\frac{\singS+\gamma}{2}}}\right)^{\frac{1}{2}}
\\
	\lesssim  2^{\frac{-\gamma i}{2}}\left|f\right|_{L^2_{-m}}\left|h \right|_{I^{\singS,\gamma}_{l}}
	|\eta|_{I^{\singS,\gamma}_{l}},
\end{multline}
where the third inequality is by (\ref{LPd1}).  Finally, we take the sum in $i$ from $i=0$ to $\infty$ on both $S_3$ and $S_4$ and obtain that
\begin{equation}
	\label{mainup2}
\sum_{i=0}^\infty\sum_{j=0}^\infty \langle w^{2l}\Gamma(f,h_{j}),\eta_{i+j}\rangle\lesssim \left|w^{l}f\right|_{L^2}\left|h \right|_{I^{\singS,\gamma}_{l}}
	|\eta|_{I^{\singS,\gamma}_{l}}.
\end{equation}
Then \eqref{mainup} and \eqref{mainup2} with \eqref{trilinearsum} completes the proof of Theorem \ref{thm1} as a special case when $l=0$.

Lemma \ref{Lemma1} can also be proven as follows.  We first take a spatial derivative $\partial^\alpha$ on the non-linear collision operator $\Gamma$ with $|\alpha|\leq N$ for some $N\geq 2$ to observe \eqref{gamma.deriv}.  We multiply $\partial^\alpha\eta$ on both sides of \eqref{gamma.deriv} and integrate with respect to $x$ and $p$. Then by \eqref{mainup} and \eqref{mainup2}, we obtain 
$$
\left|\left( w^{2l}\Gamma(\partial^{\alpha-\alpha'}f,\partial^{\alpha'}h),\partial^\alpha\eta\right)\right|\lesssim \|w^lf\|_{H^N_xL^2_v}\|h\|_{H^N_xI_l^{\singS,\gamma}}\|\partial^\alpha\eta\|_{I_l^{\singS,\gamma}}.
$$   This proves Lemma \ref{Lemma1} for the both hard \eqref{hard} and soft-interactions \eqref{soft}.

Furthermore, we would like to mention a proposition that contains other useful compact estimates.

\begin{proposition}\label{prop:basisEst}
	Let $\phi$ be a function satisfying \eqref{derivESTa}.
	Then we have that
	\begin{equation}
	\label{C1}
	|\langle w^{2l} \Gamma(\phi,h),\eta\rangle |\lesssim |h|_{I^{\singS,\gamma}_l}|\eta|_{I^{\singS,\gamma}_l},
	\end{equation}
	Further, for any fixed $\epsilon>0$ small enough, we have that 
	\begin{equation}
	\label{C2}
	|\langle  w^{2l}\Gamma(f, \phi),h\rangle |\lesssim |w^lf|_{L^2_{\frac{\singS}{2}-\epsilon}}|h|_{I^{\singS,\gamma}_l}.
	\end{equation}
	Additionally, for any $m\geq 0$, we have
	\begin{equation}
	\label{C3}
	|\langle w^{2l} \Gamma(f, h),\phi\rangle |\lesssim |f|_{L^2_{-m}}|h|_{L^2_{-m}}.
	\end{equation}
\end{proposition}

\begin{proof}
	First of all, we note that \eqref{trilinearsum} with \eqref{mainup} and \eqref{mainup2} immediately imply (\ref{C1}) because $|w^l \phi|_{L^2}$ is bounded.

We now prove \eqref{C2}. In this case we also use the splitting \eqref{trilinearsum} with \eqref{196} and \eqref{1962}.  For the terms $S_2$ and $S_4$, we follow \eqref{S2cancellationsummation} and \eqref{S2cancellationsummation2} to obtain $$|S_2|,|S_4|\lesssim 2^{\frac{-\gamma i}{2}}|f|_{L^2_{-m}}|h|_{I^{\singS,\gamma}_l}.$$ 
Here we used that $|\phi|_{I^{\singS,\gamma}_l}\lesssim 1$.

For the upper-bound estimate of $|S_1|$ and $|S_3|$, we use the upper bounds of $\left|T^{k,l}_+\right|(f, \phi,h)$ and $\left|T^{k,l}_-\right|(f, \phi,h)$. 
We further split $S_1$  as 
$S_1 = S_1^+ - S_1^{-}$
where from \eqref{196} we have 
\begin{equation}\notag
    S_1^{\pm}
    \eqdef
    \sum_{j=0}^\infty\sum_{k=-\infty}^{j}T_{\pm}^{k,l}(f,\phi_{i+j}  ,h_{j}).
\end{equation}
We similarly split $S_3 = S_3^+ - S_3^{-}$
where from \eqref{1962} we have 
\begin{equation}\notag
    S_3^{\pm}
    \eqdef
    \sum_{j=0}^\infty\sum_{k=-\infty}^{j}T_{\pm}^{k,l}(f,\phi_{j}  ,h_{i+j}).
\end{equation}
For the upper bound of $\left|T^{k,l}_-(f, \phi,h)\right|$ we use \eqref{T-} and obtain that, similar to \eqref{S1.ineq} and \eqref{S3.ineq}, $S_1^{-}$ and $S_3^{-}$ are bounded as
\begin{equation}\notag
    \left| S_1^{-} \right|+\left| S_3^{-} \right|
    \lesssim
2^{-\frac{\gamma i}{2}} |f|_{L^2_{-m}}|w^l h|_{L^2_{\frac{\singS+\gamma}{2}}}.
\end{equation}
Note that $|w^l\phi|_{L^2_{\frac{\singS+\gamma}{2}}} \lesssim 1$  as before. Now we further use \eqref{tminushRAPplus} to obtain that
\begin{equation}\notag
    \left| S_1^{+} \right|+\left| S_3^{+} \right|
    \lesssim
2^{-\frac{\gamma i}{2}}  | w^{l} f|_{L^2_{\frac{\singS+\gamma-1}{2}}} 
|w^l h|_{L^2_{\frac{\singS+\gamma}{2}}}.
\end{equation}
This yields \eqref{C2} since $\epsilon>0$ can be chosen sufficiently small such that $\gamma -1 < -2\epsilon$.

	For (\ref{C3}), we write the trilinear form as the sum
    \begin{align*}
\langle  w^{2l}\Gamma(f,h),\phi\rangle
&=\sum_{k=0}^\infty (T^{k,l}_+-T^{k,l}_-)(f,h,\phi)+\sum_{k=-\infty}^{-1} (T^{k,l}_+-T^{k,l}_-)(f,h,\phi)\\
&\eqdef S_5+S_6.
\end{align*}
Note that here $\phi$ has rapid decay using (\ref{derivESTa}).
We obtain the upper bound for $S_6$ from \eqref{tminushRAPplus2} and summing over $k\le -1$ since $\gamma>0$.  Our estimate for $S_5$ follows from Proposition \ref{compactCANCELe} because $\gamma-1<0$. This completes the proof.  \end{proof}

Note that (\ref{C1}) implies Lemma \ref{Nfupperboundlemma} together with the frequency multiplier asymptotics \eqref{Paos}. Also, Proposition \ref{prop:basisEst} further implies the following lemma:

\begin{lemma} We have the uniform estimate
	\begin{equation}
    \label{compactest}
	|\langle w^{2l} \mathcal{K}f,h\rangle | \lesssim |w^{l}f|_{L^2_{\frac{\singS}{2}}}|  h|_{I^{\singS,\gamma}_l}.
	\end{equation}
	\end{lemma}

From \eqref{defK},
this lemma is a direct consequence of \eqref{C2} and the estimate on $|\zeta_\mathcal{K}(p)|$ by choosing $\varepsilon \in (0, \gamma/4)$ in \eqref{Paos}.  Note that \eqref{compactest} implies Lemma \ref{Lemma2} by letting $h=f$. More precisely, we see that for any small $\epsilon>0$, the upper bound of \eqref{compactest} is bounded above by
 $$
| w^{l} f|_{L^2_{\frac{\singS}{2}}}|  f|_{I^{\singS,\gamma}_l} \leq \frac{\epsilon}{2}|  f|^2_{I^{\singS,\gamma}_l}+C_\epsilon| w^{l} f|^2_{L^2_{\frac{\singS}{2}}}.
$$
For the term $C_\epsilon|w^{l}f|^2_{L^2_{\frac{\singS}{2}}}$, we split the region into $|p|\leq R$ and $ |p|\geq R$. When $|p|\leq R$, the term is bounded above by $C_\epsilon|f|^2_{L^2({B_R})}$. Outside the ball, we choose $R>0$ sufficiently large enough so that $C_\epsilon R^{-\frac{\gamma}{2}}\leq \frac{\epsilon}{2}$.
Then we obtain Lemma \ref{Lemma2}.

This concludes our discussion of upper-bound estimates for the linearized Boltzmann operator. In
the next section we will make coercive lower-bound estimates.

\section{Main coercive estimates}
\label{main coercive estimates}
In this section, for any Schwartz function $f$, we consider the quadratic difference arising in the inner product of the norm part $\mathcal{N}f$ with $f$. Our goal is to prove Lemma \ref{Nfcoercivitylemma}.  The key point  is to estimate the norm  $|f|^2_{\mathtt{B}}$ which arises in the inner product $\langle \mathcal{N}f,f\rangle$ from \eqref{mainPartNorm.Nf} 
and will be defined as follows:
\begin{align*}
|f|^2_{\mathtt{B}}&\eqdef\frac{1}{2}\int_\rth dp\int_\rth dq\ \int_{\mathbb{S}^2} d\omega ~ v_{\text{\o}} \sigma(g,\theta)(f(p')-f(p))^2\sqrt{J(q)J(q')}.
\end{align*}
\subsection{Pointwise estimates}
The norm $|f|^2_{\mathtt{B}}$ can be further estimated using Lemma \ref{lemma.reduction} as
\begin{align*}
|f|^2_{\mathtt{B}} =&\  \frac{1}{2}\int_{\mathbb{R}^3}\frac{dp}{{p^0}}\int_{\mathbb{R}^3}\frac{dp'\ }{p'^0}(f(p')-f(p))^2 \\
& \hspace{5mm}\times\int_{\rth}\frac{dq}{q^0}\int_{\rth}\frac{dq'}{q'^0}\  \delta^{(4)}(p'^\mu+q'^\mu-p^\mu-q^\mu) s\sigma(g,\theta)e^{-\frac{q^0+q'^0}{2}}\\
  \approx &\  \frac{1}{2}\int_{\mathbb{R}^3}\frac{dp}{{p^0}}\int_{\mathbb{R}^3}\frac{dp'\ }{p'^0}(f(p')-f(p))^2 \\
 &\hspace{5mm}\times\int_{\rth}\frac{dq}{q^0}\int_{\rth}\frac{dq'}{q'^0}\  \delta^{(4)}(p'^\mu+q'^\mu-p^\mu-q^\mu) sg^{\singS}\left(\frac{\bar{g}}{g}\right)^{-2-\gamma}e^{-\frac{q^0+q'^0}{2}},
\end{align*} 
where we used \eqref{soft} and \eqref{angassumption} and \eqref{bargoverg}.  We conclude that
\begin{align*}
|f|^2_{\mathtt{B}} 
 \approx&\ \frac{1}{2}\int_{\mathbb{R}^3}\frac{dp}{{p^0}}\int_{\mathbb{R}^3}\frac{dp'\ }{p'^0}\frac{(f(p')-f(p))^2}{\bar{g}^{3+\gamma}} \\
 &\hspace{5mm}\times\int_{\rth}\frac{dq}{q^0}\int_{\rth}\frac{dq'}{q'^0}\  \delta^{(4)}(p'^\mu+q'^\mu-p^\mu-q^\mu) sg^{\singS+\gamma+2}\bar{g}e^{-\frac{q^0+q'^0}{2}}.
\end{align*} 
 Then we define a kernel $K(p,p')$ as follows:
 \begin{equation}\notag
 	K(p,p')\eqdef \int_{\rth}\frac{dq}{q^0}\int_{\rth}\frac{dq'}{q'^0}\  \delta^{(4)}(p'^\mu+q'^\mu-p^\mu-q^\mu) sg^{\singS+\gamma+2}\bar{g}e^{-\frac{q^0+q'^0}{2}}.
 \end{equation}
 Thus we have 
 \begin{equation}\label{BK}
 	|f|^2_{\mathtt{B}}\approx \frac{1}{2}\int_{\mathbb{R}^3}\frac{dp}{{p^0}}\int_{\mathbb{R}^3}\frac{dp'\ }{p'^0}\frac{(f(p')-f(p))^2}{\bar{g}^{3+\gamma}}K(p,p').
 \end{equation}
We will prove a pointwise lower bound for the kernel $K(p,p')$ as follows.
\begin{proposition}
	\label{coer}
	The kernel $K(p,p')$ is bounded from below as 
	 $$
	K(p,p')\gtrsim1_{\bar{g}\leq 1}1_{|p'^0-p^0|\leq \bar{g}}(p^0p'^0)^{\frac{\singS+\gamma}{4}+1}.
	$$
\end{proposition}

\begin{proof} We use \eqref{FREQ:s.ge.g2}, \eqref{FREQ:g.ge.lower}, and $\singS+\gamma+4>0$ for both hard- and soft-interactions to obtain that 
 $$
 sg^{\singS+\gamma+2}\geq g^{\singS+\gamma+4}\geq \frac{|p-q|^{\singS+\gamma+4}}{(p^0q^0)^{\frac{\singS+\gamma+4}{2}}}.
 $$
 Also, note that $g^2\ge  \tilde{g}^2$ by \eqref{triangle.id}. Therefore, using \eqref{gtilde} and  \eqref{eq.pqp'q'} we obtain
 $$
sg^{\singS+\gamma+2}\geq g^{\singS+\gamma+4} \geq \tilde{g}^{\singS+\gamma+4}\geq \frac{|p-q'|^{\singS+\gamma+4}}{(p^0q'^0)^{\frac{\singS+\gamma+4}{2}}}.
 $$
 With the extra assumption that $\bar{g}\leq 1$, using Lemma \ref{lemmaqq}, we obtain $q^0\approx q'^0\approx q^0+q'^0.$ Thus, we have
\begin{equation}
\label{q+q'}
sg^{\singS+\gamma+2}\gtrsim \frac{|p-q|^{\singS+\gamma+4}}{(p^0q^0)^{\frac{\singS+\gamma+4}{2}}}+\frac{|p-q'|^{\singS+\gamma+4}}{(p^0q'^0)^{\frac{\singS+\gamma+4}{2}}}\gtrsim \frac{|2p-(q+q')|^{\singS+\gamma+4}}{(p^0(q^0+q'^0))^{\frac{\singS+\gamma+4}{2}}} .
\end{equation}
 We then raise the kernel $K(p,p')$ to the 8-dimensional integral in $q^\mu$ and $q'^\mu$ using Lemma \ref{7.5ofCMP} with $dq^\mu$ and $dq'^\mu$ as below:
 \begin{multline*}
 	K(p,p')=\frac{1}{16}\int_{\rfo}dq^\mu\int_{\rfo}dq'^\mu\  \delta^{(4)}(p'^\mu+q'^\mu-p^\mu-q^\mu)sg^{\singS+\gamma+2}\bar{g}e^{-\frac{q^0+q'^0}{2}}\\\times u(q^0+q'^0)u(\dbar{s}-4)\delta(\dbar{s}-\dbar{g}^2-4)\delta((q^\mu+q'^\mu)(q_\mu-q'_\mu)),
 \end{multline*}
 where $\dbar{s}\eqdef s(q^\mu,q'^\mu)$ and $\dbar{g}\eqdef g(q^\mu,q'^\mu)$. 
 By \eqref{q+q'}, we have
 \begin{multline*}
 	K(p,p')\gtrsim \int_{\rfo}dq^\mu\int_{\rfo}dq'^\mu\  \delta^{(4)}(p'^\mu+q'^\mu-p^\mu-q^\mu)\frac{|2p-(q+q')|^{\singS+\gamma+4}}{(p^0(q^0+q'^0))^{\frac{\singS+\gamma+4}{2}}} \bar{g}e^{-\frac{q^0+q'^0}{2}}\\\times u(q^0+q'^0)u(\dbar{s}-4)\delta(\dbar{s}-\dbar{g}^2-4)\delta((q^\mu+q'^\mu)(q_\mu-q'_\mu))1_{\{\bar{g}\leq 1 \}}.
 \end{multline*}
 Now apply the change of variables
 $$\bar{p}^\mu=q^\mu+q'^\mu,\ \bar{q}^\mu=q^\mu-q'^\mu.$$
 This transformation has Jacobian equal to 16.
 With this change of variable the integral becomes
  \begin{multline*}
 K(p,p')\gtrsim\int_{\rfo}d\bar{q}^\mu\int_{\rfo}d\bar{p}^\mu\  \delta^{(4)}(p'^\mu-p^\mu-\bar{q}^\mu)  \frac{|2p-\bar{p}|^{\singS+\gamma+4}}{(p^0\bar{p}^0)^{\frac{\singS+\gamma+4}{2}}}\bar{g}e^{-\frac{\bar{p}^0}{2}}\\\times u(\bar{p}^0)u(-\bar{p}^\mu \bar{p}_\mu-4)\delta(-\bar{p}^\mu \bar{p}_\mu-\bar{q}^\mu \bar{q}_\mu-4)\delta(\bar{p}^\mu \bar{q}_\mu)1_{\{\bar{g}\leq 1 \}}.
 \end{multline*} 
 Note that $\bar{g}$ did not change under this change of variables.
 We next carry out the delta function argument for $\delta^{(4)}(p'^\mu-p^\mu-\bar{q}^\mu)$ to obtain
 \begin{multline*}
 K(p,p')\gtrsim \int_{\rfo}d\bar{p}^\mu\  \frac{|2p-\bar{p}|^{\singS+\gamma+4}}{(p^0\bar{p}^0)^{\frac{\singS+\gamma+4}{2}}}1_{\{\bar{g}\leq 1 \}}\bar{g}e^{-\frac{\bar{p}^0}{2}}\\\times u(\bar{p}^0)u(-\bar{p}^\mu \bar{p}_\mu-4)\delta(-\bar{p}^\mu \bar{p}_\mu-\bar{g}^2-4)\delta(\bar{p}^\mu (p'_\mu-p_\mu)).
 \end{multline*}
 Since $\bar{s}=\bar{g}^2+4,$ using \eqref{tilde.calc.LT} with $\tilde{g}$ replaced by $\bar{g}$ we have 
 \begin{equation*}
 u(\bar{p}^0)\delta(-\bar{p}^\mu \bar{p}_\mu-\bar{g}^2-4)=u(\bar{p}^0)\delta(-\bar{p}^\mu \bar{p}_\mu-\bar{s})
 =\frac{\delta(\bar{p}^0-\sqrt{|\bar{p}|^2+\bar{s}})}{2\sqrt{|\bar{p}|^2+\bar{s}}}.
 \end{equation*}
Again using $\bar{s}=\bar{g}^2+4$, we have $-\bar{p}^\mu\bar{p}_\mu-4=\bar{s}-4=\bar{g}^2\geq 0$ to guarantee that $u(-\bar{p}^\mu\bar{p}_\mu-4)=1$ by \eqref{def.u}.  We evaluate one integral using the delta function:
\begin{equation}\notag
K(p,p')\gtrsim\int_{\rth}\frac{d\bar{p}}{\bar{p}^0}\  \frac{|2p-\bar{p}|^{\singS+\gamma+4}}{(p^0\bar{p}^0)^{\frac{\singS+\gamma+4}{2}}}1_{\{\bar{g}\leq 1 \}}\bar{g}\delta(\bar{p}^\mu (p'_\mu-p_\mu))e^{-\frac{\bar{p}^0}{2}},
\end{equation}
where $\bar{p}^0=\sqrt{|\bar{p}|^2+\bar{s}}$.

Now express $\bar{p}$ using polar coordinates $\bar{p}\mapsto (r,\theta,\phi)$. We further choose the $z$-axis parallel to $p'-p$ such that the angle between $\bar{p}$ and $p'-p$ is equal to $\phi$. Note that the Jacobian that we get from this change of variables is equal to $r^2 \sin \phi$. Then the terms in the delta function can be written as
$$
\bar{p}^\mu (p'_\mu-p_\mu)=-\bar{p}^0(p'^0-p^0)+\bar{p}\cdot(p'-p)=-\sqrt{r^2+\bar{s}}(p'^0-p^0)+r|p'-p|\cos\phi.
$$
By using $|2p-\bar{p}|\geq ||2p|-|\bar{p}||$, we have
\begin{multline*}
K(p,p')\gtrsim\int_{0}^\infty \frac{dr}{\sqrt{r^2+\bar{s}}} \int_0^{2\pi} d\theta \int_0^{\pi} d\phi \ r^2\sin\phi
\ \delta\left(\cos\phi -\frac{\sqrt{r^2+\bar{s}}(p'^0-p^0)}{r|p'-p|}\right)\\
\times \frac{\bar{g}}{r|p'-p|}  \frac{|2|p|-r|^{\singS+\gamma+4}}{(p^0\sqrt{r^2+\bar{s}})^{\frac{\singS+\gamma+4}{2}}}1_{\{\bar{g}\leq 1 \}}e^{-\frac{\sqrt{r^2+\bar{s}}}{2}},
\end{multline*}
We define the set $S_{p,p'}$ as
 $$S_{p,p'}\eqdef \left\{r\in [0,\infty): \left|\frac{\sqrt{r^2+\bar{s}}(p'^0-p^0)}{r|p'-p|}\right|\leq 1\right\}.
 $$ 
 By a simple calculation, we can see that this set is equivalent to 
 $$S_{p,p'}=\left\{r\in [0,\infty): r^2\geq \frac{\bar{s}|p'^0-p^0|^2}{\bar{g}^2}\right\}.$$ 
Then by carrying out one integration with the delta function, and the change of variables $u=\cos\phi$, we get
$$
K(p,p')\gtrsim\int_{S_{p,p'}} \frac{rdr}{\sqrt{r^2+\bar{s}}} 
 \frac{\bar{g}}{|p'-p|}  \frac{|2|p|-r|^{\singS+\gamma+4}}{(p^0\sqrt{r^2+\bar{s}})^{\frac{\singS+\gamma+4}{2}}}e^{-\frac{\sqrt{r^2+\bar{s}}}{2}}1_{\{\bar{g}\leq 1 \}}.
$$
Now we will make the extra assumption that 
\begin{equation}\label{extra.assumption}
    |p'^0-p^0|\leq \bar{g}.
\end{equation}
We point out that this assumption \eqref{extra.assumption} is necessary in the following sense:  There is a region in $(p,p')$ space where $\bar{g}\le 1$ and $|p'^0-p^0|$ can be arbitrarily large, which then makes $\frac{\bar{s}|p'^0-p^0|^2}{\bar{g}^2}$ become similarly large.  This would prevent the lower bound in this proposition due to the exponential decay in $e^{-\frac{\sqrt{r^2+\bar{s}}}{2}}$.

This assumption, $|p'^0-p^0|\leq \bar{g}$, further implies that 
$$
|p'-p|^2=\bar{g}^2+|p'^0-p^0|^2\leq 2\bar{g}^2,
$$
and
$$
\frac{\bar{s}|p'^0-p^0|^2}{\bar{g}^2}\leq \bar{s}=\bar{g}^2+4\leq 5.$$
Therefore, we have
$$
K(p,p')\gtrsim\int_{5}^{\infty} \frac{rdr}{\sqrt{r^2+\bar{s}}} 
  \frac{|2|p|-r|^{\singS+\gamma+4}}{(p^0\sqrt{r^2+\bar{s}})^{\frac{\singS+\gamma+4}{2}}}e^{-\frac{\sqrt{r^2+\bar{s}}}{2}}1_{\{\bar{g}\leq 1 \}}1_{\{|p'^0-p^0|\leq \bar{g} \}}.
$$
Using $\bar{s}\leq 5,$ $\frac{r}{\sqrt{r^2+5}}>\frac{1}{2},$ and $p^0\approx p'^0$, we can further have that 
\begin{multline*}
K(p,p')\gtrsim1_{\{\bar{g}\leq 1 \}}1_{\{|p'^0-p^0|\leq \bar{g} \}}\int_{5}^{\infty} dr\ 
  \frac{|2|p|-r|^{\singS+\gamma+4}}{(p^0\sqrt{r^2+5})^{\frac{\singS+\gamma+4}{2}}}e^{-\frac{\sqrt{r^2+5}}{2}}\\
  \approx 1_{\{\bar{g}\leq 1 \}}1_{\{|p'^0-p^0|\leq \bar{g} \}} (p^0)^{\frac{\singS+\gamma+4}{2}}
  \approx 1_{\{\bar{g}\leq 1 \}}1_{\{|p'^0-p^0|\leq \bar{g} \}} (p^0p'^0)^{\frac{\singS+\gamma}{4}+1}.
\end{multline*}
This completes the proof.
\end{proof}

Together with \eqref{BK}, Proposition \ref{coer} implies that 
\begin{align*}
|f|^2_{\mathtt{B}}&\gtrsim  \int_\rth dp\int_\rth dp'\ \frac{(f(p')-f(p))^2}{\bar{g}^{3+\gamma}}({p'^0}{p^0})^{\frac{\singS+\gamma}{4}} 1_{\bar{g}\leq 1 }1_{|p'^0-p^0|\leq \bar{g} }.
\end{align*}
Since $\bar{g}^2=-|p'^0-p^0|^2+|p'-p|^2,$ the cutoff function $1_{|p'^0-p^0|\leq \bar{g} }$ is equal to $1_{|p'^0-p^0|\leq \frac{1}{\sqrt{2}}|p'-p| }$. Since $\bar{g}\leq |p'-p|$,  we can conclude that 
\begin{multline}
\label{coercivelowerbound}
|f|^2_{\mathtt{B}}\gtrsim  
\int_\rth dp\int_\rth dp'\ \frac{(f(p')-f(p))^2}{|p'-p|^{3+\gamma}}({p'^0}{p^0})^{\frac{\singS+\gamma}{4}} 1_{|p'-p|\leq 1 }1_{|p'^0-p^0|\leq \frac{1}{\sqrt{2}}|p'-p| }.
\end{multline}
However, this pointwise lower bound is not a sufficient coercivity estimate because of the extra cut-off restricting $1_{|p'^0-p^0|\leq \frac{1}{\sqrt{2}}|p'-p| }.$ Recall that we want to show 
 \begin{equation}
 \label{wanttoshowcoer}
|f|^2_{L^2_{\frac{\singS+\gamma}{2}}}+|f|^2_{\mathtt{B}}\gtrsim |f|^2_{I^{\singS,\gamma}},
 \end{equation}
in order to obtain Lemma \ref{Nfcoercivitylemma} for the case $l=0$ using also \eqref{Paos}. 

A direct pointwise comparison is not possible for this desirable coercivity because of the extra cut-off restricting $1_{|p'^0-p^0|\leq \frac{1}{\sqrt{2}}|p'-p| }$. In the next section, we will use the technique called ``Fourier redistribution'' from \cite{MR2784329} to get around this obstruction. 

\subsection{Fourier redistribution}\label{9.2}Essentially the key idea is to take an advantage from the Fourier transform in the situation where the pointwise bound is not available.
More precisely, we use the following proposition of \cite{MR2784329}:
\begin{proposition}[Proposition 7.1 of \cite{MR2784329}]\label{7.1}
Suppose $K_1$ and $K_2$ are even, nonnegative, measurable functions on $\rth$ satisfying
$$\int_\rth du \ K_l(u)|u|^2<\infty, \ \ \ l=1,2.$$
Suppose $\phi$ is any smooth, non-negative function on $\rth$ and that there is some constant $C_\phi$ such that $|\nabla^2\phi(u)|\leq C_\phi$ for all $u$. For $l=1,2$, consider the following quadratic forms (defined for arbitrary real-valued Schwartz functions $f$):$$|f|^2_{K_l}\eqdef \int_\rth dp \int_\rth dp'\  \phi(p)\phi(p') K_l(p-p')(f(p)-f(p'))^2.$$
If there exists a finite, non-negative constant $C$ such that, for all $\zeta\in \rth$, 
$$\int_\rth du \ K_1(u)|e^{2\pi i \langle \zeta,u\rangle}-1|^2\leq C+\int_\rth du \ K_2(u)|e^{2\pi i \langle \zeta,u\rangle}-1|^2,$$ then for all Schwartz functions $f$, 
$$|f|^2_{K_1}\leq |f|^2_{K_2}+C'C_\phi\int_\rth dp \ \phi(p)(f(p))^2,$$
where the constant $C'$ satisfies $C'\lesssim 1+C+\int_\rth du\ (K_1(u)+K_2(u))|u|^2$ uniformly in $K_1,$ $K_2$, $\phi$ and C.
\end{proposition}
In order to obtain the favorable coercivity estimates, we fix functions $K_1$ and $K_2$ on $\rfo$ given by $$K_1(u^0,u)=|u|^{-3-\gamma}1_{|u|\leq 1}$$ 
and 
$$K_2(u^0,u)=|u|^{-3-\gamma}1_{|u|\leq 1}1_{|u^{0}|\leq \epsilon|u|},$$ 
with $u\eqdef(u_1,u_2,u_3)\in\rth$, $u^0 \in \mathbb{R}$
and  
$|u|^2 = u_1^2+u_2^2+u_3^2$.
Note that the particular choice of $u^0=p'^0-p^0$, $u=p'-p$ and $\epsilon=\frac{1}{\sqrt{2}}$ corresponds to the coercive lower bound \eqref{coercivelowerbound} with $\phi(p)\eqdef (\sqrt{1+|p|^2})^{\frac{\singS+\gamma}{4}}\tilde{\phi}(p^\mu)$ where $\{\tilde{\phi}(p^\mu)\}$ is a smooth partition of unity in $\rfo$ which will be defined just below.

In order to use Proposition \ref{7.1}, we need to establish that the estimates of the derivatives of $\phi$ and the Fourier condition on $K_1$ and $K_2$. Let $\{\tilde{\phi}\}$ be a smooth partition of unity of $\rfo$ which is locally finite and such that their zeroth, first, and second derivatives are uniformly bounded. Then the estimates on the derivatives of $\phi(p)\eqdef (\sqrt{1+|p|^2})^{\frac{\singS+\gamma}{4}}\tilde{\phi}(p^\mu)$ are straight forward.

In order to prove the Fourier condition, we suppose further that each $\tilde{\phi}$ is supported only on a Euclidean ball of radius $\frac{\epsilon}{16}$ for a small $\epsilon>0$. Also, we consider a fixed $v\in \rth$ such that $\tilde{\phi}(v^\mu)\neq 0$ with $v^\mu\eqdef (\sqrt{1+|v|^2},v)$ for some fixed $\tilde{\phi}$. We then write $p=v+u$ and $p'=v+u'$. If $\tilde{\phi}(p)$ and $\tilde{\phi}(p')$ are both not equal to zero, then both $|u|,|u'|\leq \frac{\epsilon}{8}$ by the support condition of $\tilde{\phi}$. Then, we have
\begin{multline}\label{p'pu'u}
	|p'^0-p^0|=|\sqrt{1+|p'|^2}-\sqrt{1+|p'|^2}|=\frac{||p'|^2-|p|^2|}{\sqrt{1+|p|^2}+\sqrt{1+|p'|^2}}\\=\frac{|(p'-p)\cdot(p+p')|}{\sqrt{1+|p|^2}+\sqrt{1+|p'|^2}}=\frac{|(u'-u)\cdot(2v+u+u')|}{\sqrt{1+|v+u|^2}+\sqrt{1+|v+u'|^2}}\\
\leq 	\frac{|2v\cdot(u'-u)|}{\sqrt{1+|v+u|^2}+\sqrt{1+|v+u'|^2}}+\frac{|(u'-u)\cdot(u+u')|}{\sqrt{1+|v+u|^2}+\sqrt{1+|v+u'|^2}}\\
\leq	\frac{|2v\cdot(u'-u)|}{\sqrt{1+|v+u|^2}+\sqrt{1+|v+u'|^2}}+\frac{|u+u'|}{2}|u-u'|.
\end{multline}
For the upper-bound estimate of the first fraction, we split into two cases: $|v|\leq \frac{\epsilon}{4}$ and $|v|>\frac{\epsilon}{4}$. In the former case, we have
$$\frac{|2v\cdot(u'-u)|}{\sqrt{1+|v+u|^2}+\sqrt{1+|v+u'|^2}}\leq |v||u'-u|\leq \frac{\epsilon}{4}|u'-u|.$$ In the latter case, we have both $|v+u'|, |v+u|\geq \frac{|v|}{2}$ because $|u'|,|u|\leq \frac{\epsilon}{8}<\frac{|v|}{2}.$ Thus, we have
$$\frac{|2v\cdot(u'-u)|}{\sqrt{1+|v+u|^2}+\sqrt{1+|v+u'|^2}}\leq \frac{2|v\cdot(u'-u)|}{\sqrt{1+|v|^2}}.$$ In \eqref{p'pu'u}, we also recall that $\frac{|u+u'|}{2}\leq \frac{\epsilon}{4}$. 
Then in \eqref{p'pu'u}, the additional condition that $\frac{|v\cdot(u'-u)|}{\sqrt{1+|v|^2}}\leq \frac{\epsilon}{8}|u'-u|$ further guarantees that 
$$|p'^0-p^0|\leq \frac{\epsilon}{2}|u'-u|.$$
Now the inequality above holds for any $|v|$.  
Therefore, we also have that 
$$
|u-u'|^{-3-\gamma}1_{|u-u'|\leq 1}1_{\frac{|v\cdot(u'-u)|}{\sqrt{1+|v|^2}}\leq \frac{\epsilon}{8}|u'-u|}\tilde{\phi}(p'^\mu)\tilde{\phi}(p^\mu)\lesssim K_2(p'^\mu-p^\mu)\tilde{\phi}(p'^\mu)\tilde{\phi}(p^\mu),
$$
where $p'^\mu=(\sqrt{1+|v+u'|^2},v+u')$ and $p^\mu=(\sqrt{1+|v+u|^2},v+u)$.  Then for fixed $v\ne 0$,  we define
$
E_2 = \{u: |u|\leq \frac{1}{2}, \quad 
\frac{|v\cdot u|}{\sqrt{1+|v|^2}}\leq \frac{\epsilon}{8}|u|\}
$
and 
we can choose a coordinate system such that this is the set $E_2$ in  Proposition \ref{prop.compare} below.

\begin{proposition}[Proposition 7.2 of \cite{MR2784329}]\label{prop.compare}
Fix any $\varepsilon>0$, and let $E_1$ and $E_2$ be the sets in $\rth$ given by $E_1\eqdef\{u\in \rth : |u|\leq 2\}$ and $E_2\eqdef \{u\in \rth : |u|\leq \frac{1}{2} \ \text{and}\ |u_3|\leq \varepsilon|u|\} $. Then $$\int_{E_1}du|e^{2\pi i\langle \zeta,u\rangle}-1|^2|u|^{-3-\gamma}\lesssim 1+\int_{E_2}du|e^{2\pi i\langle \zeta,u\rangle}-1|^2|u|^{-3-\gamma},$$ uniformly for all $\zeta\in \rth$. 
\end{proposition}

Thus Proposition \ref{prop.compare} verifies the Fourier condition that is required in Proposition \ref{7.1}. Thus we have
\begin{multline*}
\int_\rth dp \int_\rth dp' \ \tilde{\phi}(p)\tilde{\phi}(p')(p^0p'^0)^{\frac{\singS+\gamma}{4}} K_1(p^\mu-p'^\mu)(f(p)-f(p'))^2\\
\lesssim \int_\rth dp \int_\rth dp'\  \tilde{\phi}(p)\tilde{\phi}(p')(p^0p'^0)^{\frac{\singS+\gamma}{4}} K_2(p^\mu-p'^\mu)(f(p)-f(p'))^2\\+\int_\rth dp\  (p^0)^{\frac{\singS+\gamma}{4}}\tilde{\phi}(p)(f(p))^2.
\end{multline*} 
Then by summing over the partition and using Proposition \ref{coer} we obtain 
\begin{multline*}
\int_\rth dp \int_\rth dp' \
(p^0p'^0)^{\frac{\singS+\gamma}{4}} K_1(p^\mu-p'^\mu)(f(p)-f(p'))^2
\sum_{\tilde{\phi}} \tilde{\phi}(p)\tilde{\phi}(p')
\\
\lesssim 
|f|^2_{\mathtt{B}}
+
|f|^2_{L^2_{\frac{\singS+\gamma}{2}}}.
\end{multline*} 
The above holds since $\tilde{\phi}(p)$ is chosen to have compact support in a small ball so the integral $$\int_\rth dp \ (p^0)^{\frac{\singS+\gamma}{4}}\tilde{\phi}(p)(f(p))^2$$ is not higher order than $|f|^2_{L^2_{\frac{\singS+\gamma}{2}}}$.

Now let $0<M<\infty$ be the maximal number of partition elements that can be non-zero at any specific point.  Then for any $p$, there must be an element of the partition such that we have $\tilde{\phi}(p) \ge \frac{1}{M}$.  Then since the partition of unity was chosen in such a way that we have uniform bounds on the first derivatives, then we can choose a radius $\delta>0$ such that $\tilde{\phi}(q) \ge \frac{1}{2M}$ for any $q$ in the ball of center $p$ and radius $\delta$.  If $0<\epsilon< \frac{1}{2}\delta$ then we then have that
\begin{equation}\notag
  \int_\rth dp \int_\rth dp' \
(p^0p'^0)^{\frac{\singS+\gamma}{4}} \frac{(f(p)-f(p'))^2}{|p-p'|^{3+\gamma}}
1_{|p-p'|\leq \epsilon}
\lesssim 
|f|^2_{\mathtt{B}}
+
|f|^2_{L^2_{\frac{\singS+\gamma}{2}}}.  
\end{equation}
Note that the integral of the lower bound above over $\epsilon< |p-p'|\leq 1$ is clearly bounded by the upper bound above.  Thus we conclude
$$
|f|^2_{L^2_{\frac{\singS+\gamma}{2}}}+|f|^2_{\mathtt{B}}\gtrsim |f|^2_{I^{\singS,\gamma}},
$$  This completes the proof for our main coercive inequality stated in  Lemma \ref{Nfcoercivitylemma} for the case $l=0$ using the Fourier redistribution.

\subsection{Coercivity with extra weights $w^{l}$ for $l>0$.}
Notice that Lemma \ref{Nfcoercivitylemma} for the case $l=0$ has been proven above. Now we will prove Lemma \ref{Nfcoercivitylemma} when $l>0$. 

\begin{proof}[Proof of Lemma \ref{Nfcoercivitylemma}]
	Suppose $l\neq 0$.
	Because of the presence of the weight $w^{2l}(p)$ in the integration, the change of variables $(p,q)\mapsto (p',q')$ creates another difference of $w^{2l}(p)-w^{2l}(p')$ inside the inner product of \eqref{defN} as seen below:
	\begin{align*}
&	-\int_\rth dp\int_{\rth}dq\int_{\mathbb{S}^2}d\omega\  w^{2l}(p)
v_{\text{\o}} \sigma(g,\theta)
(f(p')-f(p))f(p)\sqrt{J(q')}\sqrt{J(q)}\\
	=&-\frac{1}{2}\int_\rth dp\int_{\rth}dq\int_{\mathbb{S}^2}d\omega\  w^{2l}(p)
v_{\text{\o}} \sigma(g,\theta)
	(f(p')-f(p))f(p)\sqrt{J(q')}\sqrt{J(q)}\\
	&-\frac{1}{2}\int_\rth dp\int_{\rth}dq\int_{\mathbb{S}^2}d\omega\  w^{2l}(p')v_{\text{\o}} \sigma(g,\theta)
	(f(p)-f(p'))f(p')\sqrt{J(q)}\sqrt{J(q')}\\
	=&\ \frac{1}{2}\int_\rth dp\int_{\rth}dq\int_{\mathbb{S}^2}d\omega\   w^{2l}(p)
	v_{\text{\o}} \sigma
	(f(p')-f(p))^2\sqrt{J(q')}\sqrt{J(q)}\\
	& +\frac{1}{2}\int_\rth dp\int_{\rth}dq\int_{\mathbb{S}^2}d\omega\  (w^{2l}(p)-w^{2l}(p'))
v_{\text{\o}} \sigma 
	(f(p')-f(p))f(p')\sqrt{J(q)}\sqrt{J(q')}\\
&	\eqdef |f|^2_{\mathtt{B}_{2l}}+I'.
	\end{align*}
Here we are also using \eqref{weighted.non.local.fractional.norm}.	
	We express the inner product of $\langle w^{2l}\mathcal{N}f,f\rangle$ from \eqref{defN} as
$$
	\langle w^{2l}\mathcal{N}f,f\rangle=\int_ \rth dp\  \zeta(p)w^{2l}(p)|f(p)|^2+|f|^2_{\mathtt{B}_{2l}}+I'\approx|w^lf|^2_{L^2_{\frac{\singS+\gamma}{2}}}+|f|^2_{\mathtt{B}_{2l}}+I',
$$by \eqref{Paos}.
	We notice that we already obtain the following coercivity by using the methods in \secref{9.2}: 
$$
	|w^lf|^2_{L^2_{\frac{\singS+\gamma}{2}}}+|f|^2_{\mathtt{B}_{2l}}\gtrsim |f|^2_{I^{\singS,\gamma}_l},
$$as we can let $\frac{\singS+\gamma}{2}\mapsto \frac{\singS+\gamma}{2}+2l$ in \secref{9.2}.

We now estimate the upper-bound for the term $I'$.
	We first take the change of variables $(p',q')\mapsto (p,q)$ from \eqref{prepost.change} again and use the Cauchy-Schwarz inequality to obtain
	\begin{multline}
	\label{original}
		|I'|
		\lesssim |f|_{\mathtt{B}_{2l}}\Big(\int_\rth dp\int_{\rth}dq\int_{\mathbb{S}^2}d\omega\ \frac{ (w^{2l}(p)-w^{2l}(p'))^2}{w^{2l}(p)}\\
	\times  v_{\text{\o}} \sigma(g,\theta)
	|f(p)|^2\sqrt{J(q)}\sqrt{J(q')}\Big)^{\frac{1}{2}}.
		\end{multline}
	Now, by the fundamental theorem of calculus, we have
	$$
	w^{2l}(p')-w^{2l}(p)=\int_0^1d\tau\ (p'-p)\cdot (\nabla w^{2l})(\zeta(\tau))
	$$
	where $\zeta(\tau)=p+\tau(p'-p)$. Since $p'-p=q-q'$ from \eqref{conservation}, we have that 
	$$
	\sqrt{1+|\zeta(\tau)|^2}\lesssim p^0\sqrt{1+|q-q'|^2}.
	$$
	Thus, 
	we obtain
	$$
	\frac{(w^{2l}(p)-w^{2l}(p'))^2}{w^{2l}(p)}\sqrt{J(q)J(q')}\lesssim |p-p'|^2w^{2l}(p)(p^0)^{-2\delta}(J(q)J(q'))^\frac{1}{4}.
	$$
	Above $\delta\in (0,1]$ is a small constant that satisfies $\delta < 2l$ which is possible since $l>0$; if $2l-1>0$ then we can take $\delta=1$.  	We write $|p-p'|^2=|p-p'|^{\gamma+\epsilon}|q'-q|^{2-\gamma-\epsilon}$ for some small $\epsilon\in(0,2-\gamma)$. We recall from \eqref{bargoverg} that $\frac{\bar{g}}{g}=\sin\frac{\theta}{2}$. We also use that $|p-p'|\leq \bar{g}\sqrt{q^0q'^0}$ by \eqref{FREQ:g.ge.lower} and $p-p'=q'-q$ from \eqref{conservation}. Then, we have
\begin{multline}\notag
	|p-p'|^2(J(q)J(q'))^\frac{1}{4}\leq \bar{g}^{\gamma+\epsilon}(q^0q'^0)^{\frac{\gamma+\epsilon}{2}}|q'-q|^{2-\gamma-\epsilon}(J(q)J(q'))^\frac{1}{4}
	\\
	\lesssim g^{\gamma+\epsilon}\left(\sin\frac{\theta}{2}\right)^{\gamma+\epsilon}(J(q))^{\epsilon'},
\end{multline}
	where $\epsilon'>0$ is sufficiently small and we also use $|q'-q|^{2-\gamma-\epsilon}(J(q)J(q'))^{\frac{1}{4}-\tilde{\delta}}\lesssim 1$ for any $\tilde{\delta}\in(0,\frac{1}{8})$. Then the extra powers on $\sin\frac{\theta}{2}$ will control the angular singularity and we obtain that
	\begin{multline}\notag
	\int_{\mathbb{S}^2}d\omega
	~ v_{\text{\o}} \sigma(g,\theta)
	g^{\gamma+\epsilon}\left(\sin\frac{\theta}{2}\right)^{\gamma+\epsilon}(J(q))^{\epsilon'}
	\\
	\lesssim g^{\singS+\gamma+\epsilon}(J(q))^{\epsilon'} \lesssim (p^0q^0)^{\frac{\singS+\gamma+\epsilon}{2}}(J(q))^{\epsilon'}\lesssim (p^0)^{\frac{\singS+\gamma+\epsilon}{2}}(J(q))^{\epsilon''},
	\end{multline}
	if $\singS+\gamma+\epsilon\ge 0$ by \eqref{FREQ:g.le.sqrtpq}. If $\singS+\gamma+\epsilon<0$, then we use \eqref{FREQ:g.ge.lower} to obtain
		\begin{multline}\notag
	\int_{\mathbb{S}^2}d\omega
	~ v_{\text{\o}} \sigma(g,\theta)
	g^{\gamma+\epsilon}\left(\sin\frac{\theta}{2}\right)^{\gamma+\epsilon}(J(q))^{\epsilon'}
	\\
	\lesssim g^{\singS+\gamma+\epsilon}(J(q))^{\epsilon'} \lesssim |p-q|^{\singS+\gamma+\epsilon}(p^0q^0)^{-\frac{\singS+\gamma+\epsilon}{2}}(J(q))^{\epsilon'}.
	\end{multline}
	We put these back into (\ref{original}) to get for any small $\eta_1>0$ and any small $\eta_2>0$ that
$$
	|I'|\lesssim |f|_{\mathtt{B}_{2l}}|w^l f|_{L^2_{\frac{\singS+\gamma-\delta'}{2}}}\leq \eta_1 |f|^2_{\mathtt{B}_{2l}}+\eta_2|w^l f|^2_{L^2_{\frac{\singS+\gamma}{2}}}+C|w^l f|^2_{L^2(B_C)},
$$
	where $\delta'=4\delta-\epsilon>0$ and $C= C(\eta_1, \eta_2)>0$. This proves Lemma \ref{Nfcoercivitylemma}.
\end{proof}

This concludes our discussion of the main coercive estimates and Lemma \ref{Nfcoercivitylemma}. In
the next section we will prove the global existence and uniqueness of the solutions to the Boltzmann equation by using the non-linear energy method.

\section{Local and global existence}
\label{globalExist}

In this last section, we will establish global existence based on the modern method of separating the needed space-time estimates. This methodology goes back to the cut-off Boltzmann theory \cite{MR2000470}. We will show that the sharp estimates proved in the previous sections can be utilized for the method. The system of space-time relativistic macroscopic equations (\ref{macroscopic}) and the system of relativistic conservation laws (\ref{Conservation Laws}) will be derived and used to prove the coercive lower bound for the linear Boltzmann operator $L$ in our isotropic fractional derivative norm (\ref{fractional}).
In \secref{sec:local} we will explain the local existence argument, and then in 
\secref{sec:global} we establish the global existence, uniqueness and asymptotic decay rates to equilibrium.

\subsection{Local existence}\label{sec:local}
We now use the estimates that we made in the previous sections to sketch the local existence proof for small initial data.  Full details of a similar local existence argument can be found in \cite{1904.12086}. We will use the standard iteration method and uniform energy estimates for the iterated sequence of approximate solutions. 
The iteration starts at $f^0(t,x,p)=0$. We solve for $f^{m+1}(t,x,p)$ such that 
\begin{equation}
\label{8.1}
(\partial_t+\hat{p}\cdot\nabla_x+\mathcal{N})f^{m+1}+\mathcal{K}f^{m}=\Gamma(f^m,f^{m+1}), \hspace{5mm} f^{m+1}(0,x,p)=f_0(x,p).
\end{equation}
It can be proven with our estimates as the main tool that the linear equation (\ref{8.1}) admits smooth solutions with the same regularity in $H^N_l$ as a given smooth small initial data and that the solution has a gain of $L^2((0,T);I^{\singS,\gamma}_{l,N})$.  We omit these standard details.  
We will set up some estimates which are necessary to find a local classical solution as $m\rightarrow \infty$. 
As mentioned before, we will use the norm $\|\cdot \|_H$ for $\|\cdot\|_{H^N_l }$ for convenience and also use the norm $\|\cdot\|_I$ for the norm $\|\cdot\|_{I^{\singS,\gamma}_{l,N}}$. 
Define the total norm as 
 $$
M(f(t))=\|f(t)\|^2_H+\int_0^td\tau\ \|f(\tau)\|^2_I.
$$
In this section we will also use the notation $|f|_{I^{\singS,\gamma}}$ to denote $\langle \mathcal{N}f,f\rangle$ to simplify the notation below (this is justified by Lemmas \ref{Nfupperboundlemma} and \ref{Nfcoercivitylemma}). 

Here we state our main energy estimate:
\begin{lemma}
	\label{8.2}
Let $\{f^m\} $ be the sequence of iterated approximate solutions. There exists a short time $T^*=T^*(\|f_0\|^2_{H})>0$ such that for $\|f_0\|^2_H$ sufficiently small, there is a uniform constant $C_0>0$ such that 
	 $$
	\sup_{m\geq 0}\sup_{0\leq \tau \leq T^*} M(f^m(\tau))\leq 2 C_0\|f_0\|^2_H.
$$
\end{lemma}
\begin{proof} We will only write down the proof in the case $l=0$. For $l>0$ the proof is analogous, using the weighted norm as in (\ref{weightE}).	We prove this lemma by induction over $k$. If $k=0$, the lemma is trivially true. Suppose that the lemma holds for $k=m$. Let $f^{m+1}$ be the solution to the linear equation (\ref{8.1}) with given $f^m$. We take the spatial derivative $\partial^\alpha$ on the linear equation (\ref{8.1}) and obtain
	 $$
	(\partial_t+\hat{p}\cdot\nabla_x)\partial^\alpha f^{m+1}+\mathcal{N}(\partial^\alpha f^{m+1})+\mathcal{K}(\partial^\alpha f^{m})=\partial^\alpha\Gamma(f^m,f^{m+1}).
$$
	Then, we take a inner product with $\partial^\alpha f^{m+1}$. The trilinear estimate of Lemma \ref{Lemma1}  and \eqref{NL2equiv} implies that 
	\begin{multline*}
	\frac{1}{2}\frac{d}{dt}\|\partial^\alpha f^{m+1}\|^2_{L^2_pL^2_x}+\|\partial^\alpha f^{m+1}\|^2_{I^{\singS,\gamma}}+(\mathcal{K}(\partial^\alpha f^{m}),\partial^\alpha f^{m+1})
	\\
	\approx(\partial^\alpha \Gamma(f^{m},f^{m+1}),\partial^\alpha f^{m+1})\lesssim \|f^m\|_H \|f^{m+1}\|^2_I.
	\end{multline*}
	We integrate over time to obtain that
	\begin{multline}
	\label{8.5}
	\frac{1}{2}\|\partial^\alpha f^{m+1}(t)\|^2_{L^2_pL^2_x}+\int_0^t d\tau \|\partial^\alpha f^{m+1}(\tau)\|^2_{I^{\singS,\gamma}}+\int_0^t d\tau (\mathcal{K}(\partial^\alpha f^{m}),\partial^\alpha f^{m+1})\\
	\leq \frac{1}{2}\|\partial^\alpha f_0\|^2_{L^2_pL^2_x}+C\int_0^td\tau\|f^m\|_H\|f^{m+1}\|^2_I.
	\end{multline}
	From the compact estimate \eqref{compactest}, 
	for any small $\epsilon>0$ we have
	\begin{align*}
	&\left|\int_0^t d\tau (\mathcal{K}(\partial^\alpha f^{m}),\partial^\alpha f^{m+1})\right|\\ &\leq C_{\epsilon} \int_0^t d\tau\  \|w^{\singS/2}\partial^\alpha f^{m}(\tau)\|^2_{L^2}+\epsilon \int_0^t d\tau\  \|\partial^\alpha f^{m+1}(\tau)\|^2_{I^{\singS,\gamma}}.
	\end{align*}
We can interpolate for any small $\varepsilon'>0$ there is a large $C_{\varepsilon'}>0$ such that
\begin{equation*}
\|w^{\singS/2}\partial^\alpha f^{m}(\tau)\|^2_{L^2}
\leq C_{\varepsilon'}  
\|\partial^\alpha f^{m}(\tau)\|^2_{L^2}
+
\varepsilon' \|\partial^\alpha f^{m}(\tau)\|^2_{I}.
\end{equation*}
We use these estimates in (\ref{8.5}) and take a sum over all the derivatives such that $|\alpha|\leq N$ to obtain
\begin{multline}
	\label{continuity1}
M(f^{m+1}(t))\leq \|f_0\|^2_H+2C\sup_{0\leq\tau\leq t} M(f^{m+1}(\tau))\sup_{0\leq\tau\leq t} M^{1/2}(f^{m}(\tau))
	\\
+2C_{\epsilon}C_{\varepsilon'}   \int_0^t d\tau \|f^{m}(\tau)\|^2_{H}+2\epsilon \int_0^t d\tau \| f^{m+1}(\tau)\|^2_{I}
    +2C_{\epsilon}\varepsilon' \int_0^t d\tau \| f^{m}(\tau)\|^2_{I}
    \\
\leq \|f_0\|^2_H+2C\sup_{0\leq\tau\leq t} M(f^{m+1}(\tau))\sup_{0\leq\tau\leq t} M^{1/2}(f^{m}(\tau))
+2C_{\epsilon}C_{\varepsilon'}  t\sup_{0\leq\tau\leq t} M(f^{m}(\tau))
\\
+2\epsilon\sup_{0\leq\tau\leq t} M(f^{m+1}(\tau))
    +2C_{\epsilon}\varepsilon'\sup_{0\leq\tau\leq t} M(f^{m}(\tau)).
\end{multline}
	Then by the induction hypothesis on $M(f^m(\tau))$, we obtain that
\begin{multline*}
    M(f^{m+1}(t))\leq   \|f_0\|^2_H+2C\sqrt{2C_0}\|f_0\|_H\sup_{0\leq\tau\leq t} M(f^{m+1}(\tau))\\
    +4C_0C_{\epsilon}C_{\varepsilon'}  t \|f_0\|^2_H
    +4C_0C_{\epsilon}\varepsilon' \|f_0\|^2_H
    +2\epsilon\sup_{0\leq\tau\leq t} M(f^{m+1}(\tau)).
\end{multline*}
Thus we obtain
\begin{multline*}
    (1-2\epsilon-2C\sqrt{2C_0}\|f_0\|_H)\sup_{0\leq\tau\leq T^*} M(f^{m+1}(t))
    \\
\leq 
(1+4C_0C_{\epsilon}C_{\varepsilon'}  T^* +4C_0C_{\epsilon}\varepsilon' )\|f_0\|^2_H.
\end{multline*}
Thus we choose $\|f_0\|_H$, $\epsilon>0$, $\varepsilon'>0$, and $T^*>0$ sufficiently small (in that order) to obtain that 
	 $$
	\sup_{0\leq\tau\leq t} M(f^{m+1}(t))\leq 2C_0\|f_0\|^2_H.
	$$
	This proves the lemma by induction.
\end{proof}

Now, we prove the local existence result with the uniform control on each iteration.

\begin{theorem}
	\label{local existence}
	For any sufficiently small $M_0>0$, there exists a time $T^*=T^*(M_0)>0$ and $M_1>0$ such that if $\|f_0\|^2_H\leq M_1,$ then there exists a unique solution $f(t,x,p)$ to the linearized relativistic Boltzmann equation (\ref{Linearized B}) on $[0,T^*)\times \mathbb{T}^3\times \rth $ such that 
	 $$
	\sup_{0\leq t \leq T^* } M(f(t))\leq M_0.
$$
	Also, $M(f(t))$ is continuous on $[0,T^*)$. 
\end{theorem}

\begin{proof}
	\textit{Existence and Uniqueness.}
	By letting $m\rightarrow \infty$ in the previous lemma, we obtain sufficient compactness to obtain the local existence of a strong solution $f(t,x,p)$ to (\ref{Linearized B}). 
	For the uniqueness, suppose there exists another solution $h$ to   (\ref{Linearized B}) with the same initial data satisfying $\sup_{0\leq t \leq T^* } M(h(t))\leq \epsilon$ for a sufficiently small $\epsilon>0$. Then, by the equation, we have 
	\begin{equation}
	\label{8.6}
	\{\partial_t +\hat{p}\cdot \nabla_x\}(f-h)+L(f-h)=\Gamma(f-h,f)+\Gamma(h,f-h).
	\end{equation} 
	Then, by the Sobolev embedding $H^2(\mathbb{T}^3)\subset L^\infty (\mathbb{T}^3)$, Theorem \ref{thm1}, and the Cauchy-Schwarz inequality, we have
	\begin{multline*}
	    	|(\{\Gamma(h,f-h)+\Gamma(f-h,f)\},f-h)|
	    	\\
    \lesssim \|h\|_{L^2_pH_x^2}\|f-h\|^2_{I^{\singS,\gamma}} +\|f-h\|_{L^2_{p,x}}\|f\|_{H_x^2I^{\singS,\gamma}}\|f-h\|_{I^{\singS,\gamma}}
	= T_1+T_2.
	\end{multline*}
	For $T_1$, we have 
	 $$
	\int_0^t d\tau \hspace*{1mm} T_1(\tau) \lesssim \sqrt{\epsilon} \int_0^t d\tau \|f(\tau)-h(\tau)\|^2_{I^{\singS,\gamma}}
	$$
	because we have $\sup_{0\leq t \leq T^* } M(h(t)) \leq \epsilon.$
	For $T_2,$ we use the Cauchy-Schwarz inequality and obtain
	\begin{align*}
	\int_0^t d\tau \hspace*{1mm}T_2(\tau)\leq & \sqrt{\epsilon} \left(\sup_{0\leq \tau\leq t}\|f(\tau)-h(\tau)\|^2_{L^2_{p,x}}\int_0^t d\tau \|f(\tau)-h(\tau)\|^2_{I^{\singS,\gamma}}\right)^{1/2}\\
	\lesssim & \sqrt{\epsilon} \left(\sup_{0\leq \tau\leq t}\|f(\tau)-h(\tau)\|^2_{L^2_{p,x}}+\int_0^t d\tau \|f(\tau)-h(\tau)\|^2_{I^{\singS,\gamma}}\right)\\
	\end{align*}
	because $f$ also satisfies $\sup_{0\leq t \leq T^* } M(f(t)) \lesssim \epsilon.$
	For the linearized Boltzmann operator $L$ on the left-hand side of  (\ref{8.6}), we use Lemma \ref{2.10} to obtain 
	 $$
	(L(f-h),f-h)\ge c\|f-h\|^2_{I^{\singS,\gamma}}-C\|f-h\|^2_{L^2(\mathbb{T}^3\times B_R )}
$$
	for some small $c>0$ and some $R>0$. 
	We finally take the inner product of   (\ref{8.6}) with $(f-h)$ and integrate over $[0,t]\times \mathbb{T}^3\times \rth$ and use the estimates above to obtain
	\begin{align*}
	\frac{1}{2}\|f(t)-h(t&)\|^2_{L^2(\mathbb{T}^3\times \rth)}+c\int_0^t d\tau \  \|f(\tau)-h(\tau)\|^2_{I^{\singS,\gamma}}\\
	\lesssim &\sqrt{\epsilon}\left( \sup_{0\leq \tau \leq t}\|f(\tau)-h(\tau)\|^2_{L^2(\mathbb{T}^3\times \rth)}+\int_0^t d\tau \|f(\tau)-h(\tau)\|^2_{I^{\singS,\gamma}}   \right)\\
	&+\int_0^t d\tau \|f(\tau)-h(\tau)\|^2_{L^2(\mathbb{T}^3\times \rth)}.
	\end{align*}
	By Gronwall's inequality, we obtain that $f=h$ because $f$ and $h$ satisfy the same initial conditions. This proves the uniqueness of the solution.

	\textit{Continuity.} 
	Let $[s,t]$ be a time interval. We follow the simliar argument as in  (\ref{8.5}) and   (\ref{continuity1}) with the time interval $[s,t]$ instead of $[0,t]$ and let $f^m=f^{m+1}=f$ and obtain that
	\begin{align*}
	|M(f(t))-M(f(s))|&=\left|\frac{1}{2}\|f(t)\|^2_H-\frac{1}{2}\|f(s)\|^2_H+\int_s^t d\tau\   \|f(\tau)\|^2_I\right|\\
	&\lesssim \left(\int_s^t d\tau\   \|f(\tau)\|^2_I\right)\left(1+\sup_{s\leq\tau\leq t}M^{1/2}(f(\tau)) \right).
	\end{align*}
	As $s\rightarrow t$, we obtain that $|M(f(t))-M(f(s))|\rightarrow 0$  because $\|f\|^2_I$ is integrable in time. This proves the continuity of $M$. 
\end{proof}

This concludes the proof of the local existence.

\begin{remark}  In this remark we will briefly outline two different approaches to proving the positivity of a solution $F(t,x,p) = J(p) + \sqrt{J(p)} f(t,x,p) \ge 0$ when we initially have that $F_0 = J + \sqrt{J} f_0 \ge 0$.  

One approach, which was used in particular in \cite{MR2784329}, is to use the approximation of a solution to the cut-off relativistic Boltzmann equation where positivity is already known.  In this approach (a) one has a solution to the relativistic Boltzmann equation with a cut-off angular kernel singularity that can be shown to be non-negative.  Then (b) one proves that you can choose an approximate cut-off kernel $\sigma_n$ which converges to the non-cutoff kernel $\sigma$ satisfying \eqref{angassumption}-\eqref{soft}.   Then (c) prove that the solutions to the equation with the approximate kernel $\sigma_n$ that are known to be positive converge to the solutions to the equation with kernel $\sigma$ in a strong enough sense to conclude that the solutions remain positive in the limit.   This was done in the non-relativistic situation for instance in \cite{MR2679369}.  However all of the steps (a)-(c) contain substantial lengthy details that would need to be worked out carefully and the existing literature does not contain precisely what we would need for the relativistic Boltzmann equation.   

In order to handle lower regularity solutions one may  choose $F_0 = J + \sqrt{J} f_0 \ge 0$ and still need to approximate by $F_{0, \epsilon} = J + \sqrt{J} f_{0, \epsilon} \ge 0$ such that $F_{0, \epsilon} \to F_0$ in a suitable space.    One elementary way to make this choice is as follows.  Given $F_0 = J + \sqrt{J} f_0 \ge 0$, then  let $\phi_\epsilon \ge 0$ be a standard mollifier.  Then further define $f_{0, \epsilon} \eqdef \phi_\epsilon *(f_0 + \sqrt{J})- \sqrt{J}$.   Then $f_{0, \epsilon}$ is smooth and we will have $f_{0, \epsilon} \to f_{0}$ and $F_{0, \epsilon} \to F_{0}$ as $\epsilon \to 0$ in suitable spaces.  Further
$
F_{0, \epsilon} = J + \sqrt{J} f_{0, \epsilon} = \phi_\epsilon *(f_0 + \sqrt{J}) \ge 0,
$
since $f_0 + \sqrt{J} \ge 0$ by assumption.

Another approach was used in \cite{MR2847536}.  They used a maximum principle style argument to prove the positivity without using approximation.  We believe this method may also work to to prove the positivity for the the relativistic Boltzmann equation without angular cut-off.  

However both methods would need to be worked out in full detail, and both approaches require significant additional lengthy calculations.  Due to the current length of the present paper, we leave these developments for future work.  
\end{remark}

In the next subsection, we will show global existence using the nonlinear energy method from \cite{MR2000470}.  We point out that this approach is substantially more difficult in the special relativistic case as we will observe in the developments below.

\subsection{Global existence}\label{sec:global}
In order to prove the global-in-time existence of solutions, we will have to prove the global energy inequality. The main point for this is to obtain the uniform lower bound estimate on the Dirichlet form, $\langle L f, f \rangle$, of the linear operator $L$. Note that the linear Boltzmann operator $L$ has a very large null space $Pf$ as it will be introduced below in \eqref{Pfbasis}.  To study this null space we will derive the system of macroscopic equations \eqref{macroscopic} and balance laws \eqref{Conservation Laws} with respect to the coefficients appearing in the expression for the hydrodynamic part $Pf$.  Then we prove a coercive inequality for the microscopic part $\{I-P\}f$.   Using these coercivity estimates for the non-linear local solutions to the relativistic Boltzmann system, we will show that these solutions must be global in time by proving energy inequalities and using the standard continuity argument. We will also prove the rapid time decay of the solutions to equilibrium in the later part of this section.     

Now we use the relativistic Maxwellian solution $J(p)$ from \eqref{jutter.equilibrium}, and recall that  $\int_{\rth}J(p)dp=1$. We introduce the following notations for the integrals:
\begin{equation}\notag
    \lambda^0=\int_\rth {p^0}Jdp, 
    \quad 
    \lambda^{00}=\int_\rth ({p^0})^2Jdp, 
     \quad 
    \lambda^{11}=\int_\rth (p_1)^2Jdp,
     \quad 
\lambda^{11}_0=\int_\rth \frac{p_1^2}{{p^0}}Jdp.
\end{equation}
The 5-dimensional null space of the linearized Boltzmann operator $L$ is given by
 $$
nl(L)=\text{span}\{\sqrt{J},p_1\sqrt{J},p_2\sqrt{J},p_3\sqrt{J},{p^0}\sqrt{J}\}.
$$
Then we define the orthogonal projection from $L^2(\rth)$ onto $nl(L)$ by $P$. Then we can write $Pf$ as a linear combination of the basis as
\begin{equation}
\label{Pfbasis}
Pf=\left( \mathcal{A}^f(t,x)+\sum_{i=1}^{3}\mathcal{B}^f_i(t,x)p_i+\mathcal{C}^f(t,x){p^0}\right) \sqrt{J},
\end{equation}
where the coefficients are given by
$$\mathcal{A}^f=\int_\rth f\sqrt{J}dp - \lambda^0\mathcal{C}^f, \hspace{3mm} \mathcal{B}_i^f= \frac{\int_\rth fp_i\sqrt{J}dp}{\lambda^{11}}, \hspace{3mm} \mathcal{C}^f=\frac{\int_{\rth}f({p^0}\sqrt{J}-\lambda^0\sqrt{J}) dp}{\lambda^{00}-(\lambda^0)^2}.
$$
This choice of the basis was given in \cite{MR2911100}. We now observe the null space (\ref{Pfbasis}) and the positivity of the linear operator, as proven in \cite{MR1211782}.

\begin{lemma}
\label{llinear}
$\langle Lg,h\rangle =\langle Lh,g\rangle ,\;\langle
Lg,g\rangle \ge 0.$ And $Lg=0$ if and only if 
$
g=P g.
$
\end{lemma}

Then we can decompose $f(t,x,p)$ as
\begin{equation}
\label{decomp}
f=Pf+\{I-P\}f.
\end{equation}
We start from plugging the expression  (\ref{decomp}) into  (\ref{Linearized B}). Then we obtain
\begin{equation}
\label{hydro}
\{\partial_t+\hat{p}\cdot \nabla_x\}Pf=-\partial_t\{I-P\}f-(\hat{p}\cdot \nabla_x+L)\{I-P\}f+\Gamma(f,f).
\end{equation}
Note that we have expressed the hydrodynamic part $Pf$ in terms of the microscopic part $\{I-P\}f$ and the higher-order term $\Gamma$. We define the operator $$l(\{I-P\}f)=-(\hat{p}\cdot \nabla_x+L)(\{I-P\}f).$$ 
Using the expression  \eqref{Pfbasis} of $Pf$ with respect to the basis elements, we obtain that the left-hand side of  \eqref{hydro} can be written as
\begin{multline*}
\partial_t \mathcal{A}\sqrt{J}+\sum_{i=1}^3\partial_i(\mathcal{A}+\mathcal{C}{p^0})\frac{p_i}{{p^0}}\sqrt{J} +\partial_t \mathcal{C} {p^0}\sqrt{J}+\sum_{i=1}^{3}\partial_t\mathcal{B}_ip_i\sqrt{J}\\
 +\sum_{i=1}^3\partial_i\mathcal{B}_i\frac{p_i^2}{{p^0}}\sqrt{J}+\sum_{i=1}^{3} \sum_{i\neq j}\partial_j \mathcal{B}_i\frac{p_ip_j}{{p^0}}\sqrt{J},
\end{multline*}
where $\mathcal{A}=\mathcal{A}^f$, $\mathcal{B}=\mathcal{B}^f$, $\mathcal{C}=\mathcal{C}^f$ and $\partial_i=\partial_{x_i}$. 
For fixed $(t,x)$ we can write the left-hand side with respect to the following basis, $\{e_k\}_{k=1}^{14}$, which consists of 
\begin{equation}
\label{basissis}
\sqrt{J}, \hspace{3mm}\left(\frac{p_i}{{p^0}}\sqrt{J}\right)_{1\leq i \leq 3}, \hspace{3mm}{p^0}\sqrt{J}, \hspace{3mm}\left(p_i\sqrt{J}\right)_{1\leq i \leq 3},\hspace{3mm} \left(\frac{p_ip_j}{{p^0}}\sqrt{J}\right)_{1\leq i\leq j\leq 3}.
\end{equation} 
Then we can rewrite the left-hand side of \eqref{hydro} as
\begin{multline*}
\partial_t \mathcal{A}\sqrt{J}+\sum_{i=1}^3\partial_i \mathcal{A}\frac{p_i}{{p^0}} \sqrt{J}+\partial_t \mathcal{C} {p^0}\sqrt{J}+\sum_{i=1}^3 (\partial_i\mathcal{C}+\partial_t\mathcal{B}_i)p_i\sqrt{J}\\
 +\sum_{i=1}^{3} \sum_{j=1}^3((1-\delta_{ij})\partial_i\mathcal{B}_j+\partial_j\mathcal{B}_i)\frac{p_ip_j}{{p^0}}\sqrt{J}.
\end{multline*}
By the comparison of coefficients, we can obtain a system of macroscopic equations
\begin{equation}
\label{macroscopic}
\begin{split}
\partial_t \mathcal{A}&=-\partial_tm_a+l_a+G_a,\\
\partial_i \mathcal{A}&=-\partial_tm_{ia}+l_{ia}+G_{ia},\\
\partial_t \mathcal{C}&= -\partial_tm_{c}+l_{c}+G_{c},\\
\partial_i \mathcal{C}+\partial_t\mathcal{B}_i&= -\partial_tm_{ic}+l_{ic}+G_{ic},\\
(1-\delta_{ij})\partial_i\mathcal{B}_j+\partial_j\mathcal{B}_i&=-\partial_tm_{ij}+l_{ij}+G_{ij},
\end{split}
\end{equation}
where the indices are from the index set defined as $D=\{a,ia,c,ic,ij\ |\  1\leq i\leq j\leq 3\}$ and $m_\mu$, $l_\mu$, and $G_\mu$ for $\mu \in D$ are the coefficients of $\{I-P\}f$, $l(\{I-P\}f)$, and $\Gamma(f,f)$ with respect to the basis $\{e_k\}_{k=1}^{14}$ respectively.

Also, we derive the local conservation laws. The derivation of the local conservation laws for the relativistic Boltzmann equation has already been introduced in \cite{MR2911100} and we introduce the full details here for the sake of completeness. We first multiply the linearized Boltzmann equation by $\sqrt{J}, p_i\sqrt{J},{p^0}\sqrt{J}$ and integrate them over $\rth$ to obtain that 
\begin{equation}
\label{c1}
\begin{split}
\partial_t\int_{\rth}f\sqrt{J} dp+\int_{\rth}\hat{p}\cdot\nabla_xf\sqrt{J}dp&=0,\\
\partial_t\int_{\rth}f\sqrt{J}p_i dp+\int_{\rth}\hat{p}\cdot\nabla_xf\sqrt{J}p_idp&=0,\\
\partial_t\int_{\rth}f\sqrt{J}{p^0} dp+\int_{\rth}\hat{p}\cdot\nabla_xf\sqrt{J}{p^0}dp&=0.
\end{split}
\end{equation}
These hold because $1, p_i, {p^0}$ are collisional invariants using \eqref{eq.colop.property}.  We will plug the decomposition $f=Pf+\{I-P\}f$ into  (\ref{c1}). We first consider the microscopic part. 
Note that 
\begin{multline}
    \label{c2}
\int_ \rth \hat{p}\cdot \nabla_x \{I-P\}f\sqrt{J}\left(\begin{array}{c}
1\\
p_i\\
{p^0}\end{array}\right)dp=\sum_{j=1}^3 \int_ \rth \frac{p_j}{{p^0}}\partial_j \{I-P\}f\sqrt{J}\left(\begin{array}{c}
1\\
p_i\\
{p^0}\end{array}\right)dp
\\
=\sum_{j=1}^3 \partial_j \left\langle \{I-P\}f, \sqrt{J}\left(\begin{array}{c}
\frac{p_j}{{p^0}}\\
\frac{p_ip_j}{{p^0}}\\
p_j\end{array}\right)\right\rangle. 
\end{multline}
We also notice that 
$\int_{\rth}\{I-P\}f\sqrt{J}(p)
p_j dp = 0$.  Also, we have that
\begin{equation}
\label{c3}
\begin{split}
\partial_t\int_ \rth \{I-P\}f\sqrt{J}\left(\begin{array}{c} 1\\p_i\\{p^0}\end{array}\right)=\partial_t \left\langle \{I-P\}f,\sqrt{J}\left(\begin{array}{c} 1\\p_i\\{p^0}\end{array}\right)\right\rangle=0.
\end{split}
\end{equation}
On the other hand, the hydrodynamic part $Pf=(\mathcal{A}+\mathcal{B}\cdot p+\mathcal{C}{p^0})\sqrt{J}$ satisfies
\begin{equation}\notag
    \partial_t \int_{\rth} \left(\begin{array}{c} 1\\p_i\\{p^0}\end{array}\right) Pf\sqrt{J}dp
=\partial_t \int_{\rth} \left(\begin{array}{c} \mathcal{A}+\mathcal{B}\cdot p+\mathcal{C}{p^0}\\\mathcal{A}p_i+\mathcal{B}\cdot pp_i+\mathcal{C}{p^0}p_i
\\\mathcal{A}{p^0}+\mathcal{B}\cdot p{p^0}+\mathcal{C}{(p^0)}^2\end{array}\right) \sqrt{J}(p)dp,
\end{equation}
and
\begin{equation}\notag
\int_ \rth \hat{p}\cdot \nabla_x Pf\sqrt{J}\left(\begin{array}{c} 1\\p_i\\{p^0}\end{array}\right) dp 
=
\sum_{j=1}^3\int_ \rth \partial_j \left(\begin{array}{c} \frac{p_j}{{p^0}}(\mathcal{A}+\mathcal{B}\cdot p+\mathcal{C}{p^0})\\\frac{p_ip_j}{{p^0}}(\mathcal{A}+\mathcal{B}\cdot p+\mathcal{C}{p^0})\\p_j\mathcal{A}+\mathcal{B}\cdot pp_j+\mathcal{C}{p^0}p_j\end{array}\right)\sqrt{J}(p) dp. 
\end{equation}
Thus we obtain
\begin{multline}\label{c4}
\partial_t \int_{\rth} \left(\begin{array}{c} 1\\p_i\\{p^0}\end{array}\right) Pf\sqrt{J}dp
+
\int_ \rth \hat{p}\cdot \nabla_x Pf\sqrt{J}\left(\begin{array}{c} 1\\p_i\\{p^0}\end{array}\right) dp 
\\ 
=\left(\begin{array}{c} \partial_t \mathcal{A}+\lambda^0\partial_t \mathcal{C}\\ \lambda^{11}\partial_t\mathcal{B}_i\\ \lambda^0 \partial_t \mathcal{A}+\lambda^{00}\partial_t \mathcal{C} \end{array}\right)+\left(\begin{array}{c} \lambda^{11}_0\nabla_x\cdot \mathcal{B}\\ \lambda^{11}_0\partial_i \mathcal{A}+\lambda^{11}\partial_i \mathcal{C}\\ \lambda^{11}\nabla_x \cdot \mathcal{B}\end{array}\right).
\end{multline}
Also, we have that $L(f)=L\{I-P\}f$. Together with  (\ref{c1}), (\ref{c2}), (\ref{c3}), and (\ref{c4}), we finally obtain the local conservation laws satisfied by $(\mathcal{A},\mathcal{B},\mathcal{C})$:
\begin{align*}
\partial_t \mathcal{A}+\lambda^0\partial_t \mathcal{C}+\lambda^{11}_0\nabla_x\cdot \mathcal{B} &=-\nabla_x \cdot \left\langle \{I-P\}f, \sqrt{J}\frac{p}{{p^0}}\right\rangle,\\
\lambda^{11}\partial_t \mathcal{B} +\lambda^{11}_0\nabla_x \mathcal{A}+\lambda^{11}\nabla_x \mathcal{C} &=-\nabla_x \cdot \left\langle \{I-P\}f, \sqrt{J}\frac{p\otimes p}{{p^0}}\right\rangle,\\\lambda^0\partial_t \mathcal{A}+\lambda^{00}\partial_t \mathcal{C}+\lambda^{11} \nabla_x \cdot \mathcal{B} &= 0.
\end{align*} 
Comparing the first and the third conservation laws, we obtain the following local conservation laws:
\begin{equation}
\label{Conservation Laws}
\begin{split}
\mu_1 \partial_t \mathcal{A}+\mu_2 \nabla_x\cdot \mathcal{B}&=-\nabla_x \cdot \left\langle \{I-P\}f, \sqrt{J}\frac{p}{{p^0}}\right\rangle,\\
\lambda^{11}\partial_t \mathcal{B} +\lambda^{11}_0\nabla_x \mathcal{A}+\lambda^{11}\nabla_x \mathcal{C}&=-\nabla_x \cdot \left\langle \{I-P\}f, \sqrt{J}\frac{p\otimes p}{{p^0}}\right\rangle,\\
\mu_3 \partial_t \mathcal{C}+\mu_4 \nabla_x \cdot \mathcal{B}&= -\nabla_x \cdot \left\langle \{I-P\}f, \sqrt{J}\frac{p}{{p^0}}\right\rangle.\\
\end{split}
\end{equation}  
Above the constants are given by $\mu_1 = \left(1-\frac{(\lambda^0)^2}{\lambda^{00}}\right)>0$, $\mu_2 = \left(\lambda^{11}_0-\frac{\lambda^0\lambda^{11}}{\lambda^{00}}\right)$, $\mu_3 = \left(\lambda^0-\frac{\lambda^{00}}{\lambda^0}\right)<0$ and $\mu_4 = \left(\lambda^{11}_0-\frac{\lambda^{11}}{\lambda^0}\right)$.

We also mention that we have the following lemma on the coefficients $\mathcal{A},\mathcal{B},\mathcal{C}$ directly from the conservation of mass, momentum, and energy:
\begin{lemma}
	\label{L8.5}
	Let $f(t,x,p)$ be the local solution to the linearized relativistic Boltzmann equation (\ref{Linearized B}) which is shown to exist in Theorem \ref{local existence} which satisfies the mass, momentum, and energy conservation laws (\ref{zero}). Then we have
	 $$
	\int_{\mathbb{T}^3} \mathcal{A}(t,x)dx=\int_{\mathbb{T}^3} \mathcal{B}_i(t,x)dx=\int_{\mathbb{T}^3} \mathcal{C}(t,x)dx=0,
	$$
	where $i\in\{1,2,3\}$.
\end{lemma}
We also list two lemmas that help us to control the coefficients in the linear microscopic term $l$ and the non-linear higher-order term $\Gamma$. 
\begin{lemma}
	\label{L8.6}
	For any coefficient $l_\mu$ for the microscopic term $l$, and for any $m\ge 0$ we have
	 $$
	 \|l_\mu\|_{H_x^{N-1}}\lesssim \sum_{|\alpha|\leq N}\|\{I-P\}\partial^\alpha f\|_{L^2_{-m}(\mathbb{T}^3\times\rth)},
	 \quad \forall \mu \in D.
	$$
\end{lemma}

\begin{proof}	In order to estimate the size for the $H^{N-1}$ norm, we use any $e_k$ from \eqref{basissis} to observe that
	 $$
	\langle \partial^\alpha l(\{I-P\}f),e_k\rangle=-\langle\hat{p}\cdot\nabla_x(\{I-P\}\partial^\alpha f),e_k\rangle-\langle L(\{I-P\}\partial^\alpha f),e_k\rangle.
	$$
	For any $|\alpha|\leq N-1$, the $L^2$-norm of the first part of the right-hand side is
	\begin{align*}  
	\|\langle\hat{p}\cdot\nabla_x(\{I-P\}\partial^\alpha f),e_k\rangle\|^2_{L^2_x}&\lesssim \int_ {\mathbb{T}^3\times\rth} dxdp \ |e_k||\{I-P\}\nabla_x\partial^\alpha f|^2\\
	&\lesssim \|\{I-P\}\nabla_x\partial^\alpha f\|^2_{L^2_{-m}(\mathbb{T}^3\times\rth)}.
	\end{align*}
	On the other hand, by \eqref{L.def} and \eqref{C3} we have
	\begin{align*}
	\|\langle L(\{I-P\}\partial^\alpha f),e_k\rangle\|^2_{L^2_x}&\lesssim \Big\||\{I-P\}\partial^\alpha f|_{L^2_{-m}}|J^{1/4}|_{L^2_{-m}} \Big\|^2_{L^2_x}\\
	&\lesssim \|\{I-P\}\partial^\alpha f\|^2_{L^2_{-m}(\mathbb{T}^3\times\rth)},
	\end{align*}for any $m\ge 0.$
	This completes the proof.
\end{proof}

\begin{lemma}
	\label{L8.7}
Let $\|f\|^2_H\leq M_0$ for some $M_0>0$. Then, for any $m\ge 0$, we have 
$$
\|G_\mu\|_{H_x^{N-1}}\lesssim \sqrt{M_0}\sum_{|\alpha|\leq N-1}\|\partial^\alpha f\|_{L^2_{-m}(\mathbb{T}^3\times\rth)},
\quad \forall \mu \in D.
$$
\end{lemma}

\begin{proof}	In order to estimate the size the for $H^{N-1}$ norm, we consider $\langle \Gamma(f,f), e_k\rangle$. By  (\ref{C3}), for any $m\geq 0$, we have
	\begin{align*}
	\|\langle \Gamma(f,f),e_k\rangle\|_{H^{N-1}_x}&\lesssim \sum_{|\alpha|\leq N-1}\sum_{\alpha'\leq\alpha}\Big\||\partial^{\alpha-\alpha'}f|_{L^2_{-m}}|\partial^{\alpha'}f|_{L^2_{-m}}\Big\|_{L^2_x}\\
	&\lesssim \|f\|_{L^2_{-m}H^{N-1}_x}\sum_{|\alpha|\leq N-1}\|\partial^\alpha f\|_{L^2_{-m}}\\
	&\lesssim \sqrt{M_0}\sum_{|\alpha|\leq N-1}\|\partial^\alpha f\|_{L^2_{-m}(\mathbb{T}^3\times\rth)}.
	\end{align*}
	This completes the proof.
\end{proof}

These two lemmas above, the macroscopic equations, and the local conservation laws will together prove the following theorem on the coercivity estimate for the microscopic term $\{I-P\}f$ which is crucial for the coercivity of the linearized operator $L$ in the energy form and hence is crucial for the energy inequality which will imply the global existence of the solution with the continuity argument. 
\begin{theorem}
	\label{8.4}
	Given the initial condition $f_0\in H$ which satisfies the mass, momentum, and energy conservation laws (\ref{zero}) and the assumptions in Theorem \ref{local existence}, we can consider the local solution $f(t,x,p)$ to the linearized relativistic Boltzmann equation (\ref{Linearized B}). Then, there is a constant $M_0>0$ such that if 
	$$
	\|f(t)\|^2_H\leq M_0,
	$$
	then there are universal constants $\delta>0$ and $C>0$ such that 
	 $$
	\sum_{|\alpha|\leq N} \|\{I-P\}\partial^\alpha f\|^2_{I^{\singS,\gamma}}(t)\geq \delta \sum_ {|\alpha|\leq N}\|P\partial^\alpha f\|^2_{I^{\singS,\gamma}}(t)-C\frac{dI(t)}{dt},
	$$
	where $I(t)$ is an interaction potential defined as 
	 $$
	I(t)=\sum_ {|\alpha|\leq N-1} \{I^\alpha_a(t)+I^\alpha_b(t)+I^\alpha_c(t)\} ,
	$$
Above, each of the sub-potentials $I^\alpha_a(t)$, $I^\alpha_b(t)$, and $I^\alpha_c(t)$ are defined during the proof in \eqref{interaction.a}, \eqref{interaction.c} and \eqref{interaction.b}.
\end{theorem}

\begin{proof}  Since we have the expression $Pf=\left(\mathcal{A}+\mathcal{B}\cdot p+\mathcal{C}{p^0}\right) \sqrt{J}$ as in \eqref{Pfbasis}, we have that
	 $$
	\|P\partial^\alpha f(t)\|^2_{I^{\singS,\gamma}}\lesssim \|\partial^\alpha \mathcal{A}(t)\|^2_{L^2_x}+\|\partial^\alpha \mathcal{B}(t)\|^2_{L^2_x}+\|\partial^\alpha \mathcal{C}(t)\|^2_{L^2_x}.
	$$
	Thus, it suffices to prove the following estimate:
	\begin{multline}
	\label{8.19}
	\| \mathcal{A}(t)\|^2_{H^N_x}+\| \mathcal{B}(t)\|^2_{H^N_x}+\| \mathcal{C}(t)\|^2_{H^N_x}\\
	\lesssim \sum_{|\alpha|\leq N}\|\{I-P\}\partial^\alpha f(t)\|^2_{L^2_{\frac{\singS+\gamma}{2}}}+ M_0\sum_{|\alpha|\leq N}\|\partial^\alpha f(t)\|^2_{L^2_{\frac{\singS+\gamma}{2}}}+\frac{dI(t)}{dt}.
	\end{multline}
	Then note that the term $ M_0\sum_{|\alpha|\leq N}\|\partial^\alpha f(t)\|^2_{L^2_{\frac{\singS+\gamma}{2}}}$ on (RHS) of \eqref{8.19} can be treated by using
	\begin{align*}
	&\sum_{|\alpha|\leq N}\|\partial^\alpha f(t)\|^2_{L^2_{\frac{\singS+\gamma}{2}}}\lesssim  \sum_{|\alpha|\leq N}\|P\partial^\alpha f(t)\|^2_{L^2_{\frac{\singS+\gamma}{2}}}+ \sum_{|\alpha|\leq N}\|\{I-P\}\partial^\alpha f(t)\|^2_{L^2_{\frac{\singS+\gamma}{2}}}\\
	&\lesssim \| \mathcal{A}(t)\|^2_{H^N_x}+\| \mathcal{B}(t)\|^2_{H^N_x}+\| \mathcal{C}(t)\|^2_{H^N_x}+\sum_{|\alpha|\leq N}\|\{I-P\}\partial^\alpha f(t)\|^2_{L^2_{\frac{\singS+\gamma}{2}}},
	\end{align*}
	and the terms in $ \mathcal{A}$, $\mathcal{B}$, and $\mathcal{C}$ can be absorbed by the (LHS) of \eqref{8.19} for a sufficiently small $M_0>0$. Therefore, we obtain Theorem \ref{8.4} from \eqref{8.19}.
	
	In order to prove  (\ref{8.19}), we will estimate each of the $\partial^\alpha$ derivatives of $\mathcal{A}, \mathcal{B}, \mathcal{C}$ for $0<|\alpha|\leq N$ separately. Later,  we will use Poincar\'e inequality to estimate the $L^2$-norm of $\mathcal{A}$, $\mathcal{B}$ and $\mathcal{C}$ to finish the proof. 
	
	For the estimate for $\mathcal{A}$, we use the second equation in the system of macroscopic equations \eqref{macroscopic}$_2$ which tells $\partial_i \mathcal{A}=-\partial_tm_{ia}+l_{ia}+G_{ia}$. 
	We take $\partial_i\partial^\alpha$ onto this equation for $|\alpha|\leq N-1$ and sum over $i$ and obtain that
	 $$
	-\Delta\partial^\alpha \mathcal{A}= \sum_{i=1}^3 (\partial_t\partial_i \partial^\alpha m_{ia} -\partial_i\partial^\alpha(l_{ia}+G_{ia})). 
	$$
	We now multiply $\partial^\alpha \mathcal{A}$ and integrate over $\mathbb{T}^3$ to obtain
	\begin{align*}
	\|\nabla\partial^\alpha \mathcal{A}\|^2_{L^2_x}&\leq \|\partial^\alpha (l_{ia}+G_{ia})\|_{L^2_x}\|\nabla\partial^\alpha \mathcal{A}\|_{L^2_x}+\frac{d}{dt}\sum_{i=1}^3\int_ {\mathbb{T}^3} \partial_i\partial^\alpha m_{ia}\partial^\alpha \mathcal{A}(t,x) dx\\
	& \qquad\qquad\qquad\qquad\qquad\qquad-\sum_{i=1}^3\int_ {\mathbb{T}^3} \partial_i\partial^\alpha m_{ia}\partial_t\partial^\alpha \mathcal{A}(t,x) dx. 
	\end{align*} 
	We define the interaction functional 
\begin{equation}\label{interaction.a}
I^\alpha_a(t)=\sum_{i=1}^3\int_ {\mathbb{T}^3} \partial_i\partial^\alpha m_{ia}\partial^\alpha \mathcal{A}(t,x) dx.
\end{equation}
	For the last term, we use the first equation of the local conservation laws \eqref{Conservation Laws}$_1$ to obtain that
	 $$
	\int_ {\mathbb{T}^3} \sum_{i=1}^3|\partial_i\partial^\alpha m_{ia}\partial_t\partial^\alpha \mathcal{A}(t,x)| dx\leq \zeta \|\nabla\cdot \partial^\alpha \mathcal{B}\|^2_{L^2_x}+C_\zeta \|\{I-P\}\nabla\partial^\alpha f\|^2_{L^2_{\frac{\singS+\gamma}{2}}},
	$$
	for any $\zeta>0$, since $m_\mu$ are the coefficients of $\{I-P\}f$ with respect to the basis $\{e_k\}. $ Together with Lemma \ref{L8.6} and Lemma \ref{L8.7}, we obtain that 
	\begin{equation}
	\label{maina}
    \begin{split}
    &\sum_{|\alpha|\le N-1}	\left(\|\nabla \partial^\alpha \mathcal{A}\|^2_{L^2_x}-\zeta \|\nabla\cdot \partial^\alpha \mathcal{B}\|^2_{L^2_x}\right)\\
    &\lesssim C_\zeta \sum_{|\alpha'|\leq N} \|\{I-P\}\partial^{\alpha'} f\|^2_{L^2_{\frac{\singS+\gamma}{2}}}+ \sum_{|\alpha|\le N-1}\frac{dI^\alpha_a}{dt}+M_0\sum_{|\alpha'|\leq N}\|\partial^{\alpha'} f\|^2_{L^2_{\frac{\singS+\gamma}{2}}}.
	\end{split}
    \end{equation}
	This completes our main estimate for $\mathcal{A}$.
	
We now explain the estimate for $\mathcal{C}$, we use the fourth equation in the system of macroscopic equations \eqref{macroscopic}$_4$ which tells $\partial_i \mathcal{C}+\partial_t\mathcal{B}_i=-\partial_tm_{ic}+l_{ic}+G_{ic}$. 
	We take $\partial_i\partial^\alpha$ onto this equation for $|\alpha|\leq N-1$ and sum over $i$ and obtain that
	 $$
	-\Delta\partial^\alpha \mathcal{C}= \frac{d}{dt}(\nabla\cdot \partial^\alpha \mathcal{B})+\sum_{i=1}^3 (\partial_t\partial_i \partial^\alpha m_{ic} -\partial_i\partial^\alpha(l_{ic}+G_{ic})). 
	$$
	We now multiply $\partial^\alpha \mathcal{C}$ and integrate over $\mathbb{T}^3$ 
	to obtain
	\begin{align*}
	\|\nabla\partial^\alpha \mathcal{C}\|^2_{L^2_x}&
	\leq \frac{d}{dt}\int_{\mathbb{T}^3}(\nabla\cdot \partial^\alpha \mathcal{B})\partial^\alpha \mathcal{C}(t,x)dx
	-\int_{\mathbb{T}^3}(\nabla\cdot \partial^\alpha \mathcal{B})\partial_t\partial^\alpha \mathcal{C}(t,x)dx
	\\
	&+
	\sum_{i=1}^3\|\partial^\alpha (l_{ic}+G_{ic})\|_{L^2_x}\|\nabla\partial^\alpha \mathcal{C}\|_{L^2_x}+\frac{d}{dt}\sum_{i=1}^3\int_ {\mathbb{T}^3} \partial_i\partial^\alpha m_{ic}\partial^\alpha \mathcal{C}(t,x) dx
	\\
	& -\sum_{i=1}^3\int_ {\mathbb{T}^3} \partial_i\partial^\alpha m_{ic}\partial_t\partial^\alpha \mathcal{C}(t,x) dx. 
	\end{align*} 
	We define the interaction functional 
\begin{equation}\label{interaction.c}
I^\alpha_c(t)=\int_{\mathbb{T}^3}(\nabla\cdot \partial^\alpha \mathcal{B})\partial^\alpha \mathcal{C}(t,x)dx+\sum_{i=1}^3\int_ {\mathbb{T}^3} \partial_i\partial^\alpha m_{ic}\partial^\alpha \mathcal{C}(t,x) dx.
\end{equation}
	We also use the third equation of the local conservation laws  \eqref{Conservation Laws}$_3$ to obtain that
	$$
	\int_ {\mathbb{T}^3} \sum_{i=1}^3|\partial_i\partial^\alpha m_{ic}\partial_t\partial^\alpha \mathcal{C}(t,x)| dx
    \lesssim 
	\|\nabla\cdot \partial^\alpha \mathcal{B}\|^2_{L^2_x}+ \|\{I-P\}\nabla\partial^\alpha f\|^2_{L^2_{\frac{\singS+\gamma}{2}}}.
	$$
We use \eqref{Conservation Laws}$_3$ again to estimate 
\begin{equation}
    \int_{\mathbb{T}^3}
    \left| (\nabla\cdot \partial^\alpha \mathcal{B})\partial_t\partial^\alpha \mathcal{C}(t,x) \right| dx
    \lesssim 
    \|\nabla\cdot \partial^\alpha \mathcal{B}\|^2_{L^2_x}
    +
    \|\{I-P\}\nabla\partial^\alpha f\|^2_{L^2_{\frac{\singS+\gamma}{2}}}.
\end{equation}	
	Together with Lemma \ref{L8.6} and Lemma \ref{L8.7}, we obtain that 
	\begin{equation}
	\label{mainc}
    \begin{split}
	&\sum_{|\alpha|\le N-1}	\left(\|\nabla \partial^\alpha \mathcal{C}\|^2_{L^2_x}-\lambda \|\nabla\cdot \partial^\alpha \mathcal{B}\|^2_{L^2_x}\right)\\&\lesssim C_\zeta \sum_{|\alpha'|\leq N} \|\{I-P\}\partial^{\alpha'} f\|^2_{L^2_{\frac{\singS+\gamma}{2}}}+\sum_{|\alpha|\le N-1}	\frac{dI^\alpha_c}{dt}+M_0\sum_{|\alpha'|\leq N}\|\partial^{\alpha'} f\|^2_{L^2_{\frac{\singS+\gamma}{2}}},
	\end{split}
    \end{equation}
where above $\lambda>0$.
This is our main estimate for the $\mathcal{C}$ terms.
	
Next we estimate the terms involving $\mathcal{B}$, we use the last equation in the system of macroscopic equations \eqref{macroscopic}$_5$ which tells $(1-\delta_{ij})\partial_i\mathcal{B}_j+\partial_j\mathcal{B}_i=-\partial_tm_{ij}+l_{ij}+G_{ij}$. 
	Note that when $i=j$, we have 
	$$
	\partial_j\mathcal{B}_j=-\partial_tm_{jj}+l_{jj}+G_{jj}.
	$$
We take $\partial_j\partial^\alpha$ on \eqref{macroscopic}$_5$ for $|\alpha|\leq N-1$ and sum on $j$ to obtain
\begin{multline*}
\Delta \partial^\alpha\mathcal{B}_i
	    	=\sum_{j=1}^3\left(-\partial_j\partial^\alpha (1-\delta_{ij})\partial_i\mathcal{B}_j-\partial_j\partial^\alpha\partial_tm_{ij}+\partial_j\partial^\alpha l_{ij}+\partial_j\partial^\alpha G_{ij}\right)
	    	\\
	    	=-
	    	\sum_{j\neq i}\partial_j\partial^\alpha\partial_i\mathcal{B}_j
	    	+\sum_{j=1}^3\bigg(-\partial_j\partial^\alpha\partial_tm_{ij}+\partial_j\partial^\alpha l_{ij}+\partial_j\partial^\alpha G_{ij}\bigg).
\end{multline*}
	Then by using 
$\partial_i\partial_j\mathcal{B}_j=-\partial_i\partial_tm_{jj}+\partial_il_{jj}+\partial_iG_{jj},$ we have
\begin{multline*}
	    	\Delta \partial^\alpha\mathcal{B}_i
	=\sum_{j\neq i}\bigg(\partial^\alpha\partial_i\partial_tm_{jj}-\partial^\alpha\partial_il_{jj}-\partial^\alpha\partial_iG_{jj}\bigg)
	\\+\sum_{j=1}^3\bigg(-\partial_j\partial^\alpha\partial_tm_{ij}+\partial_j\partial^\alpha l_{ij}+\partial_j\partial^\alpha G_{ij}\bigg).
\end{multline*}
We now multiply $\partial^\alpha \mathcal{B}_i$ and integrate over $\mathbb{T}^3$ to obtain
\begin{multline*}
	\|\nabla\partial^\alpha \mathcal{B}_i\|^2_{L^2_x}
\lesssim 
\frac{d}{dt}\sum_{j=1}^3\int_{\mathbb{T}^3}\partial^\alpha(\partial_j m_{ij}-\partial_i m_{jj}(1-\delta_{ij}))\partial^\alpha \mathcal{B}_idx\\
	-\sum_{j=1}^3\int_{\mathbb{T}^3}\partial^\alpha( \partial_jm_{ij}-\partial_i m_{jj} (1-\delta_{ij}))\partial_t\partial^\alpha \mathcal{B}_i dx+
	\sum_{\mu\in D}\|\partial^\alpha (l_{\mu}+G_{\mu})\|_{L^2_x}. 
\end{multline*}
We define the interaction functional 
\begin{equation}\label{interaction.b}
I^\alpha_b(t)=\sum_{i=1}^{3}\sum_{j=1}^3\int_{\mathbb{T}^3}\partial^\alpha(\partial_j m_{ij}-\partial_i m_{jj}(1-\delta_{ij}))\partial^\alpha \mathcal{B}_idx.
\end{equation}
We also use the second equation of \eqref{Conservation Laws}$_2$ to obtain  for any $\zeta>0$ that
	\begin{align*}
	\sum_{i=1}^3\sum_{j=1}^3\int_ {\mathbb{T}^3}&|\partial^\alpha( \partial_jm_{ij}-\partial_i m_{jj}(1-\delta_{ij}))\partial_t\partial^\alpha \mathcal{B}_i(t,x)| dx\\
	&\leq \zeta( \|\nabla\cdot \partial^\alpha \mathcal{A}\|^2_{L^2_x}+\|\nabla\cdot \partial^\alpha \mathcal{C}\|^2_{L^2_x})+C_\zeta \|\{I-P\}\nabla\partial^\alpha f\|^2_{L^2_{\frac{\singS+\gamma}{2}}}.
	\end{align*}
	Together with Lemma \ref{L8.6} and Lemma \ref{L8.7}, we obtain for $\zeta'>0$ sufficiently small that 
	\begin{equation}
	\label{mainb}
	\begin{split}
&	\sum_{|\alpha|\le N-1}	\left(\|\nabla \partial^\alpha \mathcal{B}\|^2_{L^2_x}-\zeta' (\|\nabla\cdot \partial^\alpha \mathcal{A}\|^2_{L^2_x}+\|\nabla\cdot \partial^\alpha \mathcal{C}\|^2_{L^2_x})\right)
\\
    \lesssim& \ C_{\zeta'} \sum_{|\alpha'|\leq N} \|\{I-P\}\partial^{\alpha'} f\|^2_{L^2_{\frac{\singS+\gamma}{2}}}+\sum_{|\alpha|\le N-1}	\frac{dI^\alpha_b}{dt}+M_0\sum_{|\alpha'|\leq N}\|\partial^{\alpha'} f\|^2_{L^2_{\frac{\singS+\gamma}{2}}}.
	\end{split}
	\end{equation}
This is our main estimate for the 	$\mathcal{B}$ terms.
	
Now we first consider the lower bounds in  the estimates of \eqref{maina}, \eqref{mainc}, and \eqref{mainb}. 
We multiply the lower bound in \eqref{mainc} by a small constant $\epsilon>0$ and then add to it the lower bounds in \eqref{maina} and \eqref{mainb} to obtain
\begin{multline*}
    \sum_{|\alpha|\le N-1}	
    \left(\|\nabla \partial^\alpha \mathcal{A}\|^2_{L^2_x}
    +
    \|\nabla \partial^\alpha \mathcal{B}\|^2_{L^2_x}
    +
    \epsilon \|\nabla \partial^\alpha \mathcal{C}\|^2_{L^2_x}
    \right)
\\
    	-\sum_{|\alpha|\le N-1}	\left(\zeta \|\nabla\cdot \partial^\alpha \mathcal{B}\|^2_{L^2_x}+\zeta' (\|\nabla\cdot \partial^\alpha \mathcal{A}\|^2_{L^2_x}+\|\nabla\cdot \partial^\alpha \mathcal{C}\|^2_{L^2_x})
    	+
    	\epsilon\lambda \|\nabla\cdot \partial^\alpha \mathcal{B}\|^2_{L^2_x}\right)
    	\\
\approx 
    \sum_{|\alpha|\le N-1}	
    \left(\|\nabla \partial^\alpha \mathcal{A}\|^2_{L^2_x}
    +
    \|\nabla \partial^\alpha \mathcal{B}\|^2_{L^2_x}
    +
    \epsilon \|\nabla \partial^\alpha \mathcal{C}\|^2_{L^2_x}
    \right).
\end{multline*}	
The last line above is obtained by first choosing $\epsilon>0$ small, and second choosing $\zeta>0$ sufficiently small, and lastly choosing  $\zeta'>0$ sufficiently small.  Then  (\ref{maina}), (\ref{mainc}), and (\ref{mainb}) imply that 
	\begin{equation}
	\label{mainabc}
	\begin{split}
	\|\nabla \mathcal{A}\|^2_{H^{N-1}_x}+&\|\nabla \mathcal{B}\|^2_{H^{N-1}_x}+\|\nabla \mathcal{C}\|^2_{H^{N-1}_x}\\\lesssim &\sum_{|\alpha|\leq N} \|\{I-P\}\partial^\alpha f\|^2_{L^2_{\frac{\singS+\gamma}{2}}}+\frac{dI}{dt}
	+M_0\sum_{|\alpha|\leq N}\|\partial^\alpha f\|^2_{L^2_{\frac{\singS+\gamma}{2}}}.
	\end{split}	
	\end{equation}
On the other hand, with the Poincar\'e inequality and Lemma \ref{L8.5}, we obtain that 
	$$
	\|\mathcal{A}\|^2\lesssim \left(\|\nabla \mathcal{A}\|+\left|\int_{\mathbb{T}^3}\mathcal{A}(t,x) dx\right|\right)^2\approx \|\nabla \mathcal{A}\|^2.
	$$
	This same estimate holds for $\mathcal{B}$ and $\mathcal{C}$. Therefore, the inequality (\ref{8.19}) holds and this completes the proof for the theorem.
\end{proof}

We now use this coercive estimate to prove that the local solutions from Theorem \ref{local existence} will be global-in-time solutions by the standard continuity argument. 
We will also prove that the solutions have exponential decay in time for hard interactions \eqref{hard} and polynomial decay in time for soft interactions \eqref{soft}.

Before we go into the proof for the global existence, we would like to mention a coercive lower bound for the linearized Boltzmann collision operator $L$ which also gives the positivity of the operator:
\begin{lemma}
	\label{coercive L} Assume \eqref{angassumption}-\eqref{singS.defin} hold. 
	Then there is a constant $\delta>0$ such that 
	$$
	\langle Lf,f\rangle \geq \delta |\{I-P\}f|^2_{I^{\singS,\gamma}}.
	$$
\end{lemma}

\begin{proof}	
We give the standard proof of Lemma \ref{coercive L} using the method of contradiction. 
In this proof we denote $|g|_{\mathcal{N}}^2 \eqdef \langle \mathcal{N}g,g\rangle$ and $\langle g,g\rangle_{\mathcal{N}} \eqdef \langle \mathcal{N}g,g\rangle$ recalling \eqref{mainPartNorm.Nf}.
Assuming the lemma is false, we obtain a sequence of normalized functions $\{g^n(p)\}_{n \ge 1}$ satisfying that 
$|g^n|_{\mathcal{N}} =1$ for all $n\ge 1$.  By Lemma \ref{llinear} we also have
\begin{equation}\label{zero.lemma}
    \int_{\mathbb{R}^3}g^nJ^{1/2}(p)dp =\int_{\mathbb{R}^3} g^n p_i J^{1/2}(p)dp=\int_{
\mathbb{R}^3}g^n p^0 J^{1/2}(p)dp=0, 
\end{equation}
and for some uniform constant $C>0$ and $\forall n\ge 1$ we have
\begin{equation}\label{l1n}
\langle Lg^n,g^n\rangle =\langle \mathcal{N} g^n,g^n\rangle +\langle \mathcal{K} g^n,g^n\rangle
\le C/n.  
\end{equation}
We denote the weak limit, with respect to the inner product $\langle \cdot
,\cdot \rangle_{\mathcal{N}}$, of $g^n$ (up to a subsequence) by $g^{0}$. Lower semi-continuity of the weak limit implies
$
|g^0|_{\mathcal{N}} \le 1. 
$

From \eqref{defK} and \eqref{defN} we have 
\begin{equation*}
    \langle Lg^n,g^n\rangle 
=
|g^n|_{\mathcal{N}}^2
+\langle \mathcal{K}g^n,g^n\rangle. 
\end{equation*}
By Lemma \ref{Lemma2} and \eqref{compactest}, for any small $\epsilon>0$, we have
\begin{equation}\notag
|\langle \mathcal{K}g^n,g^n\rangle | 
\leq 
\epsilon |g^n|^2_{\mathcal{N}}
+
C_\epsilon |g^n|^2_{L^2(B_{C_\epsilon})}.
\end{equation}
Here we use Lemma \ref{Nfupperboundlemma} and Lemma \ref{Nfcoercivitylemma} to bound 
$\langle \mathcal{N}g,g\rangle \approx |g |^2_{I^{\singS,\gamma}}$.  Now by the fractional-order Rellich–Kondrachov Theorem we have (up to taking a sub-sequence) that $|g^n-g^0|^2_{L^2(B_{C_\epsilon})} \to 0$ as $n\to\infty$.  By first choosing $\epsilon>0$ small and then letting $n\rightarrow \infty$, we
conclude that $\langle \mathcal{K} g^n,g^n\rangle \rightarrow \langle \mathcal{K} g^0,g^0\rangle$.

Letting $n\rightarrow \infty $ in (\ref{l1n}), we have shown that
$$
0=1+\langle \mathcal{K} g^0,g^0\rangle.
$$
Equivalently 
\[
0=\left(1-|g^0|_{\mathcal{N}}^2\right)+\langle Lg^0,g^0\rangle . 
\]
Now both terms are non-negative by Lemma \ref{llinear}.  Hence $|g^0|_{\mathcal{N}}^2=1$ and $\langle
Lg^0,g^0\rangle =0.$  
Again using Lemma \ref{llinear} we have $g^0=P g^0$.  Alternatively, letting $n\rightarrow \infty$ in (\ref{zero.lemma}) we deduce that $g^0=\left( I-P\right)g^0$ or $g^0\equiv 0$;  this contradicts $|g^0|_{\mathcal{N}}^2=1$.  
\end{proof}

Now, we define the dissipation rate $\mathcal{D}_l  $ as
$$
\mathcal{D}_l  =\sum_{|\alpha|\leq N} \|\partial^\alpha f(t)\|^2_{I^{\singS,\gamma}_l  }. 
$$
We will use the energy functional $\mathcal{E}_l  (t)$ to be a high-order norm which satisfies 
\begin{equation}
\label{energy}
\mathcal{E}_l  (t)\approx \sum_{|\alpha|\leq N} \| w^l \partial^\alpha f(t)\|^2_{L^2(\mathbb{T}^3\times \rth)}.
\end{equation}
This functional will be precisely defined during the proof.
Then, we would like to set up the following energy inequality:
\begin{equation}
	\label{weightE}
\frac{d}{dt}\mathcal{E}_l  (t)+\mathcal{D}_l  (t)\leq C\sqrt{\mathcal{E}_l  (t)}\mathcal{D}_l  (t).
\end{equation}  
We will prove this energy inequality and use this to show the global existence. 

\begin{proof}[Proof of Theorem \ref{MAIN}] 
We first prove the energy inequality for the $l=0$ case. We denote $\mathcal{D}\eqdef \mathcal{D}_0$ and $\mathcal{E}\eqdef \mathcal{E}_0$. By the definitions of the interaction functionals defined in Theorem \ref{8.4}, for any $C'>0$ there exists $C''=C''(C')>0$ sufficiently large such that 
$$
\|f(t)\|^2_{L^2_pH^N_x}\leq (C''+1)\|f(t)\|^2_{L^2_pH^N_x}-C'I(t)\lesssim \|f(t)\|^2_{L^2_pH^N_x}.
$$
We remark that $C''$ doesn't depend on $f(t,x,p)$ but only on $C'$ and the structure of $I$. Here we define the energy functional $\mathcal{E}(t)$ as
$$
\mathcal{E}(t)=(C''+1)\|f(t)\|^2_{L^2_pH^N_x}-C'I(t).
$$
Then, the above inequalities show that the definition of $\mathcal{E}$ satisfies  (\ref{energy}).
	
Recall the local existence Theorem \ref{local existence}, and Theorem \ref{8.4}, and choose $M_0\leq 1$ so that both theorems hold. We choose $0<M_1\leq \frac{M_0}{2}$ and consider initial data $\mathcal{E}(0)$ so that
$$
\mathcal{E}(0)\leq M_1<M_0.
$$
From the local existence theorem, we define $T>0$ so that
$$
T=\sup\{t\geq 0| \mathcal{E}(t)\leq 2M_1\}.
$$
	By taking the spatial derivative $\partial^\alpha$ onto the linearized relativistic Boltzmann equation (\ref{Linearized B}), multiplying by $\partial^\alpha f$ and integrating over $(x,p)$, and summing over $\alpha$, we obtain
	\begin{equation}
	\label{last}
	\frac{1}{2}\frac{d}{dt}\|f(t)\|^2_{L^2_pH^N_x}+\sum_{|\alpha|\leq N}(L\partial^\alpha f,\partial^\alpha f)=\sum_{|\alpha|\leq N}(\partial^\alpha\Gamma(f,f),\partial^\alpha f).
	\end{equation} 
	By the estimates from Lemma \ref{Lemma1}, we have
	$$
	\sum_{|\alpha|\leq N}(\partial^\alpha\Gamma(f,f),\partial^\alpha f)\lesssim \sqrt{\mathcal{E}}\mathcal{D}.
$$
	Since our choice of $M_1$ satisfies $\mathcal{E}(t)\leq 2M_1\leq M_0,$ we see that the assumption for Theorem \ref{8.4} is satisfied.
	Then, Theorem \ref{8.4} and Lemma \ref{coercive L} tell us that
	\begin{align*}
	\sum_{|\alpha|\leq N}(L\partial^\alpha f,\partial^\alpha f)&\geq \delta \|\{I-P\}f\|^2_{I^{\singS,\gamma}}\\
	&\geq \frac{\delta}{2}\|\{I-P\}f\|^2_{I^{\singS,\gamma}}+\frac{\delta\delta'}{2}\sum_ {|\alpha|\leq N}\|P\partial^\alpha f\|^2_{I^{\singS,\gamma}}(t)-\frac{\delta C}{2}\frac{dI(t)}{dt}.
	\end{align*}
	Let $\delta''=\min\{\frac{\delta}{2},\frac{\delta\delta'}{2}\}$ and let $C'=\delta C$. Then, we have
	$$
	\frac{1}{2}\frac{d}{dt}\left( \|f(t)\|^2_{L^2_pH^N_x}-C'I(t)\right)+\delta''\mathcal{D}\lesssim \sqrt{\mathcal{E}}\mathcal{D}.
	$$ 
	We multiply  (\ref{last}) by $C''$ 
	and add this onto the last inequality above using the positivity of $L$ to conclude that
	$$
	\frac{d\mathcal{E}(t)}{dt}+\delta''\mathcal{D}(t)\leq C\sqrt{\mathcal{E}(t)}\mathcal{D}(t),
$$
	for some $C>0$. 
	Suppose $M_1=\min\{\frac{\delta''^2}{8C^2},\frac{M_0}{2}\}.$ Then, we have
	\begin{equation}\notag
	\frac{d\mathcal{E}(t)}{dt}+\delta''\mathcal{D}(t)\leq C\sqrt{\mathcal{E}(t)}\mathcal{D}(t)\leq C\sqrt{2M_1}\mathcal{D}(t)\leq \frac{\delta''}{2}\mathcal{D}(t).
	\end{equation}
	Now, we integrate over $t$ for $0\leq t\leq \tau<T$ and obtain
	$$
	\mathcal{E}(\tau)+\frac{\delta''}{2}\int_{0}^{\tau}\mathcal{D}(t)dt\leq \mathcal{E}(0)\leq M_1< 2M_1.
$$	
	Since $\mathcal{E}(\tau)$ is continuous in $\tau$, $\mathcal{E}(\tau)\leq M_1$ if $T<\infty$. This contradicts the definition of $T$ and hence $T=\infty$. This proves the global existence.
	
	If we have $l>0$, we recall Lemma \ref{2.10} and deduce that for some $C>0$ that
	\begin{equation}
	\label{lastlower}
	\left(w^{2l}L\partial^\alpha f,\partial^\alpha f\right)\gtrsim \frac{1}{2}\|\partial^\alpha f\|^2_{I^{\singS,\gamma}_l}-C\|\partial^\alpha f\|^2_{L^2(B_C)}.
	\end{equation}
	We also take the $\partial^\alpha $ derivative on the linearized Boltzmann equation (\ref{Linearized B}), take the inner product with $w^{2l}\partial^\alpha f$, integrate both sides, and use Lemma \ref{Lemma1} to obtain that
	$$
	\sum_{|\alpha|\leq N}\left(\frac{1}{2}\frac{d}{dt}\|w^l \partial^\alpha f(t)\|^2_{L^2}+(w^{2l}L\partial^\alpha f,\partial^\alpha f)\right)\lesssim \sqrt{\mathcal{E}_l(t)}\mathcal{D}_l(t).
$$
	Then we apply the lower bound estimate (\ref{lastlower}). Finally, we add the energy inequality for $l=0$ case multiplied by sufficiently large constant $C'>0$ to obtain
\begin{equation}\label{lastlast}
	\frac{d}{dt}\mathcal{E}_l(t)+\mathcal{D}_l(t)\lesssim \sqrt{\mathcal{E}_l(t)}\mathcal{D}_l(t),
\end{equation}	where we define the energy functional $\mathcal{E}_l$ as $$\mathcal{E}_l(t)\eqdef \frac{1}{2}\sum_{|\alpha|\leq N}\|w^{l}\partial^\alpha f(t)\|^2_{L^2}+C'\mathcal{E}(t).$$ Thus, we obtain the energy inequality for the case $l>0$. 	In the hard-interaction case with \eqref{hard}, note that $\mathcal{E}_l(t)\lesssim \mathcal{D}_l(t)$. This and the equation (\ref{lastlast}) show the exponential time decay for $\mathcal{E}_l(t) \lesssim \mathcal{E}_l(0)$ sufficiently small.
	
On the other hand, in the soft-interaction case \eqref{soft} we do not have $\mathcal{E}_l(t)\lesssim \mathcal{D}_l(t)$ because $\singS+\gamma<0$. Instead, to obtain the rapid polynomial decay we use the interplation technique from \cite{SG-CPDE}.  The inequality that we do have is
$\mathcal{E}_{l+(\singS+\gamma)/4}(t)\lesssim \mathcal{D}_l(t)$.  Using that inequality, we perform the following interpolation for fixed $l\geq |\singS+\gamma|/4$ and $m>0$
\begin{multline} \label{key soft}
\mathcal{E}_l(t)
\lesssim 
\mathcal{E}^{\frac{2m}{2m+|\singS+\gamma|}}_{l+(\singS+\gamma)/4}(t)
\mathcal{E}^{\frac{|\singS+\gamma|}{2m+|\singS+\gamma|}}_{l+m}(t)
\lesssim 
\mathcal{D}^{\frac{2m}{2m+|\singS+\gamma|}}_{l}(t)
\mathcal{E}^{\frac{|\singS+\gamma|}{2m+|\singS+\gamma|}}_{l+m}(t)
\\
\lesssim 
\mathcal{D}^{\frac{2m}{2m+|\singS+\gamma|}}_{l}(t)
\mathcal{E}^{\frac{|\singS+\gamma|}{2m+|\singS+\gamma|}}_{l+m}(0).
\end{multline}
Thus from \eqref{lastlast} we have 
$$ 
\frac{d}{dt} \mathcal{E}_l(t)
+
C_{l,m}
\mathcal{E}^{-\frac{|\singS+\gamma|}{2m}}_{l+m}(0)
\mathcal{E}^{\frac{2m+|\singS+\gamma|}{2m}}_l(t)
\leq 0
$$
for some $C_{l,m}>0$. Thus we have 
$$
\frac{d\left(\mathcal{E}^{-\frac{|\singS+\gamma|}{2m}}_l(t)\right)}{dt}
\geq 
C_{l,m}
\frac{|\singS+\gamma|}{2m}
\mathcal{E}^{-\frac{|\singS+\gamma|}{2m}}_{l+m}(0).
$$	 
By integrating over $[0,t]$ we obtain  
$$
\mathcal{E}^{-\frac{|\singS+\gamma|}{2m}}_l(t)
\geq 
\mathcal{E}^{-\frac{|\singS+\gamma|}{2m}}_l(0)
+
t C_{l,m}
\frac{|\singS+\gamma|}{2m}
\mathcal{E}^{-\frac{|\singS+\gamma|}{2m}}_{l+m}(0).
$$ 
Now we use that $\mathcal{E}_l(0)\lesssim \mathcal{E}_{l+m}(0)$ to conclude the polynomial decay for the soft potentials as in Theorem \ref{MAIN}.  This concludes the proof of our main theorem.
\end{proof}

  In the following section we will establish the relativistic Carleman representation.

\section{Relativistic Carleman representation}\label{CarlemanAppendix}

In this Section we will introduce the relativistic Carleman representation 
for the gain and loss terms which have arisen many times throughout this paper.    We will introduce two Carleman representations in \secref{CarlemanAppendix.Carleman} and \secref{app.dual.rep} that are not the same. The one in \secref{CarlemanAppendix.Carleman} is based on the reduction of the space $\mathbb{R}^3_q\times \mathbb{R}^3_{q'}$ onto the hyperplane $$E^q_{p'-p}=\{q\in\mathbb{R}^3:(p'^\mu-p^\mu)(p_\mu+q_\mu)=0\}$$ via the evaluation of the delta function of the energy-momentum conservation laws.   On the other hand, the Carleman representation that we introduce in \secref{app.dual.rep} is based on the derivation of Hilbert-Schmidt operator via taking the specific choice of the Lorentz transformation \eqref{eq.LT} and this derivation is similar to those introduced in \cite[Appendix A]{MR2728733}. Each derivation has their own advantage; the former one in  \secref{CarlemanAppendix.Carleman} is more appropriate for the case when the unknowns are written in the variables of $p$ and $p'$ only (c.f., \eqref{I2.eta.int} and \eqref{claimc2}), whereas the latter one in \secref{app.dual.rep} is more explicit and is powerful for a general situation especially when we need a dual cancellation estimate (c.f. Proposition \ref{prop:cancellation2}).  
In  \secref{CarlemanAppendix.Carleman} we will prove Lemma \ref{lemma.reduction} which allows us to integrate over the surface of the collisional geometry \eqref{conservation}.  Then in \secref{app.dual.rep} we will prove Lemma \ref{transformation.Lemma.appendix} which allows us to present the dual representation \eqref{dual4} for the trilinear form \mbox{$\langle w^{2l} \Gamma(f,h),\eta\rangle$}  from \eqref{eqn.trilinear.form} and more generally for \eqref{original eq}.

\subsection{Carleman dual representation}\label{CarlemanAppendix.Carleman}

We consider the collision integral from \eqref{omegaint}.
The purpose of this subsection will be to prove Lemma \ref{lemma.carleman}.   By Lemma \ref{lemma.reduction} the integral   \eqref{omegaint} can be written in the following form:
\begin{equation}
\label{original eq}
\int_{\mathbb{R}^3}\frac{dp}{{p^0}}\int_{\mathbb{R}^3}\frac{dq}{{q^0}}\int_{\mathbb{R}^3}\frac{dq'\ }{{q'^0}}\int_{\mathbb{R}^3}
\frac{dp'\ }{{p'^0}}s\sigma(g,\theta)\delta^{(4)}(p'^\mu+q'^\mu-p^\mu-q^\mu)G(p,q,p'),
\end{equation}
where $G$ can be defined to suitably represent \eqref{omegaint}.  More generally we will assume that the function $G$ has a sufficient vanishing condition so that the integral in \eqref{original eq} is well-defined.  We will prove that we can write the integral \eqref{original eq} as one on the set $\mathbb{R}^3\times\mathbb{R}^3\times E^{q}_{p'-p}$ where $E^{q}_{p'-p}$ is the hyperplane 
$$
E^{q}_{p'-p}=\{q\in\mathbb{R}^3:(p'^\mu-p^\mu)(p_\mu+q_\mu)=0\}.
$$
This will be the main result of this subsection.

To this end, we now use Lemma \ref{7.5ofCMP} to rewrite  (\ref{original eq}) as 
$$
\int_{\mathbb{R}^3}\frac{dp}{{p^0}}\int_{\mathbb{R}^3}\frac{dp'}{{p'^0}}B(p,q,p'),
$$
where $B=B(p,q,p')$ is defined as
\begin{multline*}
B =\int_{\mathbb{R}^3}\frac{dq\ }{{q^0}}\int_{\mathbb{R}^3}\frac{dq'\ }{{q'^0}}s\sigma(g,\theta)\delta^{(4)}(p'^\mu+q'^\mu-p^\mu-q^\mu)G(p,q,p')\\
 =\int_{\mathbb{R}^4\times\mathbb{R}^4}d\Theta(q^\mu,q'^\mu)s\sigma(g,\theta)\delta^{(4)}(p'^\mu+q'^\mu-p^\mu-q^\mu)G(p^\mu,q^\mu,p'^\mu)  ,
\end{multline*}
with as in 
$$d\Theta(q^\mu,q'^\mu)\eqdef dq^\mu dq'^\mu u({q'^0})u({q^0})\delta(s-g^2-4)\delta((q^\mu-q'^\mu)(q^\mu+q'^\mu)),$$ and $u(x)$ is defined in \eqref{def.u}. 

Next we apply the following change of variable
$$
\bar{q}^\mu=q'^\mu-q^\mu.
$$
Then with this change of variable the integral becomes
$$
B=\int_{\mathbb{R}^4\times\mathbb{R}^4}d\Theta(\bar{q}^\mu,q^\mu)s\sigma(g,\theta)\delta^{(4)}(p'^\mu+\bar{q}^\mu-p^\mu)G(p^\mu,q^\mu,p'^\mu) , 
$$
where $$d\Theta(\bar{q}^\mu,q^\mu)\eqdef dq^\mu d\bar{q}^\mu u(\bar{q}^0+{q^0})u({q^0})\delta(s-g^2-4)\delta(\bar{q}^\mu(2q^\mu+\bar{q}^\mu)).$$
This change of variables gives us the Jacobian $= 1$. 
Finally we evaluate the delta function, $\delta^{(4)}$, to obtain
$$
B=\int_{\mathbb{R}^4}d\Theta(q^\mu)s\sigma(g,\theta)G(p^\mu,q^\mu,p'^\mu)  ,
$$
where we are now integrating over the four vector $q^\mu$ and
$$d\Theta(q^\mu)=dq^\mu u({p^0}-{p'^0}+{q^0})u({q^0})\delta(s-g^2-4)\delta((p^\mu-p'^\mu)(2q_\mu+p_\mu-p'_\mu)).$$
Here, we note that 
\begin{multline*} \int_{\mathbb{R}} dq^0 u(q^0) \delta(s-g^2 -4)\\
=\int_{\mathbb{R}} dq^0 u(q^0) \delta(-(p^\mu+q^\mu)(p_\mu+q_\mu)-(p^\mu-q^\mu)(p_\mu-q_\mu) -4)\\
=\int_{\mathbb{R}} dq^0 u(q^0) \delta(-(p^\mu+q^\mu)(p_\mu+q_\mu)-(p^\mu-q^\mu)(p_\mu-q_\mu) -4)\\
=\int_{\mathbb{R}} dq^0 u(q^0) \delta(2-2q^\mu q_\mu -4)\\
=\int_{\mathbb{R}} dq^0 u(q^0) \delta(2(q^0-\sqrt{1+|q|^2})(q^0+\sqrt{1+|q|^2}))\\=\frac{1}{4}\int_{\mathbb{R}} \frac{dq^0}{\sqrt{1+|q|^2}} u(q^0) \left(\delta(q^0-\sqrt{1+|q|^2})+\delta(q^0+\sqrt{1+|q|^2})\right)\\
=\frac{1}{4}\int_{\mathbb{R}} \frac{dq^0}{\sqrt{1+|q|^2}} u(q^0) \delta(q^0-\sqrt{1+|q|^2}).
\end{multline*}
We thus conclude that the integral is given by 
\begin{equation}
\label{E2}
B=\int_{E^{q}_{p'-p}} \frac{d\pi_{q}}{8\bar{g}{q^0}}s\sigma(g,\theta)G(p,q,p'),
\end{equation}
where  
$d\pi_{q}=dq\ u({p^0}+q^0-{p'^0})\delta\left(\frac{\bar{g}}{2}+\frac{q^\mu(p_\mu-p'_\mu)}{\bar{g}}\right)$ with $q^0\eqdef \sqrt{1+|q|^2}$. 
This is a $2$-dimensional surface measure on the hypersurface $E^{q}_{p'-p}$ in $\mathbb{R}^3$.

\subsection{Dual representation for a trilinear term}\label{app.dual.rep}
In this section we will prove Lemma \ref{transformation.Lemma.appendix}.  
For concreteness, we focus on the derivation of the Carleman dual representation of the trilinear term 
\mbox{$\langle w^{2l} \Gamma(f,h),\eta\rangle$}.  We explain how to generalize to the full proof of Lemma \ref{transformation.Lemma.appendix} at the end.  

After applying the pre-post change of variables, as in 
\eqref{prepost.change},
to the $T^l_+$ part of 
\eqref{eqn.trilinear.form},
then this term is given by
\begin{multline}\label{eq.Igainloss1}
I=\langle w^{2l}\Gamma(f,h),\eta\rangle
=\int_\rth dp \int_\rth dq \int_{\mathbb{S}^2}d\omega \ v_{\text{\o}} \sigma(g,\theta)f(q)h(p)\\\times\left(w^{2l}(p')\eta(p')\sqrt{J(q')}-w^{2l}(p)\eta(p)\sqrt{J(q)}\right) \eqdef I_{gain}-I_{loss}.
\end{multline}
We initially suppose that $\int_{\mathbb{S}^2} d\omega\   |\sigma_0(\cos\theta)| <\infty$ and that
$$
\int_{\mathbb{S}^2} d\omega\   \sigma_0(\cos\theta)=0.
$$
Then, under that condition, the loss term vanishes $I_{loss}=0$ and we obtain
\begin{equation}\notag
I=I_{gain}=\int_\rth dp \int_\rth dq \int_{\mathbb{S}^2}d\omega \ v_{\text{\o}} \sigma(g,\theta)f(q)h(p)w^{2l}(p')\eta(p')\sqrt{J(q')}.
\end{equation}
By applying Lemma \ref{lemma.reduction}
we obtain another representation of $I$:
\begin{multline}\label{Igainrepresentation}
I= \int_{\rth}\frac{dp}{p^0}\int_{\rth}\frac{dq}{q^0}\int_{\rth}\frac{dp'}{p'^0}\int_{\rth}\frac{dq'}{q'^0}s\sigma(g,\theta)\delta^{(4)}(p'^\mu+q'^\mu-p^\mu-q^\mu)\\\times f(q)h(p)w^{2l}(p')\eta(p')\sqrt{J(q')}.
\end{multline}
Here from \eqref{g}, \eqref{FREQ:s.ge.g2}, \eqref{gbar}, and \eqref{gtilde} we have $g=g(p^\mu,q^\mu)$, $s=g^2+4$, $\bar{g}\eqdef g(p^\mu,p'^\mu)=g(q^\mu,q'^\mu)$, and $\tilde{g} = g(p^\mu,q'^\mu)$.   Also from \eqref{triangle.id} and \eqref{bargoverg} we have
$$
\cos\theta=2\frac{\tilde{g}^2}{g^2}-1.
$$
We further {\it claim} that 
\begin{equation}\label{gg}
g^2=\tilde{g}^2-\frac{1}{2}(p^\mu+q'^\mu)(p'_\mu+q_\mu-p_\mu-q'_\mu).
\end{equation} 
Recall that $\tilde{s}\eqdef \tilde{g}^2+4$. Then, using \eqref{g} and \eqref{s}, (\ref{gg}) is  equivalent to
\begin{multline*}
g^2=\tilde{g}^2-\frac{1}{2}\tilde{s}-\frac{1}{2}(p^\mu+q'^\mu)(p'_\mu+q_\mu)
=\frac{1}{2}\tilde{g}^2-2-\frac{1}{2}(p^\mu+q'^\mu)(p'_\mu+q_\mu)\\
=\frac{1}{2}\tilde{g}^2+g^2+2p^\mu q_\mu-\frac{1}{2}(p^\mu+q'^\mu)(p'_\mu+q_\mu).
\end{multline*}
Thus we prove (\ref{gg}) by showing that
$$\frac{1}{2}\tilde{g}^2+2p^\mu q_\mu-\frac{1}{2}(p^\mu+q'^\mu)(p'_\mu+q_\mu)=0.$$
Expanding the left-hand side of this equation, we obtain
$$
-p^\mu q'_\mu-1+2p^\mu q_\mu-\frac{1}{2}p^\mu p'_\mu-\frac{1}{2}q'^\mu p'_\mu-\frac{1}{2}p^\mu q_\mu-\frac{1}{2}q'^\mu q_\mu.$$
 Therefore, using \eqref{eq.pqp'q'} we obtain $$
-1+p^\mu q_\mu-\frac{1}{2}p^\mu p'_\mu-\frac{1}{2}p^\mu q'_\mu-\frac{1}{2}p'^\mu q_\mu-\frac{1}{2}q'^\mu q_\mu,$$ 
which by \eqref{conservation} is equal to
$$-1+p^\mu q_\mu-\frac{1}{2}(p^\mu+q^\mu)(p'_\mu+q'_\mu)=-1+p^\mu q_\mu+\frac{1}{2}s=0.$$ This finishes the proof of the  {\it claim} \eqref{gg}.

\begin{remark}\label{comparison.gg}
Combining \eqref{gg} and \eqref{triangle.id}, we see that  $\bar{g}^2$ can be represented as
$$\bar{g}^2=-\frac{1}{2}(p^\mu+q'^\mu)(p'_\mu+q_\mu-p_\mu-q'_\mu).$$ 
In the rest of this section we will use this representation and follow the formula for $\bar{g}^2$ as we perform the changes of variables below.
\end{remark}

Then exchanging $p$ and $p'$ in \eqref{Igainrepresentation}, we have
\begin{multline*}I= \int_{\rth}\frac{dp'}{p'^0}\int_{\rth}\frac{dq}{q^0}\int_{\rth}\frac{dp}{p^0}\int_{\rth}\frac{dq'}{q'^0}\tilde{s}\sigma(\tilde{g},\theta')\delta^{(4)}(p^\mu+q'^\mu-p'^\mu-q^\mu)\\\times f(q)h(p')w^{2l}(p)\eta(p)\sqrt{J(q')},
\end{multline*}
where the angle $\theta'$ is now redefined as
\begin{equation}\label{angle.theta.prime}
    \cos\theta'\eqdef 2\frac{g^2}{\tilde{g}^2}-1,
\end{equation}
and from Remark \ref{comparison.gg} and \eqref{gg} 
we have
\begin{equation}
\notag	\tilde{g}^2=g^2+\bar{g}^2,
\quad 
\bar{g}^2=-\frac{1}{2}(p'^\mu+q'^\mu)(p_\mu+q_\mu-p'_\mu-q'_\mu).
\end{equation} 
And we further use $\tilde{s}\eqdef \tilde{g}^2+4$.  As we change variables below we will refer to the transformed $\bar{g}$ as $g_L$.

We now define the functional $i(p,q)$ as
\begin{equation}\label{i}
i(p,q)\eqdef \frac{1}{p^0q^0} \int_{\rth}\frac{dp'}{p'^0}\int_{\rth}\frac{dq'}{q'^0}\tilde{s}\sigma(\tilde{g},\theta')\delta^{(4)}(p^\mu+q'^\mu-p'^\mu-q^\mu)h(p')\sqrt{J(q')},
\end{equation} so that we have
\begin{equation}\label{Ipq}
I=\int_\rth\int_\rth i(p,q)f(q)w^{2l}(p)\eta(p)dqdp.\end{equation}
We first translate (\ref{i}) into an expression involving the total and relative momentum
variables, $p'^\mu+q'^\mu$ and $p'^\mu-q'^\mu$ respectively. Define $u$ as in \eqref{def.u}. Let $\underline{g}\eqdef g(p'^\mu,q'^\mu)$ and $\underline{s}\eqdef s(p'^\mu,q'^\mu).$ Then by Lemma \ref{7.5ofCMP} we have
$$i(p,q)=\frac{1}{16p^0q^0}\int_{\rfo\times\rfo}d\Theta(p'^\mu,q'^\mu)h(p')\sqrt{J(q')}\tilde{s}\sigma(\tilde{g},\theta')\delta^{(4)}(p'^\mu+q^\mu-p^\mu-q'^\mu),$$
where
$$d\Theta(p'^\mu,q'^\mu)\eqdef dp'^\mu dq'^\mu u(p'^0+q'^0)u(\underline{s}-4)\delta(\underline{s}-\underline{g}^2-4)\delta((p'^\mu+q'^\mu)(p'_\mu-q'_\mu)).$$
Thus we have lifted to an integral over $\rfo\times\rfo$ from one over $\rth\times\rth$.

Now we apply the change of variables $\bar{p}^\mu=p'^\mu+q'^\mu$ and $\bar{q}^\mu=p'^\mu-q'^\mu$. Then the Jacobian is 16. Since $q'=\frac{\bar{p}-\bar{q}}{2}$ and $p'=\frac{\bar{p}+\bar{q}}{2}$, we have
\begin{equation}\notag
i(p,q)= \frac{c'}{p^0q^0}\int_{\rfo\times\rfo}d\Theta(\bar{p}^\mu,\bar{q}^\mu)\tilde{s}\sigma(\tilde{g},\theta')\delta^{(4)}(q^\mu-p^\mu+\bar{q}^\mu)h\left(\frac{\bar{p}+\bar{q}}{2}\right)e^{\frac{-\bar{p}^0+\bar{q}^0}{4}},
\end{equation} for some constant $c'>0$ (whose value below can change from line to line), where
$$d\Theta(\bar{p}^\mu,\bar{q}^\mu)\eqdef d\bar{p}^\mu d\bar{q}^\mu u(\bar{p}^0)u(-\bar{p}^\mu\bar{p}_\mu-4)\delta(-\bar{p}^\mu\bar{p}_\mu-\bar{q}^\mu\bar{q}_\mu-4)\delta(\bar{p}^\mu\bar{q}_\mu).$$
We now carry out $\delta^{(4)}(q^\mu-p^\mu+\bar{q}^\mu)$ to obtain
\begin{equation}\notag
i(p,q)= \frac{c'}{p^0q^0}\int_{\rfo}d\Theta(\bar{p}^\mu)\tilde{s}\sigma(\tilde{g},\theta')h\left(\frac{\bar{p}+p-q}{2}\right)\exp\left(\frac{-\bar{p}^0+p^0-q^0}{4}\right),
\end{equation}
where the measure $d\Theta(\bar{p}^\mu)$ is now equal to
$$d\Theta(\bar{p}^\mu)\eqdef d\bar{p}^\mu  u(\bar{p}^0)u(-\bar{p}^\mu\bar{p}_\mu-4)\delta(-\bar{p}^\mu\bar{p}_\mu-g^2-4)\delta(\bar{p}^\mu(p_\mu-q_\mu)).$$
Since $s=g^2+4$ from \eqref{s}, we have
\begin{multline*}
u(\bar{p}^0)\delta(-\bar{p}^\mu\bar{p}_\mu-g^2-4)=u(\bar{p}^0)\delta(-\bar{p}^\mu\bar{p}_\mu-s)=u(\bar{p}^0)\delta((\bar{p}^0)^2-|\bar{p}|^2-s)\\=\frac{\delta(\bar{p}^0-\sqrt{|\bar{p}|^2+s})}{2\sqrt{|\bar{p}|^2+s}}.
\end{multline*}
Then we carry out one integration using this delta function to obtain
\begin{multline}\notag
i(p,q)= \frac{c'}{2p^0q^0}\int_{\rth}\frac{d\bar{p}}{\bar{p}^0}u(-\bar{p}^\mu\bar{p}_\mu-4)\delta(\bar{p}^\mu(p_\mu-q_\mu))\tilde{s}\sigma(\tilde{g},\theta')\\\times h\left(\frac{\bar{p}+p-q}{2}\right)\exp\left(\frac{-\sqrt{|\bar{p}|^2+s}+p^0-q^0}{4}\right),
\end{multline}
where $\bar{p}^0=\sqrt{|\bar{p}|^2+s}$. Using $s=g^2+4$ again, we have $$-\bar{p}^\mu\bar{p}_\mu-4=s-4=g^2\geq 0$$ to guarantee that $u(-\bar{p}^\mu\bar{p}_\mu-4)=1$. Thus
\begin{multline}\notag
i(p,q)= \frac{c'}{2p^0q^0}\exp\left(\frac{p^0-q^0}{4}\right)\int_{\rth}\frac{d\bar{p}}{\bar{p}^0}\delta(\bar{p}^\mu(p_\mu-q_\mu))\tilde{s}\sigma(\tilde{g},\theta')\\\times h\left(\frac{\bar{p}+p-q}{2}\right)\exp\left(\frac{-\sqrt{|\bar{p}|^2+s}}{4}\right),
\end{multline}
where $\bar{p}^0=\sqrt{|\bar{p}|^2+s}$. In this representation we have that the angle $\theta'$ is still given by \eqref{angle.theta.prime} but now we have
\begin{equation} \notag	
\tilde{g}^2=g^2+g_L^2,
\quad 
g_L^2=-\frac{1}{2}\bar{p}^\mu (p_\mu+q_\mu-\bar{p}_\mu)
=-\frac{1}{2}\bar{p}^\mu (p_\mu+q_\mu) -\frac{1}{2} s.
\end{equation} 
And again $\tilde{s}\eqdef \tilde{g}^2+4$.

We finish off our reduction by moving to a new Lorentz frame.  We choose the Lorentz transformation from \eqref{eq.LT} which importantly satisfies the condition \eqref{center.momentum.condition}. Then, using the change of variables \eqref{eq.LT}, with $U^\mu=(1,0,0,0)^\top$, we have
\begin{multline*}
\int_{\rth}\frac{d\bar{p}}{\bar{p}^0}\delta(\bar{p}^\mu(p_\mu-q_\mu))\tilde{s}\sigma(\tilde{g},\theta')h\left(\frac{\bar{p}+p-q}{2}\right)e^{\left(\frac{\bar{p}^\mu U_\mu}{4}\right)}\\
=\int_{\rth}\frac{d\bar{p}}{\bar{p}^0}\delta(\bar{p}^\mu B_\mu)s_\Lambda\sigma(g_\Lambda,\theta_\Lambda)h\left(\frac{((\Lambda^{-1})^\nu{}_\mu\bar{p}^\mu)_{1\le \nu \le 3}+p-q}{2}\right)e^{\left(\frac{\bar{p}^\mu \bar{U}_\mu}{4}\right)}.
\end{multline*}
We used that $\frac{d\bar{p}}{\bar{p}^0}$ is Lorentz invariant.
Here $\bar{p}^0=\sqrt{|\bar{p}|^2+s}$ and $s_\Lambda$, $g_\Lambda\geq0$ are
$$
g^2_\Lambda\eqdef g^2 + g_L^2, \quad g_L^2 = -\frac{1}{2}\bar{p}^\mu A_\mu-\frac{1}{2} s= \frac{1}{2}\sqrt{s}(\bar{p}^0-\sqrt{s}),
$$
where
\begin{equation}\label{slambda.def}
s_\Lambda\eqdef g^2_\Lambda+4,
\end{equation}
and
\begin{equation}\label{coslam}
\cos\theta_\Lambda\eqdef 2\frac{g^2}{g^2_\Lambda}-1.
\end{equation}
Also, $\bar{U}^\mu$ is given by
$\bar{U}^\mu = \Lambda^{\mu}{}_\nu U^\mu
= \left(\frac{p^0+q^0}{\sqrt{s}},\frac{2|p\times q|}{g\sqrt{s}},0,\frac{p^0-q^0}{g}\right)$.

We now switch to polar coordinates in the form $$d\bar{p}=r^2 dr \sin\psi d\psi d\phi,\hspace{5mm} \bar{p}\eqdef r(\sin \psi \cos \phi, \sin \psi \sin \phi, \cos \psi).$$
Then we obtain $$\bar{p}^\mu B_\mu = gr\cos\psi.$$
Then the integral $i(p,q)$ is now equal to
\begin{multline}\notag
i(p,q)= \frac{c'}{2p^0q^0}\exp\left(\frac{p^0-q^0}{4}\right)\int_{0}^{2\pi}d\phi \int_{0}^{\pi}d\psi \sin\psi \\\times\int_{0}^{\infty} \frac{r^2dr}{\sqrt{r^2+s}}\delta(gr\cos\psi)s_\Lambda\sigma(g_\Lambda,\theta_\Lambda)h\left(\frac{((\Lambda^{-1})^\nu{}_\mu\bar{p}^\mu)_{1\le \nu \le 3}+p-q}{2}\right) e^{\left(\frac{\bar{p}^\mu \bar{U}_\mu}{4}\right)}.
\end{multline}
We evaluate the last delta function at $\psi = \pi/2$ to write $i(p, q)$ as
\begin{multline}\notag
i(p,q)= \frac{c'}{2gp^0q^0}\exp\left(\frac{p^0-q^0}{4}\right) \int_{0}^{2\pi}d\phi \int_{0}^{\infty} \frac{rdr}{\sqrt{r^2+s}}s_\Lambda\sigma(g_\Lambda,\theta_\Lambda) \\\times h\left(\frac{((\Lambda^{-1})^\nu{}_\mu\bar{p}^\mu)_{1\le \nu \le 3}+p-q}{2}\right)\exp\left(-\bar{p}^0\frac{p^0+q^0}{4\sqrt{s}}+\frac{|p\times q|}{2g\sqrt{s}}r\cos\phi\right),
\end{multline}
where
$$\bar{p}^\mu = (\sqrt{r^2+s},r\cos\phi,r\sin\phi,0).$$
Then we have for $(i=1,2,3)$ that
\begin{multline*}
    ((\Lambda^{-1})^\nu{}_\mu\bar{p}^\mu)_{\nu=i}= (\Lambda_\mu{}^\nu\bar{p}^\mu)_{\nu=i} =\sum_{\mu = 0}^3 \Lambda_\mu{}^i\bar{p}^\mu \\
    =\frac{p_i+q_i}{\sqrt{s}}\sqrt{r^2+s} +\Lambda^1{}_i ~r\cos\phi+\frac{(p\times q)_i}{|p\times q|}r\sin\phi\eqdef a_i,
\end{multline*}and we have
\begin{multline*}
    (\Lambda^{-1})^0{}_\mu\bar{p}^\mu= \Lambda_\mu{}^0\bar{p}^\mu =\sum_{\mu = 0}^3 \Lambda_\mu{}^0\bar{p}^\mu 
    =-\frac{p^0+q^0}{\sqrt{s}}\sqrt{r^2+s} +\Lambda^1{}_0 ~r\cos\phi\\=-\frac{p^0+q^0}{\sqrt{s}}\sqrt{r^2+s} +\frac{2|p\times q|}{g\sqrt{s}} ~r\cos\phi\eqdef a^0.
\end{multline*}
Define $x=(x_1,x_2)\eqdef(r\cos\phi,r\sin\phi),$ and denote \begin{equation}\label{azero.notation.def}
    a^0(p,q,x)=a^0(p,q,x)=a^0,
\end{equation}and
\begin{equation}\label{a.notation.def}
    a(p,q,x)=(a_1(p,q,x),a_2(p,q,x),a_3(p,q,x))=(a_1,a_2,a_3).
\end{equation} Then we note that $$ a^\mu a_\mu= -(a^0)^2+|a|^2 =-s ,$$ since $\Lambda$ is a Lorentz transformation. 
Also note that $$a(p,q,0)=(p+q)\text{ and }a^0(p,q,0)=-(p^0+q^0).$$
Then we have
\begin{multline}\notag
i(p,q)= \frac{c'}{2gp^0q^0}\exp\left(\frac{p^0-q^0}{4}\right) \int_{\mathbb{R}^2} \frac{dx}{\sqrt{|x|^2+s}}s_\Lambda\sigma(g_\Lambda,\theta_\Lambda) \\\times h\left(\frac{a(p,q,x)+p-q}{2}\right)\exp\left(-\frac{p^0+q^0}{4\sqrt{s}}\sqrt{|x|^2+s}+\frac{|p\times q|}{2g\sqrt{s}}x_1\right).
\end{multline}
Now, with \eqref{slambda.def} and \eqref{coslam}, $g_\Lambda\geq0$ is given by
\begin{equation}\label{g2}
g^2_\Lambda = g^2+g_L^2, \quad g_L^2 = \frac{1}{2}\sqrt{s}(\sqrt{|x|^2+s}-\sqrt{s}).
\end{equation}
So by \eqref{Ipq} we obtain a new representation of our gain term
\begin{multline}\label{final gain}
I=I_{gain}=\frac{c'}{2}\iint_{\mathbb{R}^6}\frac{dp}{p^0} \frac{dq}{q^0}\exp\left(\frac{p^0-q^0}{4}\right) \int_{\mathbb{R}^2} \frac{dx}{g\sqrt{|x|^2+s}}s_\Lambda\sigma(g_\Lambda,\theta_\Lambda) w^{2l}(p)\eta(p)\\  \times f(q) h\left(\frac{a(p,q,x)+p-q}{2}\right)\exp\left(-\frac{p^0+q^0}{4\sqrt{s}}\sqrt{|x|^2+s}+\frac{|p\times q|}{2g\sqrt{s}}x_1\right).
\end{multline}
This completes our transformation of the gain term $I_{gain}$.

We now return to the loss term $I_{loss}$.  We recall that $I_{loss}=0$ under the assumptions that $\int_{\mathbb{S}^2} d\omega |\sigma_0(\cos\theta)| <\infty$ and $\int_{\mathbb{S}^2} d\omega \sigma_0(\cos\theta)=0$.  We will now find a different expression than $I_{loss}$ which also integrates to zero under the same conditions, that will provide suitable cancellation for the term in (\ref{final gain}) even when we no longer assume that $\int_{\mathbb{S}^2} d\omega |\sigma_0(\cos\theta)| <\infty$ and $\int_{\mathbb{S}^2} d\omega \sigma_0(\cos\theta)=0$.

To this end, we recall the definition \eqref{coslam}.  This using \eqref{g2} we have
\begin{equation}\label{newcos}
\cos\theta_\Lambda=\frac{2g^2}{g^2_\Lambda}-1=\frac{g^2-\frac{1}{2}\sqrt{s}(\sqrt{r^2+s}-\sqrt{s})}{g^2+\frac{1}{2}\sqrt{s}(\sqrt{r^2+s}-\sqrt{s})}.
\end{equation}
Differentiating $\cos\theta_\Lambda$ with respect to $r$, we have
$$
\frac{d(\cos\theta_\Lambda)}{dr}=-\frac{g^2\sqrt{s}r}{g^4_\Lambda\sqrt{r^2+s}}.
$$  Since we have assumed that $\int_{-1}^1d(\cos\theta_\Lambda) \sigma_0(\cos\theta_\Lambda)=0,$ then we further have
$$
\int_{0}^{\infty}dr \frac{g^2\sqrt{s}r}{g^4_\Lambda\sqrt{r^2+s}}\sigma_0(\cos\theta_\Lambda)=0.
$$ 
This follows since $\cos\theta_\Lambda = 1\text{ and }-1$ correspond to $r=0$ and $r=\infty$ respectively, using \eqref{newcos}.  

Thus we obtain
\begin{multline}
\notag
\frac{c'}{2}\int_\rth\frac{dp}{p^0}\int_{\rth}\frac{dq}{q^0}\exp\left(\frac{p^0-q^0}{4}\right)\int_{0}^{2\pi}d\phi \int_{0}^\infty \frac{rdr}{g\sqrt{r^2+s}}s\Phi(g)\sigma_0(\cos \theta_\Lambda)\frac{g^4}{g^4_\Lambda}\\\times w^{2l}(p)\eta(p)f(q)h(p)\exp\left(-\frac{p^0+q^0}{4}\right)=0.
\end{multline}
The above is the same as 
\begin{multline}
\notag
\frac{c'}{2}\int_\rth\frac{dp}{p^0}\int_{\rth}\frac{dq}{q^0}\exp\left(\frac{p^0-q^0}{4}\right)
\int_{\mathbb{R}^2} \frac{dx}{g\sqrt{|x|^2+s}}s\Phi(g)\sigma_0(\cos \theta_\Lambda)\frac{g^4}{g^4_\Lambda}\\\times w^{2l}(p)\eta(p)f(q)h(p)\exp\left(-\frac{p^0+q^0}{4}\right)=0.
\end{multline}
Subtracting this zero integral from (\ref{final gain}), we further obtain
\begin{multline}\label{dual}
I=\frac{c'}{2}\int_\rth\frac{dp}{p^0}\int_{\rth}\frac{dq}{q^0}\exp\left(\frac{p^0-q^0}{4}\right) \int_{\mathbb{R}^2} \frac{dx}{g\sqrt{|x|^2+s}}s_\Lambda\sigma(g_\Lambda,\theta_\Lambda)w^{2l}(p)\eta(p)f(q)\\\times\Big[h\left(\frac{a(p,q,x)+p-q}{2}\right)\exp\left(-\frac{p^0+q^0}{4\sqrt{s}}\sqrt{|x|^2+s}+\frac{|p\times q|}{2g\sqrt{s}}x_1\right) \\-\frac{s \Phi(g)g^4}{s_\Lambda \Phi(g_\Lambda)g^4_\Lambda}h(p)\exp\left(-\frac{p^0+q^0}{4}\right)\Big].
\end{multline}
This is equal to the original integral $I$ when the mean value of $\sigma_0$ is zero. 

Since we are working with the Schwartz functions, by a standard approximation argument we can directly prove that \eqref{dual} also holds even when the mean value of $\sigma_0$ is not zero and $\sigma_0$ is not integrable such as in \eqref{define.kernel} with \eqref{angassumption} and \eqref{hard} or \eqref{soft}.  We refer to \cite[Appendix A]{MR2784329} and \eqref{FREQ:dual} for full details of analogous approximation arguments.
 
 Now by making the change of variables $x\mapsto z=\frac{x}{\sqrt{s}}$ with $dz=s^{-1}dx,$ we have
 \begin{multline}\notag
I=\frac{c'}{2}\int_\rth\frac{dp}{p^0}\int_{\rth}\frac{dq}{q^0} \frac{\sqrt{s}}{g}\int_{\mathbb{R}^2} \frac{dz}{\sqrt{|z|^2+1}}s_\Lambda\sigma(g_\Lambda,\theta_\Lambda)\sqrt{J(q)}w^{2l}(p)\eta(p)f(q)\\\times\bigg[h\left(\frac{a(p,q,\sqrt{s}z)+p-q}{2}\right)\exp\left(-\frac{p^0+q^0}{4}(\sqrt{|z|^2+1}-1)+\frac{|p\times q|}{2g}z_1\right)\\
-\frac{s \Phi(g)g^4}{s_\Lambda \Phi(g_\Lambda)g^4_\Lambda}h(p)\bigg].
\end{multline}
Next we recover the original variables by relabelling $p$ and $p'$ above, we then have 
\begin{multline}\notag
\langle w^{2l}\Gamma(f,h),\eta\rangle=I\\
=
\frac{c'}{2}\int_\rth\frac{dp'}{p'^0}\int_{\rth}\frac{dq}{q^0} \frac{\sqrt{\tilde{s}}}{\tilde{g}}\int_{\mathbb{R}^2} \frac{dz}{\sqrt{|z|^2+1}}s_\Lambda\sigma(g_\Lambda,\theta_\Lambda)\sqrt{J(q)}w^{2l}(p')\eta(p')f(q)\\\times\bigg[h\left(\frac{a(p',q,\sqrt{\tilde{s}}z)+p'-q}{2}\right)\exp\left(-\frac{p'^0+q^0}{4}(\sqrt{|z|^2+1}-1)+\frac{|p'\times q|}{2\tilde{g}}z_1\right) \\-\frac{\tilde{s} \Phi(\tilde{g})\tilde{g}^4}{s_\Lambda \Phi(g_\Lambda)g^4_\Lambda}h\left(
p'\right)\bigg],
\end{multline}
where $g_\Lambda,\ s_\Lambda,$ and $\theta_\Lambda$ are defined as\begin{equation}\label{g2.eq.lambda}
g^2_\Lambda = \tilde{g}^2+g_L^2, \quad g_L^2 = \frac{1}{2}\tilde{s}(\sqrt{|z|^2+1}-1),
\end{equation}
\begin{equation}\label{theta.eq.lambda}
\cos\theta_\Lambda\eqdef 2\frac{\tilde{g}^2}{g^2_\Lambda}-1,\text{ and }s_\Lambda=g_\Lambda^2+4.
\end{equation}
Here we also have from \eqref{azero.notation.def} and \eqref{a.notation.def} that$$a^0(p',q,\sqrt{\tilde{s}}z)=-(p'^0+q^0)\sqrt{|z|^2+1} +\Lambda^{1}{}_0\sqrt{\tilde{s}}z_1,$$ and
$$a(p',q,\sqrt{\tilde{s}}z)=(p'+q)\sqrt{|z|^2+1} +\Lambda^{1}\sqrt{\tilde{s}}z_1+\frac{p'\times q}{|p'\times q|}\sqrt{\tilde{s}}z_2,$$
where 
$\Lambda^{1}=\Lambda^1(p',q)=(\Lambda^{1}{}_{1}(p',q),\Lambda^{1}{}_{2}(p',q),\Lambda^{1}{}_{3}(p',q))$ from \eqref{eq.LT}.
Further define 
\begin{multline}\label{Azero.def.lambda}
-A^0=-A^0(p',q,z)=\frac{a^0(p',q,\sqrt{\tilde{s}}z)+(p'^0+q^0)}{2}\\
=-\frac{(p'^0+q^0)}{2}(\sqrt{|z|^2+1}-1) +
\frac{|p'\times q|}{\tilde{g}}
z_1
=-2l(\sqrt{|z|^2+1}-1) +
2j
z_1,\end{multline}and
\begin{multline}\label{A.def.lambda}
A=A(p',q,z)=\frac{a(p',q,\sqrt{\tilde{s}}z)-(p'+q)}{2}\\
=\frac{(p'+q)}{2}(\sqrt{|z|^2+1}-1) +\frac{1}{2}\Lambda^{1}\sqrt{\tilde{s}}z_1
+
\frac{1}{2}
\frac{p'\times q}{|p'\times q|}\sqrt{\tilde{s}}z_2.\end{multline}
Then we have
\begin{multline}\label{dual4}
\langle w^{2l}\Gamma(f,h),\eta\rangle\\
=\frac{c'}{2}\int_\rth\frac{dp'}{p'^0}\int_{\rth}\frac{dq}{q^0} \frac{\sqrt{\tilde{s}}}{\tilde{g}}
\int_{\mathbb{R}^2} \frac{dz}{\sqrt{|z|^2+1}}
s_\Lambda\sigma(g_\Lambda,\theta_\Lambda)
\sqrt{J(q)}w^{2l}(p')\eta(p')f(q)\\\times\bigg[h\left(A(p',q,z)+p'\right)\exp\left(-\frac{A^0}{2}\right) -\frac{\tilde{s} \Phi(\tilde{g})\tilde{g}^4}{s_\Lambda \Phi(g_\Lambda)g^4_\Lambda}h(p')\bigg].
\end{multline}
This is the main expression for the dual representation that we will use to prove our cancellation estimates which land on the function $h$.

We have actually proven a more general integral formula. Now we consider \eqref{original.eq.Ig}.  The transformation from \eqref{original.eq.Ig} to \eqref{second.eq.Ig} or \eqref{third.eq.Ig} incorporates the series of changes of variables discussed previously in this section.  The transformation from \eqref{original.eq.Ig} to \eqref{second.eq.Ig} follows exactly from the arguments between \eqref{eq.Igainloss1}  to \eqref{final gain}.   Then to additionally derive \eqref{third.eq.Ig}  we further follow the arguments between \eqref{final gain} and \eqref{dual}.  This proves Lemma \ref{transformation.Lemma.appendix}.

 \section{Collision frequency multiplier derivation}\label{FREQ:sec:derivation}

We now explain a derivation of an alternative form of $\tilde{\zeta}(p)$ from \eqref{FREQ:tildezeta}, and give the new decomposition of  $\tilde{\zeta}(p)$ that has been explained in \secref{FREQ:sec:main.decomp}.

\subsection{Derivation of a new representation of $\tilde{\zeta}(p)$}
For a fixed $p\in\rth$, recalling \eqref{FREQ:tildezeta}, we would like to have an alternative representation of the following integral:
\begin{equation*}
I\eqdef -\tilde{\zeta}(p)=	\int_{\rth\times\mathbb{S}^2} v_{\text{\o}} \sigma(g,\theta) \sqrt{J(q)}\left(\sqrt{J(q')}-\sqrt{J(q)}\right)dqd\omega
\eqdef I_{gain}-I_{loss}.
\end{equation*}
Initially, suppose that $\int_{\mathbb{S}^2} d\omega\   |\sigma_0(\cos\theta)| <\infty$ and that
$$
\int_{\mathbb{S}^2} d\omega\   \sigma_0(\cos\theta)=0.
$$
Then, under that condition, the loss term vanishes $I_{loss}=0$ and we obtain
\begin{equation}\label{FREQ:Igainrepresentation}
I=I_{gain}=\int_{\rth}dq\int_{\mathbb{S}^2}d\omega\  v_{\text{\o}} \sigma(g,\theta)\sqrt{J(q)}\sqrt{J(q')}.
\end{equation}
By recovering the delta function involving the energy-momentum convervation laws, we obtain another representation of $I$:
$$I= \frac{1}{p^0}\int_{\rth}\frac{dq}{q^0}\int_{\rth}\frac{dp'}{p'^0}\int_{\rth}\frac{dq'}{q'^0}s\sigma(g,\theta)\delta^{(4)}(p'^\mu+q'^\mu-p^\mu-q^\mu)\sqrt{J(q)}\sqrt{J(q')}.$$
Here $g=g(p^\mu,q^\mu)$, $s=g^2+4$, $\bar{g}\eqdef g(p^\mu,p'^\mu)=g(q^\mu,q'^\mu)$, $\tilde{g} = g(p'^\mu,q^\mu)$, and
$$
\cos\theta=2\frac{\tilde{g}^2}{g^2}-1,
$$
by \eqref{FREQ:cos}.
We further {\it claim} that
\begin{equation}\label{FREQ:gg}
g^2=\tilde{g}^2-\frac{1}{2}(p^\mu+q'^\mu)(p'_\mu+q_\mu-p_\mu-q'_\mu).
\end{equation} Let $\tilde{s}\eqdef \tilde{g}^2+4$. Then \eqref{FREQ:gg} is equivalent to
\begin{multline*}
g^2=\tilde{g}^2-\frac{1}{2}\tilde{s}-\frac{1}{2}(p^\mu+q'^\mu)(p'_\mu+q_\mu)\\
=\frac{1}{2}\tilde{g}^2-2-\frac{1}{2}(p^\mu+q'^\mu)(p'_\mu+q_\mu)\\
=\frac{1}{2}\tilde{g}^2+g^2+2p^\mu q_\mu-\frac{1}{2}(p^\mu+q'^\mu)(p'_\mu+q_\mu).
\end{multline*}
Thus we prove \eqref{FREQ:gg} by showing that
$$\frac{1}{2}\tilde{g}^2+2p^\mu q_\mu-\frac{1}{2}(p^\mu+q'^\mu)(p'_\mu+q_\mu)=0.$$
Expanding the left-hand side of this equation, we obtain$$
-p^\mu q'_\mu-1+2p^\mu q_\mu-\frac{1}{2}p^\mu p'_\mu-\frac{1}{2}q'^\mu p'_\mu-\frac{1}{2}p^\mu q_\mu-\frac{1}{2}q'^\mu q_\mu.$$
By the result of the conservation laws $p^\mu+q^\mu=p'^\mu+q'^\mu$, we have $p^\mu q_\mu=p'^\mu q'_\mu$ and $p'^\mu q_\mu=p^\mu q'_\mu$. Therefore, we obtain $$
-1+p^\mu q_\mu-\frac{1}{2}p^\mu p'_\mu-\frac{1}{2}p^\mu q'_\mu-\frac{1}{2}p'^\mu q_\mu-\frac{1}{2}q'^\mu q_\mu,$$ which is equal to
$$-1+p^\mu q_\mu-\frac{1}{2}(p^\mu+q^\mu)(p'_\mu+q'_\mu)=-1+p^\mu q_\mu+\frac{1}{2}s=0.$$ This finishes the proof of the  {\it claim} \eqref{FREQ:gg}.

By exchanging $q$ and $q'$, we have
$$I= \frac{1}{p^0}\int_{\rth}\frac{dq}{q^0}\sqrt{J(q)}\int_{\rth}\frac{dp'}{p'^0}\int_{\rth}\frac{dq'}{q'^0}\sqrt{J(q')}\tilde{s}\sigma(\tilde{g},\theta')\delta^{(4)}(p'^\mu+q^\mu-p^\mu-q'^\mu),$$
where the angle $\theta'$ is now defined as
$$\cos\theta'\eqdef 2\frac{g^2}{\tilde{g}^2}-1,$$ 
and
\begin{equation}\notag	\tilde{g}^2=g^2-\frac{1}{2}(p^\mu+q^\mu)(p'_\mu+q'_\mu-p_\mu-q_\mu).\end{equation} 
We have the new argument in the delta function and $\tilde{s}\eqdef \tilde{g}^2+4$.

We now define the functional $i(p,q)$ as
\begin{equation}\label{FREQ:i}
i(p,q)\eqdef \frac{1}{p^0q^0} \int_{\rth}\frac{dp'}{p'^0}\int_{\rth}\frac{dq'}{q'^0}\sqrt{J(q')}\tilde{s}\sigma(\tilde{g},\theta')\delta^{(4)}(p'^\mu+q^\mu-p^\mu-q'^\mu),
\end{equation} so that we have
\begin{equation}\label{FREQ:Ipq}
I=\int_\rth i(p,q)\sqrt{J(q)}dq.\end{equation}
We first translate \eqref{FREQ:i} into an expression involving the total and relative momentum
variables, $p'^\mu+q'^\mu$ and $p'^\mu-q'^\mu$ respectively. Define $u$ by $u(x) = 0$ if $x < 0$ and
$u(x) = 1$ if $x\geq 0$. Let $g'\eqdef g(p'^\mu,q'^\mu)$ and $s'\eqdef s(p'^\mu,q'^\mu).$ Then by the claim (7.5) of \cite{MR2728733}, we have
$$i(p,q)=\frac{1}{16p^0q^0}\int_{\rfo\times\rfo}d\Theta(p'^\mu,q'^\mu)\frac{e^{-q'^0/2}}{4\pi}\tilde{s}\sigma(\tilde{g},\theta')\delta^{(4)}(p'^\mu+q^\mu-p^\mu-q'^\mu),$$
where
$$d\Theta(p'^\mu,q'^\mu)\eqdef dp'^\mu dq'^\mu u(p'^0+q'^0)u(s'-4)\delta(s'-g'^2-4)\delta((p'^\mu+q'^\mu)(p'_\mu-q'_\mu)).$$
Thus we have lifted to an integral over $\rfo\times\rfo$ from one over $\rth\times\rth$.

Now we apply the change of variables $\bar{p}^\mu=p'^\mu+q'^\mu$ and $\bar{q}^\mu=p'^\mu-q'^\mu$. Then the Jacobian is 16. Since $q'^0=\frac{\bar{p}^0-\bar{q}^0}{2}$, we have
\begin{equation}\notag
i(p,q)= \frac{c'}{p^0q^0}\int_{\rfo\times\rfo}d\Theta(\bar{p}^\mu,\bar{q}^\mu)\tilde{s}\sigma(\tilde{g},\theta')\delta^{(4)}(q^\mu-p^\mu+\bar{q}^\mu)\exp\left(\frac{-\bar{p}^0+\bar{q}^0}{4}\right)
\end{equation} for some constant $c'>0$ (whose value can change from line to line), where
$$d\Theta(\bar{p}^\mu,\bar{q}^\mu)\eqdef d\bar{p}^\mu d\bar{q}^\mu u(\bar{p}^0)u(-\bar{p}^\mu\bar{p}_\mu-4)\delta(-\bar{p}^\mu\bar{p}_\mu-\bar{q}^\mu\bar{q}_\mu-4)\delta(\bar{p}^\mu\bar{q}_\mu).$$

We now carry out $\delta^{(4)}(q^\mu-p^\mu+\bar{q}^\mu)$ to obtain
\begin{equation}\notag
i(p,q)= \frac{c'}{p^0q^0}\int_{\rfo}d\Theta(\bar{p}^\mu)\tilde{s}\sigma(\tilde{g},\theta')\exp\left(\frac{-\bar{p}^0+p^0-q^0}{4}\right),
\end{equation}
where the measure $d\Theta(\bar{p}^\mu)$ is now equal to
$$d\Theta(\bar{p}^\mu)\eqdef d\bar{p}^\mu  u(\bar{p}^0)u(-\bar{p}^\mu\bar{p}_\mu-4)\delta(-\bar{p}^\mu\bar{p}_\mu-g^2-4)\delta(\bar{p}^\mu(p_\mu-q_\mu)).$$
Since $s=g^2+4,$ we have
\begin{multline*}
u(\bar{p}^0)\delta(-\bar{p}^\mu\bar{p}_\mu-g^2-4)=u(\bar{p}^0)\delta(-\bar{p}^\mu\bar{p}_\mu-s)\\=u(\bar{p}^0)\delta((\bar{p}^0)^2-|\bar{p}|^2-s)=\frac{\delta(\bar{p}^0-\sqrt{|\bar{p}|^2+s})}{2\sqrt{|\bar{p}|^2+s}}.
\end{multline*}
Then we carry out one integration using this delta function to obtain
\begin{multline}\notag
i(p,q)= \frac{c'}{2p^0q^0}\int_{\rth}\frac{d\bar{p}}{\bar{p}^0}u(-\bar{p}^\mu\bar{p}_\mu-4)\delta(\bar{p}^\mu(p_\mu-q_\mu))\tilde{s}\sigma(\tilde{g},\theta')\\\times\exp\left(\frac{-\sqrt{|\bar{p}|^2+s}+p^0-q^0}{4}\right),
\end{multline}
where $\bar{p}^0=\sqrt{|\bar{p}|^2+s}$. Using $s=g^2+4$ again, we have $$-\bar{p}^\mu\bar{p}_\mu-4=s-4=g^2\geq 0$$ to guarantee that $u(-\bar{p}^\mu\bar{p}_\mu-4)=1$. Thus
\begin{equation}\notag
i(p,q)= \frac{c'}{2p^0q^0}\exp\left(\frac{p^0-q^0}{4}\right)\int_{\rth}\frac{d\bar{p}}{\bar{p}^0}\delta(\bar{p}^\mu(p_\mu-q_\mu))\tilde{s}\sigma(\tilde{g},\theta')e^{\left(\frac{\bar{p}^\mu U_\mu}{4}\right)},
\end{equation}
where $\bar{p}^0=\sqrt{|\bar{p}|^2+s}$ and $U^\mu=(1,0,0,0)$.
We finish off our reduction by moving to a new Lorentz frame. We consider a Lorentz transformation $\Lambda$ which maps into the center-of-momentum system as 
$$
A_\nu\eqdef \Lambda^{\mu}{}_{\nu}(p_\mu+q_\mu) =(\sqrt{s},0,0,0),\hspace{10mm} B_\nu\eqdef -\Lambda^{\mu}{}_{\nu} (p_\mu-q_\mu)=(0,0,0,g).
$$
The explicit form of the matrix $\Lambda$ was given p. 593 of \cite{MR2728733}, and also in \cite{MR2707256}.  More precisely, we consider
\begin{equation}\label{FREQ:eq.LT}
\Lambda=(\Lambda^{\mu}{}_\nu)=\left(\begin{array}{cccc}\frac{p^0+q^0}{\sqrt{s}}& -\frac{p_1+q_1}{\sqrt{s}} &-\frac{p_2+q_2}{\sqrt{s}} &-\frac{p_3+q_3}{\sqrt{s}}\\ \Lambda^1{}_0 &\Lambda^1{}_1&\Lambda^1{}_2&\Lambda^1{}_3\\ 0 & \frac{(p\times q)_1}{|p\times q|} &\frac{(p\times q)_2}{|p\times q|}&\frac{(p\times q)_3}{|p\times q|}\\ \frac{p^0-q^0}{g} &-\frac{p_1-q_1}{g}&-\frac{p_2-q_2}{g}&-\frac{p_3-q_3}{g}\end{array}\right),
\end{equation}
with the second row given by
$$\Lambda^1{}_0=\Lambda^1{}_0(p,q)=\frac{2|p\times q|}{g\sqrt{s}},$$ 
and for $i=1,2,3$ we have 
$$\Lambda^1{}_i=\Lambda^1{}_i(p,q)=\frac{2\left(p_i\{p^0+q^0p^\mu q_\mu\}+q_i\{q^0+p^0p^\mu q_\mu\}\right)}{g\sqrt{s}|p\times q|}.$$
Then, using this change of variables, we have
\begin{equation*}
\int_{\rth}\frac{d\bar{p}}{\bar{p}^0}\delta(\bar{p}^\mu(p_\mu-q_\mu))\tilde{s}\sigma(\tilde{g},\theta')e^{\left(\frac{\bar{p}^\mu U_\mu}{4}\right)}
=\int_{\rth}\frac{d\bar{p}}{\bar{p}^0}\delta(\bar{p}^\mu B_\mu)s_\Lambda\sigma(g_\Lambda,\theta_\Lambda)e^{\left(\frac{\bar{p}^\mu \bar{U}_\mu}{4}\right)}.
\end{equation*}
Note that $\frac{d\bar{p}}{\bar{p}^0}$ is Lorentz invariant.
Here $\bar{p}^0=\sqrt{|\bar{p}|^2+s}$ and $s_\Lambda$, $g_\Lambda\geq0$ are
$$
g^2_\Lambda\eqdef g^2-\frac{1}{2}A^\mu(\bar{p}_\mu-A_\mu)= g^2+\frac{1}{2}\sqrt{s}(\bar{p}^0-\sqrt{s}),
$$
where
\begin{equation}\label{FREQ:slambda.def}
s_\Lambda\eqdef g^2_\Lambda+4,
\end{equation}
and
\begin{equation}\label{FREQ:coslam}
\cos\theta_\Lambda\eqdef 2\frac{g^2}{g^2_\Lambda}-1.
\end{equation}
Also, $\bar{U}^\mu$ is defined as $\bar{U}^\mu= \left(\frac{p^0+q^0}{\sqrt{s}},\frac{2|p\times q|}{g\sqrt{s}},0,\frac{p^0-q^0}{g}\right)$.
We switch to polar coordinates in the form $$d\bar{p}=r^2 dr \sin\psi d\psi d\phi,\hspace{5mm} \bar{p}\eqdef r(\sin \psi \cos \phi, \sin \psi \sin \phi, \cos \psi).$$
Then we obtain $$\bar{p}^\mu B_\mu = gr\cos\psi.$$
Then the integral $i(p,q)$ is now equal to
\begin{multline}\notag
i(p,q)= \frac{c'}{2p^0q^0}\exp\left(\frac{p^0-q^0}{4}\right)\int_{0}^{2\pi}d\phi \int_{0}^{\pi}d\psi \sin\psi \\\times\int_{0}^{\infty} \frac{r^2dr}{\sqrt{r^2+s}}\delta(gr\cos\psi)s_\Lambda\sigma(g_\Lambda,\theta_\Lambda) e^{\left(\frac{\bar{p}^\mu \bar{U}_\mu}{4}\right)}.
\end{multline}
We evaluate the last delta function at $\psi = \pi/2$ to write $i(p, q)$ as
\begin{multline}\notag
i(p,q)= \frac{c'}{2gp^0q^0}\exp\left(\frac{p^0-q^0}{4}\right)  \\\times\int_{0}^{2\pi}d\phi \int_{0}^{\infty} \frac{rdr}{\sqrt{r^2+s}}s_\Lambda\sigma(g_\Lambda,\theta_\Lambda)\exp\left(-\bar{p}^0\frac{p^0+q^0}{4\sqrt{s}}+\frac{|p\times q|}{2g\sqrt{s}}r\cos\phi\right).
\end{multline}
By using the modified Bessel function of index zero given by \eqref{bessel0} we have
\begin{multline}\label{FREQ:finali}
i(p,q)= \frac{c'}{2gp^0q^0}\exp\left(\frac{p^0-q^0}{4}\right)  \\\times \int_{0}^{\infty} \frac{rdr}{\sqrt{r^2+s}}s_\Lambda\sigma(g_\Lambda,\theta_\Lambda)\exp\left(-\frac{p^0+q^0}{4\sqrt{s}}\sqrt{r^2+s}\right)I_0\left(\frac{|p\times q|}{2g\sqrt{s}}r\right).
\end{multline}
Now $g_\Lambda\geq0$ is given by
\begin{equation}\label{FREQ:g2}
g^2_\Lambda = g^2+\frac{1}{2}\sqrt{s}(\sqrt{r^2+s}-\sqrt{s}),
\end{equation}
with \eqref{FREQ:slambda.def} and \eqref{FREQ:coslam}.
So by \eqref{FREQ:Ipq} we obtain a new representation of our gain term
\begin{multline}\label{FREQ:final gain}
I=I_{gain}=\frac{c'}{p^0}e^{\frac{p^0}{4}}\int_{\rth}\frac{dq}{q^0}\frac{e^{-\frac{3}{4}q^0}}{g} \int_{0}^{\infty} \frac{rdr}{\sqrt{r^2+s}}s_\Lambda\sigma(g_\Lambda,\theta_\Lambda)\\\times \exp\left(-\frac{p^0+q^0}{4\sqrt{s}}\sqrt{r^2+s}\right)I_0\left(\frac{|p\times q|}{2g\sqrt{s}}r\right),
\end{multline}
We recall that $I_{loss}=0$. We will now find a different expression than $I_{loss}$ which is also equal to zero, that will provide suitable cancellation for the term in \eqref{FREQ:final gain} when we no longer assume that $\int_{\mathbb{S}^2} d\omega\  |\sigma_0(\cos\theta)| <\infty$ and $\int_{\mathbb{S}^2} d\omega\  \sigma_0(\cos\theta)=0$.

To this end, we recall the definitions \eqref{FREQ:g2} and \eqref{FREQ:coslam}.  This yields
\begin{equation}\label{FREQ:newcos}
\cos\theta_\Lambda=\frac{2g^2}{g^2_\Lambda}-1=\frac{g^2-\frac{1}{2}\sqrt{s}(\sqrt{r^2+s}-\sqrt{s})}{g^2+\frac{1}{2}\sqrt{s}(\sqrt{r^2+s}-\sqrt{s})},
\end{equation}
using \eqref{FREQ:g2} above.  Then we have
$$\frac{dg^2_\Lambda}{dr}= \frac{\sqrt{s}r}{2\sqrt{r^2+s}}.$$
By differentiating $\cos\theta_\Lambda$ with respect to $r$, we have
\begin{multline*}
\frac{d(\cos\theta_\Lambda)}{dr}
=\frac{d}{dr}\frac{g^2-\frac{1}{2}\sqrt{s}(\sqrt{r^2+s}-\sqrt{s})}{g^2_\Lambda}\\=\frac{-\frac{1}{2}\sqrt{s}\frac{2r}{2\sqrt{r^2+s}}g_\Lambda^2-(g^2-\frac{1}{2}\sqrt{s}(\sqrt{r^2+s}-\sqrt{s}))\frac{d g^2_\Lambda}{dr}  }{g^4_\Lambda}\\
=-\frac{1}{2}\sqrt{s}\frac{r}{\sqrt{r^2+s}g^2_\Lambda}-\frac{(g^2-\frac{1}{2}\sqrt{s}(\sqrt{r^2+s}-\sqrt{s})) }{g^4_\Lambda} \frac{\sqrt{s}r}{2\sqrt{r^2+s}}\\
=\frac{\sqrt{s}r}{2g^4_\Lambda\sqrt{r^2+s}}\left(-g^2_\Lambda -(g^2-\frac{1}{2}\sqrt{s}(\sqrt{r^2+s}-\sqrt{s}))\right)\\
=\frac{\sqrt{s}r}{2g^4_\Lambda\sqrt{r^2+s}}\left(-2g^2\right)=-\frac{g^2\sqrt{s}r}{g^4_\Lambda\sqrt{r^2+s}}.
\end{multline*}Therefore,
$$\frac{d(\cos\theta_\Lambda)}{dr}=-\frac{g^2\sqrt{s}r}{g^4_\Lambda\sqrt{r^2+s}}.$$
Since we have assumed that $$\int_{-1}^1d(\cos\theta_\Lambda) \sigma_0(\cos\theta_\Lambda)=0,$$ we further have
$$
\int_{0}^{\infty}dr \frac{g^2\sqrt{s}r}{g^4_\Lambda\sqrt{r^2+s}}\sigma_0(\cos\theta_\Lambda)=0,
$$ as $\cos\theta_\Lambda = 1\text{ and }-1$ correspond to $r=0$ and $r=\infty$ respectively, by \eqref{FREQ:newcos}.  Thus we obtain
\begin{equation}
\notag
\frac{c'}{p^0}e^{\frac{p^0}{4}}\int_{\rth}\frac{dq}{q^0}\frac{e^{-\frac{3}{4}q^0}}{g} \int_{0}^\infty \frac{rdr}{\sqrt{r^2+s}}s\Phi(g)\sigma_0(\cos \theta_\Lambda)\frac{g^4}{g^4_\Lambda}\exp\left(-\frac{p^0+q^0}{4\sqrt{s}}\sqrt{s}\right)=0.
\end{equation}
Subtracting this zero integral from \eqref{FREQ:final gain}, we obtain
\begin{multline}\label{FREQ:dual}
I=\frac{c'}{p^0}e^{\frac{p^0}{4}}\int_{\rth}\frac{dq}{q^0}\frac{e^{-\frac{3}{4}q^0}}{g}\int_{0}^\infty \frac{rdr}{\sqrt{r^2+s}}s_\Lambda\sigma(g_\Lambda,\theta_\Lambda)\\\times\Big[\exp\left(-\frac{p^0+q^0}{4\sqrt{s}}\sqrt{r^2+s}\right)I_0\left(\frac{|p\times q|}{2g\sqrt{s}}r\right)\\ -\exp\left(-\frac{p^0+q^0}{4}\right)\frac{s\Phi(g)g^4}{s_\Lambda  \Phi(g_\Lambda)g^4_\Lambda}\Big].
\end{multline}
This is equal to the original integral $I=-\tilde{\zeta}(p)$ when the mean value of $\sigma_0$ is zero.

 We also note that \eqref{FREQ:dual} also holds for \eqref{FREQ:tildezeta} even when the mean value of $\sigma_0$ is not zero.  
	Suppose that $\int_{\mathbb{S}^2} d w \hspace{1mm} |\sigma_0(\theta)| <\infty$ and that $\int_{\mathbb{S}^2} d w \hspace{1mm} \sigma_0(\theta) = 2\pi c_0\neq0$.
	Define 
	$$
	\sigma_0^\epsilon(\theta) = \sigma_0(\theta)-1_{[1-\epsilon,1]}(\cos\theta)\int_{-1}^1dt'\frac{\sigma_0(t')}{\epsilon}. 
	$$
	Then, we have $\int_{-1}^1\sigma_0^\epsilon(\theta)d(\cos\theta)=0$ vanishing on $\omega\in \mathbb{S}^2$.
	Now, define 
	\begin{align*}
	&\tilde{\zeta}^\epsilon(p)
	=\int_{\rth\times\mathbb{S}^2} v_{\text{\o}} \Phi(g)\sigma^\epsilon_0(\theta) \sqrt{J(q)}\left(\sqrt{J(q)}-\sqrt{J(q')}\right)dqd\omega.
	\end{align*} Then, also using  \eqref{FREQ:tildezeta}, we have
	\begin{multline}
	\label{FREQ:notmeanzero}
	\left|\tilde{\zeta}(p) -\tilde{\zeta}^\epsilon(p) \right|
	\\
	=\bigg|c_0 \int_{\rth\times\mathbb{S}^2} v_{\text{\o}} \Phi(g) \sqrt{J(q)}\left(\sqrt{J(q)}-\sqrt{J(q')}\right) \frac{1_{[1-\epsilon,1]}(\cos\theta)}{\epsilon} dqd\omega\bigg|.
	\end{multline}
 If $\cos\theta=1$, by Remark \ref{FREQ:angle.remark} and e.g. \eqref{FREQ:cosine.angle.formula}, \eqref{gbar} and \eqref{conservation}  we have $p'^\mu=p^\mu$ and $q'^\mu=q^\mu$. 
	Thus, as $\epsilon \rightarrow 0$, the difference term in \eqref{FREQ:notmeanzero} $\rightarrow 0$ as $\sqrt{J(q)}-\sqrt{J(q')}$ has a higher order cancellation and hence the integrand vanishes on the set $\cos\theta=1$. 
 By the higher-order cancellation, an additional cutoff argument shows that the identity \eqref{FREQ:dual} holds for the noncutoff kernel $\sigma_0$ from \eqref{angassumption}.

\subsection{First representation of $\tilde{\zeta}$}
We will now further split $\tilde{\zeta}=\zeta_0+\zetaL$. From \eqref{FREQ:dual} for simplicity we write
$$
-I=\frac{c'}{p^0}e^{\frac{p^0}{4}}\int_{\rth}\frac{dq}{q^0}\frac{e^{-\frac{3}{4}q^0}}{g}\newK(p,q),
$$
where
\begin{multline}\notag
\newK(p,q) \eqdef \int_{0}^\infty \frac{rdr}{\sqrt{r^2+s}}s_\Lambda\sigma(g_\Lambda,\theta_\Lambda)
\\
\times
\Big[\exp\left(-\frac{p^0+q^0}{4}\right)\frac{s\Phi(g)g^4}{s_\Lambda  \Phi(g_\Lambda)g^4_\Lambda} - \exp\left(-\frac{p^0+q^0}{4\sqrt{s}}\sqrt{r^2+s}\right)I_0\left(\frac{|p\times q|}{2g\sqrt{s}}r\right) \Big].
\end{multline}
Notice that both terms of the integral converge for large $r\ge 1$.  We further split
\begin{multline}\notag
\exp\left(-\frac{p^0+q^0}{4}\right)\frac{s\Phi(g)g^4}{s_\Lambda  \Phi(g_\Lambda)g^4_\Lambda}-\exp\left(-\frac{p^0+q^0}{4\sqrt{s}}\sqrt{r^2+s}\right)I_0\left(\frac{|p\times q|}{2g\sqrt{s}}r\right)
\\
=\left(\exp\left(-\frac{p^0+q^0}{4}\right) - \exp\left(-\frac{p^0+q^0}{4\sqrt{s}}\sqrt{r^2+s}\right)I_0\left(\frac{|p\times q|}{2g\sqrt{s}}r\right)\right) \frac{s\Phi(g)g^4}{s_\Lambda  \Phi(g_\Lambda)g^4_\Lambda}\\
+\exp\left(-\frac{p^0+q^0}{4\sqrt{s}}\sqrt{r^2+s}\right) I_0\left(\frac{|p\times q|}{2g\sqrt{s}}r\right)
\left( \frac{s\Phi(g)g^4}{s_\Lambda  \Phi(g_\Lambda)g^4_\Lambda} -1\right).
\end{multline}
This motivates the following splitting of $\tilde{\zeta}=\zeta_0+\zetaL$ with
\begin{multline}\label{FREQ:zeta0B.appendix}
\zeta_0 \eqdef \frac{c'}{p^0}e^{\frac{p^0}{4}}\int_{\rth}\frac{dq}{q^0}\frac{e^{-\frac{3}{4}q^0}}{g}\int_{0}^\infty \frac{rdr}{\sqrt{r^2+s}}s_\Lambda\sigma(g_\Lambda,\theta_\Lambda) \frac{s\Phi(g)g^4}{s_\Lambda  \Phi(g_\Lambda)g^4_\Lambda}
\\
\times
\Big[\exp\left(-\frac{p^0+q^0}{4}\right) - \exp\left(-\frac{p^0+q^0}{4\sqrt{s}}\sqrt{r^2+s}\right)I_0\left(\frac{|p\times q|}{2g\sqrt{s}}r\right) \Big],
\end{multline}and
\begin{multline}\label{FREQ:zetaLB}
\zetaL \eqdef \frac{c'}{p^0}e^{\frac{p^0}{4}}\int_{\rth}\frac{dq}{q^0}\frac{e^{-\frac{3}{4}q^0}}{g}\int_{0}^\infty \frac{rdr}{\sqrt{r^2+s}}s_\Lambda\sigma(g_\Lambda,\theta_\Lambda)   \\
\times
\exp\left(-\frac{p^0+q^0}{4\sqrt{s}}\sqrt{r^2+s}\right) I_0\left(\frac{|p\times q|}{2g\sqrt{s}}r\right)
\left( \frac{s\Phi(g)g^4}{s_\Lambda  \Phi(g_\Lambda)g^4_\Lambda} -1\right).
\end{multline}
This completes the derivation of our first representation of $\tilde{\zeta}(p)$.

\subsection{Derivation of an alternative representation of $\tilde{\zeta}(p)$}\label{FREQ:sec:alternative.deriv}
For a fixed $p\in\rth$, we would like to have an alternative representation of \eqref{FREQ:tildezeta}:
\begin{equation}\label{FREQ:Irepresentation}
 -\tilde{\zeta}(p)=	\int_{\rth\times\mathbb{S}^2} v_{\text{\o}} \sigma(g,\theta) \sqrt{J(q)}\left(\sqrt{J(q')}-\sqrt{J(q)}\right)dqd\omega\eqdef I_{gain}-I_{loss}.
\end{equation}
Then exactly as previously we can derive \eqref{FREQ:final gain} for the term $I_{gain}$.  We will now find an alternative expression for $I_{loss}$ that will provide suitable cancellation for the term in \eqref{FREQ:final gain}. 

To this end, using the definition of $I_{loss}$ from \eqref{FREQ:Irepresentation} yields
\begin{equation}\label{FREQ:Ilossrepresentation}
I_{loss}=\int_{\rth}dq\int_{\mathbb{S}^2}d\omega\  v_{\text{\o}} \sigma(g,\theta)J(q).
\end{equation}Since the gain term representation \eqref{FREQ:Igainrepresentation} results in \eqref{FREQ:i} and \eqref{FREQ:Ipq}, the loss term \eqref{FREQ:Ilossrepresentation} would yield
\begin{equation}\notag
I_{loss}=\int_{\rth}i_{loss}(p,q)dq,
\end{equation}where 
\begin{equation}\label{FREQ:iloss}
i_{loss}(p,q)\eqdef \frac{1}{p^0q^0} \int_{\rth}\frac{dp'}{p'^0}\int_{\rth}\frac{dq'}{q'^0}J(q')\tilde{s}\sigma(\tilde{g},\theta')\delta^{(4)}(p'^\mu+q^\mu-p^\mu-q'^\mu),
\end{equation} 
by following the same argument between \eqref{FREQ:Igainrepresentation} and \eqref{FREQ:Ipq}. Note that we have exchanged $q$ and $q'$ variables in the procedure. Then we can easily see that the only difference between \eqref{FREQ:i} and \eqref{FREQ:iloss} is the power on the term $J(q')$; i.e., the power on $J(q')$ in $i_{loss}(p,q)$ is twice of that in $i(p,q)$. Therefore, the same derivation results in the new representation of the loss term similar to  \eqref{FREQ:finali}:
\begin{multline}\label{FREQ:finaliloss}
i_{loss}(p,q)= \frac{c'}{gp^0q^0}\exp\left(\frac{p^0-q^0}{2}\right)  \\\times \int_{0}^{\infty} \frac{rdr}{\sqrt{r^2+s}}s_\Lambda\sigma(g_\Lambda,\theta_\Lambda)\exp\left(-\frac{p^0+q^0}{2\sqrt{s}}\sqrt{r^2+s}\right)I_0\left(\frac{|p\times q|}{g\sqrt{s}}r\right).
\end{multline}
In particular we have
\begin{multline}\label{FREQ:final loss}
I_{loss}=\frac{c'}{p^0}e^{\frac{p^0}{2}}\int_{\rth}\frac{dq}{q^0}\frac{e^{-\frac{1}{2}q^0}}{g} \int_{0}^{\infty} \frac{rdr}{\sqrt{r^2+s}}s_\Lambda\sigma(g_\Lambda,\theta_\Lambda)\\\times \exp\left(-\frac{p^0+q^0}{2\sqrt{s}}\sqrt{r^2+s}\right)I_0\left(\frac{|p\times q|}{g\sqrt{s}}r\right).
\end{multline}
Subtracting this integral from \eqref{FREQ:final gain}, we obtain
\begin{multline}\notag
I=\frac{c'}{p^0}e^{\frac{p^0}{4}}\int_{\rth}\frac{dq}{q^0}\frac{e^{-\frac{3}{4}q^0}}{g}\int_{0}^\infty \frac{rdr}{\sqrt{r^2+s}}s_\Lambda\sigma(g_\Lambda,\theta_\Lambda)\\\times\Big[\exp\left(-\frac{p^0+q^0}{4\sqrt{s}}\sqrt{r^2+s}\right)I_0\left(\frac{|p\times q|}{2g\sqrt{s}}r\right)\\ -\exp\left(\frac{p^0+q^0}{4}\right)\exp\left(-\frac{p^0+q^0}{2\sqrt{s}}\sqrt{r^2+s}\right)I_0\left(\frac{|p\times q|}{g\sqrt{s}}r\right)\Big].
\end{multline}
This is equal to the original integral $I=-\tilde{\zeta}(p)$. 
The representation above also holds when $\sigma_0$ does not have mean zero or does not have a bounded integral, as we discussed in \eqref{FREQ:notmeanzero}.

As in \eqref{FREQ:zeta0By} and \eqref{FREQ:zetaLBy}, we take the change of variables $r\mapsto y=\frac{r}{\sqrt{s}}$ in $I$ above. 
Then, we can write $\tilde{\zeta}$ as follows
\begin{multline}\label{FREQ:eq:tildezetanew1.appendix}
\tilde{\zeta}(p)=\frac{c'}{\pi p^0}\int_{\rth}\frac{dq}{q^0}\frac{e^{-q^0}\sqrt{s}}{g}\int_{0}^\infty \frac{ydy}{\sqrt{y^2+1}}s_\Lambda\sigma(g_\Lambda,\theta_\Lambda)\int_0^{\pi} d\phi\\\times\Big[\exp (2l-2l\sqrt{y^2+1}+2jy\cos\phi)-\exp(l-l\sqrt{y^2+1} +jy\cos\phi)\Big],
\end{multline}
where we use the notations \eqref{FREQ:lj} and \eqref{FREQ:g2y.variable}. This completes the derivation.

\section{Proofs of the pointwise estimates}\label{FREQ:sec:pointwiseProof}

In this subsection we give the proofs of Lemma \ref{FREQ:lem:useful.ests}, Lemma \ref{FREQ:lem:integral.ests}, and Lemma \ref{pointwise.lemma2}.

\begin{proof}[Proof of Lemma \ref{FREQ:lem:useful.ests}]
The proof of \eqref{FREQ:s.ge.g2} is direct, and \eqref{FREQ:s.le.pq} follows from \eqref{s} and the Cauchy-Schwartz inequality.  Then \eqref{FREQ:g.ge.lower}, \eqref{FREQ:g.ge.2lower} and \eqref{FREQ:g.le.upper} follow from \eqref{g.ineq.sharp}.  For \eqref{FREQ:p0q0.le.pq} notice that that
$$
p^0 - q^0 = \frac{|p|^2 - |q|^2}{p^0 + q^0} = \frac{(p - q) \cdot (p+q)}{p^0 + q^0} 
\le |p - q|.
$$
Then \eqref{FREQ:p0.plus.q0.le.p0q0} is automatic.

Now using \eqref{FREQ:lj} then equation \eqref{FREQ:j.le.l} follows from \eqref{FREQ:g.ge.2lower}.  And \eqref{FREQ:l.upper.ineq} is automatic.  The proof of \eqref{FREQ:l2j2} requires some development and is from \cite{MR1211782}.  Now using \eqref{FREQ:lj} we have
$$
    l^2-j^2 = \left(\frac{p^0+q^0}{4}\right)^2-\left(\frac{|p\times q|}{2g}\right)^2
    =\frac{(p^0+q^0)^2g^2-4|p\times q|^2}{16g^2}.$$
    By the definition of $g$ in \eqref{g}, we have
\begin{multline*}
    (p^0+q^0)^2g^2-4|p\times q|^2
    =(p^0+q^0)^2(-2-2p^\mu q_\mu)-4|p\times q|^2\\
    =(2+|p|^2+|q|^2+2p^0q^0)(-2+2p^0q^0-2p\cdot q)-4|p\times q|^2\\
    =(2+|p|^2+|q|^2+2p^0q^0)(-2-|p|^2-|q|^2+2p^0q^0+|p-q|^2)-4|p\times q|^2\\
    =(2p^0q^0)^2-(2+|p|^2+|q|^2)^2 +(2+|p|^2+|q|^2+2p^0q^0)|p-q|^2-4|p\times q|^2\\
    =(2p^0q^0)^2-(2+|p|^2+|q|^2)^2 +(p^0+q^0)^2|p-q|^2-4|p\times q|^2.
\end{multline*}
We calculate that
\begin{multline*}
    (2p^0q^0)^2-4|p\times q|^2=4+4|p|^2|q|^2+4|p|^2+4|q|^2-4|p\times q|^2\\
    =4+4(p\cdot q)^2+4|p|^2+4|q|^2,
\end{multline*}
and$$
    (2+|p|^2+|q|^2)^2=4+|p|^4+|q|^4+4|p|^2+4|q|^2+2|p|^2|q|^2.$$
Thus, using also \eqref{s}, we have
\begin{multline*}
    (2p^0q^0)^2-(2+|p|^2+|q|^2)^2 +(p^0+q^0)^2|p-q|^2-4|p\times q|^2\\
    =(p^0+q^0)^2|p-q|^2-|p|^4-|q|^4+4(p\cdot q)^2-2|p|^2|q|^2\\
    =(p^0+q^0)^2|p-q|^2-(|p|^2+|q|^2)^2 +4(p\cdot q)^2\\
    =(p^0+q^0)^2|p-q|^2 -(|p|^2+|q|^2+2p\cdot q)(|p|^2+|q|^2-2p\cdot q)\\
    =(p^0+q^0)^2|p-q|^2-|p+q|^2|p-q|^2
    =s|p-q|^2.
\end{multline*}
Therefore, we have \eqref{FREQ:l2j2}.   Then \eqref{FREQ:l2j2size} follows from \eqref{FREQ:l2j2} and \eqref{FREQ:s.ge.g2}.

We will now prove \eqref{FREQ:ineq.gL.here}.  The upper bound of $g_\Lambda^2$ in \eqref{FREQ:ineq.gL.here} follows from \eqref{FREQ:g2y.variable} with \eqref{FREQ:s.ge.g2}.  The lower bound of $g_\Lambda^2$ in \eqref{FREQ:ineq.gL.here} follows from \eqref{FREQ:glambda.calc} and \eqref{FREQ:s.ge.g2}.  
\end{proof}

\begin{proof}[Proof of Lemma \ref{FREQ:lem:integral.ests}]
We will start with \eqref{FREQ:smally.lemma}.  From \eqref{bessel0} and \eqref{FREQ:int.Kgamma} we have 
 $$
|\bar{K}_\gamma(l,j)|  \lesssim
\max_{0\leq x\leq 1}\exp(-l\sqrt{x^2+1}+jx).
$$ 
The maximum of the function $h(x)\eqdef -l\sqrt{x^2+1}+jx$ occurs at $x=0$, $x=1$, or $x=x_0=\frac{j}{\sqrt{l^2-j^2}}$ where $h'(x_0)=0$. Note that $ h(x_0)=-\sqrt{l^2-j^2}.$ When $x=1$, we have $$h(1)=-\sqrt{2}l+j\leq -\sqrt{l^2-j^2}.$$ Thus, we conclude \eqref{FREQ:max.bound} and \eqref{FREQ:smally.lemma}.
	
Then \eqref{FREQ:J2.special} is a known integral that can be calculated exactly \cite{Gradshteyn:1702455} as  \eqref{FREQ:J2.lemma}.  Further  \eqref{FREQ:k2lj.lemma} is calculated during the proof of Corollary 2  in \cite[Corollary 2, pp.~323]{MR1211782}.  In particular we can obtain \eqref{FREQ:k2lj.lemma} from $\tilde{K}_2(l,j)= \partial_l^2 J_2(l,j)$.
\end{proof}

\begin{proof}[Proof of Lemma \ref{pointwise.lemma2}]
We remark that the proof of  \eqref{jutter.integral.est} follows from well known pointwise estimates.
    Then \eqref{moller.upper.est} is a direct consequence of \eqref{FREQ:s.le.pq} and \eqref{FREQ:g.le.sqrtpq}. 
    
    For \eqref{gtildeg.equiv}, note that \eqref{FREQ:cosine.angle.formula}, \cite[Proposition 2.7]{MR4156121}, or \cite[pp.12 or pp.58]{MR3542337} 
    implies 
   \begin{equation}\label{triangle.id}g^2=\tilde{g}^2+\bar{g}^2,\end{equation} and \begin{equation}\label{bargoverg}\sin\frac{\theta}{2}=\frac{\bar{g}}{g}.\end{equation} Then, since $\theta\in \bigg(0,\frac{\pi}{2}\bigg] $, we have
    $\bar{g}\le \tilde{g}.$ Therefore, 
    $$\tilde{g}^2\le g^2\le 2\tilde{g}^2.$$ This implies \eqref{gtildeg.equiv}.
    
    For \eqref{frac sg.est}, we first mention that $\frac{\tilde{s} \Phi(\tilde{g})\tilde{g}^4}{s_\Lambda \Phi(g_\Lambda)g^4_\Lambda}$ is clearly non-negative. 
    On the other hand, using \eqref{g2y.variable} we have $$\tilde{g}\le g_\Lambda,\text{ and }\tilde{s}\le s_\Lambda.$$ Since 
    $\Phi(g)= C_\Phi g^{\singS}$ and $\singS\in (-2.5,2) $ for both \eqref{hard} and \eqref{soft}, we have $4-\singS>0$ and
    $$\frac{\tilde{s} \Phi(\tilde{g})\tilde{g}^4}{s_\Lambda \Phi(g_\Lambda)g^4_\Lambda}= \frac{\tilde{s} \tilde{g}^{4-\singS}}{s_\Lambda g^{4-\singS}_\Lambda}\le 1.$$  
    The proof for \eqref{colfre4.2} follows by \eqref{FREQ:difference.estimate.here} and \eqref{FREQ:upperbound.allcases.zeta1}. 
    
    For \eqref{eq.pqp'q'}, we note that using \eqref{conservation} we have
  $$
        -2+2p'^\mu q'_\mu=(p'^\mu+q'^\mu)(p'_\mu+q'_\mu)=(p^\mu+q^\mu)(p_\mu+q_\mu)= -2+2p^\mu q_\mu.
   $$Similarly, also using \eqref{conservation} we have
   $$
        -2-2p'^\mu q
        ^\mu=(p'^\mu-q^\mu)(p'_\mu-q_\mu)=(p^\mu-q'^\mu)(p_\mu-q'_\mu)= -2-2p^\mu q'_\mu.
 $$
 The proof that $p'^\mu p_\mu=q'^\mu q_\mu$ is the same.    This completes these proofs.
\end{proof}

\providecommand{\bysame}{\leavevmode\hbox to3em{\hrulefill}\thinspace}
\providecommand{\href}[2]{#2}

\end{document}